\documentclass[11pt, reqno]{amsart}
\usepackage{amsmath,amssymb,amsthm,mathrsfs,enumerate,bm,xcolor,multirow,pbox}
\usepackage{graphicx,color,framed,tikz,caption,subcaption}
\usepackage[shortlabels]{enumitem}
\usepackage{enumitem}
\setlist{leftmargin=5mm}
\allowdisplaybreaks[4]
\usepackage[colorlinks,linkcolor=red,citecolor=blue,urlcolor=blue]{hyperref}
\allowdisplaybreaks[4]
\numberwithin{equation}{section}
\usepackage{macros}
\usepackage[T1]{fontenc}
\usepackage[latin9]{inputenc}
\usepackage{amsmath}
\usepackage{amsthm}
\usepackage{amssymb}
\usepackage{hyperref}
\usepackage{algorithm,algpseudocode}
\usepackage{lipsum} 
\usepackage{graphicx} 
\usepackage{environ}
\usepackage{booktabs}
\usepackage{cleveref}
\usepackage[toc,page,header]{appendix}
\usepackage{minitoc}
\usepackage{accents}

\crefformat{footnote}{#2\footnotemark[#1]#3}
\newcommand{\Cov}{{\rm Cov}}
\newcommand{\Var}{{\rm Var}}

\newcommand{\wtilde}{\widetilde}

\renewcommand{\H}{\mathcal{H}}

\newcommand{\wbar}{\overline}
\newcommand{\F}{\mathcal{F}}

\newcommand{\noise}{\xi}
\newcommand{\eps}{\epsilon}

\newcommand{\HS}{{\rm HS}}
\newcommand{\obj}{\mathcal{J}}

\newcommand{\grad}{\nabla}
\newcommand{\energy}{\mathcal{E}}
\newcommand{\error}{\mathcal{E}}
\newcommand{\err}{\mathcal{E}}

\newcommand{\normP}[1]{\norm{#1}_{\mathcal{L}_2(\P)}}
\newcommand{\normbigP}[1]{\normbig{#1}_{\mathcal{L}_2(\P)}}
\newcommand{\normPn}[1]{\norm{#1}_{\mathcal{L}_2(\P_n)}}
\newcommand{\normPninfty}[1]{\norm{#1}_{\mathcal{L}_\infty(\P_n)}}
\newcommand{\normbigPn}[1]{\normbig{#1}_{\mathcal{L}_2(\P_n)}}
\newcommand{\KRR}{{\rm KRR}}
\newcommand{\event}{\Lambda}

\usepackage{xr}
\AtBeginDocument{%
	\def\MR#1{}
}

\begin{document}
\doparttoc % Tell to minitoc to generate a toc for the parts
\faketableofcontents % Run a fake tableofcontents command for the partocs

%\part{} % Start the document part
%\parttoc % Insert the document TOC

\title[Taming Nonconvexity in Kernel Feature Selection]{Taming Nonconvexity in Kernel Feature Selection---Favorable Properties 
of the Laplace Kernel}

\author[F. Ruan]{Feng Ruan}

\address[F. Ruan]{
Department of Electrical Engineering and Computer Science, the University of California, Berkeley, CA 94720, USA.
}
\email{fengruan@berkeley.edu}

\author[K. Liu]{Keli Liu}

\address[K. Liu]{
	Voleon Group, Berkeley, CA 94704, USA. 
}
\email{keli.liu25@gmail.com}

\author[M. I. Jordan]{Michael I. Jordan}

\address[M. Jordan]{
Department of Electrical Engineering and Computer Science, the University of California, Berkeley, CA 94720, USA.
}
\email{jordan@cs.berkeley.edu}

%\affil{University of California, Berkeley$^\dag$ \\
%The Voleon Group$^*$
%}

\begin{abstract}
Kernel-based feature selection is an important tool in nonparametric statistics. 
Despite many practical applications of kernel-based feature selection, there is 
little statistical theory available to support the method.  A core challenge is
the objective function of the optimization problems used to define kernel-based
feature selection are nonconvex.  The literature has only studied the statistical 
properties of the \emph{global optima}, which is a mismatch, given that the gradient-based
algorithms available for nonconvex optimization are only able to guarantee convergence
to local minima.  Studying the full landscape associated with kernel-based methods,
we show that feature selection objectives using the Laplace kernel (and other 
$\ell_1$ kernels) come with statistical guarantees that other kernels, including
the ubiquitous Gaussian kernel (or other $\ell_2$ kernels) do not possess.
Based on a sharp characterization of the gradient of the objective function, 
we show that $\ell_1$ kernels eliminate unfavorable stationary points that appear 
when using an $\ell_2$ kernel.  Armed with this insight, we establish statistical 
guarantees for $\ell_1$ kernel-based feature selection which do not require reaching 
the global minima. In particular, we establish model-selection consistency of 
$\ell_1$-kernel-based feature selection in recovering main effects and hierarchical 
interactions in the nonparametric setting with $n \sim \log p$ samples.

\end{abstract}

\maketitle

\section{Introduction}
\indent\indent
Statistical learning problems are often characterized by data sets in which both the
number of data points, $n$, and the number of dimensions, $p$, are large.  Such scaling
is increasingly common in applied problem domains, and it is often accompanied by a focus 
on prediction and flexible nonparametric models in such domains.  Examples of such problem
domains include text classification, object recognition, and genetic 
screening~\cite{LiChWaMoTrRoTaLi17, CaiLuWaYa18, DasRa20}.  Even in such domains, however, 
there is a tension between prediction and interpretation~\cite{Arras17, Rudin19}, and increasingly a
call for ``white-box'' nonparametric modeling, where effective prediction and interpretability 
are both required~\cite{GuidottiMoRuTuGiPe18, MurdochSiKuAbYu19, Miller19}.

\newcommand{\HSIC}{{\rm HSIC}}
\newcommand{\T}{\mathcal{T}}
One general approach to addressing this challenge involves the use of \emph{kernel-based
feature selection}.  Kernel-based methods are nonparametric and yet have mathematical 
structure that can be exploited for interpretability.  In particular algorithms, 
kernel-based feature selection methods have the advantage of being able to find 
reduced-dimensional representations of regression functions, while capturing nonlinear
relationships between the features and response.  Moreover, kernel-based feature selection
methods are expressed as objective functions in an optimization framework, and blend 
appealingly with the modern focus on gradient-based optimization methods for fitting
models.  Two main objectives have become dominant in the literature on kernel-based 
feature selection:
\begin{enumerate}
	\item \emph{Hilbert-Schmidt Independence Criterion} (HSIC). 
This is a nonparametric dependence measure based on the Hilbert-Schmidt norm of 
a covariance operator~\cite{GrettonBoSmSc05}.  This dependence measure can be used
for feature selection in the following way~\cite{SongSmGrBoBe07, SongSmGrBeBo12}.
Let $(X, Y)$ denote the data where $X \in \R^p$ is the feature vector and $Y \in \R$ 
is the response. Let $k(x, x')$ be a positive definite kernel. For any vector 
$x \in \R^p$ and any subset $\T \subseteq \{1, 2, \ldots, p\}$, denote $x^\T \in \R^p$ 
with components $x_i^\T = x_i$ if $i \in \T$ and let $x_i^\T = 0$ if $i \not\in \T$. 
Let  $(X', Y')$ denote an independent copy of $(X, Y)$.  The HSIC-based approach 
to feature selection finds a subset of features by optimizing 
\begin{equation*}
\begin{split}
	\maximize_{\T: \T \subseteq \{1, 2, \ldots, p\}}~~ &\HSIC(\T) \\
	\text{where}~~~~&\HSIC(\T) = \E[Y Y' k(X^\T, (X')^\T)].
\end{split}
\end{equation*}	
Subsequent work studied continuous relaxations of this 
objective~\cite{MasaeliFuDy10, YamadaJiSiXi14}.  Most of the focus in this literature
is, however, computational, and there are currently no general statistical guarantees 
available for the HSIC-based approach. 
\item \emph{Kernel ridge regression} (KRR). In this framework the features are multiplied 
by a set of weights (either discrete or continuous), and the following objective is 
formed~\cite{WestonMuChPoPoVa00, GrandvaletCa02, CaoShSuYaCh07, Allen13, ChenStWaJo17}:
\begin{equation*}
\begin{split}
	\minimize_{\T: \T \subseteq \{1, 2, \ldots, p\}} ~~&\KRR(\T) \\
	\text{where}~~~~&\KRR(\T) = \half \E[(Y-f(X^\T))^2] + \frac{\lambda}{2} \norm{f}_\H^2, 
\end{split}
\end{equation*}
where $\norm{\cdot}_{\H}$ denotes the norm of $\H$, a reproducing kernel Hilbert space. 
This objective is optimized jointly over the weights and the regression function.
\cite{ChenStWaJo17} prove that the \emph{global} optima of the KRR objective are 
feature-selection consistent.  No statistical guarantees are available for continuous 
relaxations of the discrete objective. 
%Among the existent work,
%\cite{ChenStWaJo17} proves that the \emph{global} optima of the discrete KRR objective (based on subset selection) is feature selection consistent. No other statistical guarantees 
%on feature selection consistency are available.
\end{enumerate} 
Both of the discrete and continuous HSIC and KRR objectives are nonconvex. The 
difficulty of analyzing such nonconvex objectives has led to a lack of understanding 
of the statistical properties of the resulting feature-selection algorithms.  Indeed,
for HSIC, the most recent work has been disappointing---it has been shown via counterexamples 
that the \emph{global} optima of the HSIC objective (discrete or continuous) can fail
to select important features and the overall procedure is therefore inconsistent~\cite{LiuRu20}.  
The picture is slightly more favorable for KRR, in that the \emph{global} optima of the 
\emph{discrete} objective is selection consistent; however, this is the lone guarantee 
available in the literature \cite{ChenStWaJo17}.  No other guarantees exist regarding 
the \emph{local optima} or \emph{stationary points} for any \emph{continuous} relaxation
of the KRR objective---yet these relaxations are the most critical to algorithmic 
success in practice. 

% \rfcomment{I do not know how to make the literature review like a story. But here is the actual (and accurate) material for the above paragraph.

% Existent kernel-based approaches operates in one of the following two manners. 
% \begin{itemize}
% \item Define the loss criterion for each subset of features. Optimize the loss criterion over all possible subsets of features.
% 	This leads to a discrete optimization. This is not a real algorithm (computationally intractable).
% \item Multiply each feature by a non-negative weight. Define the loss criterion for each weight. Optimize the 
% 	loss criterion over the weights. This leads to a continuous optimization. Nonconvex in its nature. 
% \end{itemize}
% There are three different approaches to define the loss criterion. 
% \begin{itemize}
% \item The first criterion is based on HSIC. Inconsistent for global minimum.  
% \item The second criterion is based on KRR. Consistent for its global minimum. 
% \item The third criterion is based on CCO. Consistent for its global minimum.
% \end{itemize}
% All existent work only studies the global minimum of the nonconvex/subset-selection objective. Our work studies the 
% stationary point of the objective. %Fundamental gap. 
% }

Our work studies the landscape of the \emph{continuous} KRR objective, most notably
we study \emph{all} of the stationary points (not simply the global optima).  Despite 
the nonconvexity of the objective, we show that, with a carefully designed kernel, 
such \emph{stationary point} have provably benign statistical guarantees. Formally, 
assuming without loss of generality that $\E[Y] = 0$ (an assumption that we make
throughout the paper),\footnote{In general case where $\E[Y] \neq 0$, we need to add an intercept term in the KRR objective.} 
in this paper we consider minimizing the following form of KRR-based objective: \vskip 0.05in
\begin{equation}
\label{eqn:population-objective}
\begin{split}
	\minimize_{\beta: \beta \ge 0, \norm{\beta}_1 \le M}~~ &\obj_\gamma(\beta) \defeq \obj(\beta) + \gamma \norm{\beta}_1 \\
	~~\text{where}~~&\obj(\beta) = \min_{f\in \H}~ \half \E\left[(Y - f(\beta^{1/q} \odot X))^2\right] + \frac{\lambda}{2} \norm{f}_\H^2,
\end{split}
\end{equation}
and where $\lambda, \gamma, M \ge 0$ are regularization parameters. 
Above we use the notation shorthand
$(\beta^{1/q}\odot X) = (\beta^{1/q}X_1, \beta_2^{1/q} X_2, \ldots, \beta_p^{1/q}X_p)$
for a vector $\beta \in \R_+^p$.
We take the reproducing kernel Hilbert space (RKHS) $\H$ to be 
of $\ell_q$ type, where $q = 1, 2$, meaning that the kernel $k$ 
associated with the RKHS $\H$ in the objective takes the form 
$k(x, x') = h(\norm{x-x'}_{q}^q)$, where the notation $\norm{z}_q$ 
refers to the Euclidean $\ell_q$ norm of a vector $z$.  Examples of 
the $\ell_q$ type RKHS include the Gaussian RKHS, where 
$k(x, x') = \exp(-\norm{x-x'}_2^2)$, and the Laplace RKHS, where 
$k(x, x') = \exp(-\norm{x-x'}_1)$.  One of our major findings is 
the choice of the $\ell_1$ kernel (e.g., the Laplace kernel) rather 
than an $\ell_2$ kernel (e.g., the Gaussian) yields significant
improvements to the landscape of the (nonconvex) objective function (both 
the population case and the finite-sample case).  This is suggested
by the following example, which shows how the choice of an $\ell_1$ 
kernel eliminates bad stationary points that would otherwise appear 
for an $\ell_2$ kernel.

\vspace{0.25cm}
\begin{exampleintro}
\label{ex:intro-one}
Consider an additive model where the response $Y$ is the sum of individual 
independent main effects, $f_i^*(X_i)$; i.e., $Y = \sum_{i=1}^p f^*_i(X_i)$, 
where $X_1 \perp X_2 \perp \ldots \perp X_p$. Consider the KRR objective function 
$\mathcal{J} (\beta)$ (see equation~\eqref{eqn:population-objective}). 
We have the following description of the population landscape of the 
KRR objective $\mathcal{J}(\beta)$: %when the kernel is of $\ell_q$ type ($q=1, 2$): 
\begin{itemize}
\item For $q= 1, 2$, the global minimum of $\mathcal{J}(\beta)$ satisfies $\beta_j > 0$ for all $1\le j\le p$.
\item For $q = 1$, any \emph{stationary point} $\beta$ of $\mathcal{J}(\beta)$ satisfies~$\beta_j > 0$ for all $1\le j\le p$.
\item For $q = 2$, $\beta = 0$ is a \emph{stationary point} if $\Cov(f_j^*(X_j), X_j) = 0$ for all $1\le j\le p$. \nolinebreak
\item For $q = 2$, there is a stationary point $\beta$ with $\beta_j = 0$ if $\Cov(f_j^*(X_j), X_j^r) = 0$ for $r = 1, 2$.
\end{itemize}
Under the additive model considered in the example, all the features are important. 
Thus we would like our feature selection algorithm to converge to some 
$\beta$ such that $\beta_j > 0$ for all $j \in [p]$.  Our example shows,
however, that although the global minimum for both $\ell_1$ and $\ell_2$ 
kernels satisfy this desideratum, a gradient-descent algorithm may become 
trapped at a bad stationary point (where $\beta_j = 0$ for some $j$) if one 
uses an $\ell_2$ kernel. This does not occur if one uses an $\ell_1$ kernel. 
\end{exampleintro}

%Proposition~\ref{ex:intro-one} below 
%demonstrates the point, showing how the choice of an $\ell_1$ kernel eliminates bad stationary points that would otherwise appear 
%for an $\ell_2$ kernel.
%\begin{proposition}[Informal Discussion of Proposition~\ref{proposition:landscape-analysis-main-effect}]
%\label{ex:intro-one}
%Consider an additive model where the response $Y$ is precisely the sum of individual independent main effect $f_i^*(X_i)$, i.e.,
%$Y = \sum_{i=1}^p f^*_i(X_i)$ where $X_1 \perp X_2 \perp \ldots \perp X_p$. Consider the population CCO-based kernel objective 
%$\mathcal{J} (\beta)$. Assume necessary regularity conditions. Then: 
%\begin{itemize}
%\item For both $q= 1, 2$, the global minimum of $\mathcal{J}(\beta)$ satisfies $\beta_j > 0$ for all $1\le j\le p$.
%\item When $q = 1$, any stationary point $\beta$ of $\mathcal{J}(\beta)$ satisfies~$\beta_j > 0$ for all $1\le j\le p$.
%\item When $q = 2$, $\beta = 0$ is stationary if $\Cov(f_j^*(X_j), X_j) = 0$ for all $1\le j\le p$.
%\item When $q = 2$, there is a stationary point with $\beta_j = 0$ if $\Cov(f_j^*(X_j), X_j^r) = 0$	for $r = 1, 2$.
%\end{itemize}
%\end{proposition}

The previous example demonstrates the clear advantage of the $\ell_1$ kernel over the $\ell_2$ kernel in the context of an
additive model.  This same advantage in fact holds under more general models.    
We sketch why this is the case---why the $\ell_1$ type RKHS leads to a better 
objective landscape than the $\ell_2$ type RKHS---with formal details to follow 
in subsequent sections.  The key to our result is a sharp characterization of 
the gradient of the KRR objective $\grad_\beta \obj(\beta)$ in the context of
any joint distribution for $(X, Y)$.  
Let $f_\beta$ be the minimum of the KRR in equation~\eqref{eqn:population-objective}, 
and let $r_\beta$ denote the residual, $r_\beta(x, y) = y - f_\beta(x)$. 
Equations~\eqref{eqn:result-of-l-1} and~\eqref{eqn:result-of-l-2} characterize the leading terms 
of the gradient. Letting $\bar{\mu}$ denote a measure implicitly determined solely by the 
kernel \nolinebreak $k$, we have the following characterization of the gradient $\grad_\beta \obj(\beta)$
(below $o_\lambda(1)$ denotes a quantity that tends to $0$
as $\lambda$ tends to $0$):
\begin{itemize}
\item For $q = 1$, the gradient of the objective $\grad \obj(\beta)$ takes 
the form 
	\begin{equation}
	\label{eqn:result-of-l-1}
		\partial_{\beta_l} \obj(\beta) = -\frac{1}{\lambda} 
			\left(\iint \left|\Cov\left(r_\beta(\beta\odot X, Y)e^{i\langle \beta \odot X, \omega \rangle}, e^{i \zeta_l X_l}\right)\right|^2 
				\cdot \frac{d\zeta_l}{\pi\zeta_l^2}\cdot \bar{\mu}(d\omega) + o_\lambda(1)\right).
	\end{equation}
\item For $q = 2$, the gradient of the objective $\grad \obj(\beta)$ takes 
the form 
	\begin{equation}
	\label{eqn:result-of-l-2}
		\partial_{\beta_l} \obj(\beta) = -\frac{1}{\lambda} 
			\left(\int \left|\Cov\left(r_\beta(\beta^{1/2}\odot X, Y)e^{i\langle \beta^{1/2} \odot X, \omega \rangle}, X_l\right)\right|^2 \cdot 
				\bar{\mu}(d\omega) + o_\lambda(1)\right).
	\end{equation}
\end{itemize}
Compare the leading terms of the gradient $\partial_{\beta_l} \obj(\beta)$ in 
equations~\eqref{eqn:result-of-l-1} and~\eqref{eqn:result-of-l-2}.
In the case of $q = 1$, the gradient is a weighted
average of the square of the covariance between a (modified) residual and the exponential function $e^{i \zeta_l X_l}$. Because 
$\{e^{i \zeta_l X_l}\}_{\zeta_l \in \R}$ forms a basis, the gradient with 
respect to $X_l$ captures all functions of $X_l$ that remain in the residual.  
This is in stark contrast to the case of $q = 2$ where the gradient is only able 
to capture signal that is \emph{linear} in $X_l$.  This shows the necessity of 
using an $\ell_1$ kernel in order to capture \emph{nonlinear} signals. Underlying 
the derivations of equations~\eqref{eqn:result-of-l-1} and~\eqref{eqn:result-of-l-2} 
is the development of novel Fourier analytic techniques to \emph{analytically} 
characterize the connections among the solutions of \emph{a family of} kernel ridge regression 
problems indexed by the parameter $\beta$. 

The organization of the paper is as follows. Section~\ref{sec:preliminary-main-text} sets out notation and preliminary details. Section~\ref{sec:distinction-l-1-l-2} 
formalizes the characterization of the gradient for $\ell_1$ and $\ell_2$ kernels 
alluded to above. Using our characterization of the gradient, we show how to 
provide statistical guarantees for kernel feature selection \emph{without} requiring the algorithm to find a global 
minimum. Section~\ref{sec:population-guarantees} gives our first set of results, 
showing that, in the population, the KRR-based objective has the following two 
desirable properties:
\begin{itemize}
\item Any \emph{stationary point} reached by the algorithm excludes noise variables. This applies to both $\ell_1$ and $\ell_2$ kernels. 
\item The algorithm is able to recover main effects and hierarchical interactions as long as the regularization parameters 
	$\lambda, \gamma$ are sufficiently small compared to the signal size. This result applies only to the $\ell_1$ kernel. Our result 
	provides a precise mathematical characterization of signals for which recovery is feasible. 
\end{itemize}

Section~\ref{sec:finite-sample-guarnatees} contains our second set of results 
which translate the population guarantees of Section~\ref{sec:population-guarantees} 
into finite-sample guarantees. We show that with a careful choice of the 
regularization parameters $\lambda, \gamma, M \ge 0$, any \emph{stationary point} of the finite sample KRR objective can achieve (with high probability) precisely the same statistical guarantees as the population version whenever the sample size satisfies $n \gg \log p$. 
The key mathematical result that allows this translation is a high-probability 
concentration statement which shows that the empirical gradient is uniformly close to the 
population gradient when $n \gg \log p$. The derivation of the concentration result 
is non-trivial; it leverages the following ideas: (i) a functional-analytic 
characterization of \emph{a family of} kernel ridge regression problems; (ii) 
Maurey's empirical method to bound the metric entropy; and (iii) large-deviation 
results for the supremum of sub-exponential \nolinebreak processes. The result
is that we are able to provide finite-sample statistical guarantees for the kernel 
feature selection algorithm without requiring the algorithm to reach a global minimum.

\subsection{Notation.}
\label{sec:notation}
\begin{table}[!ht]%\caption{Notation List}
\begin{center}% used the environment to augment the vertical space
% between the caption and the table
\begin{tabular}{r c p{10cm} }
\toprule
$\R$      & $\triangleq$ &the set of real numbers\\
$\R_+$ & $\triangleq$ & the set of nonnegative reals \\
$\C$     & $\triangleq$ &the set of  complex numbers\\
$[p]$ & $\triangleq$ &the set $\{1, 2, \ldots, p\}$ \\
$2^{[p]}$ & $\triangleq$ & the set of all subsets of $[p]$ \\
$\wbar{z}$ & $\triangleq$ & the conjugate of a complex number $z \in \C$ \\
$v_S$ & $\triangleq$ & restriction of a vector $v$ to the index set $S$\\
$\supp(\mu)$ & $\triangleq$ & the support for a measure $\mu$ \\
$\norm{v}_q$ & $\triangleq$ & the $\ell_q$ norm of the vector $v$: $\norm{v}_q = (\sum_i |v_i|^q)^{1/q}$\\  
$\norm{v}_{q, \beta}$ & $\triangleq$ & the weighted $\ell_q$ norm of the vector $v$: \\
	&&$\norm{v}_{q, \beta} = (\sum_i \beta_i |v_i|^q)^{1/q}$\\  
$\mathcal{C}^\infty(\R_+)$  & $\triangleq$ & the set of functions $f$ infinitely differentiable 
				 on $(0, \infty)$ whose derivatives are right continuous at $0$ \\
$f^{(l)}(x)$ & $\triangleq$ & the $l$-th derivative of a function $f$ at $x$ \\
$f^{(l)}(0)$ & $\triangleq$ & $\lim_{x\to 0^+} f^{(l)}(x)$ for any $f \in \mathcal{C}^\infty(\R_+)$ \\
$\mathcal{L}_r(\R^p)$ & $\triangleq$ & the set of all functions such that $\int |f(x)|^r dx < \infty$ \\
$\norm{f}_\infty$ & $\triangleq$ & $\sup_{x \in \R^p} |f(x)|$ the supremum of $|f|$\\
$\norm{Z}_{\mathcal{L}_r(\P)}$ & $\triangleq$ & $(\E[|Z|^r])^{1/r}$ the $L_r$ norm of a random variable $Z$\\
$\mathcal{F}(f)(\omega)$ & $\triangleq$ & Fourier transform of the function $f: \R^p \to \R$: \\
&&	$\mathcal{F}(f)(\omega) = \frac{1}{(2\pi)^{p/2}}\int e^{-i \langle x, \omega\rangle} f(x) dx$. \\
\bottomrule
\end{tabular}
\caption{Notation}
\label{table:notation}
\end{center}
\end{table}

The notation $\P, \E$ are reserved for the population distribution of the data $(X, Y)$, and 
$\what{\P}, \what{\E}$ are reserved for the empirical distribution. 
%For a function $f \in \mathcal{L}_2(\R^p)$, we denote its Fourier transform 
%$\what{f}(\omega) = \frac{1}{(2\pi)^{p/2}}\int e^{-i \langle x, \omega\rangle} f(x) dx$. 
The notation $\H$ stands for the $\ell_q$-type RKHS associated with the kernel function $k(x,x') = h(\norm{x-x'}_q^q)$.
%where $h \in \mathcal{C}^\infty(\R_+) \cap  \mathcal{L}_1(\R_+)$. 
The notation $\mu, Q, q_t, \psi_t, p$ are reserved to denote the measure and functions as they appeared in 
equations~\eqref{eqn:l_1_kernel}--\eqref{eqn:notation-of-p-2}. The notation $\mathbf{1}$ stands for the all $1$ vector
in $\R^p$. %Many other notation can be found in Table~\ref{table:notation}.

The function $f_\beta \in \H$ denotes the minimum of the KRR at population level, i.e.,
\begin{equation}
\label{eqn:solution-KRR-beta-population}
	f_\beta =\argmin_{f\in \H}~ \half \E\left[(Y - f(\beta^{1/q} \odot X))^2\right] + \frac{\lambda}{2} \norm{f}_\H^2,
\end{equation}
and the function $r_\beta(x; y) = y- f_\beta(x)$ denotes the residual at the population level. 

The function $\what{f}_\beta \in \H$ denotes the minimum of the KRR in finite samples, i.e., 
\begin{equation}
\label{eqn:solution-KRR-beta-empirical}
	\what{f}_\beta =\argmin_{f\in \H}~ \half \what{\E}\left[(Y - f(\beta^{1/q} \odot X))^2\right] + \frac{\lambda}{2} \norm{f}_\H^2,
\end{equation}
and the function $\what{r}_\beta(x; y) = y- \what{f}_\beta(x)$ denotes the empirical residual function.

\section{Reproducing Kernel Hilbert Space of $\ell_q$ Type}
\label{sec:preliminary-main-text}
\newcommand{\Lap}{{\rm Lap}}

This section reviews the basic properties of ``$\ell_q$-type'' RKHS $\H$ whose associated reproducing 
kernel $k(x, x')$ takes the form $k(x, x')= h (\norm{x - x'}_q^q)$, where 
$h\in \mathcal{C}^\infty[0, \infty)$ ($\mathcal{C}^\infty[0, \infty)$ denotes the set of functions 
that are infinitely differentiable on $[0, \infty)$, see Table~\ref{table:notation} in Section~\ref{sec:notation}).  
Throughout the paper we focus on the cases $q=1$ and $q = 2$.

\subsection{Characterization of the $\ell_q$-type Positive Definite Kernels.}
Proposition~\ref{proposition:k-pos-definite-representation} identifies all the functions $h$ such that the 
mapping $k: \R^p \times \R^p \mapsto \R$ where $k(x, x')= h (\norm{x - x'}_q^q)$ is a positive 
definite kernel on $\R^p \times \R^p$ for all integer $p \in \N$.

\begin{proposition}
\label{proposition:k-pos-definite-representation}
Let $h \in \mathcal{C}^\infty[0, \infty)$ and $q \in \{1, 2\}$. The following statements are equivalent: 
\begin{itemize}
\item The mapping $k: \R^p \times \R^p \mapsto \R$ where 
	$k(x, x') = h(\norm{x-x'}_q^q)$ is positive definite for all positive integer $p \in \N$.
\item The following representation on $h$ holds: for some nonnegative finite measure $\mu$ on 
	$[0, \infty)$ with $\mu((0, \infty)) > 0$, 
	\begin{equation}
		\label{eqn:l_1_kernel-prop}
		h(x) = \int_0^\infty e^{-t x} \mu(dt).
	\end{equation}
\end{itemize}
\end{proposition}

\begin{proof}
A proof of Proposition~\ref{proposition:k-pos-definite-representation} when $q = 2$ is known as Schoenberg's 
Theorem~\cite[Theorem 7.13]{Wendland04}. For convenience of the reader, we give a proof in 
Appendix~\ref{sec:proof-k-pos-definite-representation}. 
\end{proof}

Throughout the rest of the paper, we assume that $h$ admits the representation as described 
in equation~\eqref{eqn:l_1_kernel-prop}. In particular, the $\ell_q$ kernel $k(x, x')$ admits the following 
representation: 
\begin{equation}
\label{eqn:l_1_kernel}
k(x, x') =  \int_0^\infty e^{-t \norm{x-x'}_q^q} \mu(dt). 
\end{equation}
Equation~\eqref{eqn:l_1_kernel} essentially ssays that the $\ell_q$-type positive definite kernel 
$k(x, x')= h (\norm{x - x'}_q^q)$ can be regarded as a 
weighted average of the kernel $k_{t}(x, x') = \exp(- t \norm{x-x'}_q^q)$ over different scales $t \ge 0$. 
 We will make use of this integral representation throughout the paper.
 
%\paragraph{Convention} 
%Throughout the rest of the paper, we assume $\mu((0, \infty)) > 0$. This eliminates the edge case where $\mu$ is a point mass 
%at $0$ (whose corresponding kernel is $k(x, x') \equiv 1$).

%\rfcomment{Need to think about where to add this $(2\pi)^p$: this affects many proofs in the later playground. It affects the 
%two approximation bounds (surrogate gradients, and its surrogate in the proof)! Please remember on it.}

\subsection{Characterization of the $\ell_q$-type RKHS space.}
Proposition~\ref{proposition:norm-of-H} provides an \emph{analytic} characterization of the 
space and inner product of the $\ell_q$-type RKHS $\H$. The derivation is straightforward using 
the existing theory on RKHS ~\cite{Wendland04, BerlinetTh11, Paulsen16}. Denote the Fourier
transform of a function $t \mapsto f(t)$ to be $\omega \mapsto \mathcal{F}(f)(\omega)$
(see Table~\ref{table:notation} in Section~\ref{sec:notation} for a formal definition of the Fourier 
transform $\mathcal{F}(f)(\omega)$).

\begin{proposition}
\label{proposition:norm-of-H}
Let $\H$ be the $\ell_q$-type RKHS associated with the kernel $k(x, x')$ in equation~\eqref{eqn:l_1_kernel}. 
Assume that $h$ is integrable on $\R_+$. Then $\H$ has the below characterization
\begin{equation}
\label{eqn:characterization-of-the-space}
\H = \left\{f: \text{$f$ is continuous, square integrable on $\R^p$ and}~\int \frac{|\mathcal{F}(f)(\omega)|^2}{Q(\omega)} d\omega < \infty \right\}.
\end{equation}
Above, $Q(\omega) = \int_0^\infty q_t(\omega) \mu(dt)$ where $\omega \mapsto q_t(\omega)$ is defined as follows: 
\begin{equation}
	\label{eqn:notation-of-q-t-p-t}
		q_t(\omega) = \prod_{i=1}^p \psi_t(\omega_i)~~\text{where}~~\psi_t(\omega) = \frac{1}{t} \cdot \psi\left(\frac{\omega}{t}\right),
\end{equation}
with the function $\psi$ defined below, whose definition depends on the choice of $q = 1$ or $q = 2$.
\begin{itemize}
	\item For $q=1$ (i.e., $\H$ is a $\ell_1$-type RKHS), the function $\psi$ is the \emph{Cauchy} \nolinebreak density:
		\begin{equation}
		\label{eqn:notation-of-p-1}
			\psi(\omega) = \frac{1}{\pi(1+\omega^2)}.
		\end{equation}
	\item For $q=2$ (i.e., $\H$ is a $\ell_2$-type RKHS), the function $\psi$ is the \emph{Gaussian} \nolinebreak density:
		\begin{equation}
		\label{eqn:notation-of-p-2}
			\psi(\omega) = \frac{1}{2\sqrt{\pi}} e^{-\omega^2/4}.
		\end{equation}
\end{itemize}
Additionally, we can analytically characterize the inner product $\langle f, g\rangle_\H$ as follows: 
\begin{equation}
	\label{eqn:norm-of-H}
		\langle f, g\rangle_{\H} = \frac{1}{(2\pi)^p} 
			\cdot \int_{\R^p} \frac{\mathcal{F}(f)(\omega)\wbar{\mathcal{F}(g)(\omega)}}{Q(\omega)} d\omega.
\end{equation}
\end{proposition}

\begin{proof}
The proof is standard~\cite{Wendland04, BerlinetTh11, Paulsen16}. The only thing to note is that the 
Fourier transform of the Laplace $\exp(-|t|)$ is the Cauchy density $\frac{1}{\pi (1+\omega^2)}$, and the Fourier 
transform of the Gaussian $\exp(-t^2)$ is Gaussian, $\frac{1}{2\sqrt{\pi}}\exp(- \omega^2/4)$. 
For completeness, we provide a proof in Appendix~\ref{sec:proof-norm-of-H}. 
\end{proof}

%\begin{remark}
%This remark helps build basic intuitions for readers unfamiliar with the topic. 
%Equation~\eqref{eqn:characterization-of-the-space} essentially says that $\H$ is composed of 
%continuous (basically smooth) functions $f$ whose Fourier transform $\omega \mapsto 
%\F(f)(\omega)$ has a sufficiently fast tail decay compared to $\omega \mapsto Q(\omega)$ in a way
%that $\omega \mapsto |\mathcal{F}(f)(\omega)|^2/Q(\omega)$
%is integrable. Consequently, the tail behavior of the weight function $\omega \mapsto Q(\omega)$ 
%completely determines the smoothness properties of the functions in the space $\H$. See Section~\ref{sec:example-of-H}
%for concrete illustrations of Proposition~\ref{proposition:norm-of-H}. % of $\H$ with discussions. 
%\end{remark}

\subsection{Examples of the $\ell_q$-type RKHS $\H$.}
\label{sec:example-of-H}

We give concrete examples to illustrate Proposition~\ref{proposition:k-pos-definite-representation} and Corollary~\ref{corrolary:norm-of-H}. \vskip 0.1in

\begin{example}
Consider the Laplace RKHS $\H$ whose associated kernel is the Laplace function
$k(x, x') = e^{-\norm{x-x}_1}$. The corresponding measure $\mu$ is the atom at $1$.
The norm of the Laplace RKHS $\H$ is $\norm{f}_{\H}^2 = 2^{-p} 
\cdot \int |\mathcal{F}(f)(\omega)|^2 \prod_i (1+\omega_i^2) d\omega$. 
\end{example}

\begin{example}
Consider the Gaussian RKHS $\H$ whose associated kernel is the Gaussian kernel
$k(x, x') = e^{-\norm{x-x}_2^2}$. The measure $\mu$ is the atom at $1$.
The norm of the Gaussian RKHS $\H$ is $\norm{f}_{\H}^2 = \pi^{-p/2} \cdot\int |\mathcal{F}(f)(\omega)|^2 e^{\norm{\omega}_2^2/4} d\omega$. 
\end{example}

%\rfcomment{This is what I added.}
\subsection{Regularity on $\mu$.}
Throughout the paper, we assume the following regularity conditions on the measure $\mu$ to avoid unnecessary technicalities.  
%Let $\supp(\mu)$ be the support of the measure $\mu$. %Denote $m_\mu = \inf \{x: x\in \supp(\mu)\}$ and 
%$M_\mu = \sup\{x: x\in \supp(\mu)\}$ so that $\supp(\mu) \subseteq [m_\mu, M_\mu]$. %Note $m_\mu, M_\mu \in [0, \infty]$.
\begin{assumption}
\label{assumption:mu-compact}
Assume that $\mu$ satisfies (i) $\int_0^\infty \frac{1}{t} \mu(dt) < \infty$, (ii) the support of the measure $\mu$, 
$\supp(\mu)$, is compact when $q=1,2$,
and (iii) $0 \not\in \supp(\mu)$ when $q=2$.
\end{assumption}

\begin{remark}
Assumption (i) is equivalent to the condition that $h$ is integrable on $\R_+$. This assumption is sufficient and necessary to give the 
Fourier-analytic characterization of the RKHS (Proposition~\ref{proposition:norm-of-H}). Assumption (ii) is equivalent to the 
condition that $h(x)$ satisfies an exponential lower bound; i.e., $h(x) \ge c\exp(-Cx)$ for some $c, C > 0$. Assumption (iii)
requires that $h(x)$ satisfies the upper bound $h(x) \le c\exp(-Cx)$ for some $c, C > 0$. 
\end{remark}

Our overarching goal is to document the superiority of the $\ell_1$ kernel over 
the $\ell_2$ kernel. Note that Assumption~\ref{assumption:mu-compact} places very mild 
conditions on the $\ell_1$ kernel and covers a wide range of the $\ell_1$ kernels 
commonly discussed in the literature. 
Examples include the Laplace kernel $k(x, x') = \exp(-\norm{x-x'}_1)$ and the inverse $\ell_1$ kernel: 
$k(x, x') = 1/(\normsmall{x-x'}_1^2 + 1)^{\alpha}$ where $\alpha > 0$. Assumption~\ref{assumption:mu-compact}
places slightly more stringent conditions on the $\ell_2$ kernel. Yet the condition still holds for a broad family of $\ell_2$
kernels which includes the Gaussian kernel $k(x, x') = \exp(-\normsmall{x-x'}_2^2)$ and finite mixtures of Gaussian kernels.

\section{$\ell_1$ versus $\ell_2$ Kernel: Why It Matters} %---A Landscape Analysis of $\obj(\beta)$ 
	%via Analysis of $\grad \obj(\beta)$}
\label{sec:distinction-l-1-l-2}

In this section, we show that one can improve the landscape of the population objective, 
$\obj(\beta)$, by choosing an $\ell_1$ rather than an $\ell_2$ kernel. In particular, 
Section~\ref{sec:landscape-toy} gives a concrete example showing
that using an $\ell_1$ kernel can eliminate bad stationary points and local minima that would 
otherwise appear when using an $\ell_2$ kernel. To understand this phenomenon we develop a novel
characterization of the gradient $\grad \obj(\beta)$ in Section~\ref{sec:analysis-of-gradient}. The 
results in Section~\ref{sec:analysis-of-gradient} are important; they bring to us deep insights regarding 
the precise statistical information that is contained in the gradient. As a demonstration, in 
Section~\ref{sec:proof-landscape-toy} we illustrate 
how to use these insights to obtain a quick proof of the landscape result in Section~\ref{sec:landscape-toy}.

\subsection{The Landscape for $\obj(\beta)$ under the Additive Model.}
\label{sec:landscape-toy}
We illustrate our claim that the $\ell_1$ kernel leads to a more benign landscape than the $\ell_2$ kernel 
using a concrete example. Consider an additive model with the following characteristics:
\begin{itemize}
\item Noiseless additive signal: $Y = \sum_{i}  f_i^*(X_i)$ for functions $f_i^*: \R \mapsto \R$. 
Without loss of generality we take $\E[ f_i^*(X_i)] = 0$ (since $\E[Y] = 0$ by assumption).
\item Independent covariates: $X_1 \perp X_2 \perp \ldots \perp X_p$. 
\end{itemize}
Under this model, Proposition~\ref{proposition:landscape-analysis-main-effect} shows that the landscape of the population objective 
$\obj(\beta)$ exhibits qualitatively different behavior when using an $\ell_1$ versus 
an $\ell_2$ kernel.

\begin{proposition}
\label{proposition:landscape-analysis-main-effect}
Given Assumption~\ref{assumption:mu-compact}, assume $\E[Y^2] < \infty$ and 
$\max_{i \in [p]} \E[X_i^4] < \infty$. % and $\supp(\mu)$ is compact. 
%Assume that $\supp(\mu)$ is 
%bounded away from zero when $q=2$. 
Consider an additive model where $f_i^*(X_i) \neq 0$ for all $i \in [p]$. Write 
$\mathcal{B}_M = \{\beta \in \R_+^p: \norm{\beta}_1 \le M\}$ be the constraint set where  $M < \infty$.
\begin{enumerate}
\item Assume that $\supp(\P_X)$ is compact where $\P_X$ is the distribution of X. 
	For both $q = 1, 2$, there exists $\lambda^* > 0$ such that whenever the 
ridge penalty satisfies $\lambda \le \lambda^*$,
	\begin{equation*}
		\text{the global minimum $\beta$ of $\obj(\beta)$ in $\mathcal{B}_M$ satisfies}~\beta_i > 0~~\text{for all $i \in [p]$}.
	\end{equation*}
\item When $q = 1$, there exists $\lambda^* > 0$ such that whenever the ridge penalty 
	$\lambda \le \lambda^*$, 
	\begin{equation*}
		\text{any stationary point $\beta$ of $\obj(\beta)$ in $\mathcal{B}_M$\footnotemark satisfies}~\beta_i > 0~~\text{for all $i \in [p]$}.
	\end{equation*}
\item When $q = 2$ and if $\Cov(f_i^*(X_i), X_i) = 0$ for all $i \in [p]$, then for 
	all values of $\lambda \ge 0$, 
	\begin{equation*}
		\text{zero is a stationary point of $\obj(\beta)$ in $\mathcal{B}_M$\cref{note1}}.
	\end{equation*}
\item When $q= 2$ and if $\Cov(f_l^*(X_l), X_l) = \Cov(f_l^*(X_l), X_l^2) = 0$ for some $l \in [p]$, then for all values of $\lambda \ge 0$, 
	\begin{equation*}
		\text{there exists a stationary point of $\obj(\beta)$ in $\mathcal{B}_M$\cref{note1} which satisfies $\beta_l = 0$.}
	\end{equation*}
\end{enumerate}
\footnotetext{\label{note1} Consider a minimization problem $\minimize_{\beta \in \mathcal{C}} F(\beta)$ where $F$ is differentiable 
and $\mathcal{C}$ is a convex set. We say that $\beta \in \mathcal{C}$ is a stationary point of $F(\beta)$ in $\mathcal{C}$
 %$\minimize_{\beta \in \mathcal{C}} F(\beta)$ 
if it satisfies $\langle\grad F(\beta), \beta'-\beta\rangle \ge 0$ for any $\beta' \in \mathcal{C}$~\cite{Bertsekas97}.}
\end{proposition}

Under our additive model, we would like to select all 
signal variables, i.e., the algorithm should converge to some $\beta$ where $\beta_i > 0$ for all $i \in [p]$.
Proposition~\ref{proposition:landscape-analysis-main-effect} indicates that the kernel feature selection 
algorithm can achieve this goal if we choose $q = 1$ but not if we choose \nolinebreak $q = 2$.

\begin{itemize}
\item When $q = 2$, Proposition~\ref{proposition:landscape-analysis-main-effect} shows 
	that for sufficiently nonlinear signals (i.e. $\Cov(f_i^*(X_i), X_i) = 0$), $\beta = 0$ is a stationary point. More worryingly, when one adds $\ell_1$ regularization, zero becomes a 
	strict local minimum of $\obj_\gamma(\beta) = \obj(\beta) + \gamma\norm{\beta}_1$, trapping gradient descent in a basin of attraction.
	Note that when no signal exists, zero is also a local
	minimum of $\obj_\gamma(\beta)$. So the landscape of $\obj_\gamma(\beta)$ in a neighborhood around zero is identical whether signal is or isn't present.
	This is bad news for the numerical algorithms.
	%---they are not able to tell whether the signals are there or not!  
%	This is particularly true in high dimension where the initialization of the algorithm should be coordinately wise close to zero. 
\item When $q = 1$, Proposition~\ref{proposition:landscape-analysis-main-effect} shows that we will select all signal variables ($\beta_i > 0$ for all $X_i$), as long as we converge to a stationary point of $\obj(\beta)$. First-order algorithms such as gradient descent can select 
the right variables despite the nonconvexity of the objective.% at least under the additive model. 
\end{itemize}

Although the additive model is contrived, the picture it paints of the landscape of $\obj(\beta)$ under $q=1$
versus $q = 2$ generalizes to other models; see Section~\ref{sec:population-guarantees} for more examples. In particular, choosing $q=2$ can lead to bad local minima/stationary points that would be absent under $q=1$.

\subsection{Analysis of the Gradient $\grad \obj(\beta)$.}
\label{sec:analysis-of-gradient}

This section studies the gradient $\grad \obj(\beta)$ in full generality, where no distributional assumptions
are made about the distribution of $(X, Y)$. Section~\ref{sec:derivation-of-the-gradient} derives a simple 
representation of the gradient $\grad \obj(\beta)$ that serves 
as the foundation for the theoretical study. Section~\ref{sec:a-new-representation-of-grad} expands the 
gradient into the frequency domain using Fourier-analytic tools, which provides insight 
regarding the precise statistical information contained in the gradient $\grad \obj(\beta)$. 
The findings are summarized at the end of the section. 
%Section~\ref{sec:analysis-surrogate-gradient} summarizes our findings. 

\subsubsection{Derivation of $\grad \obj(\beta)$.}
\label{sec:derivation-of-the-gradient}
A simple representation of the gradient $\grad \obj(\beta)$ is crucial for understanding the landscape of the objective 
$\obj(\beta)$. Proposition~\ref{proposition:compute-grad-obj} supplies this. As far as we are aware, this representation of the gradient
$\grad \obj(\beta)$ (see equation~\eqref{eqn:gradient-first-expression}) is new in the literature. 

\begin{proposition}
\label{proposition:compute-grad-obj}
Given Assumption~\ref{assumption:mu-compact}, assume that $\supp(X)$ is compact and 
$\E[Y^2] < \infty$.  \vskip 0.075in
\begin{itemize}
\item  The gradient $\grad \obj(\beta)$ exists for all $\beta \ge 0$.\footnote{The notation 
$\partial_{\beta_l} \obj(\beta)$ is interpreted as the right derivative if $\beta_l = 0$.}
%and is interpreted as the normal partial derivative if $\beta_l > 0$.}. 
\vskip 0.05in %, where $\grad \obj(\beta)$ is defined coordinate-wise by 
%\begin{equation*}
%(\grad \obj(\beta))_l  = \begin{cases}
%		\partial_{\beta_l} \obj(\beta) ~~& \text{if}~\beta_l > 0 \\
%		\partial_{\beta_l}^+ \obj(\beta)~~& \text{if}~\beta_l = 0.
%	\end{cases}
%\end{equation*} 
\item The gradient $\grad \obj(\beta)$ has the following representation. Let  
$(X', Y')$ denote an independent copy of $(X, Y)$. We have for each coordinate $l \in [p]$, 
\begin{equation}
\label{eqn:gradient-first-expression}
(\grad \obj(\beta))_l = -\frac{1}{\lambda} \cdot \E\left[r_\beta(\beta^{1/q}\odot X; Y)r_\beta(\beta^{1/q}\odot X'; Y') 
	h'(\normsmall{X-X'}_{q, \beta}^q)|X_l- X_l'|^q\right].
\end{equation}
\end{itemize}
\end{proposition}

\begin{remark}
The very simple gradient representation~\eqref{eqn:gradient-first-expression} 
supplies the basis for all the rest of the analysis in the paper. However, 
a rigorous derivation of the representation~\eqref{eqn:gradient-first-expression} 
is indeed challenging for the following two reasons. (i) It is perhaps challenging to see 
why intuitively this representation~\eqref{eqn:gradient-first-expression}
should hold---note especially that both the objective $\obj(\beta)$ on the LHS and the residual term $r_\beta$ 
on the RHS of equation~\eqref{eqn:gradient-first-expression} are defined 
in an implicit manner (recall $r_\beta(x, y) = y- f_\beta(x)$ where
$f_\beta$ is defined implicitly as the solution of the kernel ridge 
regression~\eqref{eqn:solution-KRR-beta-population}). (ii) A rigorous derivation 
of equation~\eqref{eqn:gradient-first-expression} requires establishing 
improved smoothness properties of the solution $f_\beta$ (see the mid of the 
heuristic proof below), 
which requires additional analytic techniques from harmonic analysis~\cite{Grafakos08}. 
\end{remark}

\subsubsection{A Fourier-analytic View of $\grad \obj(\beta)$.}
\label{sec:Fourier-analysis-of-Gradient}
This section presents a novel Fourier-analytic technique for analyzing the gradient $\grad \obj(\beta)$.
The analysis brings new insights---examples include recovery of the formula of the gradient 
\eqref{eqn:result-of-l-1} and~\eqref{eqn:result-of-l-2} in the introduction, and the landscape result 
described in Proposition~\ref{proposition:landscape-analysis-main-effect}---allowing us to see 
why choosing $q=1$ vs. $q=2$ leads to qualitatively different results in a transparent manner. 
%For instance, we show in the later Section \ref{sec:analysis-surrogate-gradient} how these analytical insights 
%allow us to conclude easily the landscape result in Proposition~\ref{proposition:landscape-analysis-main-effect}.
%Our analysis brings to us statistical insights why the gradient when $q= 1$ con capture nonlinear 
%statistical signals 

At a high level, our analysis on the gradient $\grad \obj(\beta)$ is based on the following three steps: 
\begin{itemize}
\item Represent $\grad \obj(\beta)$ in frequency domain using Fourier expansion of kernel functions.
\item Construct a surrogate gradient $\wtilde{\grad \obj(\beta)} \approx \grad \obj(\beta)$ that is amenable to Fourier analysis.
\item Gain insights into the true gradient $\grad \obj(\beta)$ by analyzing the surrogate gradient $\wtilde{\grad \obj(\beta)}$.
\end{itemize}
Each of the three steps is discussed in a separate paragraph below. 
%In Section \ref{sec:analysis-surrogate-gradient}, we show how these insights lead 
%jointly to Proposition~\ref{proposition:landscape-analysis-main-effect}.

\vspace{0.2in}

\paragraph{\emph{A Frequency-domain Representation of $\grad \obj(\beta)$.}}
\label{sec:a-new-representation-of-grad}

Lemma~\ref{lemma:new-representation-of-grad-obj} expands $\grad \obj(\beta)$ in the 
frequency domain $\omega$. The idea is to expand the negative kernel $(x, x') \mapsto 
h^\prime(\norm{x-x'}_q^q)$ that appears on the RHS of equation~\eqref{eqn:gradient-first-expression}. 
The proof is deferred to Appendix~\ref{sec:proof-lemma-new-representation-of-grad-obj}.

%expands the gradient expression in 
%equation~\eqref{eqn:gradient-first-expression} in the frequency domain. Indeed equation~\eqref{eqn:new-representation-of-grad-obj} shows 
%that one can reconstruct the gradient 
%, 
%equation~\eqref{eqn:new-representation-of-grad-obj} shows that one can reconstruct the gradient 
%$\partial_{\beta_l} \obj(\beta)$ by first evaluating $\E[R_{\beta, \omega}(\beta^{1/q} \odot X; Y) \wbar{R_{\beta, \omega}(\beta^{1/q} \odot X'; Y')} |X_l - X_l'|]$ 
%for each frequency $\omega$, and then do a weighted integral over different frequency $\omega$ where the weight is precisely given by 
%$\wtilde{Q}(\omega)$. 
\begin{lemma}
\label{lemma:new-representation-of-grad-obj}
Assume Assumption~\ref{assumption:mu-compact}, $\max_{l \in [p]}\E[X_l^{4}] < \infty$ and $\E[Y^2] < \infty$.  
Then
\begin{equation}
\label{eqn:new-representation-of-grad-obj}
\partial_{\beta_l} \obj(\beta) =
	 \frac{1}{\lambda} \cdot  \int \E[R_{\beta, \omega}(\beta^{1/q} \odot X; Y) \wbar{R_{\beta, \omega}(\beta^{1/q} \odot X'; Y')} |X_l - X_l'|^q] 
	 	\cdot \wtilde{Q}(\omega) d\omega,
\end{equation}
where %$\wtilde{Q}(\omega): \R_+^p \mapsto \R_+$ and $R_{\beta, \omega}(\beta^{1/q} \odot X; Y)$ are defined by% the integral % (one can check the RHS is almost everywhere well defined)
\begin{equation}
\label{eqn:tilde-Q-R-beta-omega-definition}
	\wtilde{Q}(\omega) \defeq \int_0^\infty tq_t(\omega)\mu(dt)~~\text{and}~~
	R_{\beta, \omega}(\beta^{1/q} \odot X; Y) \defeq 
		e^{i \langle \omega, \beta^{1/q} \odot X\rangle}  r_\beta(\beta^{1/q} \odot X; Y).
\end{equation}
\end{lemma}\noindent\noindent
The Fourier expansion in equation~\eqref{eqn:new-representation-of-grad-obj} suggests to understand the gradient 
$\partial_{\beta_l} \obj(\beta)$ by studying the term
$\E[R_{\beta, \omega}(\beta^{1/q} \odot X; Y) \wbar{R_{\beta, \omega}(\beta^{1/q} \odot X'; Y')} |X_l - X_l'|^q]$ inside the integral.

\vspace{0.2in}

\paragraph{\emph{Define the surrogate gradient $\wtilde{\grad \obj(\beta)} \approx \grad \obj(\beta)$.}}
In order to understand the term 
$\E[R_{\beta, \omega}(\beta^{1/q} \odot X; Y) \wbar{R_{\beta, \omega}(\beta^{1/q} \odot X'; Y')} 
|X_l - X_l'|^q]$, we perform an additional Fourier expansion. In particular, in the case $q = 1$, we use the 
Fourier expansion of the conditionally negative definite kernel $(x, x') \to |x-x'|$ that holds for 
any function $p$ satisfying $\int p(x)dx = 0$: 
\begin{equation}
\label{eqn:absolute-cond-psd}
	\iint p(x) p(x')|x-x'| dxdx' = -\int \Big|\int p(x) e^{i\omega x}dx\Big|^2 \cdot \frac{d\omega}{\pi \omega^2}.
\end{equation}
Unfortunately, we can't directly apply formula~\eqref{eqn:absolute-cond-psd} to our analysis
%$\E[R_{\beta, \omega}(\beta^{1/q} \odot X; Y) \wbar{R_{\beta, \omega}(\beta^{1/q} \odot X'; Y')} |X_l - X_l'|^q]$ 
since $R_{\beta, \omega}(\beta^{1/q} \odot X; Y)$ has mean close to zero but not equal to \emph{exactly} 
zero.
%
%can't be directly applied % to the target 
%$\E[R_{\beta, \omega}(\beta^{1/q} \odot X; Y) \wbar{R_{\beta, \omega}(\beta^{1/q} \odot X'; Y')} |X_l - X_l'|^q]$ 
%$\E[R_{\beta, \omega}(\beta^{1/q} \odot X; Y)] \neq 0$. 

To overcome this technical issue---allowing further use of Fourier expansion---we construct a surrogate 
gradient, where in 
equation~\eqref{eqn:new-representation-of-grad-obj} we replace 
$R_{\beta, \omega}(\beta^{1/q} \odot X; Y)$ by its mean-corrected counterpart 
$\wbar{R}_{\beta, \omega}(\beta^{1/q} \odot X; Y)= R_{\beta, \omega} (\beta^{1/q} \odot X; Y)- \E[R_{\beta, \omega}(\beta^{1/q} \odot X; Y)]$:
\begin{equation}
\label{eqn:def-wtilde-partial-beta-j}
	\wtilde{\partial_{\beta_l} \obj(\beta)} \defeq
		 -\frac{1}{\lambda} \cdot  \int \E[\wbar{R}_{\beta, \omega}(\beta^{1/q} \odot X; Y) \wbar{\wbar{R}_{\beta, \omega}(\beta^{1/q} \odot X'; Y')} 
		 	|X_l - X_l'|^q] \wtilde{Q}(\omega) d\omega,
\end{equation}

Lemma~\ref{lemma:approximate-gradient-bound} bounds the difference between the gradient 
$\partial_{\beta_l} \obj(\beta)$ and the surrogate $\wtilde{\partial_{\beta_l} \obj(\beta)}$.
\begin{lemma}
\label{lemma:approximate-gradient-bound}
Assume $\E[X_l^{4}] < \infty$, $\E[Y^2] < \infty$ and $\supp(\mu) \subseteq [0, M_\mu]$ for $M_\mu < \infty$. 
There exists $C > 0$ depending only on $\E[X_l^4]$, $\E[Y^2]$, and $M_\mu$ such that 
for any $\beta \ge 0$ and $l \in [p]$: %for  $\beta \ge 0$:  
\begin{equation}
\label{eqn:deviation-between-wtilde-partial-true-partial}
\left|\wtilde{\partial_{\beta_l} \obj(\beta)} - \partial_{\beta_l} \obj(\beta)\right| \le C\cdot 
	(1+\frac{1}{\sqrt{\lambda}}).
\end{equation}
\end{lemma}

%\rfcomment{Below is not the covariance.}
\begin{remark}
We are interested in the case where the ridge penalty is small: $\lambda \ll 1$. In this regime, 
$\wtilde{\partial_{\beta_l} \obj(\beta)} \approx  \partial_{\beta_l} \obj(\beta)$ (the error bound in 
equation~\eqref{eqn:deviation-between-wtilde-partial-true-partial} is of the order
$1/\sqrt{\lambda}$, while the gradient is of the order $1/\lambda$): mean correction 
has a negligible effect.
%(notice that both gradients 
%$\wtilde{\partial_{\beta_l} \obj(\beta)}$ and $\partial_{\beta_l} \obj(\beta)$ are on the order of $\frac{1}{\lambda}$ by 
%equations~\eqref{eqn:new-representation-of-grad-obj} and \eqref{eqn:def-wtilde-partial-beta-j}, while the 
%difference of them is at most $\frac{1}{\sqrt{\lambda}}$ by equation~\eqref{eqn:deviation-between-wtilde-partial-true-partial}). 
The mean we remove, $\E[R_{\beta, \omega}(\beta^{1/q} \odot X; Y)]$, is the covariance 
between $r_\beta(\beta^{1/q}\odot X; Y)$ and the complex exponential basis
$e^{i \langle \beta^{1/q} \odot \omega, X\rangle}$. Since 
$r_\beta(\beta^{1/q}\odot X; Y)$ is the residual from a nonparametric ridge regression, 
it should be approximately uncorrelated with any basis (when $\lambda$ is small). 
%Thus 
%$\wbar{R}_{\beta, \omega}(\beta^{1/q} \odot X; Y) \approx R_{\beta, \omega}(\beta^{1/q} \odot X; Y)$ 
%and $\wtilde{\partial_{\beta_l} \obj(\beta)} \approx  \partial_{\beta_l} \obj(\beta)$.
\end{remark}

By construction, the surrogate gradient $\wtilde{\grad \obj(\beta)}$ admits a further 
Fourier-type expansion: 
\begin{itemize}
\item In the case where $q = 1$, the surrogate gradient $\wtilde{\partial_{\beta_l} \obj(\beta)}$ has the expansion: 
\begin{equation}
\label{eqn:wtilde-partial-beta-j-rep}
	\wtilde{\partial_{\beta_l} \obj(\beta)} =
		- \frac{1}{\lambda} \cdot  \int \left(\int \Cov^2 \Big({R}_{\beta, \omega}(\beta \odot X; Y), e^{i \zeta_l X_l }\Big)
				\cdot \frac{d\zeta_l}{\pi\zeta_l^2} \right)  \cdot \wtilde{Q}(\omega) d\omega.
\end{equation}
\item In the case where $q = 2$, the surrogate gradient $\wtilde{\partial_{\beta_l} \obj(\beta)}$ has the expansion: 
\begin{equation}
\label{eqn:wtilde-partial-beta-j-rep-2}
	\wtilde{\partial_{\beta_l} \obj(\beta)} =
		- \frac{1}{\lambda} \cdot  \int \Cov^2 \left({R}_{\beta, \omega}(\beta^{1/2} \odot X; Y), X_l\right)
				 \cdot	\wtilde{Q}(\omega) d\omega.
\end{equation}
\end{itemize}
Note that equation~\eqref{eqn:wtilde-partial-beta-j-rep} follows by applying the Fourier expansion of the conditional 
negative definite kernel $(x, x') \mapsto |x-x'|$ (equation~\eqref{eqn:absolute-cond-psd}) to the surrogate 
gradient (equation~\eqref{eqn:def-wtilde-partial-beta-j}). 
Equation~\eqref{eqn:wtilde-partial-beta-j-rep-2} follows from elementary algebraic manipulations. 

%As a consequence of equation~\eqref{eqn:wtilde-partial-beta-j-rep} and~\eqref{eqn:wtilde-partial-beta-j-rep-2} and 
%Lemma~\ref{lemma:approximate-gradient-bound}, we recover the foundational formula on the true gradient 
%$\grad \obj(\beta)$, i.e., equation~\eqref{eqn:result-of-l-1} and~\eqref{eqn:result-of-l-2} stated in the introduction.

\paragraph{\emph{Statistical insights on the true gradient $\grad \obj(\beta)$.}}
%\label{sec:analysis-surrogate-gradient}
Using Lemma~\ref{lemma:approximate-gradient-bound}
and formula~\eqref{eqn:wtilde-partial-beta-j-rep} and~\eqref{eqn:wtilde-partial-beta-j-rep-2}, we 
immediately recover the formula of the true gradient $\grad \obj(\beta)$ stated in equation~\eqref{eqn:result-of-l-1}
and~\eqref{eqn:result-of-l-2} in the introduction.

\begin{proposition}
\label{prop:statistical-insights}
Assume Assumption~\ref{assumption:mu-compact}, $\max_{l \in [p]}\E[X_l^{4}] < \infty$ and $\E[Y^2] < \infty$. 
Let $M_\mu < \infty$ be such that $\supp(\mu) \subseteq [0, M_\mu]$.
\begin{itemize}
\item In the case where $q = 1$, the gradient of the objective $\grad \obj(\beta)$ takes the form:
	\begin{equation*}
	%\label{eqn:partial-beta-j-rep}
		\partial_{\beta_l} \obj(\beta) = -\frac{1}{\lambda} 
			\left(\iint \Cov^2 \left({R}_{\beta, \omega}(\beta \odot X; Y), e^{i \zeta_l X_l }\right)
				\cdot \frac{d\zeta_l}{\pi\zeta_l^2} \cdot \wtilde{Q}(\omega) d\omega + O(\sqrt{\lambda})\right).
	\end{equation*}
\item In the case where $q = 2$, the gradient of the objective $\grad \obj(\beta)$ takes the form: 
	\begin{equation*}
	%\label{eqn:partial-beta-j-rep-2}
		\partial_{\beta_l} \obj(\beta) = -\frac{1}{\lambda} 
			\left(\int \Cov^2\left({R}_{\beta, \omega}(\beta^{1/2} \odot X; Y), X_l\right) \cdot 
				\wtilde{Q}(\omega) d\omega + O(\sqrt{\lambda})\right).
	\end{equation*}
\end{itemize}
In this case the notation $O(\sqrt{\lambda})$ refers to a remainder term whose 
absolute value is upper bounded by $C\sqrt{\lambda}$, where $C > 0$ is a constant 
depending only on $\E[X_l^{4}], \E[Y^2]$ and $M_\mu$. 
\end{proposition}

\begin{remark}
As discussed in the introduction, a comparison of the leading term in the gradient $\grad \obj(\beta)$ shows that the $\ell_1$
kernel can capture all types of \emph{nonlinear} signal in $X_l$, while the $\ell_2$ kernel can only capture a \emph{linear} signal.
\end{remark}

\subsection{A Proof Sketch of Proposition~\ref{proposition:landscape-analysis-main-effect}.}
\label{sec:proof-landscape-toy}
Based on the gradient characterization in Proposition~\ref{prop:statistical-insights}, we present a quick 
and informal proof of Part $(ii)$ and $(iii)$ of Proposition~\ref{proposition:landscape-analysis-main-effect}, 
which shows that the choice of $q$ \emph{impacts} the landscape of the objective $\obj(\beta)$ 
(i.e., the distribution of the stationary points). The sketch should clarify 
the basic intuition. For a rigorous treatment as well as the proof of the other 
two landscape results, $(i)$ and $(iv)$, see 
Section~\ref{sec:proof-proposition-landscape-analysis-main-effect}.
 
\begin{itemize}
\item Consider the case where $q = 1$. 
	We show that $\wtilde{\grad \obj(\beta)} < 0$ at any $\beta$ with $\beta_l = 0$.
	Suppose on the contrary that $\wtilde{\grad \obj(\beta)} = 0$ at $\beta_l = 0$. This implies that for all $\zeta_l$, $\omega$,
	\begin{equation*}
		0 =  \Cov \left({R}_{\beta, \omega}(\beta \odot X; Y), e^{i \zeta_l X_l }\right)
			= \Cov\left(Ye^{i \langle \omega, \beta\odot X\rangle}, e^{i\zeta_l X_l}\right).
	\end{equation*}
	In particular, $\Cov(Y, e^{i\zeta_l X_l}) = 0$ for all $\zeta_l$. This creates a contradiction since $e^{i\zeta_l X_l}$ 
	forms a basis and $\E[Y| X_l] = f_l^*(X_l)$ where $f_l^*(X_l) \neq 0$.
	Hence, $\wtilde{\grad \obj(\beta)} < 0$ if $\beta_l = 0$. 
 	Since $\wtilde{\grad \obj(\beta)}$ is the leading term of the true gradient,
	$\grad \obj(\beta)$, by Proposition~\ref{prop:statistical-insights},  this suggests that $\grad \obj(\beta) < 0$ 
	at all $\beta$ with $\beta_l = 0$ for small enough $\lambda$.
\item Consider the case where $q = 2$. Assume that $\Cov(f_l^*(X_l), X_l) = 0$. Then $\Cov(Y, X_l) = 0$ for all %the 
	variables $X_l$. At $\beta = 0$, we have 
	$
		\Cov\left({R}_{\beta, \omega}(\beta^{1/2} \odot X; Y), X_l\right) = \Cov(Y, X_l) = 0$.
	This shows that $\wtilde{\grad \obj(\beta)} = 0$ at $\beta = 0$. Note then $\wtilde{\partial_{\beta_l} \obj}(0) = \partial_{\beta_l}\obj(0)$ 
	holds since $\E[R_{0, \omega}(0; Y)] = \E[Y] = 0$. This shows that zero is a stationary point of $\obj(\beta)$. 
\end{itemize}

\newcommand{\regobj}{\mathcal{J}_\gamma}
\newcommand{\proj}{\Pi}
\newcommand{\B}{\mathcal{B}}
\section{Population-Level Guarantees}
\label{sec:population-guarantees}

This section describes the statistical properties of the kernel feature selection algorithm 
(see Alg.~\ref{alg:kernel-feature-selection}) at the \emph{population} level. 
None of our results require finding the global 
minimum of the kernel feature selection objective. We only require the algorithm to find a stationary point of the objective (easily achievable 
by using projected gradient descent with a sufficiently small stepsize). The fact that our theoretical results apply to any stationary point and not simply the global minimum separates our work 
from existing work on kernel feature selection. 

Let $\beta$ denote the \emph{stationary point}
found by projected gradient descent in Alg.~\ref{alg:kernel-feature-selection}). We want to know when
$\beta$ has the following two properties: 
\begin{itemize}
\item No False Positives: $\beta_{S^c} = 0$, i.e., the algorithm excludes all the noise variables $X_{S^c}$.
\item Fully Recovery: $\beta_S > 0$, i.e., the algorithm detects all the signal variables $X_S$. 
\end{itemize}

\begin{algorithm}[!tph]
\begin{algorithmic}[1]
%\label{alg:metric_hier}
\Require{Initializer $\beta^{(0)}$, stepsize $\alpha$, feature matrix $\mathbf{X} \in \mathbb{R}^{n\times p}$ and response $y \in \R^n$}
%\While{$\hat S$ not converged}
%\State{Initialize $\beta_j = \tau$ for $j\in \hat S$ and $\beta_j=0$ for $j \notin \hat S$}
\State{Run projected gradient descent (with stepsize $\alpha$, and initialization $\beta^{(0)}$) to solve %the optimization
%Find any stationary point $\beta$ of the problem 
\begin{equation*}
\begin{split}
	&\minimize_{\beta \in \mathcal{B}_M}  \obj_{\gamma}(\beta), \\
		~~\text{where}~~&
			 \mathcal{B}_{M} = \left\{\beta\in \R_+^p;  \norm{\beta}_1 \le M\right\}.
\end{split}
\end{equation*}
Denote the projected gradient descent iterates as $\{\beta^{(k)}\}_{k \in \N}$ where 
\begin{equation}
\label{eqn:pgd-kernel}
	\beta^{(k+1)} = \proj_{\B_M}(\beta^{(k)} - \alpha \grad \regobj(\beta^{(k)})).
\end{equation}
}
\State{Return $\hat{S} = \supp(\beta)$ where $\beta$ is any accumulation point of the iterates $\{\beta^{(k)}\}_{k \in \N}$. }
%\EndWhile
\end{algorithmic}
\caption{Kernel Feature Selection Algorithm}
\label{alg:kernel-feature-selection}
\end{algorithm}

\paragraph{Roadmap}
The rest of Section~\ref{sec:population-guarantees} is organized as follows. 
\begin{itemize}
\item Section~\ref{sec:population-setup} sets up the problem, supplying the definitions 
of the signal variables $X_S$ and the noise variables $X_{S^c}$. 
\item Section~\ref{sec:no-false-positive-population} shows that the algorithm excludes all
noise variables, i.e, $\beta_{S^c} = 0$. The ability to exclude noise variables does not rely on 
the type of kernel we use---both $\ell_1$ and $\ell_2$ kernels achieve this goal.  
\item Section~\ref{sec:power-guarantee-population} shows that the algorithm recovers all main effect signals and hierarchical interaction signals.
The recovery result requires the use of an $\ell_1$ kernel. As we have discussed in 
Section~\ref{sec:distinction-l-1-l-2}, using an $\ell_2$ kernel leads to an objective landscape with bad stationary points.
\end{itemize}

\subsection{Problem setup}
\label{sec:population-setup}
We assume the following relationship for $(X, Y)$: 
\begin{equation}
\label{eqn:model-assumption}
	Y = f^*(X_S) + \noise \qquad \E[\noise | X] = 0.
\end{equation}
We define the regression function $f^*$ to be any function satisfying $f^*(X_S) = \E[Y | X]$.
Equation~\eqref{eqn:model-assumption} says that the signal, $\E[Y | X]$, depends only on a small set of variables $X_S$. 
Since the components of $X$ can be dependent, there may be multiple ways to write equation~\eqref{eqn:model-assumption} using different sets $S$. To pin down a unique signal set $S$, we employ the following definition:
\begin{definition}[Signal Set $S$]
\label{definition:signal-set}
The signal set $S$ is defined as the unique minimal subset $S \subseteq [p]$ such that the following two 
conditions holds: 
\begin{itemize}
\item $\E[Y | X] = \E[Y | X_S]$, i.e., the signal $X_S$ has the full predictive power of $Y$ given $X$. % the relationship between $X$ and $Y$.
\item $X_S \perp X_{S^c}$, i.e., the noise variables are completely independent of the signal variables. 
\end{itemize}
Appendix~\ref{sec:definition-of-the-signal-set} shows that Definition~\ref{definition:signal-set} is proper and is satisfied by a unique set $S$. 

\end{definition}

\begin{remark}
There are two lines of research in the theoretical literature that provide justification
for our assumption of independence between the signal $X_S$ and the noise $X_{S^c}$:
for two reasons:
\begin{itemize}
\item There is a standard treatment in the literature which assumes that the distribution 
of $X$ is known exactly~\cite{CommingesDa12, CandesFaJaLv16}. This assumption implies the
condition $X_S \perp X_{S^c}$, which can be seen as follows. Using the distribution of $X$, 
we can reweight the data so that effectively the distribution of $X$ is uniform on 
$[0, 1]^d$ (see~\cite{CommingesDa12}). In that case, all variables are independent, 
hence $X_S \perp X_{S^c}$.	 
\item The requirement $X_S \perp X_{S^c}$ is useful for obtaining a result on false discoveries (Section~\ref{sec:no-false-positive-population}). Without this assumption, we can still obtain the recovery result for main effects and hierarchical interactions presented in Section~\ref{sec:power-guarantee-population}. Consider an example where $Y = g(X_1) + \noise$ and $X_2$ is highly correlated with $X_1$. Ideally, we'd select only $X_1$ but there may be stationary points of the kernel selection objective for which $\beta_2 > 0$. Since, we have no control over which stationary point gradient descent converges to, we can only guarantee that $\E[Y| X_{\hat S}] = \E[Y | X]$ but not that $\hat S$ is in any way minimal.
	
\end{itemize}
\end{remark}

\subsection{No-false-positive guarantee: $\beta_{S^c} = 0$}
\label{sec:no-false-positive-population}
Theorem~\ref{thm:no-false-positive} shows that, if initialized at $\beta^{(0)} = 0$, the kernel feature selection algorithm (Alg.~\ref{alg:kernel-feature-selection}) does not 
select any noise variables. To establish 
Theorem~\ref{thm:no-false-positive}, we need a mild regularity condition on the moments of $X$ and $Y$. The proof of 
Theorem~\ref{thm:no-false-positive} is simple and is given in Section~\ref{sec:thm-proof-no-false-positive} in the main text. 

\begin{assumption}
\label{assumption:X-Y-bound}
There exist $M_X, M_Y< \infty$ so that $\max_{l \in [p]} \E[X_l^4] \le M_X^4$ and $\E[Y^2] \le M_Y^2$.
\end{assumption}

\begin{theorem}
\label{thm:no-false-positive}
Given Assumptions~\ref{assumption:mu-compact} and~\ref{assumption:X-Y-bound}, consider 
the projected gradient descent algorithm in 
equation~\eqref{eqn:pgd-kernel}. Assume that the algorithm is initialized at $\beta^{(0)} = 0$. 
Then any accumulation point $\beta^*$ of the iterates $\{\beta^{(k)}\}_{k \in \N}$ must satisfy $\beta^*_{S^c} = 0$.
\end{theorem}

\begin{remark}
The reason why Theorem~\ref{thm:no-false-positive} holds is that the gradient of the objective 
with respect to any noise variable $\beta_l$, where $l \not\in S$, is positive at any 
$\beta$ where $\beta_{S^c} = 0$
(see Lemma~\ref{lemma:grad-noise-variable-beta-noise-0}). Thus the coordinate of any noise variable 
can't increase due to gradient-descent dynamics. In particular, all the iterates of the gradient dynamics exclude 
the noise variables, i.e., $\beta_l^{(k)} = 0$ for all $k \in \N$ and $l \not\in S$.
\end{remark}

%Theorem~\ref{thm:no-false-positive} considers a more general situation than Theorem~\ref{thm:no-false-positive} where 
%we allow the kernel dimension reduction algorithm to initialize at $\beta^{(0)} = \nu \mathbf{1}$ for some $\nu \ge 0$. To achieve the 
%same no false positive guarantee, the algorithm needs to set a threshold $\gamma \gtrsim \frac{\nu}{\lambda^2}$. 

%\begin{corollary}
%\label{cor:no-false-positive-general}
%\end{corollary}

%\begin{proof}
%The proof is a simple combination of that of Theorem~\ref{thm:no-false-positive} and Proposition~\ref{prop:grad-obj-lipschitz-beta}. 
%We leave the details to the readers. 
%\end{proof}

\newcommand{\effect}{\mathcal{E}}
\newcommand{\G}{\mathcal{G}}
\subsection{Power guarantees: $\beta_S > 0$}
\label{sec:power-guarantee-population}
In this subsection, we focus on the ability of kernel feature selection to recover 
signal variables. 
The recovery guarantees in this section apply to $\ell_1$ kernels but not $\ell_2$ kernels. 
%; we will need an additional mild assumption on the $\ell_1$ kernel, see Section~\ref{sec:choice-of-ell-1-kernel}. 
As discussed in Section \ref{sec:distinction-l-1-l-2}, the objective landscape under an $\ell_2$ kernel has bad stationary points unless the signals are linear.

Aside from the type of kernel we choose, the power of the algorithm also depends on the type of
signals we are trying to recover. Below, we analyze the power of the kernel feature selection algorithm under 
a classical functional ANOVA model~\cite{FriedmanHaTi01}, which we review in Section~\ref{sec:functional-ANOVA}. 
We provide recovery guarantees for two stylized types of signals---main effect signals
(Section~\ref{sec:main-effect}) and hierarchical interaction signals (Section~\ref{sec:hierarchical-effect}).
For each of these signal types, we give the precise mathematical condition under which the 
population algorithm achieves full recovery. The mathematical condition is stated in the form of an \emph{effective signal size} (appropriately defined) exceeding a \emph{threshold}.

\subsubsection{Functional ANOVA model}
\label{sec:functional-ANOVA}
The remainder of Section~\ref{sec:power-guarantee-population} assumes the following 
functional ANOVA model~\cite{FisherMa23, Stein87, FriedmanHaTi01}: 
\begin{itemize}
\item The signal admits the functional ANOVA decomposition: 
	\begin{equation*}
		f^*(X_S) = \sum_{A \subseteq S} f^*_A(X_A),
	\end{equation*}
	where the function $f_A^*: \R^{|A|} \to \R$ satisfies the mean-zero 
	condition $\E[f^*_A(X_A)] = 0$ and the orthogonality condition, 
	$\E\left[f^*_{A}(X_{A})| X_{A'}\right] = 0$, holds for any set $A'$ 
	that does not contain $A$.
\item Independent covariates: $X_{l_1} \perp X_{l_2} \perp \ldots \perp X_{l_{|S|}}$, where $S = \{l_1, l_2, \ldots, l_{|S|}\}$. 
\end{itemize}

The functional ANOVA model is simple and interpretable. The term $f^*_A(X_A)$ captures the interaction between the variables in the set $A$.

\vspace{0.3cm}
\begin{remark}
The assumption of independence between variables in the signal set is not strictly necessary. We use this assumption in the main text because it gives the cleanest result and its proof is the most insightful for understanding the algorithm. In Appendix~\ref{sec:proof-under-dependent-covariates}, we discuss the recovery of signal variables without this independence assumption, and provide a general result on the recovery of main effects under dependent covariates. 
\end{remark}

%\subsubsection{Choice of the $\ell_1$ kernel} 
%\label{sec:choice-of-ell-1-kernel}
%Our recovery guarantees requires a mild technical assumption on the $\ell_1$ kernel we choose. Any $\ell_1$ kernel has the representation:
%\begin{equation*}
%k(x, x') = h(\norm{x-x'}_1) = \int_0^\infty e^{-t \norm{x-x'}_1} d\mu(t).
%\end{equation*} 
%
%\begin{assumption}
%\label{assumption:mu-compact-old}
%The measure $\mu$ is compactly supported: $\supp(\mu) \subseteq [0, M_\mu]$ for $M_\mu < \infty$. 
%\end{assumption}
%
%\begin{remark}
%Assumption~\ref{assumption:mu-compact} is a mild condition that holds for many strictly completely monotone functions $h$ and 
%the corresponding $\ell_1$ kernel. 
%For example, Assumption~\ref{assumption:mu-compact} holds for the Laplace kernel $k(x, x') = e^{- \norm{x-x'}_1}$. 
%To check whether Assumption~\ref{assumption:mu-compact} holds for a certain function $h$, we
%can use the Paley-Wiener theorem~\cite{ReedSi75}, which says that Assumption~\ref{assumption:mu-compact} holds 
%for a completely monotone function $h$ if and only if there exists an analytical extension of the function $h$ to the complex 
%plane $\C$ satisfying the exponential growth condition, i.e.,  $h(z) \le (1+|z|)^N e^{B|\Im(z)|}$ for some $B, N < \infty$. The function $h(x) = e^{-x}$ satisfies Assumption~\ref{assumption:mu-compact}
%since the exponential function $h(z) = \exp(-z)$ grows at most exponentially in the complex plane $\C$.
%\end{remark}

\subsubsection{Recovery of main effect signal}
\label{sec:main-effect}
A variable $X_l$ has a \emph{main effect signal}
under the functional ANOVA model if and only if $f_l^*(X_l) \neq 0$. This section shows that the kernel 
feature selection algorithm (Alg.~\ref{alg:kernel-feature-selection}) can recover main effect signals at the population level. 

Before diving into the main result, Theorem~\ref{thm:add-main-mix-effect}, we start with a simple 
example (Example~\ref{example:additive-main-effect})---the additive main effect model that we introduced in Section~\ref{sec:distinction-l-1-l-2}. The proof of recovery in the additive model is conceptually much simpler than 
that of the general result (Theorem~\ref{thm:add-main-mix-effect}) and provides useful intuition. 

\vspace{0.5cm}
\begin{example}[Additive Main Effect Model]
\label{example:additive-main-effect}
Consider the following additive model: 
\begin{itemize}
\item $f^*(X_S) = \sum_{l \in S} f_l^*(X_l)$, where the functions $f_l^*: \R \to \R$ satisfy
	$\E[f_l^*(X_l)] = 0$.
\item Independent covariates: $X_{l_1} \perp X_{l_2} \perp \ldots \perp X_{l_{|S|}}$, 
	where $S = \{l_1, l_2, \ldots, l_{|S|}\}$. 
\end{itemize}
For a variable $X_l$ with a main effect, we define the \emph{effective size} of the main effect as
\begin{equation*}
	\effect_l = \left|\E[f_l^*(X_l) f_l^*(X_l') |X_l - X_l'|]\right| = \int \left|\E[|f_l^*(X_l) e^{i\omega X_l}|]\right|^2 \cdot \frac{d\omega}{\pi\omega^2} > 0.
\end{equation*}
Theorem~\ref{thm:add-main-effect} shows that Alg.~\ref{alg:kernel-feature-selection} recovers $X_l$ at the population level as long as the effective signal $\effect_l$ exceeds a threshold. The proof of Theorem~\ref{thm:add-main-effect}
is simple and given in Section~\ref{sec:proof-thm-add-main-effect} of the main text.

\renewcommand\thetheorem{\arabic{theorem}'}
\begin{theorem} [Additive Model]
\label{thm:add-main-effect}
Assume Assumptions~\ref{assumption:mu-compact} and~\ref{assumption:X-Y-bound} . 
There exists a constant $\wbar{C} > 0$ depending only on $M, M_X, M_Y, M_\mu$ such that 
the following holds. Suppose the \emph{effective signal size} of a
variable $X_l$ exceeds a threshold:
\begin{equation}
\label{eqn:add-main-effect}
	\effect_l \ge \wbar{C} \cdot (\lambda^{1/2}(1+\lambda^{1/2}) + \lambda \gamma).
\end{equation}
Consider Algorithm~\ref{alg:kernel-feature-selection} with initialization $\beta^{(0)} = 0$ and stepsize $\alpha \le \frac{\lambda^2}{\wbar{C} p}$. 
Then $l \in \hat{S}$ where \nolinebreak $\hat{S}$ is the set returned by Algorithm~\ref{alg:kernel-feature-selection}. 
\end{theorem}

\begin{remark}
Theorem~\ref{thm:add-main-effect} shows that the algorithm can recover main effects 
when the regularizers $\lambda$ and $\lambda\gamma$ are sufficiently small compared 
to the \emph{effective signal size} 
(and in particular, when $\lambda = 0$). The main technique used in the proof is 
the characterization of the gradient $\grad \obj(\beta)$ in 
Section~\ref{sec:distinction-l-1-l-2}. 
\end{remark}
\end{example}

\renewcommand\thetheorem{\arabic{theorem}}
\setcounter{theorem}{1}

We now state a more general result on recovery of
main effect signals (Theorem~\ref{thm:add-main-mix-effect}). Parallel to the statement of Theorem~\ref{thm:add-main-effect},
we first define the \emph{effective signal size} of a main effect signal $X_l$ under the more general setup of the 
functional ANOVA model. 
\begin{definition}
\label{definition:l-signal-main-mix}
Define the \emph{effective signal size} of the main effect of $X_l$ as
\begin{equation*}
\effect_l = \inf_{(T_1, \ldots, T_{|S|}) \in \mathcal{G}_l}~~ \prod_{k =1}^{|S|} \min\{\effect_l(X_{T_k}), 1\}.
\end{equation*}
Here, the set $\mathcal{G}_l $ and the quantity $\effect_l(X_T)$ for any set $T$ are defined by 
\begin{itemize}
\item $\effect_l(X_T) \defeq \E\left[F_l(X_T) F_l(X_T') h\left(\norm{X_{T} - X_{T}'}_1\right)\right]$, where 
	$F_l(X_T) = \sum_{l \in A: A \subseteq T} f_A(X_A)$.
\item $\mathcal{G}_l = \left\{(T_1, \ldots, T_{|S|}): T_1 = \{l\}, T_{|S|} = S, T_k \subsetneq T_{k+1},~\text{for all $1\le k < |S|$}\right\}$.
\end{itemize}
\end{definition}
The effective signal size so defined is strictly positive for any main effect. This is formalized 
in Proposition~\ref{prop:main-effect-positive} whose proof is given in Appendix~\ref{sec:proof-proposition-main-effect-positive}.

\begin{proposition}
\label{prop:main-effect-positive}
The effective signal size $\effect_l > 0$ holds for any variable $X_l$ where $f_l^*(X_l) \neq 0$. 
\end{proposition}

Theorem~\ref{thm:add-main-mix-effect} shows that we can recover the variable $X_l$ if the \emph{effective signal size} 
$\effect_l$ exceeds a threshold (see equation~\eqref{eqn:add-main-mix-effect}). The proof of Theorem~\ref{thm:add-main-mix-effect}
is given in Appendix~\ref{sec:proof-of-theorem-add-main-mix-effect}. 
\begin{theorem}[Functional ANOVA]
\label{thm:add-main-mix-effect}
Given Assumptions~\ref{assumption:mu-compact} and~\ref{assumption:X-Y-bound}, assume 
that the functional ANOVA model holds. 
There exists a constant $\wbar{C} > 0$ depending only on $|S|, M, M_X, M_Y, M_\mu$ such that the following holds. 
Suppose the \emph{effective signal size} of the main effect of $X_l$ exceeds a threshold: 
\begin{equation}
\label{eqn:add-main-mix-effect}
	\effect_l \ge \wbar{C} \cdot (\lambda^{1/2}(1+\lambda^{1/2}) + \lambda \gamma). %~~\text{for some $l \in S$}, 
\end{equation}
Consider Algorithm~\ref{alg:kernel-feature-selection} with initialization $\beta^{(0)} = 0$ and stepsize $\alpha \le \frac{\lambda^2}{\wbar{C} p}$. 
Then $l \in \hat{S}$ where \nolinebreak $\hat{S}$ is the set returned by Algorithm~\ref{alg:kernel-feature-selection}. 
%Consider the algorithm which initializes at $\beta^{(0)} = 0$ with stepsize $\alpha \le \frac{\lambda^2}{\wbar{C} p}$. Then any accumulation point $\beta^*$ of the algorithm iterates $\{\beta^{(k)}\}_{k \in \N}$ satisfies $\beta^*_{l} > 0$.
\end{theorem}

\begin{remark}
Theorem~\ref{thm:add-main-mix-effect} generalizes Theorem~\ref{thm:add-main-effect} by proving main effect recovery in the more general functional ANOVA setup (allowing variables to also interact). 
%when the regularizer $\lambda$ is sufficiently small compared with the \emph{effective signal size} (and in particular, when $\lambda = 0$). 
%The proof is based on the same logic as in Theorem~\ref{thm:add-main-mix-effect} where we show that any stationary point $\beta$
%encountered by the algorithm must have $\beta_l > 0$. 
Compared with the proof of  Theorem~\ref{thm:add-main-effect}, the proof of Theorem~\ref{thm:add-main-mix-effect} introduces one new argument
(see Lemma~\ref{lemma:upper-bound-tilde-J} and the accompanying remark) which captures the following phenomenon. Suppose the set
$T\backslash l$ has been selected (i.e., $\beta_{T \backslash l}$ is large) but $l$ has not been selected ($\beta_l = 0$). The size of the gradient with respect to $\beta_l$  will now depend on the signal size of $X_l$ in the context of the group of variables $T$---this signal size is measured quantitatively by the term $\effect_l(X_T)$ defined above. If $\effect_l(X_T)$ is sufficiently large, then $\beta_l$ will become non-zero once $T\backslash l$ has been selected. The definition of $\effect_l$ minimizes over all possible orderings in which the variables in $S$ might be selected and guarantees that $\beta_l$ will become non-zero no matter which variables might be selected before it in the ordering.
\end{remark}

\subsubsection{Hierarchical interaction signal}
\label{sec:hierarchical-effect}

In this section, we show how a natural variant of Algorithm \ref{alg:kernel-feature-selection} is 
able to find variables with \emph{zero} marginal effects as long as those variables participate in a hierarchical interaction. To formally define the hierarchy of a signal, we use the functional ANOVA model discussed 
in Section~\ref{sec:functional-ANOVA}. Suppose the ANOVA decomposition has the following form: 
\begin{equation}
\label{eqn:hierarchical-ANOVA}
	f^*(X_S) = \sum_{k=1}^K \sum_{l=1}^{|S_k|}  f_{S_{k, l}}^*(X_{S_{k, l}}),
\end{equation}
where we have
\begin{itemize}
\item $K$ disjoint hierarchical components: $S = \cup_{k=1}^K S_k$ and $S_i \cup S_j = \emptyset$ for $i \neq j$.
\item Hierarchical signal within each component:
	$\emptyset \subsetneq S_{k, 1} \subsetneq S_{k, 2} \subsetneq \ldots \subsetneq S_{k, |S_k|} = S_k$ where $|S_{k, l}| = l$ for $l \in [|S_k|]$, and $f_{S_{k, l}}^*(X_{S_{k, l}}) \neq 0$ for any $k,l$.
\end{itemize}

The ANOVA decomposition in equation~\eqref{eqn:hierarchical-ANOVA} defines $K$ hierarchical signals in the following sense. 
All the variables in $\cup_{k \in K} S_{k, 1}$ are main effects (level $1$ signals). 
Variables in $\cup_{k \in K} (S_{k, 2}\backslash S_{k, 1})$ have level $2$ signals---i.e. 
level 2 variables---have a conditional main effect given the level $1$ variables. 
We then recursively define the level $l$ 
variables as those in $\cup_{k \in K} (S_{k, l} \backslash S_{k, l-1})$. 
As a concrete example, suppose the signal takes the form 
\begin{equation*}
	f^*(X_{12345}) = (f_1^*(X_1) + f_{12}^*(X_{12})) + (f_3^*(X_3) + f_{34}^*(X_{34}) + f_{345}^*(X_{345})). 
\end{equation*}
In this case, we have two hierarchical components $\{1, 2\}$ and $\{3, 4, 5\}$, and within each component, 
the signals exhibit a hierarchy: the level $1$ signals are $\{1, 3\}$, the level $2$ 
signals are $\{2, 4\}$ and the level $3$ signal is $\{5\}$.

\paragraph{Notation}
For notational simplicity, we adopt the following index on the features: 
$X_{k, l} \defeq X_{S_{k, l}} \backslash X_{S_{k, l-1}}$. Hence, $X_{S_{k, 1}} = X_{k, 1}$, 
$X_{S_{k, 2}} = X_{k, 1} \cup X_{k, 2}$, $X_{S_{k, 3}} = X_{k, 1} \cup X_{k, 2} \cup X_{k, 3}$ etc.
We use $N_k = |S_k|$ to denote the size of the $k$th component. 

\vspace{0.4cm}
Now we define the \emph{effective signal size} for a signal variable $X_{k, l}$ for $k \le K$ and $l \le N_k$. 

\begin{definition}
\label{definition:hier-effect}
Define the \emph{effective signal size} of $X_{k, l}$ in the hierarchical model by
\begin{equation*}
	\effect_{k, l} = \min_{1 \le m \le l}  \Bigg\{\prod_{m \le j \le N_k} \min\left\{\effect_{k, m}(X_{S_{k, j}}), 1\right\}\Bigg\},
\end{equation*}
where we define 
$\effect_{k, m}(X_{S_{k, j}}) \defeq \E\left[F_{k, m}(X_{S_{k, j}}) F_{k, m}(X_{S_{k, j}}') h\Big(\normbig{X_{S_{k, j}} - X_{S_{k, j}}'}_1\Big)\right]$ 
for $m \le j$, where $F_{k, m}(X_{S_{k, j}}) = \sum_{m \le w\le j} f_{S_{k, w}}^*(X_{S_{k, w}})$.
\end{definition}

The effective signal size for $X_{k, l}$, the level $l$ variable in component $k$, is positive as long as all the lower level variables in component $k$ have non-zero effective signal size. More precisely, we require $f_{S_{k, j}}^*(X_{S_{k, j}}) \neq 0$ for all $1\le j\le l$.
The result is formally stated in Proposition~\ref{prop:hier-positive} with proof in Appendix~\ref{sec:proof-proposition-hier-positive}. 
\begin{proposition}
\label{prop:hier-positive}
The effective signal size $\effect_{k, l} > 0$ as long as $f_{S_{k, j}}^*(X_{S_{k, j}}) \neq 0$ 
for all $1\le j \le l$. 
\end{proposition}

\begin{algorithm}[!tph]
\begin{algorithmic}[1]
%\label{alg:metric_hier}
\Require{Initializers $\{\beta^{(0; T)}\}_{T \in 2^{[p]}}$, stepsize $\alpha$, feature matrix $\mathbf{X} \in \mathbb{R}^{n\times p}$ and response $y \in \R^n$}
\While{$\hat S$ not converged}
%\State{Initialize $\beta_j = \tau$ for $j\in \hat S$ and $\beta_j=0$ for $j \notin \hat S$}
\State{Run projected gradient descent (with stepsize $\alpha$ and initialization $\beta^{(0; \hat{S})}$) to solve %the optimization
%Find any stationary point $\beta$ of the problem 
\begin{equation*}
\begin{split}
	&\minimize_{\beta \in \mathcal{B}_{M, \hat{S}}}  \obj_{\gamma}(\beta), \\
		~~\text{where}~~&
			\mathcal{B}_{M, \hat{S}} = \left\{\beta\in \R_+^p;  \norm{\beta_{\hat{S}^c}}_1 \le M~\text{and}~ \beta_{\hat{S}} = \tau \mathbf{1}_{\hat{S}}\right\}.
\end{split}
\end{equation*}
Denote the projected gradient descent iterates as $\{\beta^{(k)}\}_{k \in \N}$ where 
\begin{equation}
%\label{eqn:pgd-kernel}
	\beta^{(k+1)} = \proj_{\B_{M, \hat{S}}}(\beta^{(k)} - \alpha \grad \regobj(\beta^{(k)})).
\end{equation}
}
\State{Update $\hat{S} = \supp(\beta) \cup \hat{S}$ where $\beta$ is any accumulation point of the iterates $\{\beta^{(k)}\}_{k \in \N}$. }
\EndWhile
\end{algorithmic}
\caption{Kernel Feature Selection Algorithm (Variant)}
\label{alg:kernel-feature-selection-hier}
\end{algorithm}

Theorem~\ref{thm:hier-effect} shows that Alg.~\ref{alg:kernel-feature-selection-hier}, a simple variant of Alg.~\ref{alg:kernel-feature-selection}, can
recover all hierarchical interactions at the population level. The idea is to run multiple rounds of Alg.~\ref{alg:kernel-feature-selection} while keeping the already discovered variables active in subsequent rounds. In the first round, we can discover all main effect signals (Theorem \ref{thm:add-main-mix-effect}); in the second round, we can discover all level 2 signals, and so on. Theorem~\ref{thm:hier-effect} formalizes this result. The proof is in Appendix~\ref{sec:proof-thm-hier-effect}. 

\begin{theorem}[Hierarchical Interaction]
\label{thm:hier-effect}
Make Assumptions~\ref{assumption:mu-compact} and~\ref{assumption:X-Y-bound} . Assume the hierarchical interaction model. 
There exists a constant $\wbar{C} > 0$ that depends only on $\tau, |S|, M, M_X, M_Y, M_\mu$ such that the following holds. 
Suppose the \emph{effective signal size} of a signal variable $X_{k, l}$ exceeds a certain threshold: 
\begin{equation}
\label{eqn:hier-effect}
	\effect_{k, l} \ge C \cdot (\lambda^{1/2}(1+\lambda^{1/2}) + \lambda \gamma).
\end{equation}
Consider Algorithm~\ref{alg:kernel-feature-selection-hier} with the initializers $\{\beta^{(0; T)}\}_{T \in 2^{[p]}}$ where $\beta^{(0; T)}$
is defined by $\beta^{(0; T)}_T = \tau \mathbf{1}_T$ and $\beta^{(0; T)}_{T^c} = 0$
and with the stepsize $\alpha \le \frac{\lambda^2}{\wbar{C} p}$. %Then %$\hat{S} = S$ where returne
Then the algorithm selects the variable $X_{k, l}$.%, i.e., the set $\hat{S}$ returned by the algorithm satisfies $X_{k, l} \in X_{\hat{S}}$. 
\end{theorem}

%\subsection{Proof of Main Theorems}
\subsection{Proof of Theorem~\ref{thm:no-false-positive}}
\label{sec:thm-proof-no-false-positive}
The key to the proof is Lemma~\ref{lemma:grad-noise-variable-beta-noise-0} which holds for both $\ell_1$ and $\ell_2$ kernels.
\begin{lemma}
\label{lemma:grad-noise-variable-beta-noise-0}
We have the following for all $\beta$ such that $\beta_{S^c} = 0$:
\begin{equation}
	\partial_{\beta_l} \obj_\gamma(\beta) \ge \gamma \ge 0 \qquad \forall l \in S^c.
\end{equation}
\end{lemma}
\noindent\noindent
With Lemma~\ref{lemma:grad-noise-variable-beta-noise-0} in hand, we can prove $\beta^{(k)}_{S^c} = 0$ for all $k \in \N$. 
The proof is via induction.  
\begin{itemize}%[(i)]
\item The base case: $\beta^{(0)}_{S^c} = 0$ (this is the only part where we use the assumption $\beta^{(0)} = 0$). 
\item Suppose $\beta^{(k)}_{S^c} = 0$. Fix a noise variable $l \in S^c$. Note then $\partial_{\beta_l} \regobj(\beta^{(k)})\ge 0$ 
	by Lemma~\ref{lemma:grad-noise-variable-beta-noise-0}. This shows the bound
	\begin{equation*}
		\beta^{(k+\half)}_l  \defeq \beta^{(k)}_l - \alpha \cdot \partial_{\beta_l} \regobj(\beta^{(k)}) \le 0.
	\end{equation*}
	According to Lemma~\ref{lemma:projection-onto-ell-one-ball}, after projection, $\beta_l^{(k+1)} = (\proj_{\B_M}(\beta^{(k+\half)}))_l = 0$.
	Since the choice of $l \in S^c$ is arbitrary, this proves $\beta^{(k+1)}_{S^c} = 0$ and completes the induction step.
\end{itemize}

\paragraph{Proof of Lemma~\ref{lemma:grad-noise-variable-beta-noise-0}}
By definition, $\partial_{\beta_l} \obj_\gamma(\beta) = \partial_{\beta_l} \obj(\beta) + \gamma$. 
Hence, it suffices to show that for all $\beta$ such that $\beta_{S^c} = 0$, the following holds:
\begin{equation}
\label{eqn:no-false-positive-key}
	\partial_{\beta_l} \obj(\beta) \ge 0.
\end{equation}
The key to the proof is to use the representation of $\grad \obj(\beta)$ in Proposition~\ref{proposition:compute-grad-obj}. Let $(X', Y')$ be 
an independent copy of $(X, Y)$. By Proposition~\ref{proposition:compute-grad-obj}, we have for all $\beta \ge 0$ and all $l \in [p]$: 
\begin{equation*}
\begin{split}
(\grad \obj(\beta))_l = -\frac{1}{\lambda} \cdot \E\left[\E[r_\beta(\beta^{1/q}\odot X; Y)|X] \cdot \E[r_\beta(\beta^{1/q}\odot X'; Y')|X'] 
	\cdot h'(\norm{X-X'}_{q, \beta}^q) \cdot |X_l- X_l'|^q\right].
\end{split}
\end{equation*}
Now, assume that $\beta$ satisfies $\beta_{S^c} = 0$. Fix $l \in S^c$. Notice the following facts:
\begin{enumerate}
\item The random variable $\E[r_\beta(\beta^{1/q}\odot X; Y)|X] \cdot \E[r_\beta(\beta^{1/q}\odot X'; Y')|X'] \cdot h'(\norm{X-X'}_{q, \beta}^q)$ depends only on 
	the random variables $(X_S, X_S')$. This is because $\E[Y| X] = f^*(X_S)$ and $\beta_{S^c} = 0$. %Indeed, $\E[r_\beta(\beta\odot X; Y)|X] = f^*(X_S) - f_\beta(\beta \odot X)$
\item The random variable $|X_l - X_l'|$ depends only on $(X_{S^c}, X_{S^c}')$ since $l \in S^c$.
\end{enumerate}
Because the signal variables $(X_S, X_S')$ are independent of the noise variables $(X_{S^c}, X_{S^c}')$ by assumption, we obtain
\begin{equation*}
	(\grad \obj(\beta))_l = -\frac{1}{\lambda} \cdot \E\left[r_\beta(\beta^{1/q}\odot X; Y)r_\beta(\beta^{1/q}\odot X'; Y') h'(\norm{X-X'}_{q, \beta}^q) \right] 
		\cdot \E[|X_l- X_l'|^q].
\end{equation*}
Now, we show that the right-hand side is non-negative. It suffices to show that 
\begin{equation}
\label{eqn:no-false-positive-neg-def}
\E\left[r_\beta(\beta^{1/q}\odot X; Y)r_\beta(\beta^{1/q}\odot X'; Y') h'(\norm{X-X'}_{q, \beta}^q) \right]\le 0.
\end{equation}
One way to show this is to notice that $(x, x') \mapsto h'(x-x')$ is a negative definite kernel since $-h'$ is strictly completely monotone. 
An alternative argument uses Fourier analysis. Note that $h'(x) = -\int te^{-tx}\mu(dt)$. We obtain the identity: 
\begin{equation}
\label{eqn:h-prime-integral-formula}
\begin{split}
	h'(\norm{x-x'}_{q, \beta}^q) 
	= - \int te^{-t \norm{x-x'}_{q, \beta}^q}\mu(dt) 
	= -\int e^{i \langle \omega,  \beta^{1/q} \odot (x-x')\rangle} \wtilde{Q}(\omega) d\omega,
\end{split}
\end{equation}
where $\wtilde{Q}: \R^p \to \R_+$ is defined in equation~\eqref{eqn:tilde-Q-R-beta-omega-definition}.
Substitute this into equation~\eqref{eqn:no-false-positive-neg-def}. We obtain %that 
\begin{equation}
\label{eqn:identity-involve-h'-Q}
\begin{split}
	&\E\left[r_\beta(\beta^{1/q}\odot X; Y)r_\beta(\beta^{1/q}\odot X'; Y') h'(\norm{X-X'}_{q, \beta}^q) \right] \\
		&= -\E\left[r_\beta(\beta^{1/q}\odot X; Y)r_\beta(\beta^{1/q}\odot X'; Y')
			\int e^{i \langle \omega,  \beta^{1/q} \odot (X-X')\rangle} \wtilde{Q}(\omega) d\omega\right] \\
		&= - \int \left|\E[r_\beta(\beta^{1/q}\odot X; Y) e^{i \langle \omega, \beta^{1/q} \odot X\rangle}]\right|^2 \wtilde{Q}(\omega) d\omega \le 0.
\end{split}
\end{equation}
This proves equation~\eqref{eqn:no-false-positive-neg-def} as desired. The proof of Lemma~\ref{lemma:grad-noise-variable-beta-noise-0} 
is thus complete.

\subsection{Proof of Theorem~\ref{thm:add-main-effect}}
\label{sec:proof-thm-add-main-effect}
According to Proposition~\ref{prop:grad-obj-lipschitz-beta}, we know that the gradient $\grad \obj_{\gamma}(\beta)$ is Lipschitz in the following 
sense: for some constants $C > 0$ depending only on $h'(0), M_X, M_Y$, the following holds for any $\beta, \beta' \in \R_+^p$: 
\begin{equation}
\label{eqn:Lipschitz-of-gradients}
	\norm{\grad \obj_{\gamma}(\beta) - \grad \obj_{\gamma}(\beta')}_2 \le C\cdot \frac{p}{\lambda^2} \norm{\beta-\beta'}_2.
\end{equation}
Consequently, a standard property of the projected gradient descent algorithm (Lemma~\ref{lemma:gradient-ascent-increases-objective})
implies that any accumulation point $\beta^*$ of the gradient descent iterates must be stationary 
when the stepsize satisfies
$\alpha \le \frac{\lambda^2}{C p}$ for the same constant $C$ that appears in equation~\eqref{eqn:Lipschitz-of-gradients}. Below we assume 
the stepsize satisfies this constraint, and show that any stationary point $\beta^*$ that is \emph{reachable} by the algorithm (i.e., is an accumulation point
of the iterates) must have $\beta_l^* > 0$. 

To see this, we proceed as follows. By Theorem~\ref{thm:no-false-positive}, any stationary point $\beta^*$ reachable by the 
algorithm must exclude noise variables, i.e., $\beta^*_{S^c} = 0$. Hence, it suffices to show that any stationary point $\beta^*$
with $\beta^*_{S^c} = 0$ must satisfy $\beta_l^* > 0$. 
Considering the contrapositive, it suffices to show that any $\beta$ with $\beta_l = 0$ and $\beta_{S^c} = 0$ can't be a stationary point. 

%The contrapositive is to say that any point $\beta$ with $\beta_l = 0$ and $\beta_{S^c} = 0$ can't be stationary. 
To prove the contrapositive, we show 
that the gradient with respect to the noise variable $\beta_l$ at any such $\beta$ (i.e., $\beta_l = 0$ and $\beta_{S^c} = 0$) is 
always strictly negative:
\begin{equation}
\label{eqn:desired-goal-on-obj-main-add}
	\partial_{\beta_l} \regobj(\beta) = \partial_{\beta_l} \obj(\beta) + \gamma < 0. 
\end{equation}
Thus such $\beta$ can't be stationary. The rest of the proof establishes 
equation~\eqref{eqn:desired-goal-on-obj-main-add}. Our core technique is 
to use a Fourier-analytic argument to analyze the gradient $\partial_{\beta_l} \obj(\beta)$ under 
Assumptions~\ref{assumption:mu-compact} and~\ref{assumption:X-Y-bound}  discussed in 
Section~\ref{sec:distinction-l-1-l-2}. Since this argument is used repeatedly, 
we detail its structure in the following paragraph.
\paragraph{General Recipe}
The general recipe to bound the true gradient $\partial_{\beta_l} \obj(\beta)$ is as follows. 
\begin{itemize}
\item First, bound the surrogate gradient $\wtilde{\partial_{\beta_l} \obj}(\beta)$ using either its definition in 
	equation~\eqref{eqn:def-wtilde-partial-beta-j} or the integral representation in equation~\eqref{eqn:wtilde-partial-beta-j-rep}.
\item Next, transform the bound on the surrogate $\wtilde{\partial_{\beta_l} \obj}(\beta)$ into a bound for the true gradient 
	$\partial_{\beta_l} \obj(\beta)$. To do this, use Lemma~\ref{lemma:approximate-gradient-bound} which bounds the deviation 
	between the surrogate and true gradient.
\end{itemize}

\paragraph{Proof of Theorem~\ref{thm:add-main-effect}}
Recall that our goal is to show that equation~\eqref{eqn:desired-goal-on-obj-main-add} holds at any $\beta$ with 
$\beta_l = 0$ and $\beta_{S^c} = 0$. We apply our general recipe to achieve this goal.

First, we bound the surrogate gradient $\wtilde{\partial_{\beta_l} \obj}(\beta)$. By equation~\eqref{eqn:def-wtilde-partial-beta-j}, we have
\begin{equation}
\label{eqn:approx-grad-tower}
\begin{split}
	&\wtilde{\partial_{\beta_l} \obj}(\beta)  \\
	&= \frac{1}{\lambda} \cdot  \int \E\left[\E[\wbar{R}_{\beta, \omega}(\beta \odot X; Y) | X_l] \cdot \E[\wbar{\wbar{R}_{\beta, \omega}(\beta \odot X'; Y')} | X_l'] 
		 	\cdot |X_l - X_l'|\right] \wtilde{Q}(\omega)  d\omega. 
\end{split}
\end{equation}
Now we evaluate $\E[\wbar{R}_{\beta, \omega}(\beta \odot X; Y) | X_l]]$. By definition, we have 
\begin{equation*}
\begin{split}
	R_{\beta, \omega}(\beta \odot X; Y) &= e^{i \langle \omega, \beta \odot X\rangle} \left(Y - f_\beta(\beta\odot X)\right) \\
		&= e^{i \langle \omega, \beta \odot X\rangle} \Big(\xi + f_l^*(X_l) + \sum_{j \in S\backslash l} f_j^*(X_j)  - f_\beta(\beta\odot X)\Big).
\end{split}
\end{equation*}
At $\beta$ where $\beta_l = 0$ and $\beta_{S^c} = 0$, the random variables $e^{i \langle \omega, \beta\odot X\rangle}$ and 
$f_\beta(\beta\odot X)$ depend only on the random variables $X_{S \backslash l}$, and are thus independent of $X_l$
by assumption. As a result, we obtain the following expression for 
$\E[\wbar{R}_{\beta, \omega}(\beta \odot X; Y) | X_l]$: 
\begin{equation*}
\begin{split}
\E[\wbar{R}_{\beta, \omega}(\beta \odot X; Y) | X_l]
		&= \E[R_{\beta, \omega}(\beta \odot X; Y) | X_l]  - \E[R_{\beta, \omega}(\beta \odot X; Y)] \\
		&= \E[e^{i \langle \omega, \beta \odot X\rangle}] \cdot f_l^*(X_l).
		%= \E[e^{i \langle \omega, \beta \odot X\rangle}] \cdot f_l^*(X_l).
\end{split}
\end{equation*}
Substitute this back into equation~\eqref{eqn:approx-grad-tower}. We obtain the identity 
\begin{equation}
\label{eqn:tilde-grad-main-add-effect}
\begin{split}
		   \wtilde{\partial_{\beta_l} \obj}(\beta) 
		= &\frac{1}{\lambda} \cdot  \int \E\left[\E[\wbar{R}_{\beta, \omega}(\beta \odot X; Y) | X_l] \E[\wbar{\wbar{R}_{\beta, \omega}(\beta \odot X'; Y')} | X_l'] 
		 	|X_l - X_l'|\right] \wtilde{Q}(\omega)  d\omega \\
		= &\frac{1}{\lambda} \cdot \int \E\left[ f_l^*(X_l)  f_l^*(X_l') 
		 	|X_l - X_l'|\right]  \cdot \left|\E[e^{i \langle \omega, \beta \odot X\rangle}]\right|^2 \cdot \wtilde{Q}(\omega)  d\omega \\
		= &\frac{1}{\lambda} \cdot \effect_l(X_l) \cdot \E[h'(\norm{X-X'}_{1, \beta})],
\end{split}
\end{equation}
where we use the integral formula in equation~\eqref{eqn:h-prime-integral-formula} to derive the last identity.

Note that $\norm{\beta}_1 \le M$. Hence, $\E[\norm{X-X'}_{1, \beta}] \le 2M M_X$ by assumption. Consequently, Jensen's inequality implies
that $\E[h'(\norm{X-X'}_{1, \beta})] \le h^\prime(\E[\norm{X-X'}_{1, \beta}]) \le h^\prime(2MM_X) \le 0$ since $h$ is completely monotone
(so we have $h^\prime \le 0$ and $h^{\prime}$ is concave). Substituting the bound 
into equation~\eqref{eqn:tilde-grad-main-add-effect}, we obtain the final bound 
on the surrogate gradient:
\begin{equation}
\label{eqn:tilde-grad-main-add-effect}
	\wtilde{\partial_{\beta_l} \obj}(\beta) = \frac{1}{\lambda} \cdot \effect_l(X_l) \cdot \E[h'(\norm{X-X'}_{1, \beta})]
		\le \frac{1}{\lambda} \cdot \effect_l(X_l) \cdot h'(2 M M_X).
\end{equation}

Now we turn the bound for the surrogate $\wtilde{\partial_{\beta_l} \obj}(\beta)$ in equation~\eqref{eqn:tilde-grad-main-add-effect}
into a bound for the true gradient $\partial_{\beta_l} \obj(\beta)$. By the triangle inequality and Lemma~\ref{lemma:approximate-gradient-bound}, 
\begin{equation*}
%\label{eqn:grad-main-add-effect}
	\partial_{\beta_l} \obj(\beta) \le \frac{1}{\lambda} \cdot \left(\effect_l(X_l) \cdot h'(2M M_X) + C\lambda^{1/2}(1+\lambda^{1/2})\right),
\end{equation*}
for some constant $C > 0$ depending only on $M_X, M_Y, M_\mu$. Consequently, we have established the following inequality 
that holds for all $\beta$ such that $\beta_l = 0$ and $\beta_{S^c} = 0$: 
\begin{equation}
\label{eqn:grad-main-add-effect}
	\partial_{\beta_l} \obj_\gamma(\beta) \le \frac{1}{\lambda} \cdot \left(\effect_l(X_l) \cdot h'(2M M_X) + C\lambda^{1/2}(1+\lambda^{1/2}) + \lambda \gamma\right).
\end{equation}

With this bound at hand, we see that the desired equation~\eqref{eqn:desired-goal-on-obj-main-add} holds for all such $\beta$ as long as 
the condition on the effective signal size, $\effect_l(X_l)$, in equation~\eqref{eqn:add-main-effect} holds for a sufficiently large constant $\wbar{C} > 0$. This completes the proof of 
Theorem~\ref{thm:add-main-effect}.

\newcommand{\dev}{\mathcal{E}}
\newcommand{\erobj}{\obj_{\gamma, n}}
\newcommand{\robj}{\regobj}

\section{Finite-Sample Guarantees}
\label{sec:finite-sample-guarnatees}

In this section, we provide finite-sample guarantees for the kernel feature 
selection algorithm.  First, we present the empirical kernel feature selection objective and 
the corresponding algorithm
(Section~\ref{sec:finite-sample-algorithms}). Next, we establish that the empirical gradients 
concentrate around their population counterparts (Section~\ref{sec:concentration-of-gradients}).
With the appropriate concentration results in hand, the finite-sample properties of the kernel feature selection algorithm 
follow as a consequence of our population results in Section~\ref{sec:population-guarantees}.
In particular, the kernel feature selection algorithm has the power
to exclude noise variables and include signal variables with high probability (Section~\ref{sec:finite-samples-all-results}). 
Finally, Section~\ref{sec:proof-of-concentration-results} describes the techniques used to prove the concentration results.

%The proof leverages related (i) classical functional analysis tools to characterize the kernel ridge regression 
%(ii) classical high dimensional geometry arguments to evaluate the metric entropy and (iii) classical leave-one-out 
%argument to decouple the statistical dependency. 

\subsection{Objective and Algorithm}
\label{sec:finite-sample-algorithms}
We introduce the empirical kernel feature selection objective: 
\begin{equation}
\begin{split}
	\minimize_{\beta: \beta \ge 0, \norm{\beta}_1 \le M}~~\mathcal{J}_{n, \gamma}(\beta) &= \mathcal{J}_{n}(\beta) + \gamma \norm{\beta}_1 \\
	\text{where}~~~~\mathcal{J}_{n}(\beta) &= \min_f \half \what{\E}[(Y - f(\beta^{1/q} \odot X)^2] + \frac{\lambda}{2} \norm{f}_{\H}^2. 
\end{split}
\end{equation}
The empirical objective replaces the population expectation $\E$ in the population objective 
(see equation~\eqref{eqn:population-objective}) with the empirical average $\what{\E}$. We extend the population level algorithm to finite samples by simply replacing the population objective with the empirical objective. 
See Algorithm~\ref{alg:kernel-feature-selection-empirical} and~\ref{alg:kernel-feature-selection-hier-empirical} below for details.

\setcounter{algorithm}{0}
\renewcommand\thealgorithm{\arabic{algorithm}'}

\begin{algorithm}[!ht]
\begin{algorithmic}[1]
%\label{alg:metric_hier}
\Require{Initializer $\beta^{(0)}$, stepsize $\alpha$, feature matrix $\mathbf{X} \in \mathbb{R}^{n\times p}$ and response $y \in \R^n$}
%\While{$\hat S$ not converged}
%\State{Initialize $\beta_j = \tau$ for $j\in \hat S$ and $\beta_j=0$ for $j \notin \hat S$}
\State{Run projected gradient descent (with stepsize $\alpha$, and initialization $\beta^{(0)}$) to solve %the optimization
%Find any stationary point $\beta$ of the problem 
\begin{equation*}
\begin{split}
	&\minimize_{\beta \in \mathcal{B}_M}  \obj_{n, \gamma}(\beta), \\
		~~\text{where}~~&
			 \mathcal{B}_{M} = \left\{\beta\in \R_+^p;  \norm{\beta}_1 \le M\right\}.
\end{split}
\end{equation*}
Denote the projected gradient descent iterates as $\{\beta^{(k)}\}_{k \in \N}$ where 
\begin{equation}
\label{eqn:pgd-kernel-empirical}
	\beta^{(k+1)} = \proj_{\B_M}(\beta^{(k)} - \alpha \grad \obj_{n, \gamma}(\beta^{(k)})).
\end{equation}
}
\State{Return $\hat{S} = \supp(\beta)$ where $\beta$ is any accumulation point of the iterates $\{\beta^{(k)}\}_{k \in \N}$. }
%\EndWhile
\end{algorithmic}
\caption{Empirical Kernel Feature Selection Algorithm}
\label{alg:kernel-feature-selection-empirical}
\end{algorithm}

\begin{algorithm}[!ht]
\begin{algorithmic}[1]
%\label{alg:metric_hier}
\Require{Initializers $\{\beta^{(0; T)}\}_{T \in 2^{[p]}}$, stepsize $\alpha$, feature matrix $\mathbf{X} \in \mathbb{R}^{n\times p}$ and response $y \in \R^n$}
\While{$\hat S$ not converged}
%\State{Initialize $\beta_j = \tau$ for $j\in \hat S$ and $\beta_j=0$ for $j \notin \hat S$}
\State{Run projected gradient descent (with stepsize $\alpha$ and initialization $\beta^{(0; \hat{S})}$) to solve %the optimization
%Find any stationary point $\beta$ of the problem 
\begin{equation*}
\begin{split}
	&\minimize_{\beta \in \mathcal{B}_{M, \hat{S}}}  \obj_{n, \gamma}(\beta), \\
		~~\text{where}~~&
			\mathcal{B}_{M, \hat{S}} = \left\{\beta\in \R_+^p;  \norm{\beta_{\hat{S}^c}}_1 \le M~\text{and}~ \beta_{\hat{S}} = \tau \mathbf{1}_{\hat{S}}\right\}.
\end{split}
\end{equation*}
Denote the projected gradient descent iterates as $\{\beta^{(k)}\}_{k \in \N}$ where 
\begin{equation}
\label{eqn:pgd-kernel-variant-empirical}
	\beta^{(k+1)} = \proj_{\B_{M, \hat{S}}}(\beta^{(k)} - \alpha \grad \obj_{n, \gamma}(\beta^{(k)})).
\end{equation}
}
\State{Update $\hat{S} = \supp(\beta) \cup \hat{S}$ where $\beta$ is any accumulation point of the iterates $\{\beta^{(k)}\}_{k \in \N}$. }
\EndWhile
\end{algorithmic}
\caption{Empirical Kernel Feature Selection Algorithm (Variant)}
\label{alg:kernel-feature-selection-hier-empirical}
\end{algorithm}

\subsection{Concentration of the gradients}
\label{sec:concentration-of-gradients}

In this section, we study the maximum deviation of the empirical gradients to the population gradients 
over the feasible set $\mathcal{B}_M$. Mathematically, we consider the error term: 
\begin{equation}
\label{eqn:def-dev-M}
\dev_n = \sup_{\beta \in \mathcal{B}_M} \norm{\grad \erobj(\beta) - \grad \robj(\beta)}_\infty
	=  \sup_{\beta \in \mathcal{B}_M} \norm{\grad \obj_n(\beta) - \grad \obj(\beta)}_\infty
%= \sup_{\beta \in \mathcal{B}_M} \norm{\grad \obj_n(\beta) - \grad \obj(\beta)}_\infty
\end{equation}
Theorem~\ref{thm:concentration-of-gradients} gives a high-probability upper bound on this deviation $\dev_n$. To obtain this result, 
we require an additional assumption that the distributions of $X$ and $Y$ are light-tailed. 

\setcounter{assumption}{1}
\renewcommand\theassumption{\arabic{assumption}'}
\begin{assumption}
\label{assumption:sub-gaussian-tail-Y}
The random variable $X$ is almost surely bounded: $\P(|X_l| \le \sigma_X) = 1$ for $l \in [p]$.
In addition, the random variable $Y$ is $\sigma_Y$-subgaussian, i.e., $\E[e^{tY}] \le e^{\half \sigma_Y^2 t^2}$ for $t \in \R$. 
\end{assumption}

\begin{remark}
The assumption that the coordinate of $X$ is bounded can be replaced with a subgaussian assumption on 
the coordinates of $X$. The stronger boundedness assumption is assumed mainly for technical convenience.
\end{remark}

Below is the main concentration result. A high level description of the proof strategy is given in 
Section~\ref{sec:proof-of-concentration-results}. The full proofs of Theorem~\ref{thm:concentration-of-gradients} 
can be found in Appendix~\ref{sec:proof-of-concentration-results-details}. 

\begin{theorem}
\label{thm:concentration-of-gradients}
Let $t > 0$. Assume that $\lambda \ge C \sqrt[4]{\log n\log p/n}$. 
The following bound holds with probability at least $1-e^{-cn}-e^{-t}$: 
\begin{equation*}
	\dev_n \le 
		\frac{C\log^2(n)}{\min\{\lambda, 1\}^{7/2}} \cdot \left(\sqrt[4]{\frac{\log p}{n}}  + \sqrt{\frac{t}{n}}\right),
\end{equation*}
where the constants $c, C > 0$ depend only on the parameters 
$M$,  $\sigma_X$, $\sigma_Y$, $\mu$.
\end{theorem}

\begin{remark}
Theorem~\ref{thm:concentration-of-gradients} shows that the empirical and population gradients 
are uniformly close to each other as long as the sample size satisfies $n \ge C \log^{8}n \log(p)$.
\end{remark}

\subsection{Statistical guarantees in finite samples}
\label{sec:finite-samples-all-results}

This section presents finite-sample guarantees for the kernel feature selection algorithm. These 
extend the population guarantees given in 
Section~\ref{sec:population-guarantees}. 

\subsubsection{No-false-positive guarantees}

Corollary~\ref{corollary:no-false-positive-finite-sample} is the finite-sample analogue of our population result 
on false positive control (Theorem~\ref{thm:no-false-positive}). Compared to Theorem~\ref{thm:no-false-positive},
Corollary~\ref{corollary:no-false-positive-finite-sample} shows that in finite samples, we need an $\ell_1$ penalty 
to promote sparsity. The size of the penalty must dominate the size of the deviation between the 
empirical and the population gradients shown in Theorem~\ref{thm:concentration-of-gradients}.

The proof of Corollary~\ref{corollary:no-false-positive-finite-sample} is given in Appendix~\ref{sec:proof-corollary-no-false-positive-finite-sample}. 
%Basically, the conditiont is that 
%we need a sufficiently large $\ell_1$ penalty $\lambda$ to control the false positives. 

\begin{corollary}
\label{corollary:no-false-positive-finite-sample}
Given Assumptions~\ref{assumption:mu-compact} and~\ref{assumption:sub-gaussian-tail-Y}, 
let $c, C > 0$ be the constants in Theorem~\ref{thm:concentration-of-gradients}. 
Consider the projected gradient descent algorithm in 
equation~\eqref{eqn:pgd-kernel-empirical} (or in equation~\eqref{eqn:pgd-kernel-variant-empirical}). 
Assume that the algorithm is initialized at $\beta^{(0)}$ such that $\beta_{S^c}^{(0)} = 0$. 
For any $t > 0$, assume %that 
\begin{equation}
\label{eqn:lower-bound-on-ell-1}
	\lambda \ge C \sqrt[4]{\log n\log p/n}~~\text{and}~~
	\gamma \ge\frac{C\log^2(n)}{\min\{\lambda, 1\}^{7/2}} \cdot \left(\sqrt[4]{\frac{\log p}{n}}  + \sqrt{\frac{t}{n}}\right).
\end{equation}
Then any accumulation point $\beta^*$ of the projected gradient descent iterates, $\{\beta^{(k)}\}_{k \in \N}$, 
satisfies $\beta^*_{S^c} = 0$, with probability at least $1-e^{-cn}-e^{-t}$.
\end{corollary}

\subsubsection{Power guarantees} %with high probability}
\label{sec:power-finite-samples}

This section presents finite-sample power guarantees for the kernel feature selection algorithm. As discussed in 
Section~\ref{sec:power-guarantee-population}, the power of the algorithm depends on both the kernel that we choose and the type of signals that we consider. 

Following Section~\ref{sec:power-guarantee-population}, we assume that the algorithm uses the $\ell_1$
kernel. Additionally, we assume the functional ANOVA model discussed in Section~\ref{sec:functional-ANOVA}. 
Next, we provide results for two types of signals: main effects and hierarchical interactions. 

\paragraph{Main Effect Signal}

Corollary~\ref{corollary:add-main-mix-effect} is a finite-sample analogue of our population guarantee on the 
recovery of the main effect signals (Theorem~\ref{thm:add-main-mix-effect}). Recall the notation $\effect_l$
that denotes the effective size of a main effect signal in Definition~\ref{definition:l-signal-main-mix}.

The proof of Corollary~\ref{corollary:add-main-mix-effect} is given in Appendix~\ref{sec:proof-of-corollary-add-main-mix-effect}. 

\begin{corollary}[Functional ANOVA]
\label{corollary:add-main-mix-effect}
Make Assumptions~\ref{assumption:mu-compact} and~\ref{assumption:sub-gaussian-tail-Y}. 
Let $c, C > 0$ be the constants in Theorem~\ref{thm:concentration-of-gradients}. Let $t > 0$. 
Assume that $\lambda, \gamma$ satisfy equation~\eqref{eqn:lower-bound-on-ell-1}. 

There exists $\wbar{C} > 0$ depending only on $|S|, M, M_X, M_Y, M_\mu$ such that the following holds. 
Consider Algorithm~\ref{alg:kernel-feature-selection-empirical} with initialization $\beta^{(0)} = 0$ and stepsize $\alpha \le \frac{\lambda^2}{\wbar{C} p}$. 
Suppose the effective signal size $\effect_l$ for the signal variable $X_l$ satisfies
\begin{equation}
\label{eqn:effect-signal-add-main-mix-finite-sample}
	\effect_l \ge \wbar{C} \cdot (\lambda^{1/2}(1+\lambda^{1/2}) + \lambda \gamma). %~~\text{for some $l \in S$}, 
\end{equation}
Then, with probability at least $1-e^{-cn}-e^{-t}$, $l \in \what{S}$ where $\what{S}$ is the set returned by Algorithm~\ref{alg:kernel-feature-selection-empirical}.
\end{corollary}

\paragraph{Hierarchical Interaction Signal}
Corollary~\ref{corollary:hier-effect} is a finite-sample analogue of our population guarantee on the recovery of 
hierarchical interaction signals (Theorem~\ref{thm:hier-effect}). We adopt the same notation 
as in Section~\ref{sec:hierarchical-effect}: recall that $X_{k, l}$ denotes the 
level $l$ signal in the $k$-th hierarchical component, and $\effect_{k, l}$ denotes 
the effective signal size of $X_{k, l}$ (see Definition~\ref{definition:hier-effect}).

The proof of Corollary~\ref{corollary:hier-effect} is given in Appendix~\ref{sec:proof-of-corollary-hier-effect}. 

\begin{corollary}[Hierarchical Interaction]
\label{corollary:hier-effect}
Make Assumptions~\ref{assumption:mu-compact} and~\ref{assumption:sub-gaussian-tail-Y}. 
Let $c, C > 0$ be the constants in Theorem~\ref{thm:concentration-of-gradients}. Let $t > 0$. 
Assume that $\lambda, \gamma$ satisfy equation~\eqref{eqn:lower-bound-on-ell-1}. 

There exists $\wbar{C} > 0$ depending only on $|S|, M, M_X, M_Y, M_\mu$ such that the following holds. 
Consider Algorithm~\ref{alg:kernel-feature-selection-hier-empirical} with the initializers $\{\beta^{(0; T)}\}_{T \in 2^{[p]}}$
where $\beta^{(0; T)}_T = \tau \mathbf{1}_T$ and  $\beta^{(0; T)}_{T^c} = 0$ and with the stepsize 
$\alpha \le \frac{\lambda^2}{\wbar{C} p}$. 
Suppose the effective signal size of the variable $X_{k, l}$ satisfies
\begin{equation}
\label{eqn:effect-signal-hier-finite-sample}
	\effect_{k, l} \ge \wbar{C} \cdot (\lambda^{1/2}(1+\lambda^{1/2}) + \lambda \gamma). %~~\text{for some $l \in S$}, 
\end{equation}
Then, with probability at least $1-e^{-cn}-e^{-t}$, Algorithm~\ref{alg:kernel-feature-selection-hier-empirical} selects the variable $X_{k, l}$.
\end{corollary}

\subsection{Proof techniques for the concentration (Theorem~\ref{thm:concentration-of-gradients})}
\label{sec:proof-of-concentration-results}

The proof of Theorem~\ref{thm:concentration-of-gradients} is non-trivial. 
It leverages diverse results from high-dimensional convex geometry, 
high-dimensional probability theory, and functional analysis. We
begin by highlighting the main technical idea that drives the proof. 
%including (i) tools from functional analysis to characterize the solution of kernel ridge regression (ii) high probability 
%bounds on the supremum of sub-exponential process (iii) arguments from convex geometry to bound the metric entropy 
%(iv) leave-one-out techniques to decouple the statistical dependency. Below we highlight the main technical idea of the proof. 

%\rfcomment{Think about how to re-write it?}

%\subsubsection{Main Ideas of the Proof} 

By Proposition~\ref{proposition:compute-grad-obj}, the empirical and population gradients admit the representations: 
\begin{align}
(\grad \obj(\beta))_l = -\frac{1}{\lambda} \cdot \E\left[r_\beta(\beta\odot X; Y)r_\beta(\beta\odot X'; Y') 
	h'(\langle \beta, |X- X'| \rangle)|X_l- X_l'|\right]
	\label{eqn:population-gradient}\\
(\grad \obj_n(\beta))_l = -\frac{1}{\lambda} \cdot \what{\E}
	\left[\what{r}_\beta(\beta\odot X; Y) \what{r}_\beta(\beta\odot X'; Y') h'(\langle \beta, |X- X'| \rangle)|X_l- X_l'|\right],
	\label{eqn:empirical-gradient}
\end{align}
where the notation $\what{r}_\beta$ denotes the empirical residual function, 
$\what{r}_\beta(x, y) = y - \what{f}_\beta(x)$.

The core of the proof is to show that $\what{r}_\beta \approx r_\beta$, 
or equivalently, $\what{f}_\beta \approx f_\beta$.  The underlying tool
comes from functional analysis~\cite{Baker73, CuckerSm02, FukumizuBaJo04, FukumizuBaJo09}. %, GrettonBoSmSc05}. 
%Recall the RKHS $\H$ and the associated kernel $(x, x') \mapsto k(x, x')$.
%Recall that $(x, x') \mapsto k(x, x')$ is the kernel associated with the base RKHS $\H$ and 
%$(x, x') \mapsto k(\beta \odot x, \beta \odot x')$ is the kernel associated with the RKHS $\H_\beta$.

\begin{definition}
\label{definition:cross-covariance-operator}
Define the empirical and population covariance operator $\Sigma_\beta: \H \mapsto \H$ by
\begin{equation*}
\begin{split}
\Sigma_\beta f &= \E\left[k(\beta^{1/q}\odot X, \cdot)f(\beta^{1/q} \odot X)\right] \\
\what{\Sigma}_\beta f& = \what{\E}\left[k(\beta^{1/q}\odot X, \cdot)f(\beta^{1/q} \odot X)\right].
\end{split}
\end{equation*}
Define the empirical and population covariance function by 
\begin{equation*}
	h_\beta = \E[k(\beta^{1/q}\odot X, \cdot)Y]~~\text{and}~~\what{h}_\beta = \what{\E}[k(\beta^{1/q}\odot X, \cdot)Y].
\end{equation*}
\end{definition}

Definition~\ref{definition:cross-covariance-operator} is useful since it gives 
a representation of the solution 
$f_\beta$ and $\what{f}_\beta$ from the perspective of solving an infinite-dimensional linear equation (Proposition~\ref{prop:kernel-beta-fix}): 
\begin{equation}
\label{eqn:f_beta-from-Sigma_beta-h_beta}
f_\beta = (\Sigma_\beta + \lambda I)^{-1} h_\beta
	~~\text{and}~~
\what{f}_\beta = (\what{\Sigma}_\beta + \lambda I)^{-1} \what{h}_\beta.
\end{equation}
As a result, to show that $\what{f}_\beta \approx f_\beta$, it suffices to show that $\what{\Sigma}_\beta \approx \Sigma_\beta$ 
and $\what{h}_\beta \approx h_\beta$ are close (in certain sense). 
This idea appears earlier in the literature~\cite{FukumizuBaJo09}, 
where the authors establish the uniform convergence of the empirical 
operators and functions $\what{h}_\beta, \what{\Sigma}_\beta$ to the population versions 
$h_\beta, \Sigma_\beta$ as $n \to \infty$ (with $p$ fixed). Our additional 
contribution involves carefully extending these results into a high-dimensional 
setting, for which we need to establish a high-probability concentration result 
with explicit rates. To obtain the new concentration result, we make use of 
techniques from high-dimensional convex geometry and high-dimensional probability 
theory~\cite{Ramon14, Vershynin18}, specifically Maurey's covering argument for 
metric entropy, the Hanson-Wright inequality for quadratic forms of subgaussian 
variables, and large-deviation results on the supremum of sub-exponential processes. 

In addition to showing that $\what{r}_\beta \approx r_\beta$, we also need to 
address the concern that arises from the statistical dependencies in the definition 
of the empirical gradients in equation~\eqref{eqn:empirical-gradient}. 
By equation~\eqref{eqn:empirical-gradient}, we construct the empirical estimate 
of the gradient $\grad \obj_n(\beta)$ from the same data that is used to construct 
the estimator $\what{r}_\beta$. Decoupling the statistical dependencies that arise 
from this re-use of the data requires additional delicate work 
(Section~\ref{sec:tackle-statistical-dependency}).

\section{Experiments}
We conduct numerical experiments to validate the performance of the kernel feature selection algorithm. 
First, we show empirically that there exist clear advantages in choosing $\ell_1$ kernels over $\ell_2$
kernels when trying to detect \emph{nonlinear} signals (Section~\ref{sec:l_1-vs-l_2}). This corroborates the theory 
developed in Section~\ref{sec:distinction-l-1-l-2}. Second, we demonstrate  the power of the algorithm 
(using an $\ell_1$ kernel) to recover the main effects and 
hierarchical interactions (cf.\ Section~\ref{sec:recovery-signal}).

\subsection{The $\ell_1$ versus $\ell_2$ kernel}
\label{sec:l_1-vs-l_2}

We show that choosing an $\ell_1$ kernel is crucial for the detection of \emph{nonlinear} signals. We generate the data $(X, Y)$ according to %the following distribution: 
\begin{equation*}
	Y = X_1 + (X_2^2 - 1) + \normal(0, \sigma^2)~~\text{where}~~X \sim \normal(0, I_p),
\end{equation*} 
where $X_1, X_2$ are the signal variables. We see that 
\begin{itemize}
\item The variable $X_1$ is a \emph{linear} signal in the sense that $\Cov(Y, X_1) \neq 0$. 
\item The variable $X_2$ is, by contrast, a \emph{nonlinear} signal where $\Cov(Y, X_2) = 0$.
\end{itemize} 
We compute the recovery probability and false positive rate of the kernel feature selection algorithm for the Laplace and Gaussian kernels; see Figure~\ref{fig:l1-vs-l2}. We summarize our findings as follows.
\begin{itemize}
\item Both $\ell_1$ and $\ell_2$ type kernel are equally effective in the detection of the \emph{linear} signals. 
\item The $\ell_1$ kernel is more effective than $\ell_2$ kernel in the detection of the \emph{nonlinear} signals. 
\end{itemize}

\begin{figure}[!tph]
\begin{centering}
\includegraphics[width=.45\linewidth]{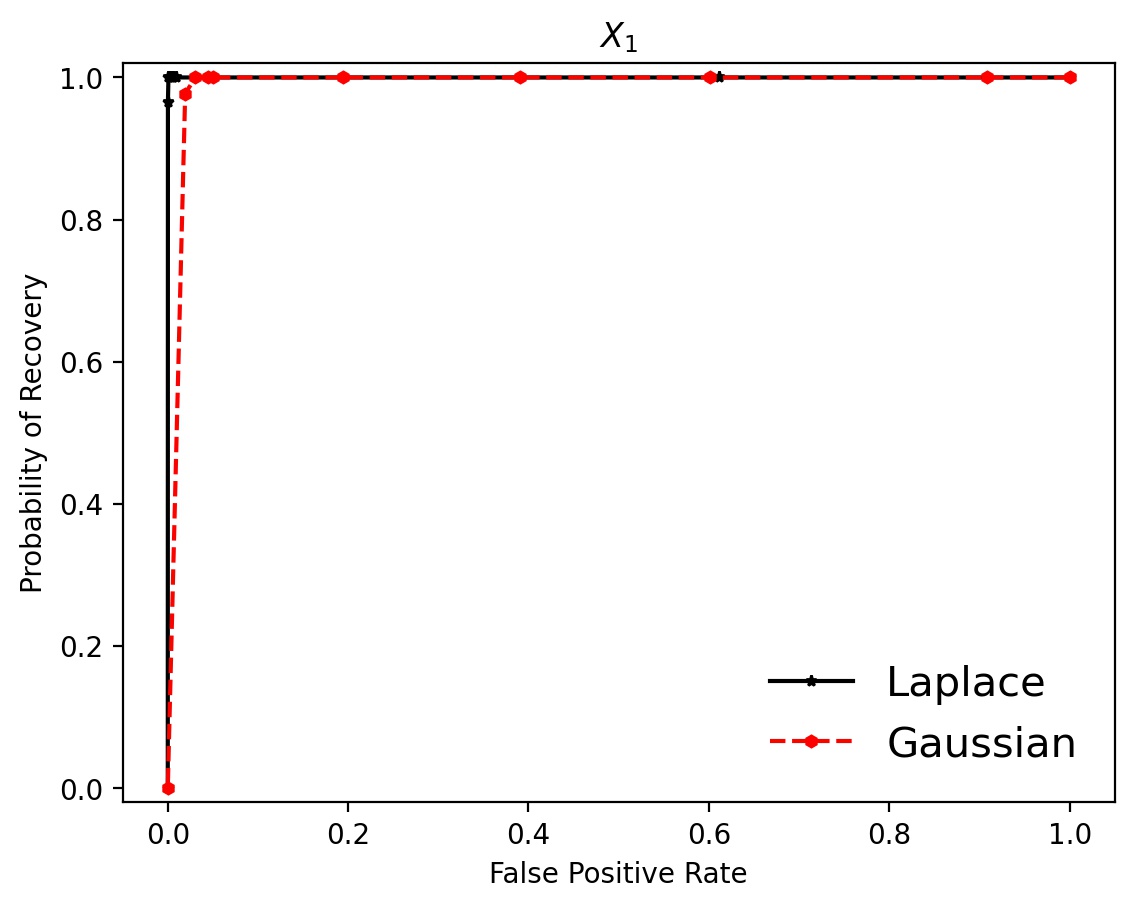}\includegraphics[width=.45\linewidth]{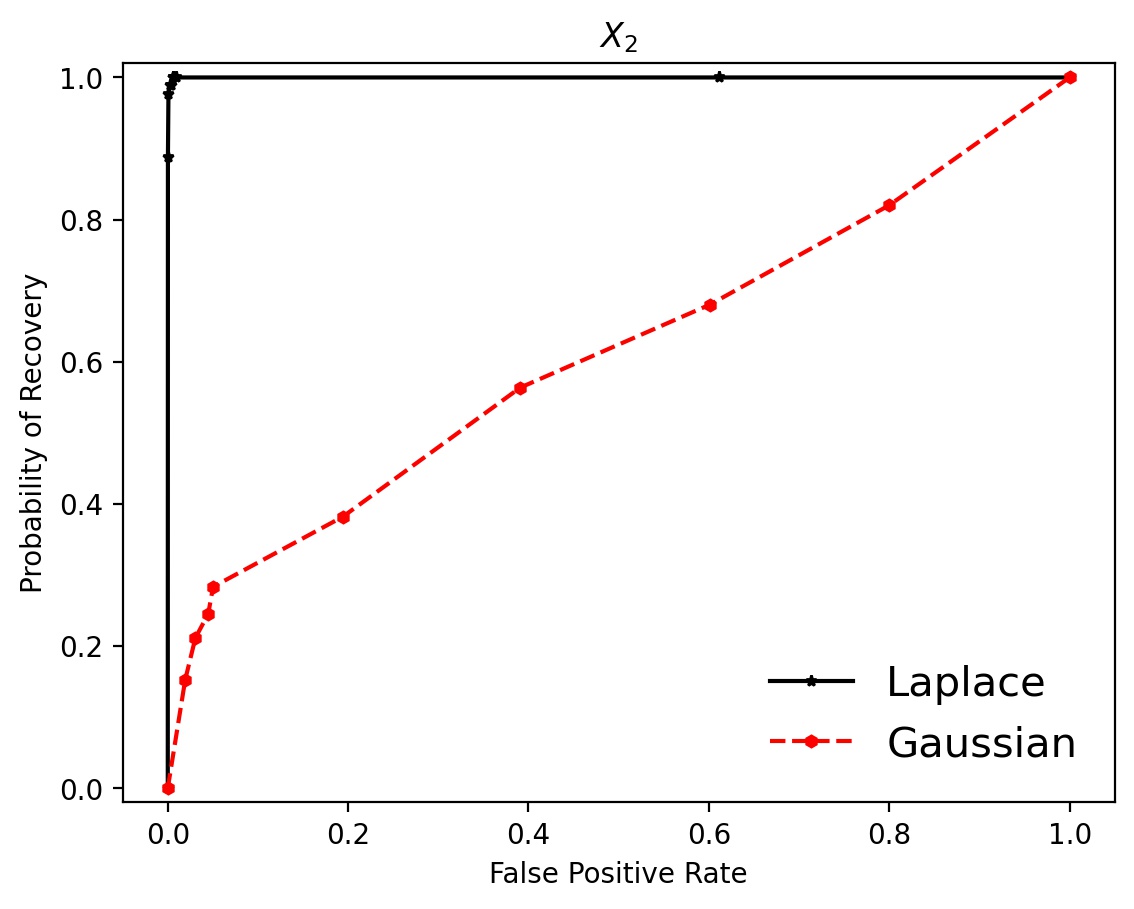}
\par\end{centering}
\caption{\small Probability of recovering a true variable against the false positive rate
in the main effect model $Y = X_1 + (X_2^2-1) +\normal(0, \sigma^2)$.
Here $\sigma^{2}=4$, $n = p = 1000$ and $\lambda = 0.01$.  To generate the ROC curve, $\gamma$ is varied over a grid of values: $(0, 0.002, 0.005, 0.01, 0.02, 0.05, 0.20, 0.6, 2.0)$.
}
\label{fig:l1-vs-l2}
\end{figure}

\subsection{Recovery of signals}
\label{sec:recovery-signal}

We investigate the power of the algorithm in recovering main effects and hierarchical interactions. We generate the data according to
\begin{equation*}
	Y = X_1 + X_1 X_2 + X_1 X_2 X_3 + \normal(0, \sigma^2)~~\text{where}~~X \sim \normal(0, I_p).
\end{equation*}
The variable $X_1$ is a main effect signal. The variables $X_2, X_3$ are level $2$ and $3$ signals respectively. 
For this experiment, we use an $\ell_1$ kernel. 

Figure~\ref{fig:hier} shows that the algorithm is able to detect high-order 
hierarchical interactions, though its power to detect interactions decreases as 
the level of the interaction increases. 

\begin{figure}[!tph]
\begin{centering}
\includegraphics[width=.6\linewidth]{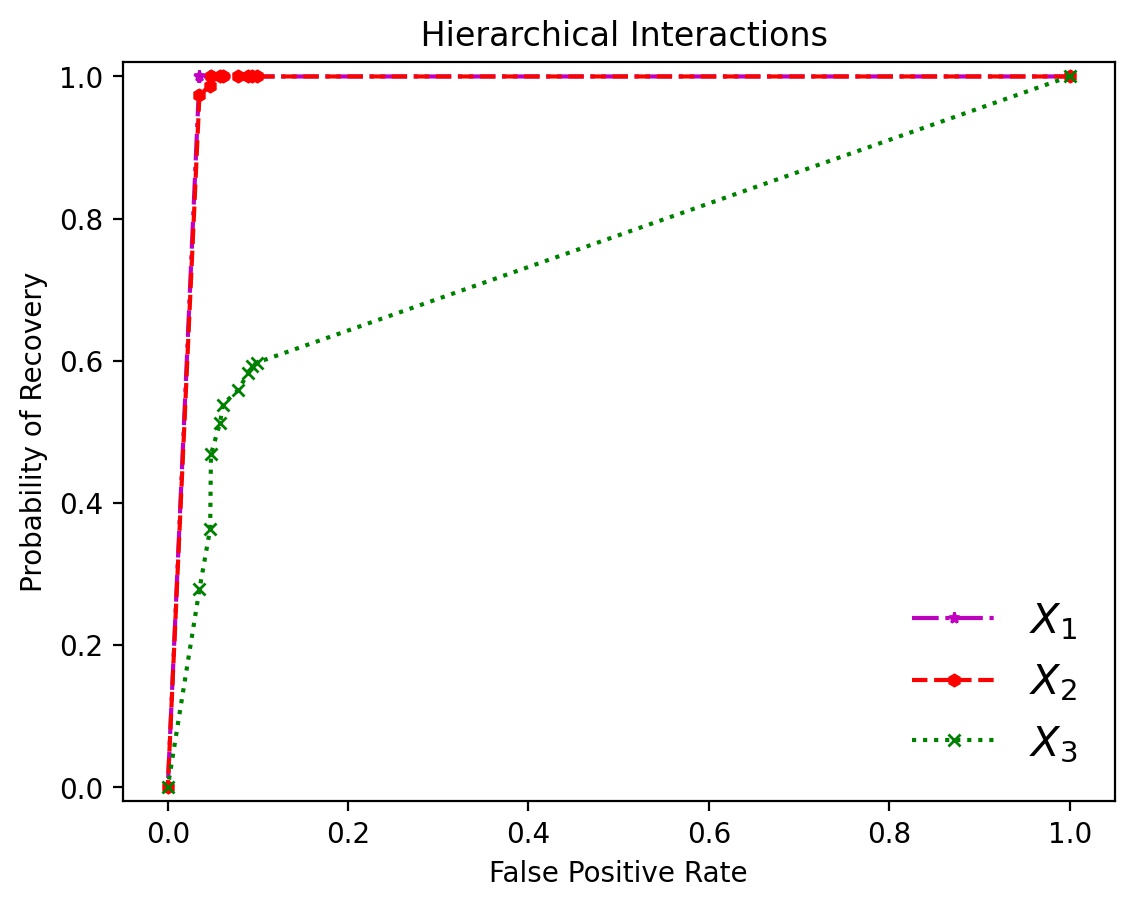}
\par\end{centering}
\caption{\small Probability of recovery of true variables against the false positive rate
 in the hierarchical interaction model $Y = X_1 +  X_1 X_2 + X_1 X_2 X_3 + \normal(0, \sigma^2)$. 
% The $y$-axis is the probability of recovery of the true variable $X_1, X_2, X_3$,
%and the $x$-axis is the probability of false discoveries.
Here $\sigma^{2}=1$, $n = p = 1000$ and $\lambda = 0.01$. To generate the ROC curve, $\gamma$ is varied over a grid of values: $(0, 0.002, 0.005, 0.01, 0.02, 0.05, 0.2, 0.5, 1.0)$. 
}
\label{fig:hier}
\end{figure}

\section{Discussion}
While kernel feature selection is a standard methodology for variable selection 
in nonparametric statistics---one which has been deployed in numerous applied
problems---there has been little statistical theory to support the methodology.
A core challenge is that the methodology is based on a nonconvex optimization
problem.  Progress has been made in studying the statistical properties of the 
global minima of the objective function, but there is a mismatch between such 
analyses and practice, given that the gradient-based methods available for
high-dimensional nonconvex optimization are only able to find local minima.

We have accordingly studied the landscape associated with kernel feature selection,
focusing on its local minima.  We have shown that the design of the kernel is 
crucial if methods that find local minima are to succeed in the task of feature
selection.  In particular, we have shown that the choice of $\ell_1$ kernel 
eliminates bad stationary points that may trap gradient descent.  We have
established this result via the development of novel techniques that may have
applications to a range of other kernel-based algorithms.

\newpage

\appendix

\section{A Roadmap to the Appendix}

%The organization of the Appendix is as follows. 
\begin{itemize}
\item Section~\ref{sec:preliminary-appendix} sets up the basics of kernel ridge regressions. %indexed by $\beta$. 
\item Section~\ref{sec:analysis-of-the-gradient} computes the gradient $\grad \obj(\beta)$. 
\item Section~\ref{sec:Lipschitzness-boundedness-of-gradient} shows the gradient $\grad \obj(\beta)$ is 
	Lipschitz and uniformly bonded. 
\item Section~\ref{sec:landscape-result} establishes the landscape result in Section~\ref{sec:distinction-l-1-l-2} of the main text. 
\item Section~\ref{sec:proof-main-theorems} establishes the population-level guarantees of the 
	kernel feature selection. 
\item Section~\ref{sec:proof-of-concentration-results-details} proves concentration results in the main text. 
\item Section~\ref{sec:proof-main-corollaries} establishes the finite-sample guarantees of the 
	kernel feature selection. 
\item Section~\ref{sec:proof-under-dependent-covariates} shows that---in a broader situation than considered 
	in the main text---kernel feature selection is able to recover the \emph{Markov blanket}.  
\item Section~\ref{sec:proof-of-section-of-preliminary-results} proves the results in
	Section~\ref{sec:preliminary-main-text} of the main text.
\item Section~\ref{sec:basics} gives basics on RKHS, functional analysis, concentration and optimization. 
\end{itemize}

\section{Preliminaries}
\label{sec:preliminary-appendix}
\label{sec:KRR-beta}
This section establishes the foundational properties of the family of kernel ridge regressions
considered in the main text (with parameter $\beta$). 
 \begin{equation}
\label{eqn:KRR-beta}
\begin{split}
\KRR(\beta):~
\minimize_{f \in \H}~~~&\energy(\beta, f) \\
~~\text{where}~~~~&\energy(\beta, f) = \half \E[(Y - f(\beta^{1/q} \odot X))^2] + \frac{\lambda}{2} \norm{f}_{\H}^2.
\end{split}
\end{equation}
Let $f_\beta$ denote the minimizer and $\obj(\beta)$ denote the minimum value of $\KRR(\beta)$. 
\begin{itemize}
\item We characterize $f_\beta$ and $\obj(\beta)$ using tools from functional analysis~\cite{FukumizuBaJo04,FukumizuBaJo09}. 
\item We prove bounds on the solution $f_\beta$ and on the residual $r_\beta$.  
\item We prove continuity of the mapping $\beta \mapsto f_\beta$, $\beta \mapsto r_\beta$ and $\beta \mapsto \obj(\beta)$. 
\end{itemize}
%Proof techniques are built from functional analytic tools in the literature~\cite{FukumizuBaJo04,FukumizuBaJo09}. 

\subsection{Characterization of $f_\beta$ and $\obj(\beta)$}
Recall the following definitions (Definition~\ref{definition:cross-covariance-operator}). 
\begin{itemize}
\item For each $\beta \ge 0$, the cross covariance operator $\Sigma_\beta: \H \mapsto \H$ is the mapping %the operator that satisfies
	\begin{equation*}
		(\Sigma_\beta f)(\cdot) = \E\left[k(\beta^{1/q}\odot X, \cdot)f(\beta^{1/q} \odot X)\right].
	\end{equation*} 
\item For each $\beta \ge 0$, the covariance function $h_\beta \in \H$ is $h_\beta(\cdot) = \E[k(\beta^{1/q}\odot X, \cdot)Y]$. 
\end{itemize}
Proposition~\ref{prop:kernel-beta-fix} characterizes the minimum $f_\beta$ and the minimum 
value $\obj_\beta$ of $\KRR(\beta)$. 
\begin{proposition}
\label{prop:kernel-beta-fix}
The minimum solution $f_\beta$ of the problem $\KRR(\beta)$ can be represented by 
\begin{equation}
\label{eqn:f-beta-representation}
f_\beta = (\Sigma_\beta + \lambda I)^{-1} h_\beta. 
\end{equation}
The minimum value $\obj(\beta)$ of the problem $\KRR(\beta)$ can be written as 
\begin{equation}
\label{eqn:minimum-beta-representation}
\obj(\beta) = \half \E[Y^2] - \langle h_\beta, (\Sigma_\beta + \lambda I)^{-1} h_\beta \rangle_{\H}.
\end{equation}
\end{proposition}

\begin{proof}
By the reproducing property of the kernel $(x, x') \mapsto k(x, x')$, we have for any $f \in \H$
$
\langle \Sigma_\beta f, f\rangle_{\H} = \E[f(\beta^{1/q} \odot X)^2]$, and
$\langle h_\beta, f\rangle_{\H} = \E[f(\beta^{1/q} \odot X) Y]$. Hence
%\end{equation*}
we can re-write the objective into %of $\KRR(\beta)$ to be
%\begin{equation}
%\label{eqn:quadratic-form-of-energy}
	%\half \E[(Y - f(\beta \odot X))^2] + \frac{\lambda}{2} \norm{f}_{\H}^2.
		%\minimize_{f \in \H}~ \energy(\beta, f) ~~\text{where}~~
$			\energy(\beta, f) = \half \langle (\Sigma_\beta+ \lambda I) f, f\rangle_{\H}- \langle h_\beta, f\rangle_{\H} + \half \E[Y^2]$. 
%\end{equation}
Note that $f \mapsto \energy(\beta, f)$ is strongly convex w.r.t $\norm{\cdot}_\H$ topology as 
$\Sigma_{\beta}$ is non-negative on $\H$. Hence the minimum $f_\beta$ is unique, and 
satisfies $(\Sigma_\beta+ \lambda I) f = h_\beta$, i.e., 
equation~\eqref{eqn:f-beta-representation}. Formula~\eqref{eqn:minimum-beta-representation} now follows. 

\end{proof}

\begin{proposition}[Variational Representation]
\label{proposition:KKT-condition}
The following holds for any $g \in \H$: 
\begin{equation}
\label{eqn:KKT-cond}
	\E[r_\beta(\beta^{1/q} \odot X; Y) g(\beta^{1/q} \odot X)] =  \lambda \langle f_\beta, g\rangle_{\H}. 
\end{equation}
\end{proposition}
\begin{proof}
The result follows by taking a first variation of the functional $f \mapsto \energy(\beta, f)$. 
%Proposition~\ref{prop:kernel-beta-fix} shows that $(\Sigma_\beta+ \lambda I) f_\beta = h_\beta$.
%Let $g \in \H$. Take inner products with respect to $g$ on both sides. By definitions, the following 
%identities hold: 
%\begin{equation*}
%\begin{split}
%	\langle (\Sigma_\beta+ \lambda I) f_\beta, g\rangle_{\H} &= \E[g(\beta^{1/q} \odot X)f_\beta(\beta ^{1/q}\odot X) ] + \lambda \langle f_\beta, g\rangle_{\H}. \\
%	\langle h_\beta, g\rangle_{\H} &= \E[g(\beta^{1/q} \odot X) Y].
%\end{split}
%\end{equation*}
%Note then $r_\beta(\beta^{1/q} \odot x; y) = y - f_\beta(\beta^{1/q} \odot x)$ by definition. This 
%proves Proposition~\ref{proposition:KKT-condition}. %^ as desired. 
\end{proof}

\subsection{Bounds on $f_\beta$ and $r_\beta$} We bound the second moments of $f_\beta$ and $r_\beta$. 

\begin{proposition}
\label{prop:second-moment-bound}
The solution $f_\beta$ and the residual $r_\beta$ satisfy 
\begin{equation*}
	\normbigP{f_\beta(\beta^{1/q}\odot X)} \le  \normP{Y},~~~
		\normbigP{r_\beta(\beta^{1/q}\odot X; Y)} \le \normP{Y}.
\end{equation*}
\end{proposition}
\begin{proof}
By Proposition~\ref{proposition:KKT-condition}, we get
$\E[r_\beta(\beta^{1/q} \odot X; Y) f_\beta(\beta^{1/q} \odot X)] = \lambda \norm{f_\beta}^2 \ge 0$.
Note then $r_\beta(\beta^{1/q} \odot x; y) = y - f_\beta(\beta^{1/q} \odot x)$ by definition. Hence, we obtain
\begin{equation*}
	\normbigP{f_\beta(\beta^{1/q}\odot X)}^2 \le \E[Y f_\beta(\beta^{1/q}\odot X)] \le \normbigP{Y}
		 \normbigP{f_\beta(\beta^{1/q}\odot X)}.
\end{equation*}
This proves the first bound. The second bound can be deduced similarly. 
%part of Proposition~\ref{prop:second-moment-bound}. Similarly, we obtain
%\begin{equation*}
%	\normbigP{r_\beta(\beta^{1/q}\odot X; Y)}^2 \le \E[Y r_\beta(\beta^{1/q}\odot X; Y)] \le \normbigP{Y}
%		 \normbigP{r_\beta(\beta^{1/q}\odot X)}.
%\end{equation*}
%This proves the second part of Proposition~\ref{prop:second-moment-bound}.
\end{proof}

\subsection{Continuity of the Mappings: $\beta \mapsto f_\beta$, $\beta \mapsto r_\beta$ and $\beta \mapsto \obj(\beta)$.}
%Proposition~\ref{prop:continuity-f-beta} 
%and Proposition~\ref{prop:continuity-f-beta-beta-X} show how we can leverage these tools to show the continuity 
%of the mappings $\beta \mapsto f_\beta$ (and many others).
%The proof of Proposition~\ref{prop:continuity-f-beta} is given in Appendix~\ref{sec:proof-prop-continuity-of-beta}.

\begin{proposition}
\label{prop:continuity-f-beta}
We have the following results.%Then the following mappings are continuous: 
\begin{enumerate}[(a)]
\item The mapping $\beta \mapsto f_\beta$ is continuous w.r.t the norm topology in $\H$, i.e., 
	\begin{equation*}
		\lim_{\beta' \to \beta} \norm{f_{\beta'} -f_\beta}_{\H} = 0
	\end{equation*}
\item The mapping $\beta \mapsto f_\beta$ and $\beta \mapsto r_\beta$ is continuous w.r.t the sup norm $\norm{\cdot}_\infty$:
	\begin{equation*}
		\lim_{\beta' \to \beta} \norm{f_{\beta'} -f_\beta}_{\infty} = \lim_{\beta' \to \beta} \norm{r_{\beta'} - r_\beta}_{\infty} = 0
	\end{equation*}
	%where the sup norm is defined by $\norm{g}_{\infty} = \sup_{x \in \R^p} |g(x)|$ for any function $g: \R^p \to \R$. 	
\item The mapping $\beta \mapsto \obj(\beta)$ is continuous. 
\end{enumerate}
\end{proposition}

\begin{proof}
The key to the proof is to show that $\beta \mapsto h_\beta$ and $\beta \mapsto \Sigma_\beta$ are continuous 
\begin{equation}
\label{eqn:continuity-of-h-Sigma}
	\lim_{\beta' \to \beta} \norm{h_{\beta'} - h_\beta}_{\H} = 0~~\text{and}~~
	\lim_{\beta' \to \beta} \opnorm{\Sigma_{\beta'} - \Sigma_\beta} = 0.
\end{equation}
Deferring its proof to the end, we
first show why Proposition~\ref{prop:continuity-f-beta} follows from 
equation~\eqref{eqn:continuity-of-h-Sigma}.
\begin{enumerate}[(a)]
\item By Proposition~\ref{prop:kernel-beta-fix}, 
	$f_\beta = (\Sigma_\beta + \lambda I)^{-1} h_\beta$. Hence, 
	the mapping $\beta \mapsto f_\beta$ is contiuous w.r.t the norm topology in $\H$
	since $\beta \mapsto h_\beta$ and 
	$\beta \mapsto \Sigma_\beta$ are continuous by equation~\eqref{eqn:continuity-of-h-Sigma}.
\item Note that $\norm{g}_\infty \le |h(0)|^{1/2}\norm{g}_{\H}$ for any $g \in \H$. Indeed, by 
	reproducing property of $\H$: 
	\begin{equation*}
		\norm{g}_\infty = \sup_x |g(x)| = \sup_x|\langle k(x, \cdot), g \rangle_{\H}| 
			\le \sup_x k(x, x)^{1/2} \norm{g}_{\H} \le |h(0)|^{1/2} \norm{g}_{\H}.
	\end{equation*}
	Hence $\lim_{\beta' \to \beta} \norm{f_{\beta'} -f_\beta}_{\infty} = 0$ as a consequence of 
	part (a). Note $r_\beta - r_{\beta'} = f_{\beta'} - f_\beta$. 
\item By Proposition~\ref{prop:kernel-beta-fix}, $\obj(\beta) = \half \E[Y^2] - \langle h_\beta, 
	(\Sigma_\beta + \lambda I)^{-1} h_\beta \rangle_{\H}$. The continuity of $\beta \mapsto \obj(\beta)$ 
	follows easily as a consequence of the fact that both 
	$\beta \mapsto h_\beta$ and $\beta \mapsto \Sigma_\beta$ are continuous. 
\end{enumerate}
It remains to prove equation~\eqref{eqn:continuity-of-h-Sigma}. The key is to notice
the  identities below (similar ones appear in the literature~\cite{FukumizuBaJo04, GrettonBoSmSc05}): 
letting $(X', Y')$ be independent copies of $(X, Y)$, 
{\small
\begin{align*}
		\norm{h_{\beta'} - h_{\beta}}_{\H}^2  %\nonumber \\
		&= \E[(k(\beta^{1/q} \odot X, \beta^{1/q} \odot X')  + k(\beta'^{1/q} \odot X, \beta'^{1/q} \odot X') - 2k(\beta^{1/q} \odot X, \beta'^{1/q} \odot X')) YY'] \\
		%	\label{eqn:evaluation-of-h-beta-var}\\
		\norm{\Sigma_{\beta'} - \Sigma_{\beta}}_{\HS}^2
		&= \E[k(\beta^{1/q} \odot X, \beta^{1/q} \odot X')^2 + k(\beta'^{1/q} \odot X, \beta'^{1/q} \odot X')^2 - 2 k(\beta^{1/q} \odot X, \beta'^{1/q} \odot X')^2]	
		%\label{eqn:evaluation-of-Sigma-beta-var}
\end{align*}
}
Note also $\opnorm{\Sigma} \le \norm{\Sigma}_{\HS}$ for any operator $\Sigma$.
As a result, equation~\eqref{eqn:continuity-of-h-Sigma} follows from the above identities and the 
fact that $(x, x') \mapsto k(x, x')$ is  continuous and  uniformly bounded. 
\end{proof}

\begin{proposition}
\label{prop:continuity-f-beta-beta-X}
The mapping $\beta \mapsto r_\beta(\beta^{1/q} \odot \cdot; \cdot)$ is continuous w.r.t the norm $\normP{\cdot}$: 
\begin{equation*}
	\lim_{\beta' \mapsto \beta} \normbigP{r_\beta(\beta^{1/q} \odot X; Y) - r_{\beta'}(\beta'^{1/q} \odot X; Y)} = 0.
\end{equation*}
\end{proposition}

\begin{proof}
It suffices to prove that $\beta \mapsto f_\beta(\beta^{1/q} \odot \cdot)$ is continuous w.r.t $\norm{\cdot}_\infty$ topology.
%\begin{equation*}
%	\lim_{\beta' \to \beta} \normbig{f_{\beta'}(\beta'^{1/q} \odot \cdot) - f_\beta(\beta^{1/q} \odot \cdot)}_\infty = 0.
%\end{equation*}
This is true since (i) $f_{\beta}(\beta^{1/q} \odot x) = \langle f_\beta, k(\beta^{1/q} \odot x, \cdot)\rangle_{\H}$,
$(ii)$ $\beta \mapsto f_\beta$ is continuous w.r.t $\norm{\cdot}_\H$, 
and $(iii)$ $\beta \mapsto k(\beta^{1/q} \odot x, \cdot)$ is uniformly (uniform w.r.t $x$) continuous w.r.t $\norm{\cdot}_\H$ topology. 
\end{proof}

\section{Computation of the Gradient $\grad \obj(\beta)$: Proof of Proposition~\ref{proposition:compute-grad-obj}}
\label{sec:analysis-of-the-gradient}
%\subsection{Characterization of the Gradient: Proof of Proposition~\ref{proposition:compute-grad-obj}}
This section substantiates the proof of Proposition~\ref{proposition:compute-grad-obj} in the main text.

\subsection{Notation.}
Recall the objective function 
\begin{equation*}
	\energy(\beta, f) =  \half ~\E[(Y - f(\beta^{1/q} \odot X))^2] + \frac{\lambda}{2} \norm{f}_{\H}^2.
\end{equation*}
To facilitate the proof, we perform a change of variable. Introduce the auxiliary objective 
\begin{equation}
\label{eqn:represent-auxiliary-energy-analytic}
\begin{split}
	\wtilde{\energy}(\beta, f) &=  \half~ \E[(Y - f(X))^2] + \frac{\lambda}{2} \norm{f}_{\H_\beta}^2, \\
	&=  \half \E[(Y - f(X))^2] + \frac{\lambda}{2 (2\pi)^p} \int |\F(f)(\omega)|^2 \cdot \frac{1}{Q_\beta(\omega)} d\omega
\end{split}
\end{equation}
where $\H_\beta$ is the RKHS with kernel $(x, x') \mapsto k(\beta^{1/q} \odot x; \beta^{1/q} \odot x')$
(Section~\ref{sec:H-beta-spaces} gives details on the construction of $\H_\beta$). A simple consequence 
of Proposition~\ref{proposition: Hilbert-beta} yields for $\beta > 0$
%and $\H_\beta = \{f(\beta^{1/q} \odot \cdot) | f\in \H\}$ (Proposition~\ref{proposition: Hilbert-beta}), %this shows that for any 
%$\beta > 0$
\begin{equation}
\label{eqn:auxiliary-energy-variation}
	\obj(\beta) = \min_{f \in \H} \energy(\beta, f) %= \min_{f \in \H} \wtilde{\energy}(\beta, f(\beta \odot \cdot))
		= \min_{f \in \H_\beta} \wtilde{\energy}(\beta, f).
\end{equation}
%As a result, equation~\eqref{eqn:auxiliary-energy-variation} gives a new variational representation of $\obj(\beta)$. 
Let $\wtilde{f}_\beta \in \H_\beta$ denote the minimum of 
$f \mapsto \wtilde{\energy}(\beta, f)$ so that $J(\beta) =  \wtilde{\energy}(\beta, \wtilde{f}_\beta)$. 
%Finally, Proposition~\ref{proposition:norm-of-H} gives the following analytical representation:
%\begin{equation}
%\label{eqn:represent-auxiliary-energy-analytic}
%	\wtilde{\energy}(\beta, f)  
%\end{equation}
%Above $\beta \mapsto Q_\beta(\omega)$ is an analytic function of $\beta > 0$ for any fixed $\omega$. 

\subsection{Main Proof.}
\indent\indent
Lemma~\ref{lemma:envelope-theorem} gives an initial analytic representation of 
$\grad \obj(\beta)$.
% whose proof is given in Section~\ref{sec:proof-lemma-envelope-theorem}. 
%Recall that $\wtilde{f}_\beta$ is the minimizer of $f \to \wtilde{\energy}(\beta, f)$.
\begin{lemma}
\label{lemma:envelope-theorem}
For any $\beta > 0$, $\grad \obj(\beta)$ exists, and satisfies
\begin{equation}
\label{eqn:formula-gradient-first-step}
	\grad  \obj(\beta) =  \frac{\lambda}{(2\pi)^p} \cdot\int |\F(\wtilde{f}_\beta)(\omega)|^2 \cdot \grad_\beta \left(\frac{1}{Q_\beta(\omega)}\right) d\omega.
\end{equation}
\end{lemma}
As explained in the main text, %a rigorous derivation of Lemma~\ref{lemma:envelope-theorem} is not easy. 
in order for one to apply the ``envelope theorem'' to the variational formula $\obj(\beta) =  
\min_{f \in \H_\beta} \wtilde{\energy}(\beta, f)$. %in conjunction with
%the analytic formula of the functional $\wtilde{\energy}(\beta, f)$ in equation~\eqref{eqn:represent-auxiliary-energy-analytic}. 
%To rigorously show equation~\eqref{eqn:formula-gradient-first-step},
one needs to establish that the solution $\wtilde{f}_\beta$ is a sufficiently \emph{smooth} function (e.g., it 
is sufficiently smooth so that the RHS of equation~\eqref{eqn:formula-gradient-first-step} exists). To 
achieve this goal, Lemma~\ref{lemma:KKT-condition} is crucial. Write $\wtilde{r}_\beta(x; y) = y - \wtilde{f}_\beta(x)$.

%certain \emph{smoothness} properties of , which basically require showing that the Fourier transform $\F(\wtilde{f}_\beta)(\omega)$ has a fast decay. 
%For instance, it is necessary to show that the integral on the RHS of equation~\eqref{eqn:formula-gradient-first-step} exists, which is equivalent
%to showing a certain rate of decay of the Fourier transform $\F(\wtilde{f}_\beta)(\omega)$---yet this is not trivially implied by the fact that 
%$\wtilde{f}_\beta \in \H_\beta$. The establishment of the so called improved \emph{smoothness} properties of the 
%solution $\wtilde{f}_\beta$ often requires peculiar analytic techniques~\cite{Hormander07}. 

%Along the path in proving Lemma~\ref{lemma:envelope-theorem}, Lemma~\ref{lemma:KKT-condition} is crucial. Basically, 
%Lemma~\ref{lemma:KKT-condition} gives a characterization of $\F(\wtilde{f}_\beta)(\omega)$ using the residual 
%$\wtilde{r}_\beta(X; Y) = Y - \wtilde{f}_\beta(X)$. 
\begin{lemma}
\label{lemma:KKT-condition}
Let $\beta > 0$. The identity below holds for almost all $\omega$ (w.r.t Lebesgue measure)
\begin{equation}
\label{eqn:KKT-feature-bar}
	\frac{1}{(2\pi)^{p/2}} \cdot \F(\wtilde{f}_\beta)(\omega) 
		= \frac{1}{\lambda} \cdot \E\left[\wtilde{r}_\beta(X; Y) e^{i \langle \omega, X \rangle}\right] \cdot Q_\beta(\omega).
\end{equation}
Equivalently, the following identity holds for almost all $\omega$ (w.r.t Lebesgue measure)
\begin{equation}
\label{eqn:KKT-feature-bar-o}
	\frac{1}{(2\pi)^{p/2}} \cdot \F(f_\beta)(\omega) = \frac{1}{\lambda} \cdot 
		\E\left[r_\beta(\beta^{1/q}\odot X; Y) e^{i \langle \omega, \beta^{1/q} \odot X \rangle}\right] \cdot Q(\omega).
\end{equation}
\end{lemma}

\begin{remark}
The main idea to prove Lemma~\ref{lemma:KKT-condition} is to use the characterization of 
$\wtilde{f}_\beta$: 
\begin{equation*}
	\E[\wtilde{r}_\beta(X; Y)\wtilde{g}(X)] =  \lambda \langle \wtilde{f}_\beta, \wtilde{g}\rangle_{\H_\beta}
		~~\text{for all $g \in \H_\beta$}
\end{equation*}
which can be derived by taking the first order variation of the objective $f \mapsto \wtilde{\energy}(\beta, f)$.
We then wish to substitute $\wtilde{g}(x) = e^{i\omega^T x}$ and obtain Lemma~\ref{lemma:KKT-condition}.
The challenge that remains is that the complex basis function $x \mapsto e^{i\omega^T x}$ does not belong to $\H_\beta$.
As a result, we apply the mollifier trick---common in harmonic analysis---to overcome this technical issue. 
\end{remark}

%\begin{remark}
%The main idea to prove Lemma~\ref{lemma:KKT-condition} is to use the KKT characterization of 
%$\wtilde{f}_\beta$
%
%substitute the 
%
%
%
%Lemma~\ref{lemma:KKT-condition} is also the foundation to establish the improved smoothness properties of $\wtilde{f}_\beta$ 
%(see how Lemma~\ref{lemma:KKT-condition} 
%is used in the proof of Lemma~\ref{lemma:envelope-theorem}). The basic idea to ``derive'' Lemma~\ref{lemma:KKT-condition} is 
%simple---all we need is to substitute the function $x \mapsto e^{i\omega^T x}$ into the KKT characterization of the minimum $\wtilde{f}_\beta$ and 
%do some heuristic calculation to ``obtain'' the result. However, the technicality is that the function 
%$e^{i\omega^T x}$ does not belong to $\H_\beta$, and thus this approach is in fact mathematically invalid. As a result, we adopt a 
%perturbation argument based on the classical mollifier trick in Fourier analysis to overcome this technical issue. 
%\end{remark}

Back to the proof of Proposition~\ref{proposition:compute-grad-obj}.
By Lemma~\ref{lemma:envelope-theorem} and Lemma~\ref{lemma:KKT-condition}, we obtain for $\beta > 0$
\begin{equation}
\label{eqn:grad-eval-beta->0}
\begin{split}
\grad \obj(\beta) &=  \frac{\lambda}{(2\pi)^p} \int |\F(\wtilde{f}_\beta)(\omega)|^2 \cdot \grad_\beta \left(\frac{1}{Q_\beta(\omega)}\right) d\omega. \\
&= - \frac{1}{\lambda} \cdot 
	\int \left|\E[\wtilde{r}_\beta(X; Y) e^{i \omega^T X}]\right|^2  \cdot \grad_\beta Q_{\beta}(\omega) d\omega \\
%&=  \frac{1}{\lambda} \cdot \int \E[\bar{r}_\beta(X) \bar{r}_\beta(X')e^{i \langle\omega, X-X' \rangle}]  \cdot \nabla_\beta Q_{\beta}(\omega) d\omega \\
&=  - \frac{1}{\lambda} \cdot 
		\E\left[\wtilde{r}_\beta(X; Y) \wtilde{r}_\beta(X'; Y')\int  e^{i \langle\omega, X-X' \rangle} \cdot \nabla_\beta Q_{\beta}(\omega) d\omega\right]  
\end{split}
\end{equation}
The next lemma evaluates the integral inside the expectation. 
\begin{lemma}
\label{lemma:technical}
For all $\beta > 0$, we have the identity that holds for $l \in [p]$: 
\begin{equation*}
h^\prime(\norm{X-X'}_{q, \beta}^q) \cdot |X_l - X_l'|^q = 
	 \int e^{i \langle\omega, X-X' \rangle} \cdot (\partial_{\beta_l}  Q_{\beta}(\omega)) d\omega.
\end{equation*}
\end{lemma}
By Lemma~\ref{lemma:technical} and equation~\eqref{eqn:grad-eval-beta->0}, we obtain for all $\beta > 0$ and $l \in [p]$
\begin{equation}
\label{eqn:grad-form}
\begin{split}
	(\grad \obj(\beta))_l &= -\frac{1}{\lambda} \cdot \E\left[\wtilde{r}_\beta(X; Y) \wtilde{r}_\beta(X'; Y')h^\prime(\norm{X-X'}_{q, \beta}^q) |X_l - X_l'|^q \right]  \\
		 & = -\frac{1}{\lambda} \cdot \E\left[r_\beta(\beta^{1/q}\odot X; Y)r_\beta(\beta^{1/q}\odot X'; Y') h^\prime(\norm{X-X'}_{q, \beta}^q) |X_l - X_l'|^q\right].
\end{split}
\end{equation}
%In above, we use $\wtilde{r}_\beta(x; y) = r_\beta(\beta^{1/q} \odot x; y)$ since $\wtilde{f}_\beta(x) = f_\beta(\beta^{1/q} \odot x)$. 

To extend the result from positive $\beta > 0$ to 
non-negative $\beta \ge 0$, we use Lemma~\ref{lemma:derivative-limit-math-analysis}.% exercise book. 
\begin{lemma}
\label{lemma:derivative-limit-math-analysis}
Let $F: \R_+ \mapsto \R$ be continuous. Suppose $F$ is differentiable on $x > 0$, and $\lim_{x \to 0^+} F'(x)$ exists. Then, 
$F'_+(0)$ exists and $F'_+(0) = \lim_{x \to 0^+} F'(x)$.
\end{lemma}
Recall that $\beta \mapsto \obj(\beta)$ is continuous for $\beta \ge 0$ (Proposition~\ref{prop:continuity-f-beta}). 
Also, the mapping 
%(cf. the RHS of equation~\eqref{eqn:grad-form})
\begin{equation*}
\beta \mapsto -\frac{1}{\lambda} \cdot \E\left[r_\beta(\beta^{1/q}\odot X; Y)r_\beta(\beta^{1/q}\odot X'; Y') h'(\norm{X-X'}_{q, \beta}^q)|X_l- X_l'|^q\right]
\end{equation*}
is continuous for $\beta \ge 0$ (Proposition~\ref{prop:second-moment-bound} and~\ref{prop:continuity-f-beta-beta-X}). 
Hence, equation~\eqref{eqn:grad-form} holds for all $\beta \ge 0$. %We leave the details for the interested readers. 

\newcommand{\convv}{*}
\newcommand{\one}{\mathbf{1}}
\newcommand{\loc}{{\rm loc}}

%\subsection{Proof of Technical Lemma}
\subsection{Proof of Lemma~\ref{lemma:KKT-condition}.}
It suffices to prove equation~\eqref{eqn:KKT-feature-bar}. Note that 
equation~\eqref{eqn:KKT-feature-bar-o} follows by a change 
of variable (by substituting $\beta^{1/q} \odot \omega$ into $\omega$ in equation~\eqref{eqn:KKT-feature-bar}).
%and use the identity (i) $\wtilde{r}_\beta(x; y)= r_\beta(\beta^{1/q}\odot x; y)$, (ii) $\F(\wtilde{f}_\beta)(\omega) 
%= \left(\prod_i \beta_i\right)^{-1/q} \cdot \F(f_beta)(\beta^{-1/q} \odot \omega)$, and (iii) 
%$Q_\beta(\omega) = \left(\prod_i \beta_i\right)^{-1/q} \cdot Q(\beta^{-1/q} \odot \omega)$). 

Below we prove equation~\eqref{eqn:KKT-feature-bar}.
By taking the first order variation of $f \mapsto \wtilde{\energy}(\beta, f)$, 
we obtain the following characterization of $\wtilde{f}_\beta$: for all $\wtilde{g} \in \H_\beta$
\begin{equation}
\label{eqn:KKT-cond-bar}
	\E[\wtilde{r}_\beta(X; Y)\wtilde{g}(X)] =  \lambda \langle \wtilde{f}_\beta, \wtilde{g}\rangle_{\H_\beta}. 
\end{equation}
As a result, using Corollary~\ref{corrolary:norm-of-H} to expand the RHS, this proves that for all 
functions $\wtilde{g} \in \H_\beta$: 
\begin{equation}
\label{eqn:KKT-cond-bar-analytic}
	\E[\wtilde{r}_\beta(X; Y)\wtilde{g}(X)] =  
		\frac{\lambda}{(2\pi)^{p}} \int \frac{\F(\wtilde{f}_\beta)(\omega') \F(\wtilde{g})(\omega')}{Q_\beta(\omega')} d\omega'.
\end{equation}

To motivate the rest of the proof, we first give a quick \emph{heuristic} derivation of Lemma~\ref{lemma:KKT-condition}.
Let $\wtilde{g}_{\omega}(x) = e^{i \omega^T x}$.
The idea is to substitute $\wtilde{g} = \wtilde{g}_{\omega}$ into equation~\eqref{eqn:KKT-cond-bar-analytic}. 
To compute the RHS, $\F(\wtilde{g}_{\omega})(\cdot)= (2\pi)^{p/2} \delta_{\omega}(\cdot)$ where 
$\delta_\omega$ is the $\delta$-function centered at $\omega$. This gives  
\begin{equation*}
	\E\left[\wtilde{r}_\beta(X; Y) e^{i\omega^T x}\right] 
		= \frac{\lambda}{(2\pi)^p}\int \frac{\F(\wtilde{f}_\beta)(\omega') \F(\wtilde{g}_\omega)(\omega')}{Q_\beta(\omega')} d\omega'
		= \frac{\lambda}{(2\pi)^{p/2}} \cdot \frac{\F(\wtilde{f}_\beta)(\omega)}{Q_\beta(\omega)}. 
\end{equation*}
Of course, the above derivation is not rigorous. The function $\wtilde{g}_{\omega}(x) = e^{i \omega^T x}$ 
lacks regularity and does not belong to $\H_\beta$. To obtain a rigorous treatment, we need to smooth 
$\wtilde{g}_{\omega}$ and borrow the regularity $\E[Y^2] < \infty$ to overcome this technical issue.% (see step 1 below). 

Below is the rigorous derivation. Define, for any $\eps > 0$, the function $\wtilde{g}_{\omega, \eps}$ by
\begin{equation*}
	\wtilde{g}_{\omega, \eps}(x) = \wtilde{g}_{\omega}(x) \cdot k_{\eps}(x)~~\text{where $k_\eps(x) = k(\eps x)$ for $k(x) = \prod_{i=1}^p (\sin(x_i)/x_i)$}.
\end{equation*}
Note that (i) $\F(k)$ is compactly supported (ii) $k$ is uniformly bounded: 
$\sup_x |k(x)| < \infty$. 
Since $\F(k)$ is compactly supported, $\wtilde{g}_{\omega, \eps} \in \H_\beta$ for any $\eps > 0$ by 
Corollary~\ref{corrolary:norm-of-H}. Additionally, %we have the identity: 
\begin{equation*}
	\F(\wtilde{g}_{\omega, \eps})(\cdot) = \F(k_{\eps})(\cdot-\omega)~~\text{where $\F(k_{\eps})(\omega') = 
		(2\pi)^{p/2} \cdot \prod_{i=1}^p (\one_{w_i' \in [-\eps, \eps]}/(2\eps)) $}.
\end{equation*}
Substitute $\wtilde{g}_{\omega, \eps}$ into $\wtilde{g}$ in 
equation~\eqref{eqn:KKT-cond-bar-analytic}. This yields the identity that holds for all $\eps > 0$
\begin{equation}
\label{eqn:before-taking-eps-to-0}
\begin{split}
	\E[\wtilde{r}_\beta(X; Y)e^{i\omega^T X} k_\eps(X)] %&=  \lambda \langle \wtilde{f}_\beta, \wtilde{g}_{\omega, \eps}\rangle_{\H_\beta} \\
		%&= \frac{\lambda}{(2\pi)^{p}} \int \frac{\F(\wtilde{f}_\beta)(\omega') \wbar{\F(k_{\eps})(\omega-\omega')}}{Q_\beta(\omega')} d\omega' \\
		&= \frac{\lambda}{(2\pi)^{p}} \cdot \left(\bigg(\frac{\F(\wtilde{f}_\beta)}{Q_\beta}\bigg) \convv \F(k_{\eps})\right)(\omega).
\end{split}
\end{equation}
Now we take $\eps \to 0^+$ on both sides. We shall show that will yield Lemma~\ref{lemma:KKT-condition}.
\begin{enumerate}
\item Take the limit $\eps \to 0^+$ on the LHS of equation~\eqref{eqn:before-taking-eps-to-0}:
	\begin{equation*}
		\lim_{\eps \to 0^+} \E\left[\wtilde{r}_\beta(X; Y)e^{i\omega^T X} k_\eps(X)\right]  = \E\left[\wtilde{r}_\beta(X; Y)e^{i\omega^T X}\right].
	\end{equation*}
	This follows from the dominated convergence theorem: (i) $\lim_{\eps \to 0^+} k_\eps(x) = 1$ 
	(ii) $\E[|[\wtilde{r}_\beta(X; Y)|] \le (\E[Y^2])^{1/2} < \infty$ by Proposition~\ref{prop:second-moment-bound}
	and (iii) $\sup_{x}|k_\eps(x)| = \sup_x |k(x)| < \infty$.
\item Take the limit $\eps \to 0^+$ on the RHS of equation~\eqref{eqn:before-taking-eps-to-0}:
	\begin{equation}
	\label{eqn:local-integrable-converge}
		\lim_{\eps \to 0^+}  \left(\bigg(\frac{\F(\wtilde{f}_\beta)}{Q_\beta}\bigg) \convv \F(k_\eps)\right)(\omega) \to 
			(2\pi)^{p/2} \cdot \frac{\F(\wtilde{f}_\beta)}{Q_\beta}(\omega)~~\text{a.e.-$\omega$}.
	\end{equation}
	This follows by applying Lebesgue's almost-everywhere differentiable theorem to the locally 
	integrable function $ \omega \mapsto \frac{\F(\wtilde{f}_\beta)}{Q_\beta}$. To see why it is
	locally integrable, we note that: (i) the function $\omega \mapsto Q_\beta(\omega)$ is positive and continuous
	(ii) $\F(f_\beta)$ is integrable since 
	\begin{equation*}
	\begin{split}
		\left(\int  |\F(\wtilde{f}_\beta(\omega))| d\omega \right)^2
			&\le \int \frac{|\F(\wtilde{f}_\beta(\omega))|^2}{Q_\beta(\omega)}  d\omega \cdot 
			 \int Q_\beta(\omega) d\omega 
			= (2\pi)^{p} |h(0)| \normbig{\wtilde{f}_\beta}_{\H}^2 < \infty. 
	\end{split}
	\end{equation*}
%	As a result, equation~\eqref{eqn:local-integrable-converge} follows by applying Lebesgue's almost-everywhere %differentiable 
%	theorem to the locally integrable function $\frac{\F(\wtilde{f}_\beta)}{Q_\beta}\in L_1^{\loc}(\R^p)$
\end{enumerate}

\subsection{Proof of Lemma~\ref{lemma:technical}.}
Our starting point is the following identity: for any $\beta > 0$
\begin{equation*}
h(\norm{x-x'}_{q, \beta}^q) =  \int e^{i \langle \omega, \beta^{1/q} \odot (x-x')\rangle} Q(\omega) d\omega
	=   \int e^{i \langle \omega, x-x'\rangle} Q_\beta(\omega) d\omega.
\end{equation*}
Take partial derivative $\partial_{\beta_l}$ on both sides. We wish to apply Lebesgue's dominated convergence
theorem to exchange the integral and the derivative operations. This requires a careful check of regularity conditions. 
We divide our discussions based on the value of $q$. 
%Recall that in the case of $q = 2$, we assume that $\supp(\mu)$ is away from $0$. Denote $m_\mu > 0$ to be such that 
%$\supp(\mu) \subseteq [m_\mu, \infty)$. 
%This requires the check of conditions. 

\begin{itemize}
\item Case $q = 1$. In this case,one can show that the following bound holds for all $\omega$: 
\begin{equation*}
	|\partial_{\beta_l} Q_\beta(\omega)| \le \frac{1}{\beta_l} |Q_\beta(\omega)|.
\end{equation*}
Note that $\sup_{\beta\in B} Q_{\beta}(\omega)(1+\omega_l^2)$ is integrable for compact $B$
which does not contain $0$. 
%Note that $\sup_{\beta\in B} Q_{\beta}(\omega)$ is integrable on any rectangular $B = \prod_j [c_{1, j}, c_{2, j}]$ away 
%from $0$ in which case $0 < c_{1, j} < c_{2, j} < \infty$.
\item Case $q = 2$. In this case, we assume $\supp(\mu)$ is away from $0$: say for some 
$m_\mu > 0$, we have $\supp(\mu) \subseteq [m_\mu, \infty)$. One can then show the 
following bound which holds for all $\omega$: 
\begin{equation*}
	|\partial_{\beta_l} Q_\beta(\omega)| \le \frac{1}{(1 \wedge \beta_l)^2 \cdot (1 \wedge m_\mu)} |Q_\beta(\omega)| (1+\omega_l^2).
\end{equation*}
Note that $\sup_{\beta\in B} Q_{\beta}(\omega)(1+\omega_l^2)$ is integrable for compact $B$
which does not contain $0$. 
%on any rectangular $B = \prod_j [c_{1, j}, c_{2, j}]$ away 
%from $0$ in which case $0 < c_{1, j} < c_{2, j} < \infty$.
\end{itemize}
Let $\beta > 0$ and $B$ be any compact set which does not contain $0$. We conclude for $q= 1, 2$
\begin{equation*}
	\int \sup_{\beta\in B} |\partial_{\beta_l} Q_\beta(\omega)| d\omega < \infty. 
\end{equation*}
As a result,  the dominated convergence theorem implies the desired identity: 
\begin{equation*}
\begin{split}
	h'(\langle \beta, |x- x'|\rangle) |x_l - x_l'|^q 
	&= \partial_{\beta_l}
		 \left(\int e^{i \langle \omega, x-x'\rangle} Q_\beta(\omega) d\omega\right) % \\
	=  \int e^{i \langle \omega, x-x'\rangle} \cdot (\partial_{\beta_l} Q_\beta(\omega)) d\omega.
\end{split}
\end{equation*}
%This proves Lemma~\ref{lemma:technical} as desired.

\subsection{Proof of Lemma~\ref{lemma:envelope-theorem}.}
\label{sec:proof-lemma-envelope-theorem}
For any $f$, let $\grad_\beta \wtilde{\energy}(\beta, f)$ denote the gradient of 
$\wtilde{\energy}(\beta, f)$ with respect to $\beta$ at $f$ (if it exists). We prove 
the following result (proof in Section~\ref{sec:proof-technical-only}). 
\begin{lemma}
\label{lemma:technical-only}
Fix $\beta > 0$. Then we have the following statements. 
\begin{enumerate}
\item Existence: $\grad_{\beta} \wtilde{\energy}(\beta{''}, f_{\beta'})$ exists for all $\beta', \beta''$ in a neighborhood of $\beta$.
\item Continuity: $\grad_{\beta} \wtilde{\energy}(\beta{''}, f_{\beta'}) \to \grad_{\beta} \wtilde{\energy}(\beta, f_\beta)$ 
	as $\beta', \beta'' \to \beta$.
\item Analytical expression: for all $\beta', \beta''$ close to $\beta$, we have 
	\begin{equation}
		\label{eqn:technical-gradient-only}
			\grad_{\beta} \wtilde{\energy}(\beta', f_{\beta''}) = 
		 		\frac{\lambda}{(2\pi)^p} \int \Big|\F(\wtilde{f_{\beta''}})(\omega)\Big|^2 \cdot 
					\grad_\beta \left(\frac{1}{Q_{\beta'}(\omega)}\right) d\omega.
	\end{equation}
\end{enumerate}
\end{lemma}
%By dominated convergence theorem, we can evaluate the gradient 
%\begin{equation}
%\label{eqn:partial-beta-energy-beta-f}
%	\grad_{\beta} \wtilde{\energy}(\beta, f) =  \lambda \int |\F(f)(\omega)|^2 \cdot \grad_\beta \left(\frac{1}{Q_\beta(\omega)}\right) d\omega.
%\end{equation} 
%According to equation~\eqref{eqn:partial-beta-energy-beta-f}, we can see that the gradient is continuous in the sense that 
%for any $f_n \to f$ in $\H$ and $\beta_n \to \beta$, 
%\begin{equation}
%\label{eqn:gradient-continuity}
%	\lim_{n \to \infty} \grad_{\beta} \wtilde{\energy}(\beta_n, f_n) \to \grad_{\beta} \wtilde{\energy}(\beta, f).
%\end{equation}
We are now ready to prove Lemma~\ref{lemma:envelope-theorem}. We first prove that 
\begin{equation}
\label{eqn:grad-lower-one} 
	\obj(\beta')  \ge \obj(\beta) + \langle \grad_\beta \energy( \beta, \wtilde{f_\beta}), \beta'-\beta \rangle + o(\ltwo{\beta'-\beta}).
\end{equation}
Indeed, note that $\obj(\beta') = \energy(\beta', \wtilde{f_{\beta'}})$ and $\energy(\beta, \wtilde{f_{\beta'}}) \ge \obj(\beta)$. 
Using Lemma~\ref{lemma:technical-only}, for $\beta'$ close to $\beta$, Taylor's intermediate theorem yields 
for some $\beta'' \in [\beta, \beta']$: 
\begin{equation*}
	\obj(\beta') = \energy(\beta, \wtilde{f_{\beta'}}) + 
		\langle \grad_\beta \energy(\beta'', \wtilde{f_{\beta'}}), \beta'-\beta \rangle
		%\ge \obj(\beta) + \langle \grad_\beta \energy(\beta'', \wtilde{f_{\beta'}}), \beta'-\beta \rangle.
		\ge \obj(\beta) + \langle \grad_\beta \energy(\beta'', \wtilde{f_{\beta'}}), \beta'-\beta \rangle.
\end{equation*}
Equation~\eqref{eqn:grad-lower-one} now follows since $\grad_\beta \energy(\beta'', \wtilde{f_{\beta'}}) 
\to \grad_{\beta} \wtilde{\energy}(\beta, f_\beta)$ by Lemma~\ref{lemma:technical-only}. With 
the same reasoning, one can analogously derive 
\begin{equation}
\label{eqn:grad-lower-two}
	\obj(\beta')  \le \obj(\beta) + \langle \grad_\beta \energy( \beta, \wtilde{f_\beta}), \beta'-\beta \rangle + o(\ltwo{\beta'-\beta}).	
\end{equation}
Equations~\eqref{eqn:grad-lower-one} and~\eqref{eqn:grad-lower-two} together yield the desired claim 
of Lemma~\ref{lemma:envelope-theorem}.

%It suffices to prove that 
%\begin{align}
%	 \obj(\beta')  \ge \obj(\beta) + \langle \grad_\beta \energy( \beta, \wtilde{f_\beta}), \beta'-\beta \rangle + o(\ltwo{\beta'-\beta})
%			\label{eqn:grad-lower-one} \\
%	\obj(\beta')  \le \obj(\beta) + \langle \grad_\beta \energy( \beta, \wtilde{f_\beta}), \beta'-\beta \rangle + o(\ltwo{\beta'-\beta})
%			\label{eqn:grad-lower-two}
%\end{align}
%%We divide the proof into two bullet points. 
%
%\begin{enumerate}
%\item To prove equation~\eqref{eqn:grad-lower-one},
%	
%\item To prove equation~\eqref{eqn:grad-lower-two}, note that 
%	$\obj(\beta) = \energy(\beta, \wtilde{f_\beta})$ and $ \obj(\beta') \le \energy(\beta', \wtilde{f_\beta})$. 
%	Using Lemma~\ref{lemma:technical-only}, for $\beta'$ close to $\beta$, Taylor's intermediate theorem yields 
%	for some $\beta'' \in [\beta, \beta']$: 
%	\begin{equation*}
%		\obj(\beta') \le \energy(\beta', \wtilde{f_\beta}) = \obj(\beta) + 
%			\langle \grad_\beta \energy(\beta'', \wtilde{f_{\beta}}), \beta'-\beta \rangle.
%			%\ge \obj(\beta) + \langle \grad_\beta \energy(\beta'', \wtilde{f_{\beta'}}), \beta'-\beta \rangle.
%	\end{equation*}
%	Equation~\eqref{eqn:grad-lower-two} follows since $\grad_\beta \energy(\beta'', \wtilde{f_{\beta'}}) 
%	\to \grad_{\beta} \wtilde{\energy}(\beta, f_\beta)$ by Lemma~\ref{lemma:technical-only}. 
%\end{enumerate}

\subsection{Proof of Lemma~\ref{lemma:technical-only}.}
\label{sec:proof-technical-only}
Recall the objective function 
\begin{equation*}
	 \wtilde{\energy}(\beta'', \wtilde{f_\beta'})
		= \half \E[(Y - \wtilde{f_\beta'}(X))^2] + \frac{\lambda}{2(2\pi)^p} 
			\int \Big|\F(\wtilde{f}_{\beta'})(\omega)\Big|^2 \cdot \frac{1}{Q_{\beta''}(\omega)} d\omega,
\end{equation*}
We wish to take the derivative w.r.t $\beta$. Let $B = \prod_{j=1}^p [c_1 \beta_j, c_2 \beta_j]$
where $c_1=0.99$, $c_2=1.01$. The key to the proof is to prove the technical result: 
%The following technical lemma is crucial. 
%\begin{lemma}
%\label{lemma:super-technical-one}
%Let $B = \prod_{j=1}^p [c_1 \beta_j, c_2 \beta_j]$ where $c_1=0.99$, $c_2=1.01$. Then we have
\begin{equation}
\label{eqn:super-technical-one}
	\int \sup_{\beta' \in B, \beta'' \in B}\Big|\F(\wtilde{f_{\beta'}})(\omega)\Big|^2 \cdot 
		 \norm{\grad_\beta \left(\frac{1}{Q_{\beta''}(\omega)}\right)}_2 d\omega < \infty.
\end{equation}
We defer the proof equation~\eqref{eqn:super-technical-one} to the end. Note then, 
given equation~\eqref{eqn:super-technical-one}, Lebesgue's dominated 
convergence theorem implies that $\grad_{\beta} \wtilde{\energy}(\beta', f_{\beta''})$ exists
and satisfies~\eqref{eqn:technical-gradient-only}. To prove the remaining claim,
$\grad \wtilde{\energy}(\beta'', \wtilde{f_{\beta'}}) \to \grad \wtilde{\energy}(\beta, \wtilde{f_{\beta}})$
as $\beta'', \beta' \to \beta$, it suffices to show that 
\begin{equation}
\label{eqn:super-technical-two}
	\limsup_{\beta'', \beta' \to \beta}
		\int  \norm{\Big|\F(\wtilde{f_{\beta'}})(\omega)\Big|^2 \cdot \grad_\beta \left(\frac{1}{Q_{\beta''}(\omega)}\right) - 
			\Big|\F(\wtilde{f_{\beta}})(\omega)\Big|^2 \cdot \grad_\beta \left(\frac{1}{Q_{\beta}(\omega)}\right)}_2 d\omega = 0.
\end{equation}
Below we prove the deferred equations~\eqref{eqn:super-technical-one} and~\eqref{eqn:super-technical-two}.

To prove equation~\eqref{eqn:super-technical-one}, we introduce the function $\omega \mapsto g(\omega)$ 
\begin{equation*}
	g(\omega) = \sup_{\beta'', \beta' \in B} 
		\bigg\{\normbigg{\grad_\beta \Big(\frac{1}{Q_{\beta''}(\omega)}\Big)}_2 \cdot Q_{\beta'}^2(\omega)\bigg\}.
\end{equation*}
Lemma~\ref{lemma:uniform-integrability-of-Q-beta} shows that $g$ is integrable. Now that 
\begin{equation*}
\begin{split}
\int \sup_{\beta' \in B, \beta'' \in B}&\Big|\F(\wtilde{f_{\beta'}})(\omega)\Big|^2 \cdot 
		 \normbig{\grad_\beta \Big(\frac{1}{Q_{\beta''}(\omega)}\Big)}_2 d\omega
	\le \int \sup_{\beta'' \in B}|\F(\wtilde{f}_{\beta'})(\omega)|^2 \cdot \frac{g(\omega)}{Q_{\beta'}^2(\omega)} d\omega \\
	&\stackrel{(i)}{=}\frac{(2\pi)^p}{\lambda^2} \cdot 
	\int \sup_{\beta'' \in B} \left|\E[\wtilde{r}_{\beta'}(X; Y) e^{i \omega^T X}]\right|^2  \cdot g(\omega) d\omega% \\
%&=  \frac{1}{\lambda} \cdot \int \E[\bar{r}_\beta(X) \bar{r}_\beta(X')e^{i \langle\omega, X-X' \rangle}]  \cdot \nabla_\beta Q_{\beta}(\omega) d\omega \\
\stackrel{(ii)}{\le}  \frac{(2\pi)^p}{\lambda^2} \cdot 
		\E\left[|Y|^2\right]  \cdot \int |g(\omega)| d\omega < \infty. 
\end{split}
\end{equation*}
where (i) is due to Lemma~\ref{lemma:KKT-condition} and (ii) is due to Proposition~\ref{prop:second-moment-bound}.
This proves equation~\eqref{eqn:super-technical-one}.

To prove equation~\eqref{eqn:super-technical-two}, we introduce the functions 
\begin{equation*}	
	h_{\beta', \beta''}(\omega) = \grad_\beta \left(\frac{1}{Q_{\beta''}(\omega)}\right)
		\cdot Q_{\beta'}^2(\omega)~~\text{and}~~z_\beta(\omega) = \grad Q_\beta(\omega).
\end{equation*}
Note that the following bound holds for all $\beta, \beta', \beta''$: 
\begin{equation*}
\begin{split}
&\int  \norm{\Big|\F(\wtilde{f_{\beta'}})(\omega)\Big|^2 \cdot \grad_\beta \left(\frac{1}{Q_{\beta''}(\omega)}\right) - 
			\Big|\F(\wtilde{f_{\beta}})(\omega)\Big|^2 \cdot \grad_\beta \left(\frac{1}{Q_{\beta}(\omega)}\right)}_2 d\omega \\
&=\int \normbigg{|\F(\wtilde{f}_{\beta'})(\omega)|^2 \cdot \frac{h_{\beta', \beta''}(\omega)}{Q_{\beta'}^2(\omega)} - 
	|\F(\wtilde{f}_\beta)(\omega)|^2 \cdot \frac{z_\beta(\omega)}{Q_{\beta}^2(\omega)}}_2 d\omega \\
&\stackrel{(i)}{=} \frac{(2\pi)^p}{\lambda^2} \cdot 
	\int \normbigg{\left|\E[\wtilde{r}_{\beta'}(X; Y) e^{i \omega^T X}]\right|^2 h_{\beta', \beta''}(\omega) - 
		\left|\E[\wtilde{r}_\beta(X; Y) e^{i \omega^T X}]\right|^2 z_\beta(\omega)}_2 d\omega \\
&\stackrel{(ii)}{\le} \frac{(2\pi)^p}{\lambda^2} \cdot 
	\left(\int \norm{h_{\beta', \beta''}(\omega)-z_\beta(\omega)}_2 d\omega \cdot \E[Y^2] + 
	\left|\E[(\wtilde{r_{\beta'}}- \wtilde{r_\beta})(X; Y)]\right| \cdot \E[Y^2]^{1/2} \cdot \int \norm{z_\beta(\omega)}_2 d\omega
	\right).
\end{split}
\end{equation*}
where $(i)$ is due to Lemma~\ref{lemma:KKT-condition} and (ii) is due to Cauchy-Schwartz. 
Note $\limsup_{\beta' \to \beta}|\E[(\wtilde{r_{\beta'}}- \wtilde{r_\beta})(X; Y)]| = 0$
by Proposition~\ref{prop:continuity-f-beta-beta-X}, and $\limsup_{\beta', \beta'' \to \beta}
\int \norm{h_{\beta', \beta''}(\omega)-z_\beta(\omega)}_2 d\omega = 0$ by 
Lemma~\ref{lemma:uniform-integrability-of-Q-beta} and 
Lebesgue's dominated convergence theorem. This proves equation~\eqref{eqn:super-technical-two}.
%and that 
%\begin{equation}
%\label{eqn:statement-2-final-2}
%	\limsup_{\beta' \to \beta} \left|\E[(\wtilde{r_{\beta'}}- \wtilde{r_\beta})(X; Y)]\right| = 0.
%\end{equation}
%By Lebesgue's dominated convergence theorem, the first inequality~\eqref{eqn:statement-2-final-1} would follow if we can show that 
%\begin{equation*}
%	\int \sup_{\beta'', \beta' \in B} \bigg\{\normbigg{\grad_\beta \Big(\frac{1}{Q_{\beta''}(\omega)}\Big)}_2 
%		\cdot Q_{\beta'}^2(\omega)\bigg\} d\omega < \infty. 
%\end{equation*}
%The above calculus result is precisely given in Lemma~\ref{lemma:uniform-integrability-of-Q-beta}. 
%The proof of the second equality~\eqref{eqn:statement-2-final-2} is a consequence of Proposition~\ref{prop:continuity-f-beta-beta-X} and 
%Cauchy-Schwartz inequality. Indeed, %we have 
%\begin{equation*}
%	\limsup_{\beta' \to \beta}|\E[(\wtilde{r_{\beta'}}- \wtilde{r_\beta})(X; Y)]| \le 
%		\limsup_{\beta' \to \beta} |\E[(\wtilde{r_{\beta'}}- \wtilde{r_\beta})^2(X; Y)]|^{1/2} = 0.
%\end{equation*}

\subsection{An Integrability Result on $Q_\beta$.}
\begin{lemma}
\label{lemma:uniform-integrability-of-Q-beta}
Let $\beta > 0$. For $B = \prod_{j=1}^p [c_1 \beta_j, c_2\beta_j]$ where $c_1=0.99$, $c_2=1.01$, we have
\begin{equation*}
	\int \sup_{\beta'', \beta' \in B} \bigg\{\normbigg{\grad_\beta \Big(\frac{1}{Q_{\beta''}(\omega)}\Big)}_2 \cdot Q_{\beta'}^2(\omega)\bigg\} d\omega < \infty. 
\end{equation*}
\end{lemma}

\begin{proof}
It suffices to show that for any $l \in [p]$: 
\begin{equation*}
	\int \sup_{\beta'', \beta' \in B} \bigg\{|\partial_{\beta_l} Q_{\beta''}(\omega)| \cdot \frac{Q_{\beta'}^2(\omega)}{Q_{\beta''}^2(\omega)}\bigg\} d\omega < \infty. 
\end{equation*}
\begin{itemize}
\item Case $q = 1$. In this case, %we have by definition 
	$ %\begin{equation*}
		Q_\beta(\omega) = \int_0^\infty \prod_{i \in [p]} \frac{1}{\pi} \frac{\beta_i t}{\beta_i^2 t^2 + \omega_i^2} \mu(dt).
	$ %\end{equation*}
	 Note the following bound
	\begin{equation*}
		\sup_{\beta', \beta''\in B} \left|\frac{Q_{\beta'}(\omega)}{Q_{\beta''}(\omega)} \right| \le \left(\frac{1.01}{0.99}\right)^p
		~~\text{and}~~
		|\partial_{\beta_l} Q_{\beta''}(\omega)| \le \frac{1}{\beta''_l} Q_{\beta''}(\omega)
	\end{equation*}
	As a result, we obtain that 
	\begin{equation*}
		\int \sup_{\beta'', \beta' \in B} \bigg\{|\partial_{\beta_l} Q_{\beta''}(\omega)| \cdot \frac{Q_{\beta'}^2(\omega)}{Q_{\beta''}^2(\omega)}\bigg\} d\omega
		\le \left(\frac{1.01}{0.99}\right)^{p+1} \int \frac{1}{\beta_l} \cdot Q_\beta(\omega) d\omega < \infty.
	\end{equation*}
\item Case $q = 2$. In this case, 	
	$Q_\beta(\omega) = \int 
			\prod_{i \in [p]}\frac{1}{2\sqrt{\pi \beta_i t}} e^{-\frac{\omega_i^2}{4\beta_i t}} \mu(dt). 
	$
	Assume W.L.O.G. $\supp(\mu) \subseteq [m_\mu, M_\mu]$ where 
	$0 < m_\mu < M_\mu < \infty$.  %Note
	Note the following elementary bound
	\begin{equation*}
		|\partial_{\beta_l} Q_{\beta''}(\omega)| \le \frac{1}{(1\wedge \beta''_l)^2 \cdot (1\wedge m_\mu)} |Q_{\beta''}(\omega)|(1+\omega_l^2).
	\end{equation*}
	As a result, it suffices to show that 
	\begin{equation}
	\label{eqn:uniform-integrability-last}
		\int_0^\infty \sup_{\beta'', \beta' \in B} \left\{\frac{Q^2_{\beta'}(\omega)}{Q_{\beta''}(\omega)} \right\} \cdot (1+\omega_l^2) d\omega < \infty. 
	\end{equation}
	Let $K$ be the integer such that $M_\mu/m_\mu < 1.01^K$. Decompose 
	$[m_\mu, M_\mu] = \bigcup_{k=0}^{K-1} [n_\mu^{(k)}, n_{\mu}^{(k+1)}]$ where 
	$n_\mu^{(0)} = m_\mu$, $n_\mu^{(K)} = M_\mu$ and $n_\mu^{(k+1)}/n_\mu^{(k)} < 1.01$. Introduce the notation
	\begin{equation*}
		Q_\beta^{(k)}(\omega) = \int_{n_\mu^{(k)}}^{n_\mu^{(k+1)}} \prod_{i \in [p]} \frac{1}{2\sqrt{\pi \beta_i t}} e^{-\frac{\omega_i^2}{4\beta_i t}}  \mu(dt).
	\end{equation*}
	Then $Q_\beta(\omega) = \sum_k Q_\beta^{(k)}(\omega)$. Hence we have the basic inequality 
	\begin{equation}
	\label{eqn:basic-decomposition}
		\frac{Q^2_{\beta'}(\omega)}{Q_{\beta''}(\omega)} \le \sum_{k=1}^K 
			\frac{(Q^{(k)}_{\beta'}(\omega))^2}{Q^{(k)}_{\beta''}(\omega)}
	\end{equation}
	Because $n_\mu^{(k+1)}/n_\mu^{(k)} < 1.01$, hence
	for some constants $c^{(k)}, C^{(k)} > 0$, we have for all $\omega$
	\begin{equation*}
		\sup_{\beta', \beta'' \in B} \frac{(Q^{(k)}_{\beta'}(\omega))^2}{Q_{\beta''}^{(k)}(\omega)} \le 
			C^{(k)} e^{-c^{(k)}\norm{\omega}^2}.
	\end{equation*} 
	This exponential tail bound in conjunction with inequality~\eqref{eqn:basic-decomposition} yields 
	equation~\eqref{eqn:uniform-integrability-last}. 
\end{itemize}
\end{proof}

\subsection{Reproducing Kernel Hilbert Spaces $\{H_\beta\}_{\beta \ge 0}$.}
\label{sec:H-beta-spaces}
\indent\indent
Let $k(x, x')$ be the kernel associated with $\H$, 
Write $k_\beta(x, x') = k(\beta^{1/q} \odot x, \beta^{1/q} \odot x')$. Then $k_\beta$
is positive definite. Moore Aronszajn Theorem (Theorem~\ref{theorem:Moore-Aronszajn}) shows
that there exists an RKHS $\H_\beta$ whose kernel is $k_\beta$. 

% the existence 
%of an RKHS whose reproducing kernel is $k_\beta(x, x')$. Denote this RKHS to be $\H_\beta$. 
%Note that $\H = \H_{\mathbf{1}}$ under this notation. 

Proposition~\ref{proposition: Hilbert-beta} builds connections between $\H_\beta$ and
and $\H = \H_{\mathbf{1}}$. %The result appears in the literature~\cite{FukumizuBaJo09}.

\begin{proposition}
\label{proposition: Hilbert-beta}
We have the following properties. 
\begin{enumerate}[(a)]
\item For any $\beta \ge 0$, the space $\H_\beta$ has the representation: $\H_\beta = \left\{f(\beta^{1/q} \odot x): f \in \H\right\}$.
\item For any $\beta > 0$, we have the identity: $\norm{f(\beta^{1/q} \odot \cdot)}_{\H_\beta} = \norm{f(\cdot)}_{\H}$ for any $f \in \H$.
\end{enumerate}
\end{proposition}

\begin{proof}
Part (a) is immediate from the characterization of the Hilbert space due to Moore-Aronszajn (Theorem~\ref{theorem:Moore-Aronszajn}). 
Part (b) follows from the definition $k_\beta(x, x') = k(\beta^{1/q} \odot x, \beta^{1/q} \odot x')$ and the reproducing property of the kernel $k_\beta(x, x')$
and $k(x, x')$ with respect to $\H_\beta$ and $\H$: 
\begin{equation*}
	\langle k(x, \beta^{1/q} \odot \cdot), k(x', \beta^{1/q} \odot \cdot)\rangle_{\H_\beta} = k(x, x') = \langle k(x, \cdot), k(x', \cdot)\rangle_{\H}.
\end{equation*}
Notice that the first identity uses the assumption that $\beta > 0$.
\end{proof}

% is 
%the following corollary that \emph{analytically} characterizes the inner product of 
%the RKHS $\H_\beta$ for all $\beta > 0$.  Note that it is a generalization of Proposition~\ref{proposition:norm-of-H} in the main text. 
\begin{corollary}
\label{corrolary:norm-of-H}
%Let $\H_\beta$ be the 
%$\ell_q$ type RKHS associated with the kernel $k_\beta(x, x') = k(\beta^{1/q} \odot x, \beta^{1/q} \odot x')$ where 
%$k(x, x')$ is define in equation~\eqref{eqn:l_1_kernel}, i.e., $k(x, x') = h(\norm{x-x'}_q^q) = \int_0^\infty e^{-t \norm{x-x'}_q^q}\mu(dt)$. 
%Assume that the function $h$ is an integrable function on $\R_+$. 
For any $\beta > 0$, the inner product $\langle f, g\rangle_{\H_\beta}$ 
has the explicit characterization
\begin{equation}
\label{eqn:norm-of-H-beta}
		\langle f, g\rangle_{\H_\beta} = \frac{1}{(2\pi)^p} \cdot \int_{\R^p} \frac{\F(f)(\omega)\wbar{\F(g)(\omega)}}{Q_{\beta}(\omega)} d\omega
		~~\text{where}~~Q_{\beta}(\omega) = \int_0^\infty q_{\beta, t}(\omega) \mu(dt).
\end{equation}
Above, %$\omega \mapsto q_{\beta, t}(\omega)$ is defined by 
%\begin{equation}
%\label{eqn:notation-of-q-t-p-t-beta}
$	q_{\beta, t}(\omega) = \prod_{i \in [p]} \psi_{\beta_i^{1/q} t}(\omega_i)$
%\end{equation}
where %$\psi_s: \R_+ \mapsto \R$ is defined by 
$\psi_s(\omega) = \frac{1}{s} \cdot \psi\left(\frac{\omega}{s}\right)$ for any $s > 0$.
\end{corollary}

\begin{proof}
This is an immediate consequence of 
Proposition~\ref{proposition: Hilbert-beta} and Proposition~\ref{proposition:norm-of-H}.
\end{proof}

\section{Lipschitzness and Boundedness of the Gradient $\grad \obj(\beta)$}
\label{sec:Lipschitzness-boundedness-of-gradient}
\subsection{Lipschitzness.} 
Proposition~\ref{prop:grad-obj-lipschitz-beta} shows that $\beta \mapsto \grad \obj(\beta)$ is Lipschitz. 
%The proof is nontrivial---it does not follow immediately from 
%Proposition~\ref{prop:continuity-f-beta} which only shows that $\beta \mapsto r_\beta$ is H\"{o}lder continuous with
%$\norm{r_\beta - r_{\beta'}}_{\H} \lesssim \norm{\beta - \beta'}_1^{1/2}$.
\begin{proposition}
\label{prop:grad-obj-lipschitz-beta}
Assume Assumption~\ref{assumption:mu-compact}---\ref{assumption:X-Y-bound}. 
The mapping $\beta \mapsto \grad \obj(\beta)$ is Lipschitz: %for any $l \in [p]$, 
	\begin{equation}
	\label{eqn:grad-obj-beta-Lipschitz-beta}
		\norm{\grad \obj(\beta) - \grad \obj(\beta')}_2 \le \frac{Cp}{\lambda^2}\cdot M_Y^2 \cdot\norm{\beta - \beta'}_2
	\end{equation}
	where the constant $C > 0$ depends only on $M_X$ and $|h'(0)|$.
\end{proposition}

\begin{proof}
%\label{sec:proof-prop-grad-obj-lipschitz-beta}
The key is Lemma~\ref{lemma:residual-Lipschitz}, whose proof is 
deferred to Section~\ref{sec:proof-lemma-residual-Lipschitz}. 
\begin{lemma}
\label{lemma:residual-Lipschitz}
Assume $\E[Y^2] \le M_Y^2$ and $\max_l \E[X_l^4] \le M_X^4$. Then for all values of $\beta, \beta'$
\begin{equation}
	\normbig{r_\beta(\beta^{1/q}\odot X; Y) - r_\beta'(\beta'^{1/q}\odot X; Y)}_{\mathcal{L}_2(\P)} \le 
		\frac{1}{\lambda} \cdot |h^\prime(0)| \cdot (2M_X)^q \cdot M_Y \cdot \norm{\beta- \beta'}_1.
\end{equation}
\end{lemma}
\noindent\noindent
Back to the proof of Proposition~\ref{prop:grad-obj-lipschitz-beta}.
By Proposition~\ref{proposition:compute-grad-obj}, we have for any $l \in [p]$
\begin{equation*}
	(\grad \obj(\beta))_l = -\frac{1}{\lambda} \cdot \E\left[r_\beta(\beta^{1/q}\odot X; Y)r_\beta(\beta^{1/q}\odot X'; Y') h'(\normsmall{X-X'}_{q,\beta}^q)|X_l- X_l'|^q\right].
\end{equation*}
By triangle inequality, we obtain for any values of $\beta, \beta'$: 
\begin{equation}
\label{eqn:decomposition-into-three-errors}
	|(\grad \obj(\beta))_l  - (\grad \obj(\beta'))_l | \le \frac{1}{\lambda} \left(|\error_1| + |\error_2| + |\error_3|\right),
\end{equation}
where 
\begin{equation*}
\begin{split}
	\error_1 &= \E[\Delta_{\beta, \beta'}(X) r_\beta(\beta^{1/q}\odot X'; Y') h'(\normsmall{X-X'}_{q,\beta}^q)|X_l- X_l'|^q] \\
	\error_2 &= \E[\Delta_{\beta, \beta'}(X') r_{\beta'}(\beta'^{1/q}\odot X; Y) h'(\normsmall{X-X'}_{q,\beta}^q)|X_l- X_l'|^q] \\
	\error_3 &= \E[ r_{\beta}(\beta^{1/q}\odot X; Y) r_{\beta'}(\beta'^{1/q}\odot X'; Y')(\delta h')_{\beta, \beta'}(X, X')|X_l- X_l'|^q]
\end{split}
\end{equation*}
and the terms $\Delta_{\beta, \beta'}$ and $(\delta h')_{\beta, \beta'}$ are defined by 
\begin{equation*}
\begin{split}
	\Delta_{\beta, \beta'}(x) &= r_\beta(\beta^{1/q}\odot x; y) - r_{\beta'}(\beta'^{1/q}\odot x; y), \\
	(\delta h')_{\beta, \beta'}(x, x') &= h'(\normsmall{x-x'}_{q,\beta}^q) - h'(\normsmall{x-x'}_{q,\beta'}^q).
\end{split}
\end{equation*}
Below we estimate the three error terms. The following facts are useful towards this end. 
\begin{itemize}
\item As $h$ is completely monotone, $h^\prime$ is $|h^{\prime\prime}(0)|$ Lipschitz. Hence, 
we obtain that 
\begin{equation*}
		\normP{(\delta h')_{\beta, \beta'}(X, X')|X_l - X_l'|^q}
			\le |h^{\prime\prime}(0)| \cdot (2M_X)^{2q} \cdot \norm{\beta- \beta'}_1.
\end{equation*}
	Furthermore, $h^\prime(\norm{X-X'}_{q, \beta}^q) \le |h^\prime(0)|$ as $h$ is completely monotone. 
\item By Lemma~\ref{lemma:residual-Lipschitz}, we have $\beta \mapsto r_\beta(\beta^{1/q}\odot X; Y)$ is Lipschitz, and hence,
	\begin{equation*}
		\normP{\Delta_{\beta, \beta'}(X)} \le \frac{1}{\lambda} \cdot 
			|h'(0)| \cdot (2M_X)^{q} \cdot M_Y \cdot \norm{\beta-\beta'}_1.
	\end{equation*}
	Furthermore, we have the bound $\normP{r_\beta(X; Y)} \le M_Y$ by Proposition~\ref{prop:second-moment-bound}. 
\end{itemize}
Now is an opportune time to establish error bounds on $\error_1, \error_2, \error_3$. %The key is to 
%use Cauchy-Schwartz inequality in a proper way. 
\begin{itemize}
\item For the first error term $\error_1$, we note that by Cauchy-Schwartz, 
\begin{equation*}
|\error_1| \le \normP{\Delta_{\beta, \beta'}(X) r_\beta(\beta^{1/q}\odot X'; Y')} \cdot \normbig{h'(\normsmall{X-X'}_{q,\beta}^q)|X_l- X_l'|^q}_{\mathcal{L}_2(\P)}.
\end{equation*}
The independence between $(X, Y)$ and $(X', Y')$ and the above facts imply that 
%\begin{equation*}
% \normP{\Delta_{\beta, \beta'}(X) r_\beta(\beta^{1/q}\odot X'; Y')}
% 	\le \frac{1}{\lambda} \cdot |h'(0)| \cdot (2M_X)^{q} \cdot M_Y^2 \cdot \norm{\beta-\beta'}_1.
%\end{equation*}
%Moreover, $\normbig{h'(\normsmall{X-X'}_{q,\beta}^q)|X_l- X_l'|^q}_{\mathcal{L}_2(\P)} \le |h'(0)| \cdot (2M_X)^q$. Hence, 
%we obtain 
\begin{equation*}
	|\error_1| \le \frac{1}{\lambda} \cdot |h'(0)|^2 \cdot (2M_X)^{2q} \cdot M_Y^2 \cdot \norm{\beta-\beta'}_1.
\end{equation*}
\item For the second error term $\error_2$, an analysis parallel to that of the first error term $\error_1$ yields 
\begin{equation*}
	|\error_2| \le \frac{1}{\lambda} \cdot |h'(0)|^2 \cdot (2M_X)^{2q} \cdot M_Y^2 \cdot \norm{\beta-\beta'}_1.
\end{equation*}
\item For the last error term $\error_3$, we note by Cauchy Schwartz's inequality 
\begin{equation*}
|\error_3| \le \normP{r_{\beta}(\beta^{1/q}\odot X; Y) r_{\beta'}(\beta'^{1/q}\odot X'; Y')}\normP{(\delta h')_{\beta, \beta'}(X, X')|X_l- X_l'|^q}
\end{equation*}
The independence between $(X, Y)$ and $(X', Y')$ and the above facts gives 
\begin{equation*}
|\error_3| \le  |h''(0)| \cdot (2M_X)^{2q} \cdot M_Y^2 \cdot \norm{\beta-\beta'}_1.
\end{equation*}
\end{itemize}
Substitute the above error bounds into equation~\eqref{eqn:decomposition-into-three-errors}. We obtain 
\begin{equation}
\label{eqn:grad-obj-beta-infty-Lipschitz-beta}
	\norm{\grad \obj(\beta) - \grad \obj(\beta')}_\infty \le \frac{1}{\lambda^2}\cdot (\lambda |h^{\prime\prime}(0)|+|h^\prime(0)|^2) 
		\cdot (2M_X)^{2q} \cdot M_Y^2 \cdot\norm{\beta - \beta'}_1
\end{equation}
Proposition~\ref{prop:grad-obj-lipschitz-beta} follows by applying H\"{o}lder's inequality
 to equation~\eqref{eqn:grad-obj-beta-infty-Lipschitz-beta}.

\end{proof}

%The proof of Proposition~\ref{prop:grad-obj-lipschitz-beta} 
%is given in Appendix~\ref{sec:proof-prop-grad-obj-lipschitz-beta}. 

%\noindent\noindent

\subsubsection{Proof of Lemma~\ref{lemma:residual-Lipschitz}.}
\label{sec:proof-lemma-residual-Lipschitz}
%
%Our starting point is based on the following bound. Following the proof of Proposition~\ref{sec:proof-prop-continuity-of-beta-beta-X}, 
%we obtain that (see equation~\eqref{eqn:bound-on-diff-beta-r-beta})
%\begin{equation}
%	\normP{r_\beta(\beta^{1/q}\odot X; Y) - r_\beta'(\beta'^{1/q}\odot X; Y)} \le 
%		 \frac{1}{\lambda} \normP{Y} \cdot \E[(\delta K)^2_{\beta, \beta'}(X, X')]^{1/2}.
%\end{equation} 
%where $(\delta K)_{\beta, \beta'}(x, x') = h(\norm{x-x'}_{q,\beta}^q) - h(\norm{x-x'}_{q,\beta'}^q)$. 
%
%
%\subsubsection{Proof of Lemma~\ref{lemma:X-Y-gradient-Lipschitz}.}
By Proposition~\ref{prop:kernel-beta-fix},
$(\Sigma_{\beta} + \lambda I) f_\beta = h_\beta$,
$(\Sigma_{\beta'} + \lambda I) f_{\beta'} = h_{\beta'}$. 
%\begin{equation}
%\label{eqn:starting-point-of-f-beta-sim}
%\begin{split}
%	\lambda f_\beta (\cdot) + \E[f_\beta(\beta^{1/q} \odot X)k(\beta^{1/q} \odot X, \cdot)] &= \E[Y k(\beta^{1/q} \odot X, \cdot)] \\
%	\lambda f_{\beta'} (\cdot) + \E[f_{\beta'}(\beta'^{1/q} \odot X)k(\beta'^{1/q} \odot X, \cdot)] 
%		&= \E[Y k(\beta'^{1/q} \odot X, \cdot)].
%\end{split}
%\end{equation}
Let $X' \sim \P$ be an independent copy of $X$. We can rewrite them into the identities: 
\begin{equation*}
\begin{split}
	f_\beta (\beta^{1/q} \odot X') + \E[f_\beta(\beta^{1/q} \odot X)k(\beta^{1/q} \odot X, \beta \odot X')\mid X'] 
		&= \E[Y k(\beta^{1/q} \odot X, \beta^{1/q} \odot X')\mid X'] \\
	f_{\beta'} (\beta'^{1/q} \odot X') + \E[f_{\beta'}(\beta^{1/q} \odot X)k(\beta'^{1/q} \odot X, \beta' \odot X') \mid X'] &= \E[Y k(\beta'^{1/q} \odot X, \beta'^{1/q} \odot X')\mid X' ]
\end{split}
\end{equation*}
Subtract the first from the second of the equations. Recall $r_\beta(x, y) = y - f_\beta(x)$. We obtain
\begin{equation}
\label{eqn:key-grad-beta-Lip-sim}
\lambda \Delta_{\beta, \beta'} (X') + \E[k(\beta'^{1/q} \odot X, \beta'^{1/q} \odot X')\Delta_{\beta, \beta'} (X) \mid X'] = 
	\E[(\delta K)_{\beta, \beta'}(X, X') r_\beta(\beta^{1/q}\odot X; Y)\mid X']
\end{equation}
where $\Delta_{\beta, \beta'}$, $(\delta K)_{\beta, \beta'}$ are defined by
\begin{equation*}
\begin{split}
\Delta_{\beta, \beta'} (x) &\defeq 
	f_{\beta'}(\beta'^{1/q} \odot x) - f_{\beta}(\beta^{1/q} \odot x) = r_{\beta}(\beta^{1/q} \odot x; y) - r_{\beta'}(\beta'^{1/q} \odot x; y). \\
(\delta K)_{\beta, \beta'}(x, x') &\defeq 
	k(\beta^{1/q} \odot x, \beta^{1/q} \odot x') - k(\beta'^{1/q} \odot x, \beta'^{1/q} \odot x').
\end{split}
\end{equation*}
Multiply both sides of~\eqref{eqn:key-grad-beta-Lip-sim} by $\Delta_{\beta, \beta'} (X')$, and take 
expectation over $X' \sim \P$. This gives
\begin{equation}
\label{eqn:key-grad-multiple-sim}
%\begin{split}
\begin{split}
	&\lambda \E[\Delta_{\beta, \beta'}^2(X)] + \E[k(\beta'^{1/q} \odot X, \beta'^{1/q} \odot X')\Delta_{\beta, \beta'} (X)\Delta_{\beta, \beta'} (X')]  \\
	&~~~= \E[(\delta K)_{\beta, \beta'}(X, X') r_\beta(\beta^{1/q}\odot X)\Delta_{\beta, \beta'}(X')] 
\end{split}
\end{equation}
We analyze both the LHS and the RHS of equation~\eqref{eqn:key-grad-multiple-sim}. 
\begin{itemize}
\item First, the LHS of equation~\eqref{eqn:key-grad-multiple-sim} is lower bounded by $\lambda \E[\Delta_{\beta, \beta'}^2(X)]$. 
	The reason is that the second term on the LHS is non-negative since 
	$(x, x') \mapsto k(x, x')$ is positive definite. 
\item Second, the RHS of equation~\eqref{eqn:key-grad-multiple-sim} has the upper bound 
	\begin{equation*}
	\begin{split}
		&\E[(\delta K)_{\beta, \beta'}(X, X') r_\beta(\beta^{1/q}\odot X)\Delta_{\beta, \beta'}(X')] \\
			&\le \E[(\delta K)^2_{\beta, \beta'}(X, X')]^{1/2} \cdot \E[r_\beta^2(\beta^{1/q}\odot X; Y)\Delta^2_{\beta, \beta'}(X')]]^{1/2} \\
			&= \E[(\delta K)^2_{\beta, \beta'}(X, X')]^{1/2} \cdot \E[r_\beta^2(\beta^{1/q}\odot X; Y)]^{1/2} \cdot \E[\Delta^2_{\beta, \beta'}(X')]^{1/2}
	\end{split}
	\end{equation*}
	where the inequality is due to Cauchy-Schwartz and the equality is due to the independence between $X, X'$. 
	Note that $\E[r_\beta^2(\beta^{1/q}\odot X; Y)] \le \E[Y^2]$ by Proposition~\ref{prop:second-moment-bound}.
\end{itemize}
Plugging these lower and upper bounds into equation~\eqref{eqn:key-grad-multiple-sim}, we obtain the inequality 
\begin{equation*}
	\lambda \E[\Delta_{\beta, \beta'}^2(X)] \le \E[Y^2]^{1/2}\cdot \E[(\delta K)^2_{\beta, \beta'}(X, X')]^{1/2} \E[\Delta_{\beta, \beta'}^2(X)]^{1/2}.
		%~~~\Longrightarrow~~~~
	%\normP{\Delta_{\beta, \beta'}} \le \frac{C}{\lambda} \norm{\beta - \beta'}_1.
\end{equation*}
Cancelling $\E[\Delta_{\beta, \beta'}^2(X)]^{1/2}$ once on both sides. Recall the definition of $\Delta_{\beta, \beta'}$. We obtain
\begin{equation}
\label{eqn:bound-on-diff-beta-r-beta}
	\normbigP{r_{\beta}(\beta^{1/q} \odot X; Y) - r_{\beta'}(\beta'^{1/q} \odot X; Y)}
		\le \frac{1}{\lambda} \normP{Y} \cdot \E[(\delta K)^2_{\beta, \beta'}(X, X')]^{1/2}.
\end{equation} 
Since $h$ is $|h'(0)|$ Lipschitz as $h$ is completely monotone, we have the estimate 
\begin{equation*}
	\E[(\delta K)_{\beta, \beta'}^2(X, X')] \le |h'(0)| \cdot \norm{\beta-\beta'}_1 \cdot (2M_X)^{2q}.
\end{equation*}
%Above, we have used that $h$ is $|h'(0)|$ Lipschitz as $h$ is completely monotone. 
Plugging the estimate into equation~\eqref{eqn:bound-on-diff-beta-r-beta} completes the proof.

\subsection{Uniform Boundedness.} 
Proposition~\ref{prop:grad-obj-bound-beta} says that $\beta\mapsto \grad \obj(\beta)$ is uniformly bounded.
\begin{proposition}
\label{prop:grad-obj-bound-beta}
Assume Assumption~\ref{assumption:mu-compact}--\ref{assumption:X-Y-bound}. For 
$C = |h^\prime(0)| \cdot (2M_X)^q$, we have for $\beta \ge 0$
\begin{equation}
	\norm{\grad \obj(\beta)}_\infty \le \frac{C}{\lambda}\cdot M_Y^2.
\end{equation}
%where the constant $C = |h^\prime(0)| \cdot (2M_X)^q > 0$. % depends only on $M_X$ and $|h'(0)|$ (and does not depend on $\beta$).
\end{proposition}
%The proof of Proposition~\ref{prop:grad-obj-bound-beta} is given in 
%Appendix~\ref{sec:proof-proposition-grad-obj-bound-beta}. 

%We prove a stronger statement than Proposition~\ref{prop:grad-obj-lipschitz-beta} 
%(see Lemma~\ref{lemma:X-Y-gradient-Lipschitz}). 
%Clearly, Proposition~\ref{prop:grad-obj-lipschitz-beta} follows by applying H\"{o}lder's inequality to equation~\eqref{eqn:grad-obj-beta-infty-Lipschitz-beta}.
%\begin{lemma}
%\label{lemma:X-Y-gradient-Lipschitz}
%Assume $\E[Y^2] \le M_Y^2$ and $\max_l \E[X_l^4] \le M_X^4$. Then for all values of $\beta, \beta'$
%\begin{equation}
%\label{eqn:grad-obj-beta-infty-Lipschitz-beta}
%	\norm{\grad \obj(\beta) - \grad \obj(\beta')}_\infty \le \frac{1}{\lambda^2}\cdot (\lambda |h^{\prime\prime}(0)|+|h^\prime(0)|^2) 
%		\cdot (2M_X)^{2q} \cdot M_Y^2 \cdot\norm{\beta - \beta'}_1
%\end{equation}
%\end{lemma}
%The key to Lemma~\ref{lemma:X-Y-gradient-Lipschitz} is to show that the residual $\beta \mapsto r_\beta(\beta^{1/q}\odot X; Y)$
%is Lipschitz. 

%Notice that $\sup |(\delta K)_{\beta, \beta'}(x, x')| \le |h(0)|$ since $k(x, x') = h(\norm{x-x'}_q^q)$ where $h$ is strictly completely monotone. 
%Since $(x, x') \mapsto k(x, x')$ is continuous, Lebesgue's dominated convergence theorem shows that 
%$\lim_{\beta' \to \beta} \E[(\delta K)^2_{\beta, \beta'}(X, X')] = \E[\lim_{\beta' \to \beta} (\delta K)^2_{\beta, \beta'}(X, X')] = 0$. Consequently, 
%equation~\eqref{eqn:bound-on-diff-beta-r-beta} indicates Proposition~\ref{sec:proof-prop-continuity-of-beta-beta-X} as desired.

\begin{proof}
By Proposition~\ref{proposition:compute-grad-obj}, we have for any $l \in [p]$, 
\begin{equation*}
	(\grad \obj(\beta))_l = -\frac{1}{\lambda} \cdot \E\left[r_\beta(\beta^{1/q}\odot X; Y)r_\beta(\beta^{1/q}\odot X'; Y') h'(\norm{X-X'}_{q, \beta}^q)|X_l- X_l'|^q\right].
\end{equation*}
Cauchy-Schwartz inequality and independence between $(X, Y)$ and $(X', Y')$ yield the bound 
\begin{equation*}
\begin{split}
	|(\grad \obj(\beta))_l| &\le \frac{1}{\lambda} \cdot \normbigP{r_\beta(\beta^{1/q}\odot X; Y)r_\beta(\beta^{1/q}\odot X'; Y')} 
		\cdot \normbig{h'(\norm{X-X'}_{q, \beta}^q)|X_l- X_l'|^q}_{\mathcal{L}_2(\P)} \\
	&=  \frac{1}{\lambda} \cdot \normbigP{r_\beta(\beta^{1/q}\odot X; Y)}^2 \cdot 
		\normbig{h'(\norm{X-X'}_{q, \beta}^q)|X_l- X_l'|^q}_{\mathcal{L}_2(\P)}
\end{split}
\end{equation*}
Now that $\normP{r_{\beta}(\beta \odot X; Y)} \le \normP{Y}$
by Proposition~\ref{prop:second-moment-bound}. Additionally, 
$\norm{h^\prime}_\infty = |h^\prime(0)|$, and $\normbigP{|X_l- X_l'|^q} \le (2M_X)^q$.
This proves Proposition~\ref{prop:grad-obj-bound-beta}.
\end{proof}

\section{Proof of Deferred Lemma in Section~\ref{sec:distinction-l-1-l-2}}
\label{sec:landscape-result}
\subsection{Proof of Lemma~\ref{lemma:new-representation-of-grad-obj}.}
\label{sec:proof-lemma-new-representation-of-grad-obj}
By Proposition~\ref{proposition:compute-grad-obj}, we have  
\begin{equation}
\label{eqn:start-grad-representation}
	\partial_{\beta_l} \obj(\beta) = -\frac{1}{\lambda} \cdot \E[r_\beta(\beta^{1/q} \odot X; Y) r_\beta(\beta^{1/q} \odot X'; Y') h'(\norm{X-X'}_{q, \beta}^q) |X_l - X_l'|^q].
\end{equation}
Note that $(x, x') \mapsto h'(\norm{x-x'}_{q, \beta}^q)$ is a negative kernel with the Bochner representation: 
\begin{equation*}
	h'(\norm{x-x'}_{q, \beta}^q) = - \int e^{i \langle \omega, \beta^{1/q} \odot (x-x') \rangle} \wtilde{Q}(\omega) d\omega~~\text{where}~~
		\wtilde{Q}(\omega) = \int_0^\infty tq_t(\omega)\mu(dt)
\end{equation*}
Substitute this representation into equation~\eqref{eqn:start-grad-representation}, we obtain that 
\begin{equation*}
%\label{eqn:real-representation-of-part-obj}
\begin{split}
	\partial_{\beta_l} \obj(\beta) &= \frac{1}{\lambda} \cdot  \int
		\E[ e^{i \langle \omega, \beta^{1/q} \odot X\rangle}  r_\beta(\beta^{1/q} \odot X; Y) 
			e^{-i \langle \omega, \beta^{1/q} \odot X'\rangle}  r_\beta(\beta^{1/q} \odot X'; Y') |X_l - X_l'|^q] \wtilde{Q}(\omega) d\omega \\
		&= \frac{1}{\lambda} \cdot  \int
		\E[ R_{\beta, \omega}(\beta^{1/q} \odot X; Y) 
			\wbar{R_{\beta, \omega}(\beta^{1/q} \odot X'; Y')} |X_l - X_l'|^q] \cdot \wtilde{Q}(\omega) d\omega 
\end{split}
\end{equation*}
This completes the proof of Lemma~\ref{lemma:new-representation-of-grad-obj}.

\subsection{Proof of Lemma~\ref{lemma:approximate-gradient-bound}.}
Fix $\beta \ge 0$ and $l \in [p]$. 
Write $\Delta_{\beta, \omega} = \E[R_{\beta, \omega}(\beta^{1/q} \odot X; Y)]$
so that 
$R_{\beta, \omega}(\beta^{1/q} \odot X; Y) = \wbar{R}_{\beta, \omega}(\beta^{1/q} \odot X; Y) + \Delta_{\beta, \omega}$.
Algebraic manipulation yields
$
\partial_{\beta_l} \obj(\beta) = \wtilde{\partial_{\beta_l} \obj}(\beta) + \error_{1, l}(\beta) + \error_{2, l}(\beta)$
where the error terms are defined by 
\begin{equation*}
\begin{split}
\error_{1, l}(\beta) &=  \frac{2}{\lambda} \cdot 
	\int  \Re \left\{\wbar{\Delta_{\beta, \omega}} \cdot \E\left[\wbar{R}_{\beta, \omega}(\beta^{1/q} \odot X'; Y') |X_l - X_l'|^q\right] \right\} 
		\cdot \wtilde{Q}(\omega) d\omega, \\
\error_{2, l}(\beta) &= \frac{1}{\lambda} \cdot  \int \left|\Delta_{\beta, \omega}\right|^2 \E[|X_l - X_l'|^q] \cdot \wtilde{Q}(\omega) d\omega.
\end{split}
\end{equation*}
By the triangle inequality, we immediately arrive at the bound
\begin{equation}
\label{eqn:deviance-approx-obj-partial-obj}
	|\partial_{\beta_l} \obj(\beta) -  \wtilde{\partial_{\beta_l} \obj}(\beta)| \le  \left|\error_{1, l}(\beta)\right| + \left|\error_{2, l}(\beta)\right|.
\end{equation}
It remains to bound $\left|\error_{1, l}(\beta)\right|$ and $ \left|\error_{2, l}(\beta)\right|$. 

Assume for now $\beta > 0$. By Lemma~\ref{lemma:KKT-condition}, we have for almost all $\omega$ (w.r.t Lebesgue measure)
\begin{equation*}
	\Delta_{\beta, \omega} = \E\left[r_{\beta}(\beta^{1/q} \odot X; Y)e^{i \langle \omega, \beta^{1/q} \odot X \rangle}\right] 
		=  \frac{\lambda}{(2\pi)^{p/2}} \cdot \frac{\F(f_\beta)(\omega)}{Q(\omega)}.
\end{equation*}
Using Proposition~\ref{prop:second-moment-bound}, $\E[|\wbar{R}_{\beta, \omega}(\beta^{1/q} \odot X'; Y')|^2] 
\le \E[Y^2]$. With Cauchy-Schwartz, we obtain 
\begin{equation}
\label{eqn:error-bounds}
\begin{split}
	|\error_{1, l}(\beta)| &\le \frac{1}{(2\pi)^{p/2}} \cdot 8
		 \cdot\E[X^{2q}]^{1/2} \cdot \E[Y^2]^{1/2} \cdot \int  |\F(f_\beta)(\omega)| \cdot \frac{\wtilde{Q}(\omega)}{Q(\omega)} d\omega. \\
	|\error_{2, l}(\beta)| &\le \frac{1}{(2\pi)^{p}} 
		\cdot 4\lambda \cdot\E[X^{2q}]^{1/2} \cdot \int  \frac{|\F(f_\beta)(\omega)|^2}{Q(\omega)} \cdot \frac{\wtilde{Q}(\omega)}{Q(\omega)} d\omega.
\end{split}
\end{equation}
To further bound the RHS of equation~\eqref{eqn:error-bounds}, we notice the following facts. 
\begin{itemize}
\item By assumption, $\supp(\mu) \subseteq [0, M_\mu]$. As a result, we have the following bound 
	\begin{equation*}
		\sup_{\omega}\left|\frac{\wtilde{Q}(\omega)}{Q(\omega)}\right| = \sup_{\omega}\left|\frac{\int tq_t(\omega) \mu(dt)}{\int q_t(\omega)\mu(dt)}\right| \le 
			M_\mu.
	\end{equation*}
\item By Cauchy Schwartz inequality, we have 
	\begin{equation*}
		\left(\int |\F(f_\beta)(\omega)| d\omega\right)^2 \le \int \frac{|\F(f_\beta)(\omega)|^2}{Q(\omega)} d\omega \cdot \int Q(\omega) d\omega
	\end{equation*}
Note then 
	\begin{equation*}
		\int \frac{|\F(f_\beta)(\omega)|^2}{Q(\omega)} d\omega = (2\pi)^p \norm{f_\beta}_{\H}^2~~\text{and}~~\int Q(\omega) d\omega = |h(0)|.
	\end{equation*}
\item The bound $\norm{f_\beta}_{\H}^2 \le \frac{1}{\lambda}\E[Y^2]$ holds. Indeed, 
	$\lambda \norm{f_\beta}_{\H}^2\le \energy(\beta, f_\beta) \le \energy(\beta, 0) = \E[Y^2]$.
\end{itemize}
Substitute the above bound into equation~\eqref{eqn:error-bounds}. This shows for $C = 4 M_\mu \cdot \E[X^{2q}]^{1/2} \cdot \E[Y^2]$
\begin{equation}
\label{eqn:error-bounds-two}
	|\error_{1, l}(\beta)| \le |h(0)|^{1/2} \cdot \frac{2C}{\sqrt{\lambda}}, ~~\text{and}~~
	|\error_{2, l}(\beta)| \le  C, 
\end{equation}
We have shown this bound holds for all $\beta > 0$. Note that the same bound holds on 
$\beta \ge 0$ since both 
$\beta \mapsto \error_{1, l}(\beta)$ and $\beta \mapsto \error_{2, l}(\beta)$ are continuous on $\beta \ge 0$
by Proposition~\ref{prop:continuity-f-beta-beta-X}.

Finally, plugging the error bounds~\eqref{eqn:error-bounds-two} into equation~\eqref{eqn:deviance-approx-obj-partial-obj}
yields Lemma~\ref{lemma:approximate-gradient-bound} as desired.

\subsection{Proof of Landscape Result---Proposition~\ref{sec:proof-landscape-toy}.}
\label{sec:proof-proposition-landscape-analysis-main-effect}

This section gives a complete proof of Proposition~\ref{sec:proof-landscape-toy}, which is divided into three parts. 
%The first part concerns the global minimum, i.e., Part (i) of Proposition~\ref{sec:proof-landscape-toy}.
%In the second part, we prove the landscape result on the stationary point when $q = 1$, i.e., Part (ii) of Proposition~\ref{sec:proof-landscape-toy}.
%In the last part, we prove the landscape result on the stationary point when $q = 2$, i.e., Part (iii), (iv) of Proposition~\ref{sec:proof-landscape-toy}.
Throughout the proof, we denote $M_X^4 = \max_{l \in [p]}  \E[X_l^4]$ and $M_Y^2 = \E[Y^2]$. 
Let $M_\mu$ be such that $\supp(\mu) \subseteq [0, M_\mu]$.

\subsubsection{Proof of Part $(i)$ of Proposition~\ref{sec:proof-landscape-toy} (Global minimum).}
We prove that the global minimum satisfies $\supp(\beta) = [p]$ for both $q = 1$ and $q = 2$. 
The proof contains two steps. 

\begin{itemize}
\item In the first step, we prove that for any $\beta$ which does not have full support, i.e., 
$\supp(\beta) \neq [p]$, the objective value $\obj(\beta)$ at $\beta$ satisfies the following lower bound: 
\begin{equation}
\label{eqn:step-one-obj-beta}
	\obj(\beta) \ge \half \cdot \min_{j\in [p]} \E[f_j^*(X_j)^2] > 0.
\end{equation}
To see this, pick any $j \not\in \supp(\beta)$. The key point is that 
$f(\beta^{1/q} \odot X)$ does not depend on $X_j$ and thus has no power on explaining the main effect $f_j^*(X_j)$. 
Formally, using the mutual independence $X_1 \perp X_2 \perp \ldots \perp X_p$, we obtain that for all function $f$,
%the following lower bound: %that holds for all possible function $f$: 
\begin{equation*}
\begin{split}
	\E[(Y - f(\beta^{1/q}\odot X))]^2% &= \E[f_j^*(X_j)^2] + \E\Big[\big(\sum_{l \neq j} f_l^*(X_l) - f(\beta^{1/q}\odot X)\big)^2\Big] 
		\ge \E[f_j^*(X_j)^2].
\end{split}
\end{equation*}
Recall $J(\beta) = \min_f \energy(\beta, f)$ where 
%The above lower bound holds for all function $f$. 
%Now, the above lower bound immediately implies the following lower bound 
%on 
$\energy(\beta, f) = \half \E[(Y - f(\beta^{1/q}\odot X))]^2 + \frac{\lambda}{2} \norm{f}_{\H}^2$. 
This proves that the desired bound~\eqref{eqn:step-one-obj-beta} holds for all $\beta$ that 
does not have full support. 
\item 
In the second step, we fix a feasible $\beta^0$ that has full support. 
We prove that% that the objective value $\mathcal{J}(\beta^0)$ tends to $0$ as $\lambda$ tends to $0$:
\begin{equation}
\label{eqn:limit-kernel-ridge-regression}
	\lim_{\lambda \to 0^+} \mathcal{J}(\beta^0) = 0. 
\end{equation}
To prove equation~\eqref{eqn:limit-kernel-ridge-regression}, the key observation is to notice that 
the kernel $(x, x') \mapsto k(x, x')$ is universal~\cite{MicchelliYuZh06}. To see this, if we express the 
translation invariant kernel $k(x,x')$ as $k(x, x') = g(x-x')$ where $g(z) = h(\norm{z}_q^q)$, then 
$g$ has the property that its Fourier transform has full support on the entire space $\R^p$, which implies 
that the kernel is universal~\cite[Proposition 15]{MicchelliYuZh06}. To see the property, note that 
$\F(g) (\omega) = \int_0^\infty q_t(\omega) \mu(dt)$ and therefore $\F(g) (\omega)$ is of full support 
since $\omega \mapsto q_t(\omega)$ is of full support for all $t > 0$.
As a consequence of the fact that $k(x, x')$ is universal, using the assumption that 
$\supp(\P_X)$ is compact, and the fact $\beta^0$ is of full support, it implies 
that~\cite{MicchelliYuZh06}
\begin{equation}
	\inf_{f\in \H} \E[(f^*(X) - f((\beta^0)^{1/q} \odot X))^2] = \inf_{f \in \H} \sup_{x \in \supp(P_X)} 
		|f^*(x) -f((\beta^0)^{1/q} \odot x) |= 0. 
\end{equation}
Hence, $\lim_{\lambda \to 0^+} \mathcal{J}(\beta^0) 
	\le \inf_{f\in \H} ~\half \E[(f^*(X) - f((\beta^0)^{1/q} \odot X))^2] = 0$.
As a consequence, 
this would imply for some $\lambda^* > 0$, we have for all $\lambda\le \lambda^*$, the objective at $\beta^0$ satisfies
\begin{equation}
\label{eqn:step-two-obj-beta}
	\obj(\beta^0) < \half \cdot \min_{j\in [p]} \E[f_j^*(X_j)^2].% < \half \cdot \min_{j\in [p]} \E[f_j^*(X_j)^2].
\end{equation}
\end{itemize}
To summarize, we can combine the claims at equation~\eqref{eqn:step-one-obj-beta} and 
equation~\eqref{eqn:step-two-obj-beta} to conclude that the global minimum must be of full 
support whenever $\lambda \le \lambda^*$ for some $\lambda^* < \infty$.

% \begin{itemize}
% \item First consider any $\beta$ who does not have full support. Select any 
% 	$j \not\in \supp(\beta)$. The point is that the regression function $f(\beta^{1/q} \odot X)$ does not depend on $X_j$, 
%	and thus has no explanation on the main effect $f_j^*(X_j)$. Hence, 
%	$\E[(Y - f(\beta^{1/q}\odot X))]^2 \ge \E[f_j^*(X_j)^2]$, which further implies that $\obj(\beta) \ge \E[f_j^*(X_j)^2] > 0$ 
%	for $j \not\in \supp(\beta)$. Hence $\obj(\beta) \ge \min_{j\in [p]} \E[f_j^*(X_j)^2] > 0$ for all $\beta$
%	that do not have full support. 
% \item Next, fix any feasible $\beta^0$ that has full support, say $\beta^0 = \frac{M}{p} \mathbf{1}$.  As both $\ell_1$ and 
% 	$\ell_2$ kernels are universal kernels, and as there are no noise in $Y$, this implies that the minimum value of the kernel 
%	ridge regression $\obj(\beta^0)$ tends to $0$ as $\lambda$ tends to $0$. In particular, we can choose $\lambda^*$
%	so that $\obj(\beta^0) < \min_{j\in [p]} \E[f_j^*(X_j)^2]$ when $\lambda\le \lambda^*$.
% \end{itemize}
%Now it's clear that the global minimum must satisfy $\supp(\beta) = [p]$ for all $\lambda \le \lambda^*$.

\subsubsection{Proof of Part $(ii)$ of Proposition~\ref{sec:proof-landscape-toy} (Stationary points for $q=1$).}
Let $q = 1$. The key to the proof is to show that for all variables $l \in [p]$ and all $\beta$ such that $\beta_l = 0$, 
\begin{equation}
\label{eqn:landscape-result}
\begin{split}
	&\partial_{\beta_l} \obj(\beta) \le - \frac{1}{\lambda} \left( |h^\prime(2MM_X)| \cdot S_l + O(\sqrt{\lambda})\right) 
	~~\text{where}~~S_l = \int \left|\E[f_l^*(X_l) e^{i\omega X_l}]\right|^2 \frac{d\omega}{\pi\omega^2} . 
\end{split}
\end{equation}
where $O(\sqrt{\lambda})$ denotes a remainder term whose absolute value is upper bounded by 
$C\sqrt{\lambda}$ where $C > 0$ depends only on $M_X, M_Y$ and $M_\mu$. Given
equation~\eqref{eqn:landscape-result}, we show that Proposition~\ref{sec:proof-landscape-toy}
holds. Indeed, $S_l > 0$ since $f_l^*(X_l) \neq 0$  by assumption, and $|h^\prime(2MM_X)| > 0$
since $h$ is strictly completely monotone. Hence, there exists some $\lambda^* > 0$ such that for all $\lambda \le \lambda^*$, 
we have $\partial_{\beta_l} \obj(\beta) < 0$ for all $\beta$ such that $\beta_l = 0$. This means 
that any $\beta$ such that 
$\beta_l = 0$ can't be a stationary point of $\obj(\beta)$. 

It remains to prove inequality~\eqref{eqn:landscape-result}. By Proposition~\ref{prop:statistical-insights}, we have the identity
\begin{equation*}
	\partial_{\beta_l} \obj(\beta) = -\frac{1}{\lambda} 
		\left(\iint \Cov^2 \left({R}_{\beta, \omega}(\beta \odot X; Y), e^{i \zeta_l X_l }\right)
			\cdot \frac{d\zeta_l}{\pi\zeta_l^2} \cdot \wtilde{Q}(\omega) d\omega. + O(\sqrt{\lambda})\right).
\end{equation*}
Hence, it suffices to show the following lower bound 
\begin{equation}
\label{eqn:landscape-result-two}
	\iint \Cov^2 \left({R}_{\beta, \omega}(\beta \odot X; Y), e^{i \zeta_l X_l }\right)
			\cdot \frac{d\zeta_l}{\pi\zeta_l^2} \cdot \wtilde{Q}(\omega) d\omega \ge |h^\prime(2MM_X)| \cdot S_l.
\end{equation}
To prove this, we first evaluate the covariance inside the integral. %Note that 
%\begin{equation*}
%	\Cov \left({R}_{\beta, \omega}(\beta \odot X; Y), e^{i \zeta_l X_l }\right)
%		= \Cov\left(\E\left[{R}_{\beta, \omega}(\beta \odot X; Y) \mid X_l\right], e^{i \zeta_l X_l }\right)
%\end{equation*}
%Now we evaluate $\E[\wbar{R}_{\beta, \omega}(\beta \odot X; Y) \mid X_l]]$. 
By definition, we have 
\begin{equation*}
\begin{split}
	R_{\beta, \omega}(\beta \odot X; Y)% &= e^{i \langle \omega, \beta \odot X\rangle} \left(Y - f_\beta(\beta\odot X)\right) \\
		&= e^{i \langle \omega, \beta \odot X\rangle} \Big(f_l^*(X_l) + \sum_{j \in S\backslash l} f_j^*(X_j)  - f_\beta(\beta\odot X)\Big).
\end{split}
\end{equation*}
At $\beta$ where $\beta_l = 0$, the random variables $e^{i \langle \omega, \beta\odot X\rangle}$ and 
$f_\beta(\beta\odot X)$ depend only on $X_{[p]\backslash \{l\}}$, and are thus independent of $X_l$
by assumption. As a result, we obtain %the following formula on $\E[\wbar{R}_{\beta, \omega}(\beta \odot X; Y) \mid X_l]$: 
\begin{equation*}
\begin{split}
\E[\wbar{R}_{\beta, \omega}(\beta \odot X; Y) \mid X_l]
		= %\int \left|\E[f_l^*(X_l) e^{i\omega X_l}]\right|^2 \frac{d\omega}{\omega^2} 
			\E[R_{\beta, \omega}(\beta \odot X; Y) \mid X_l]  - \E[R_{\beta, \omega}(\beta \odot X; Y)] %\\
		= \E[e^{i \langle \omega, \beta \odot X\rangle}] \cdot f_l^*(X_l).
		%= \E[e^{i \langle \omega, \beta \odot X\rangle}] \cdot f_l^*(X_l).
\end{split}
\end{equation*}
Consequently, we can obtain the following identity on the covariance
\begin{equation*}
	\Cov \left({R}_{\beta, \omega}(\beta \odot X; Y), e^{i \zeta_l X_l }\right) = \E[e^{i \langle \omega, \beta \odot X\rangle}] \cdot 
		\E\left[f_l^*(X_l) e^{i \zeta_l X_l}\right].
\end{equation*}
Substitute this back into the integral on the LHS of equation~\eqref{eqn:landscape-result-two}. We obtain the identity 
\begin{equation}
\label{eqn:landscape-result-three}
\begin{split}
&\iint \Cov^2 \left({R}_{\beta, \omega}(\beta \odot X; Y), e^{i \zeta_l X_l }\right)
	\cdot \frac{d\zeta_l}{\pi\zeta_l^2} \cdot \wtilde{Q}(\omega) d\omega \\
		= & \int \left|\E[f_l^*(X_l) e^{i\omega X_l}]\right|^2 \frac{d\omega}{\pi\omega^2} \cdot 
			\int \left| \E[e^{i \langle \omega, \beta \odot X\rangle}]\right|^2 \wtilde{Q}(\omega)d\omega 
		= S_l \cdot \E[h'(\norm{X-X'}_{1, \beta})],
\end{split}
\end{equation}
where we use the integral formula in equation~\eqref{eqn:h-prime-integral-formula} to derive the last identity.

Note that $\norm{\beta}_1 \le M$. Hence, $\E[\norm{X-X'}_{1, \beta}] \le 2M M_X$. Consequently, Jenson's inequality implies
that $\E[h'(\norm{X-X'}_{1, \beta})] \le h^\prime(\E[\norm{X-X'}_{1, \beta}]) \le h^\prime(2MM_X) \le 0$ since $h$ is completely monotone
($h^\prime \le 0$ and $h^{\prime}$ is concave). This proves equation~\eqref{eqn:landscape-result-two} as desired. 
%This proves the key result in equation~\eqref{eqn:landscape-result}, and hence the first part of Proposition~\ref{sec:proof-landscape-toy}.

\subsubsection{Proof of Part $(iii)$ of Proposition~\ref{sec:proof-landscape-toy} (Stationary points for $q=2$).}
Let $q = 2$. Note then $\Cov(Y, X_l) = 0$ for all $l$ since $\Cov(f_l^*(X_l), X_l) = 0$ by assumption. 
Using the representation of  
$\grad \obj(\beta)$ in Proposition~\ref{proposition:compute-grad-obj}, we obtain that at $\beta = 0$,
\begin{equation*}
\partial_{\beta_l} \obj(\beta)\mid_{\beta = 0} = 
	-\frac{1}{\lambda} \cdot h'(0) \cdot \E\left[YY' |X_l- X_l'|^2\right] = 
		 -\frac{1}{\lambda} \cdot h'(0) \cdot \Cov(Y, X_l)^2 = 0.
\end{equation*}
Accordingly, $0$ is a stationary point of $\obj(\beta)$ under the assumption. 

\subsubsection{Proof of Part $(iiv)$ of Proposition~\ref{sec:proof-landscape-toy} (Stationary points for $q=2$).}
Let $q = 2$. We prove the following key result on the gradient that holds for all $\beta$ with $\beta_l = 0$: 
	\begin{equation}
	\label{eqn:gradient-nonnegative-at-l-0}
		\partial_{\beta_l} \obj(\beta) \ge 0.
	\end{equation}
Given this result, the (restricted) minimum $\beta^{l,*}$ of $\obj(\beta)$ over 
	$\mathcal{B}_M^l = \mathcal{B}_M \cap \{\beta: \beta_l = 0\}$ is a stationary point of $\obj(\beta)$ w.r.t the original 
	feasible set $\mathcal{B}_M$. To see this, we only need to show that $\langle \grad \obj(\beta^{l, *}), \beta' - \beta^{l,*} \rangle \ge 0$
	holds for any $\beta' \in \mathcal{B}_M$. This is true because (i) $\partial_{\beta_l} \obj (\beta^{l, *}) \ge 0$ by 
	equation~\eqref{eqn:gradient-nonnegative-at-l-0} and (ii) $\langle\partial_{\beta_{[p] \backslash l}} \obj(\beta^{l, *}), (\beta' -  \beta^{l, *})_{[p] \backslash l}\rangle \ge 0$ for all $\beta' \in \mathcal{B}_M$.
	%since $\beta^{l,*}$ is the minimum of $\obj(\beta)$ over $\mathcal{B}_M^l$.
		
	Now we prove the deferred equation~\eqref{eqn:gradient-nonnegative-at-l-0} that holds for all $\beta$ with $\beta_l = 0$.
	Fix a $\beta$ such that $\beta_l = 0$. By Proposition~\ref{proposition:compute-grad-obj}, the gradient admits the 
	representation
	\begin{equation}
	\label{eqn:stationary-point-one-q-2}
	\begin{split}
		\partial_{\beta_l} \obj(\beta) &= -\frac{1}{\lambda} \cdot   
			\E\left[r_\beta(\beta^{1/2}\odot X; Y)r_\beta(\beta^{1/2}\odot X'; Y') h'(\normsmall{X-X'}_{2, \beta}^2)|X_l- X_l'|^2\right].	
	\end{split}	
	\end{equation}
	Since $\beta_l = 0$, we can decompose
		$r_\beta(\beta^{1/2}\odot X; Y) = f_l^*(X_l) + \error_l(X_{[p]\backslash l})$
	where $\error_l(X_{[p]\backslash l}) = \sum_{j \neq l} f_j^*(X_j) - f_\beta(\beta^{1/2} \odot X)$ depends only on 
	$X_{[p] \backslash l}$. Hence, we obtain %Substituting it into equation~\eqref{eqn:stationary-point-one-q-2} yields
	%After this substitution, we obtain the following expression: % of the gradient $\partial_{\beta_l} \obj(\beta)$ at the $\beta$: 
	\begin{equation*}
		\partial_{\beta_l} \obj(\beta) = \frac{1}{\lambda} \cdot \left(\wtilde{\mathcal{E}}_1 + 2\wtilde{\mathcal{E}}_2 + \wtilde{\mathcal{E}}_3 \right)%+ \wtilde{\mathcal{E}}_4\right)
	\end{equation*}
	where the error terms are defined by 
	\begin{equation*}
	\begin{split}
		\wtilde{\mathcal{E}}_1 &= -\E[f_l^*(X_l) f_l^*(X_l') h'(\normsmall{X-X'}_{2, \beta}^2)|X_l- X_l'|^2].\\
		\wtilde{\mathcal{E}}_2 &= -\E[f_l^*(X_l) \error_l(X'_{[p] \backslash l}) h'(\normsmall{X-X'}_{2, \beta}^2)|X_l- X_l'|^2]\\
		%\wtilde{\mathcal{E}}_3 &= -\E[\error_l(X_{[p] \backslash l}) f_l^*(X_l')  h'(\normsmall{X-X'}_{2, \beta}^2)|X_l- X_l'|^2]\\
		\wtilde{\mathcal{E}}_3 &= -\E[\error_l(X_{[p] \backslash l}) \error_l(X'_{[p] \backslash l}) h'(\normsmall{X-X'}_{2, \beta}^2)|X_l- X_l'|^2]
	\end{split}
	\end{equation*}
	Now we show $\wtilde{\error}_1 = \wtilde{\error}_2 = 0$ and $\wtilde{\error}_3 \ge 0$. To do so, 
	we exploit the facts: (i) $\norm{X-X'}_{2, \beta} \perp X_l$ since $X_l \perp X_{[p] \backslash l}$ and $\beta_l = 0$ 
	and %(ii) $\E[ZZ'|X_l -X_l'|^2] = \E[Z] \cdot\E[ZX_l^2] - (\E[ZX_l])^2$ and (iii) 
	(ii) $\E[f_l^*(X_l)] = \E[f_l^*(X_l) X_l] = \E[f_l^*(X_l) X_l^2] = 0$. Consequently, we obtain the desired result as follows: 
	\begin{itemize}
	\item $\wtilde{\mathcal{E}}_1= -\E[f_l^*(X_l) f_l^*(X_l') |X_l - X_l'|^2] \cdot \E[h'(\normsmall{X-X'}_{2, \beta}^2)] = 0$. 
	\item $\wtilde{\mathcal{E}}_2= -\E[f_l^*(X_l) |X_l - X_l'|^2] \cdot \E[\error_l(X'_{[p] \backslash l}) h'(\normsmall{X-X'}_{2, \beta}^2)] = 0$.
	%\item Note that $\wtilde{\mathcal{E}}_3= -\E[f_l^*(X_l') |X_l - X_l'|^2] \cdot \E[\error_l(X_{[p] \backslash l}) h'(\normsmall{X-X'}_{2, \beta}^2)] = 0$.
	\item $\wtilde{\mathcal{E}}_3= -\E[|X_l - X_l'|^2] \cdot \E[\error_l(X_{[p] \backslash l})\error_l(X'_{[p] \backslash l}) h'(\normsmall{X-X'}_{2, \beta}^2)]$. To show $\wtilde{\mathcal{E}}_3 \ge 0$, note then (a) 
		$\E[|X_l - X_l'|^2] = 2\Var(X_l) \ge 0$ and (b)
		$(x, x') \mapsto h'(\norm{x-x'}_{q, \beta}^q)$ is a negative kernel. 
%		
%
%		Below we show $\wtilde{\error}_4 \ge 0$. First, a simple calculation shows that 
%		\begin{equation*}
%			\E[|X_l - X_l'|^2] = 2\Var(X_l) \ge 0.
%		\end{equation*}
%		Next, using the representation $h^\prime(\norm{x-x'}_2^2) = \int e^{i\langle \omega, x-x'\rangle}\wtilde{Q}(\omega) d\omega$, we obtain
%		\begin{equation*}
%			-\E[\error_l(X_{[p] \backslash l})\error_l(X'_{[p] \backslash l}) h'(\normsmall{X-X'}_{2, \beta}^2)] 
%		=  \int \left|\E[\error_l(X_{[p] \backslash l}) e^{i \langle \omega, \beta^{1/2} \odot X\rangle}]\right|^2 \wtilde{Q}(\omega) d\omega \ge 0
%		\end{equation*}
%		where $\wtilde{Q}: \R^p \to \R_+$ is defined in equation~\eqref{eqn:tilde-Q-R-beta-omega-definition}. This proves
%		$\wtilde{\error}_4 \ge 0$ as desired. 
	\end{itemize}

%\rfcomment{I do not know what to revise for below. Defer it to the end}

\section{Population-level Guarantees}
\label{sec:proof-main-theorems}

\subsection{Definition of the signal set $S$.}
\label{sec:definition-of-the-signal-set}
Proposition~\ref{proposition:unique-define-S} shows that Definition~\ref{definition:signal-set} is proper. 
%and defines a unique set $S$. 
\begin{proposition}
\label{proposition:unique-define-S}
There exists a unique subset $S \subseteq [p]$ with the following three properties: 
\begin{enumerate}[(i)]
\item $\E[Y \mid X] = \E[Y \mid X_S]$
\item $X_S \perp X_{S^c}$
\item There is no strict subset $A \subsetneq S$ which satisfies (i) and (ii). 
\end{enumerate}
\end{proposition}

\begin{proof}
First we prove existence. Start with $S=\left\{ 1,\ldots,p\right\} $
and note that it trivially satisfies (i) and (ii). If no strict subset
of $\left\{ 1,\ldots,p\right\} $ satisfies (i) and (ii), then $S$
satisfies (iii) also and we are done. Otherwise if a strict subset
$A\subsetneq S$ satisfies (i) and (ii), set $S$ equal to $A$. Repeat
this process until we arrive at a set $S$ for which there is no strict subset
that satisfies (i) and (ii). This process terminates in at most $p$ steps
and the $S$ returned by the process satisfies (i), (ii), (iii).

Next, we prove uniqueness. Suppose there exist subsets $A,B$ satisfying
(i), (ii) and (iii). By (i), $\E[Y|X_{A}] = \E[Y|X_{B}]$.
Taking the conditional expectation w.r.t $X_A$ yields
\begin{equation*}
\E[Y | X_A] = \E[\E[Y | X_B]| X_A] 
	=\E [\E[Y | X_B]| X_{A \cap B}] = \E[Y | X_{A \cap B}]
\end{equation*}
where the second equality comes from the fact that $X_{A \backslash B}\perp X_{B}$ since 
$B$ satisfies (ii) and the third equality comes from the tower property of conditional 
expectation. Thus, we have
\begin{equation*}
 \E[Y | X_{A \cap B}] = \E[Y | X_A] = \E[Y|X].
\end{equation*}
Moreover, denoting $P(X_T)$ to be the density of $X_T$, we have 
\begin{equation*}
P(X) = P(X_B)P(X_{B^c}) = P(X_{A\cap B}) P(X_{B \backslash A}) P(X_{B^c})
\end{equation*}
where the first equality is from $X_{B}\perp X_{B^{c}}$ and the second
equality is from $X_{A}\perp X_{A^{c}}.$ Thus $X_{A \cap B} \perp X_{(A \cap B)^c}$. 
Hence, we have shown $A\cap B$ is a subset that satisfies (i) and (ii). 
Since $A, B$ satisfy (iii), it implies $A = A\cap B = B$. This proves the uniqueness.
\end{proof}

\subsection{Proof of Theorem~\ref{thm:add-main-mix-effect}.}
\label{sec:proof-of-theorem-add-main-mix-effect}
The proof proceeds in a similar way to that of Theorem~\ref{thm:add-main-effect}. 
%To facilitate understanding, we will highlight the difference in the proof below. 

Here is the starting point: using the fact that $\grad \obj_\gamma(\beta)$ is smooth in 
$\beta$ (Proposition~\ref{prop:grad-obj-lipschitz-beta}), any accumulation point $\beta^*$ 
of the projected gradient descent algorithm must be stationary when the stepsize is 
small (Theorem~\ref{lemma:gradient-ascent-increases-objective}). 
%Hence it 
%suffices to show that any stationary point $\beta^*$ \emph{reachable} by the algorithm 
%must satisfy $\beta_l^* > 0$. 
By Theorem~\ref{thm:no-false-positive}, $\beta^*$ must exclude 
noise variables, i.e., $\beta^*_{S^c} = 0$.
Hence, it suffices to show that any $\beta$ with $\beta_l = 0$ and $\beta_{S^c} = 0$ 
can't be stationary. To see this, pick $\beta$ such that $\beta_l = 0$ and $\beta_{S^c} = 0$.
To show that it is non-stationary, it suffices to show that the gradient w.r.t $\beta_l$ is strictly negative, i.e.
\begin{equation}
\label{eqn:desired-goal-on-obj-main-add-mix}
	\partial_{\beta_l} \regobj(\beta) = \partial_{\beta_l} \obj(\beta) + \gamma < 0. 
\end{equation}
To show equation~\eqref{eqn:desired-goal-on-obj-main-add-mix}, 
Lemma~\ref{lemma:key-to-thm-add-main-mix-effect} is the key, whose proof is deferred to
Section~\ref{sec:proof-lemma-key-to-thm-add-main-mix-effect}. 
\begin{lemma}
\label{lemma:key-to-thm-add-main-mix-effect}
The following inequality holds for all $\beta$ such that $\beta_l = 0$ and $\beta_{S^c} = 0$: 
\begin{equation}
\label{eqn:grad-main-add-mix-effect}
	\partial_{\beta_l} \obj_\gamma(\beta) \le \frac{1}{\lambda} \cdot  \left(-c \cdot \effect_l + C\lambda^{1/2}(1+\lambda^{1/2})+\lambda \gamma\right).
\end{equation}
\end{lemma}
By
Lemma~\ref{lemma:key-to-thm-add-main-mix-effect}, equation~\eqref{eqn:desired-goal-on-obj-main-add-mix} 
holds for all $\beta$ with $\beta_l = 0$ and $\beta_{S^c} = 0$, provided the constant $\wbar{C} > 0$ in 
equation~\eqref{eqn:add-main-mix-effect} is sufficiently large. This concludes the proof of  
Theorem~\ref{thm:add-main-mix-effect}.

\subsection{Deferred proof of Lemma~\ref{lemma:key-to-thm-add-main-mix-effect}.}
\label{sec:proof-lemma-key-to-thm-add-main-mix-effect}
The key to the proof is to derive a tight bound on the surrogate gradient $\wtilde{\partial_{\beta_l} \obj}(\beta)$. 
Write $q_0(\zeta) = \frac{1}{\zeta^2}$. By equation~\eqref{eqn:wtilde-partial-beta-j-rep}, $\wtilde{\partial_{\beta_l} \obj}(\beta)$
satisfies
\begin{equation}
\label{eqn:wtilde-partial-beta-j-rep-mix}
\begin{split}
\wtilde{\partial_{\beta_l} \obj}(\beta) &=
		- \frac{1}{\lambda} \cdot  \int \left(\int \Big|\E[\wbar{R}_{\beta, \omega}(\beta \odot X; Y)e^{i \zeta_l X_l} ]\Big|^2 
			\cdot q_0(\zeta_l) d\zeta_l \right) \cdot \wtilde{Q}(\omega) d\omega. 
\end{split}
\end{equation}
Lemma~\ref{lemma:eqn-crucial-identity-technical} evaluates the RHS of 
equation~\eqref{eqn:wtilde-partial-beta-j-rep-mix} and  provides
a more explicit expression of $\wtilde{\partial_{\beta_l} \obj}(\beta)$ under the functional 
ANOVA model. The proof is given in Appendix~\ref{sec:proof-lemma-eqn-crucial-identity-technical}.

%. The derivation 
%is in Appendix~\ref{sec:proof-lemma-eqn-crucial-identity-technical}.
\begin{lemma}
\label{lemma:eqn-crucial-identity-technical}
Assume the functional ANOVA model. Then $\wtilde{\partial_{\beta_l} \obj}(\beta) = U(\beta)$ 
holds at any $\beta$ with $\beta_l = 0$ and $\beta_{S^c} = 0$:
(recall the definition of $F_l(X_S)$ in Definition~\ref{definition:l-signal-main-mix})
\begin{equation}
\label{eqn:def-U-beta}
\begin{split}
%\wtilde{\partial_{\beta_l} \obj}(\beta) = U(\beta)&  \\
	U(\beta) &\defeq- \frac{1}{\lambda} \cdot 
		\iint \left|\E\Big[e^{i\zeta_l X_l + i \langle \omega_{S\backslash l}, \beta_{S\backslash l} \odot X_{S\backslash l}\rangle}
		 	\cdot  F_l(X_S)\Big]\right|^2 q_0(\zeta_l) 
		\wtilde{Q}(\omega_{S \backslash l}) d\zeta_l d\omega_{S \backslash l}.
\end{split}
\end{equation}
\end{lemma}

Lemma~\ref{lemma:upper-bound-tilde-J} analyzes and gives tight upper bounds on $U(\beta)$---this is perhaps the 
more technical part of the entire proof of Lemma~\ref{lemma:key-to-thm-add-main-mix-effect}.  The proof 
is given in Appendix~\ref{sec:proof-lemma-upper-bound-tilde-J}.

%Now we leverage Lemma~\ref{lemma:eqn-crucial-identity-technical} and equation~\eqref{eqn:wtilde-partial-beta-j-rep-mix} to show 
%the bound that holds for some constant $c > 0$ that depends only on $M, M_X, M_Y, \mu$: 
%\begin{equation}
%\label{eqn:goal-upper-bound-tilde-J}
%\wtilde{\partial_{\beta_l} \obj}(\beta) \le - \frac{c}{\lambda} \cdot \effect_l. 
%\end{equation}
%To simplify the proof, we introduce notation. Let $U(\beta)$ denote the RHS of equation~\eqref{eqn:crucial-identity-technical}:
%%\begin{equation}
%%\label{eqn:def-U-beta}
%%U(\beta) = - \frac{1}{\lambda} \cdot 
%%		\iint \left|\E\Big[e^{i\zeta_l X_l + i \langle \omega_{S\backslash l}, \beta_{S\backslash l} \odot X_{S\backslash l}\rangle}
%%		 	\cdot  F_l(X_S)\Big]\right|^2 q_0(\zeta_l) 
%%		\wtilde{Q}(\omega_{S \backslash l}) d\zeta_l d\omega_{S \backslash l}.
%%\end{equation} 
%Lemma~\ref{lemma:upper-bound-tilde-J} upper bounds $U(\beta)$. This is the core technical argument in the proof. 
%The proof of Lemma~\ref{lemma:upper-bound-tilde-J} is given in Appendix~\ref{sec:proof-lemma-upper-bound-tilde-J}. 
%We give a quick intuitive explanation of 
%the statistical meaning of Lemma~\ref{lemma:upper-bound-tilde-J} in the remark below. 
\begin{lemma}
\label{lemma:upper-bound-tilde-J}
Assume Assumptions~\ref{assumption:mu-compact} and~\ref{assumption:X-Y-bound}. Assume the functional ANOVA model. 
There exist some constants $\tilde{c}, C > 0$ that depends only on $M, M_X, M_Y, \mu$ such that 
\begin{equation}
\label{eqn:upper-bound-tilde-J}
U(\beta) \le - \frac{\tilde{c}}{\lambda} \cdot U_{T; C}(\beta)
\end{equation}
holds for any subset $T \subseteq S$ such that $l \in T$. In above, the quantity $U_{T, C}(\beta)$ is defined by 
\begin{equation}
\label{eqn:def-U-T-C}
	U_{T; C}(\beta) = \bigg(\effect_l(X_T) - C\cdot \sum_{l' \in S\backslash T} \beta_{l'}\bigg)_+ \cdot
	\prod_{\bar{l} \in T \backslash \{l\}} \beta_{\bar{l}}.
\end{equation}
\end{lemma}
%\begin{remark}
%Lemma~\ref{lemma:upper-bound-tilde-J} deserves a quick explanation. By the definition of the $U_{(T; C)}(\beta)$,  
%the bound~\eqref{eqn:upper-bound-tilde-J} basically argues that, in the functional ANOVA model, 
%\begin{equation*}
%	\wtilde{\partial_{\beta_l} \obj}(\beta) = U(\beta) \lesssim -\effect_l(X_T)
%\end{equation*}
%for those $\beta$ whose coordinates in $S\backslash T$ is ``small'', but whose coordinates in $T\backslash l$ is ``large''.
%The intuition is that, at those $\beta$, the set $T \backslash l$ is selected (since the coordinates in $T\backslash l$ is ``large'')
%and the set $S\backslash T$ is not selected (since the coordinates in $S \backslash T$ is ``small''), and thus, the gradient 
%w.r.t the signal variable $X_l$ is characterized by the partial effect of $X_l$ in the group of signals $X_T$, which is precisely 
%determined by the quantity $\effect_l(X_T)$.
%\end{remark}

By Lemma~\ref{lemma:eqn-crucial-identity-technical} and~\ref{lemma:upper-bound-tilde-J}, 
we have shown for some $C, c, c' > 0$ depending only on $M, M_X, M_Y, \mu, |S|$
\begin{equation}
\label{eqn:tilde-grad-main-add-mix-effect}
	\wtilde{\partial_{\beta_l} \obj}(\beta) = U(\beta) \le - \frac{c}{\lambda} \cdot \max_{l \in T: T \subseteq S} U_{T; C}(\beta)
		\le -\frac{c'}{\lambda} \effect_l.
\end{equation}
since $\max_{l \in T: T \subseteq S} U_{T; C}(\beta) \ge c \cdot \effect_l$ for some constant $c > 0$ depending 
only on $|S|$. 
To transfer the bound of $\wtilde{\partial_{\beta_l} \obj}(\beta)$ to $\partial_{\beta_l} \obj(\beta)$, we use
Lemma~\ref{lemma:approximate-gradient-bound} to obtain
\begin{equation}
\label{eqn:grad-main-add-mix-effect}
	\partial_{\beta_l} \obj(\beta) \le \frac{1}{\lambda} \cdot \left(-c \cdot \effect_l + C\lambda^{1/2}(1+\lambda^{1/2})\right),
\end{equation}
where $c, C > 0$ depend on $M, M_X, M_Y, \mu, |S|$. 
Lemma~\ref{lemma:key-to-thm-add-main-mix-effect} follows 
as $\partial_{\beta_l} \obj_\gamma(\beta) = \partial_{\beta_l} \obj(\beta) + \gamma$.% immediately from equation~\eqref{eqn:grad-main-add-mix-effect}. 

\subsection{Proof of Theorem~\ref{thm:hier-effect}.}
\label{sec:proof-thm-hier-effect}
For notation simplicity, throughout the proof, we use double index to index the coordinates in $S$. For instance, 
$\beta_{i, j}$ represents the coordinate that corresponds to the feature $X_{i, j}$. Also, the set 
$S = \cup\left\{(i, j) | 1\le i \le K, 1\le j \le N_i\right\}$.

Fix $k, l$. Let $\what{S}_l$ denote the variables selected by the $l$-th round of the algorithm. It suffices to prove 
the following: $S_{k, m} \subseteq \what{S}_m$ for all $0\le m \le l$. The proof is based on induction on $m$. 

Consider the $m$-th round: the algorithm runs projected gradient descent to solve % problem $(O_m)$
%in this $m$-th round: 
\begin{equation*}
(O_m):~~~~
\begin{aligned}
	\minimize_{\beta} ~~&\mathcal{J}_\gamma(\beta)  \\
	\subjectto ~~& \beta \ge 0 ~\text{and}~\beta_{\what{S}_{m-1}} = \tau  \mathbf{1}_{\what{S}_{m-1}}.
\end{aligned}
\end{equation*}
Since $S_{k, m-1} \subseteq \what{S}_{m-1}$ by induction hypothesis, in order to prove that 
$S_{k, m} \subseteq \what{S}_m$, it suffices to prove that 
$(k, m) \in \what{S}_{m-1} \cup \supp(\beta^*)$. 
To do so, we can W.L.O.G. assume that $(k, m) \not\in \what{S}_{m-1}$. Now we show that $(k, m) \in \supp(\beta^*)$. 
Note the following two facts. 
\begin{itemize}
\item A simple adaptation of the proof of Theorem~\ref{thm:no-false-positive} shows that $\beta^*$ 
must satisfy $\beta^*_{S^c} = 0$.
\item $\beta^*$ must be a stationary point of the problem $(O_m)$.
	Indeed, the objective $\beta \mapsto \obj_\gamma (\beta)$ is smooth (Proposition~\ref{prop:grad-obj-lipschitz-beta}) 
	and thus $\beta^*$ be stationary
	(Lemma~\ref{lemma:gradient-ascent-increases-objective}).
\end{itemize}
As a result, it suffices to prove that any stationary point $\beta$ of the problem $O_m$ 
with $\beta_{S^c} = 0$ must satisfy $\beta_{k, m} > 0$, or equivalently, any $\beta$ with 
$\beta_{k, m} = 0$ and $\beta_{S^c} = 0$ can't be stationary. 

Fix a feasible $\beta$ of the problem $O_m$ with $\beta_{k, m} = 0$ and $\beta_{S^c} = 0$. To show 
 it is non-stationary, it suffices to show that the gradient w.r.t $\beta_{k, m}$ is strictly negative: %we have the bound
\begin{equation}
\label{eqn:desired-goal-on-obj-hier}
	\partial_{\beta_{k, m}} \regobj(\beta) = \partial_{\beta_{k, m}} \obj(\beta) + \gamma < 0.
\end{equation}
To show equation~\eqref{eqn:desired-goal-on-obj-hier}, Lemma~\ref{lemma:key-to-thm-hier-effect} 
is the key, whose proof is deferred to Section~\ref{sec:proof-lemma-key-to-thm-hier-effect}. 

\begin{lemma}
\label{lemma:key-to-thm-hier-effect}
The following holds for all $\beta$ such that $\beta_{S_{k, m-1}} = \tau \mathbf{1}_{S_{k, m-1}}$, $\beta_{k, m} = 0$, $\beta_{S^c} = 0$: 
\begin{equation}
\label{eqn:grad-hier-hard}
	\partial_{\beta_{k, m}} \obj_\gamma(\beta) \le \frac{1}{\lambda} \cdot  \left(-c \cdot \min\{\tau^m, 1\} 
		\cdot \effect_{k, l} + C\lambda^{1/2}(1+\lambda^{1/2})+\lambda \gamma\right).
\end{equation}
\end{lemma}

By Lemma~\ref{lemma:key-to-thm-hier-effect}, any feasible $\beta$ with
$\beta_{k, m} = 0$ and $\beta_{S^c} = 0$ can't be stationary if the constant $\wbar{C} > 0$ in 
equation~\eqref{eqn:hier-effect} is sufficiently large. This completes the proof of Theorem~\ref{thm:hier-effect}.

\subsubsection{Deferred proof of Lemma~\ref{lemma:key-to-thm-hier-effect}.}
\label{sec:proof-lemma-key-to-thm-hier-effect}
The key is to derive a tight bound on the surrogate gradient $\wtilde{\partial_{\beta_{(k, m)}} \obj}(\beta)$.
Write $q_0(\zeta) = \frac{1}{\zeta^2}$. By Lemma~\ref{lemma:eqn-crucial-identity-technical}, 
$\wtilde{\partial_{\beta_{(k, m)}} \obj}(\beta) = V(\beta)$ where\footnote{We can 
apply Lemma~\ref{lemma:eqn-crucial-identity-technical} since (i) hierarchical interaction model is a special 
instance of the functional ANOVA and (ii) $\beta_{k, m} = 0$ and $\beta_{S^c} = 0$)}
\begin{equation}
\begin{split}
\label{eqn:crucial-identity-technical-hier}
V(\beta)  &= 
	- \frac{1}{\lambda} \cdot 
		\iint \left|\E\Big[e^{i\zeta_m X_{k, m} + i \langle \omega_{S\backslash (k, m)},\beta_{S\backslash (k, m)} \odot X_{S\backslash (k, m)}\rangle}
		 	\cdot  F_{k, m}(X_{S_{k, N_k}})\Big]\right|^2 \\
	&\,\,\,\,\,\,\,\,\,\,\,\,\,\,\,\,\,\,\,\,\,\,\,\,\,\,\,\,\,\,\,\,\,\,\,\,
		\cdot q_0(\omega_{(k, m)}) \wtilde{Q}(\omega_{S \backslash (k, m)}) d\zeta_{(k,m)} d\omega_{S \backslash (k,m)}.
\end{split}
\end{equation}

Lemma~\ref{lemma:upper-bound-tilde-J-hier} analzyes and gives tight upper bounds on $V(\beta)$---this is the 
core technical argument in the proof of Lemma~\ref{lemma:key-to-thm-hier-effect}.
The proof is deferred to Appendix~\ref{sec:proof-upper-bound-tilde-J-hier}.

\renewcommand\thelemma{\thesection.\arabic{lemma}'}
\setcounter{lemma}{2}
\begin{lemma}
\label{lemma:upper-bound-tilde-J-hier}
Assume Assumptions~\ref{assumption:mu-compact} and \ref{assumption:X-Y-bound} and the hierarchical model.
Assume  $\beta_{S_{k, m-1}} = \tau \mathbf{1}_{S_{k, m-1}}$.
There exist constants $\tilde{c}, C > 0$ that depends only on $M, M_X, M_Y, \mu$ such that 
\begin{equation}
\label{eqn:upper-bound-tilde-J}
V(\beta) \le - \frac{\tilde{c}}{\lambda} \cdot \tau^{m-1} \cdot V_{T; C}(\beta)
\end{equation}
holds for any subset $T$ such that $[m] \subseteq T \subseteq [N_k]$ . Above,
$V_{T; C}(\beta)$ is defined by 
\begin{equation}
\label{eqn:def-V-j-C}
	V_{T; C}(\beta) = \Bigg(\effect_{k, m}(X_{k, j(T)}) - C\cdot \sum_{w \in [N_k] \backslash T} \beta_{k, w}\Bigg)_+ \cdot
	\prod_{w \in T \backslash [m]} \beta_{k, w}.
\end{equation}
In above, the index $j(T) \defeq \argmax\{j: [j] \in T\}$.
\end{lemma}
\renewcommand\thelemma{\thesection.\arabic{lemma}}

By Lemma~\ref{lemma:upper-bound-tilde-J}, and noticing that 
$\max_{[m] \subseteq T \subseteq S} V_{T; C}(\beta) \ge C \cdot \prod_{m \le j \le N_k}\min\{\effect_{k, j}(X_{S_{k, j}}), 1\}$
where $C > 0$ depends only on $|S|$, we have shown that 
%we have for some constants $C, \tilde{c} > 0$ depending only on $M, M_X, M_Y, \mu, |S|$
\begin{equation}
\label{eqn:goal-upper-bound-tilde-J-hier}
	\wtilde{\partial_{\beta_{k, m}} \obj}(\beta) \le -\frac{c}{\lambda} \cdot \tau^{m-1} 
		\cdot \prod_{m \le j \le N_k}\min\{\effect_{k, j}(X_{S_{k, j}}), 1\} \le -\frac{c'}{\lambda} \cdot \min\{\tau, 1\}^l 
			\cdot \effect_{k, l}.
\end{equation}
where $c, c' > 0$ depend only on $M, M_X, M_Y, \mu, |S|$.  To transfer the bound from 
$\wtilde{\partial_{\beta_{k, m}} \obj}(\beta)$ to $\partial_{\beta_{k, m}} \obj(\beta)$, we use
Lemma~\ref{lemma:approximate-gradient-bound} to obtain 
\begin{equation}
\label{eqn:grad-hier}
	\partial_{\beta_{k, m}} \obj(\beta) \le \frac{1}{\lambda} \cdot \left(-c \cdot \min\{\tau^l, 1\} \cdot  \effect_{k, l} + C\lambda^{1/2}(1+\lambda^{1/2})\right),
\end{equation}
where $c, C > 0$ depend on $M, M_X, M_Y, \mu, |S|$. Lemma~\ref{lemma:key-to-thm-hier-effect} 
follows as $\partial_{\beta_l} \obj_\gamma(\beta) = \partial_{\beta_l} \obj(\beta) + \gamma$.

\subsection{Proof of Technical Lemma}
\subsubsection{Proof of Lemma~\ref{lemma:eqn-crucial-identity-technical}.}
\label{sec:proof-lemma-eqn-crucial-identity-technical}
Let $\beta$ be such that $\beta_l = 0$ and $\beta_{S^c} = 0$. It suffices to prove 
%According to equation~\eqref{eqn:wtilde-partial-beta-j-rep}, it suffices to prove that under the functional ANOVA model, 
%at any $\beta$ such that $\beta_l = 0$ and $\beta_{S^c} = 0$, 
\begin{equation}
\label{eqn:small-identity}
	\E\Big[\wbar{R}_{\beta, \omega}(\beta \odot X; Y)  e^{i \omega_l X_l}\Big] = 
		\E\Big[e^{i\omega_l X_l + i \langle \omega_{\backslash l}, \beta_{\backslash l} \odot X_{\backslash l}\rangle}
		 	\cdot  F_l(X_S)\Big].
\end{equation}
To see this, we evaluate $\E\left[\wbar{R}_{\beta, \omega}(\beta \odot X; Y) \mid X_l\right]$. By definition, we have 
\begin{equation*}
\begin{split}
R_{\beta, \omega}(\beta \odot X; Y) %&= e^{i \langle \omega, \beta \odot X \rangle} (Y - f_\beta(\beta \odot X)) \\
&= e^{i \langle \omega, \beta \odot X \rangle} \Big(F_l(X_S) + \sum_{A: l \not\in A} f_A^*(X_A) + \xi - f_\beta(\beta \odot X)\Big).
\end{split}
\end{equation*} 
Since the random variables $e^{i \langle \omega, \beta\odot X\rangle}$, 
$\sum_{A: l \not\in A} f_A^*(X_A)$ and $f_\beta(\beta\odot X)$ depend only on the 
random variables $X_{S \backslash l}$, they are independent of $X_l$ by assumption. Hence we obtain 
\begin{equation*}
\begin{split}
	\E\Big[\wbar{R}_{\beta, \omega}(\beta \odot X; Y) \mid X_l\Big] 
		&= \E\Big[R_{\beta, \omega}(\beta \odot X; Y) \mid X_l\Big] -\E\Big[\wbar{R}_{\beta, \omega}(\beta \odot X; Y)\Big] \\
		&= \E\Big[e^{i \langle \omega, \beta \odot X \rangle}\cdot F_l(X_S) \mid X_l\Big]% \\
		= \E\Big[e^{i \langle \omega_{S\backslash l}, \beta_{S\backslash l} \odot 
			X_{S\backslash l} \rangle}\cdot F_l(X_S) \mid X_l\Big] .
\end{split}
\end{equation*}
By the law of iterated conditional expectation, we obtain equation~\eqref{eqn:small-identity} as desired. 

\subsubsection{Proof of Lemma~\ref{lemma:upper-bound-tilde-J}.}
\label{sec:proof-lemma-upper-bound-tilde-J}
Fix $T \subseteq S$ where $l \in T$. Write $T^c = S \backslash T$. 
%Introduce the following notation: for any set $A$, we denote $\wtilde{Q}_{\beta_A}: \R^{|A|} \mapsto \R$ to be
%\begin{equation*}
%	\wtilde{Q}_{\beta_A}(\omega_A) = (2\pi)^{|A|} \int_0^\infty t \prod_{i\in A} 
%		\psi_{\beta_i^{1/q} t}(\omega_i) \mu(dt).%~~\text{where}~~
%		%q_{\beta_A, t}(\omega) = \prod_{i\in A} p_{\beta_i^{1/q} t}(\omega_i).
%\end{equation*}
%We also use the notation shorthand $\wtilde{Q}(\omega_A) = \wtilde{Q}_{\mathbf{1}_A}(\omega_A)$.
%Below we prove Lemma~\ref{lemma:upper-bound-tilde-J}. 
Throughout the proof, we can W.L.O.G assume that $M = 1$. Note $q= 1$. We start by proving the following: 
%\paragraph{Step 1.}
%In the first step, we prove the following upper bound on $U(\beta)$: 
\begin{equation}
\label{eqn:upper-bound-one-tilde-J}
\begin{split}
U(\beta) &\le 
	- \frac{1}{\lambda} \cdot \bigg(\prod_{i \in T \backslash l} \beta_i\bigg)  \cdot \zeta_l(\beta_{T^c})  \\
%\end{equation}
%where %$\zeta_l(\beta_{T^c})$ is defined by 
%\begin{equation*}
\text{where}~~\zeta_l(\beta_{T^c}) &=% \frac{1}{\lambda} \cdot 
	\iiint \left|\E\Big[e^{i \langle \omega_T, X_T \rangle + i \langle \omega_{T^c}, \beta_{T^c} \odot X_{T^c}\rangle}
		 	\cdot  F_l(X_S)\Big]\right|^2 q_0(\omega_l) 
		\wtilde{Q}(\omega_{S \backslash l})  d\omega_l d\omega_{S \backslash l}.
\end{split}
\end{equation}
%\end{equation*}
%\begin{equation*}
%	\wtilde{Q}(\omega_{S \backslash l}) = 
%		(2\pi)^p \int_0^\infty t \prod_{i \in S \backslash \{l\}} \psi_t(\omega_i) \mu(dt)%~~\text{where}~~q_t(\omega) = .
%\end{equation*}
To see this, recall the definition of $\wtilde{Q}(\omega_{S \backslash l})$ and $U(\beta)$. This gives the expression
\begin{equation*}
\begin{split}
	U(\beta)= - \frac{(2\pi)^p}{\lambda} \cdot 
		\iiint \left|\E\Big[e^{i\zeta_l X_l + i \langle \omega_{S\backslash l}, \beta_{S\backslash l} 
			\odot X_{S\backslash l}\rangle}
		 	\cdot  F_l(X_S)\Big]\right|^2 
	q_0(\zeta_l) 
		 t \prod_{i \in S \backslash \{l\}} \psi_t(\omega_i) \mu(dt) d\zeta_l d\omega_{S \backslash l}.
\end{split}
\end{equation*}
By performing a change of variables $\omega_i \mapsto \beta_i \omega_i$ for $i \in T\backslash l$, we obtain:
\begin{equation*}
\begin{split}
U(\beta) = - \frac{(2\pi)^p}{\lambda} \cdot 
		\iiint &\left|\E\Big[e^{i\zeta_l X_l + i \langle \beta_{T\backslash l} \odot \zeta_{T \backslash l}, X_{T\backslash l}\rangle
			+ i \langle \beta_{T^c} \odot \zeta_{T^c}, X_{T^c}\rangle}
		 	\cdot  F_l(X_S)\Big]\right|^2   \\
	&\times 
			q_0(\zeta_l) \cdot t \cdot \prod_{i \in T\backslash l} \psi_{\beta_i t}(\omega_i) 
				\cdot \prod_{i \in S\backslash T} 
			\psi_{t}(\omega_i)\mu(dt) d\zeta_l d\omega_{S \backslash T}d\omega_{T \backslash l}.
\end{split}
\end{equation*}
Here is the crucial observation: for any $\beta \le 1$, 
$\psi_\beta(\omega) = \frac{\beta \omega}{\beta^2 + \omega^2} \ge \beta \psi(\omega)$ for $\omega \in \R$. Hence, 
\begin{equation*}
\begin{split}
U(\beta) = - \frac{(2\pi)^p}{\lambda} \cdot \bigg(\prod_{i \in T \backslash l} \beta_i\bigg) \cdot 
		\iiint &\left|\E\Big[e^{i\zeta_l X_l + i \langle \zeta_{T \backslash l}, X_{T\backslash l}\rangle
			+ i \langle \beta_{T^c} \odot \zeta_{T^c}, X_{T^c}\rangle}
		 	\cdot  F_l(X_S)\Big]\right|^2   \\
	&\times 
			q_0(\zeta_l) \cdot t \cdot \prod_{i \in T\backslash l} \psi_{t}(\omega_i) 
				\cdot \prod_{i \in S\backslash T} 
			\psi_{t}(\omega_i)\mu(dt) d\zeta_l d\omega_{S \backslash T}d\omega_{T \backslash l}.
\end{split}
\end{equation*}
This is exactly the same as the desired bound~\eqref{eqn:upper-bound-one-tilde-J}, after we substitute 
$\wtilde{Q}(\omega_{S \backslash l})$. 
Below we lower bound $\zeta_l(\beta_{T^c})$: for some constant $c, C > 0$ depending only on $M_X, M_Y, \mu$, 
%we have 
\begin{equation}
\label{eqn:upper-bound-two-tilde-J}
\zeta_l(\beta_{T^c}) \ge 
	\Big(c \cdot \effect_l(X_T) - C \cdot \sum_{l' \in S\backslash T} \beta_{l'}\Big)_+.
\end{equation}
To simplify notation, we introduce 
$
R_{l, T}(\beta_{T^c}) = \E[e^{i \langle \omega_T, X_T \rangle + i \langle \omega_{T^c}, \beta_{T^c} \odot X_{T^c}\rangle}
		 	\cdot  F_l(X_S)]
$. Hence,
%According to this notation, we have a simpler expression for the target $\zeta_l(\beta_{T^c})$:
\begin{equation*}
\zeta_l(\beta_{T^c}) = \iint \left|R_{l, T}(\beta_{T^c})\right|^2 q_0(\omega_l) 
		\wtilde{Q}(\omega_{S \backslash l}) d\omega_l d\omega_{S \backslash l}.
\end{equation*}
To analyze $\zeta_l(\beta_{T^c})$, we decompose $R_{l, T}$ into two terms $R_{l, T} = R_{l, T, 1} + R_{l, T, 2}$, where 
\begin{equation*}
\begin{split}
R_{l, T, 1}(\beta_{T^c}) &=  \E\Big[e^{i \langle \omega_T, X_T \rangle + i \langle \omega_{T^c}, \beta_{T^c} \odot X_{T^c}\rangle}
		 	\cdot  F_l(X_T)\Big] \\
R_{l, T, 2}(\beta_{T^c}) &= \E\Big[e^{i \langle \omega_T, X_T \rangle + i \langle \omega_{T^c}, \beta_{T^c} \odot X_{T^c}\rangle}
		 	\cdot  (F_l(X_S) - F_l(X_T))\Big]
\end{split}.
\end{equation*}
As $|z_1 + z_2|^2 \ge \half |z_1|^2 - 2|z_2|^2$ for $z_1, z_2$, we obtain 
$\zeta_l(\beta_{T^c}) \ge \half \zeta_{l, 1}(\beta_{T^c}) - 2 \zeta_{l, 2}(\beta_{T^c})$ where 
\begin{equation*}
\zeta_{l, j}(\beta_{T^c}) = \iint \left|R_{l, T, j}(\beta_{T^c})\right|^2 q_0(\omega_l) 
		\wtilde{Q}(\omega_{S \backslash l}) d\omega_l d\omega_{S \backslash l}.
\end{equation*}
Lemma~\ref{lemma:really-technical} lower bounds $\sigma_{l, 1}(\beta_{T^c})$ and 
upper bounds $\zeta_{l, 2}(\beta_{T^c})$. The proof 
is in Section~\ref{sec:proof-really-technical}. 

\begin{lemma}
\label{lemma:really-technical}
The following bound holds for constants $c, C > 0$ depending only on $M_X, M_Y, \mu$:
\begin{equation*}
\begin{split}
	\zeta_{l, 1}(\beta_{T^c}) \ge c \cdot  \effect_l(X_T) ~~\text{and}~~
	\zeta_{l, 2}(\beta_{T^c}) \le  C \cdot \sum_{l' \in S\backslash T} \beta_{l'}.
\end{split}
\end{equation*}
\end{lemma} \noindent\noindent
As $\zeta_l(\beta_{T^c}) \ge 0$, the desired equation~\eqref{eqn:upper-bound-two-tilde-J} follows 
from Lemma~\ref{lemma:really-technical}. With equations~\eqref{eqn:upper-bound-one-tilde-J} 
and~\eqref{eqn:upper-bound-two-tilde-J} at hand, we get
$U(\beta) \le -\frac{c}{\lambda} \cdot U_{T; C} (\beta)$ where
$c, C > 0$ depend only on $M_X, M_Y, \mu$. This finishes the proof of 
Lemma~\ref{lemma:upper-bound-tilde-J}. 

\subsubsection{Proof of Lemma~\ref{lemma:really-technical}.}
\label{sec:proof-really-technical}
Lemma~\ref{lemma:really-technical} consists of two parts. 

\begin{enumerate}
\item We lower bound $\zeta_{l, 1}(\beta_{T^c})$. By the independence between $X_T$ and $X_{T^c}$, 
%we have the following identity (recall that $X'$ is an independent copy of $X$): 
\begin{equation*}
\begin{split}
	|R_{l, T, 1}(\beta_{T^c})|^2 &=  \left|\E[e^{i \langle \omega_T, X_T\rangle} F_l(X_T)]\right|^2 \cdot 
		\left| \E[e^{i \langle \omega_{T^c}, \beta_{T^c} \odot X_{T^c}\rangle}]\right|^2. % \\
	%	&= \E\left[e^{i \langle \omega_T, X_T - X_T'\rangle} F_l(X_T) F_l(X_T')\right] \cdot 
	%		\E\left[e^{i \langle \omega_{T^c}, \beta_{T^c} \odot (X_{T^c} - X_{T^c}')\rangle}\right].
\end{split}
\end{equation*}
Next, note that $t q_o(\omega_l) \ge \psi_t(\omega_l)$. As a result, we obtain 
\begin{equation*}
	q_0(\omega_l) \wtilde{Q}(\omega_{S \backslash l}) = q_0(\omega_l) \cdot \int_0^\infty t \prod_{i \in S \backslash l} \psi_t(\omega_i) \mu(dt) 
		\ge \int_0^\infty \prod_{i \in S} \psi_t(\omega_i) \mu(dt).
\end{equation*}
Now, using the above identity and inequality, we obtain the following lower bound: 
\begin{equation}
\label{eqn:two-step-zeta-l-1}
\begin{split}
	\zeta_{l, 1}(\beta_{T^c})
		&\ge \iiint \left|\E[e^{i \langle \omega_T, X_T\rangle} F_l(X_T)]\right|^2 \cdot 
		\left| \E[e^{i \langle \omega_{T^c}, \beta_{T^c} \odot X_{T^c}\rangle}]\right|^2
				 \prod_{i \in S} \psi_t(\omega_i) d\omega_S \mu(dt) \\
		&= \int \bigg(\int \left|\E[e^{i \langle \omega_T, X_T\rangle} F_l(X_T)]\right|^2 \prod_{i\in T} \psi_t(\omega_i) d\omega_T\bigg) \\
		&\,\,\,\,\,\,\,\,\,\,\,\,\,\,\,\,\,\,\,\,\,\,\,\,\,\,\,\,\,\,\,\,\,\,\,\,\,\,
			\times \bigg(\int \left| \E[e^{i \langle \omega_{T^c}, \beta_{T^c} \odot X_{T^c}\rangle}]\right|^2
				\prod_{i\in T^c} \psi_t(\omega_i) d\omega_{T^c}\bigg) \mu(dt)
\end{split}
\end{equation}
Below we lower bound the two integrals in the brackets. Let $X'$ denote an independent copy of $X$. 
Since the Fourier transform of the Cauchy density is Laplace, we obtain
\begin{equation*}
\begin{split}
	\int \left| \E[e^{i \langle \omega_{T^c}, \beta_{T^c} \odot X_{T^c}\rangle}]\right|^2\prod_{i\in T^c} \psi_t(\omega_i) d\omega_{T^c}
		&= \E\Big[ \int e^{i \langle \omega_{T^c}, \beta_{T^c} \odot (X_{T^c} - X'_{T^c}\rangle}\prod_{i\in T^c} \psi_t(\omega_i) d\omega_{T^c} \Big] \\
		&= \E \left[e^{- t \cdot \norm{X_{T^c} - X'_{T^c}}_{1, \beta_{T^c}}}\right] \ge e^{-2tM_X}
\end{split}
\end{equation*} 
where the last step is due to Jensen's inequality and the fact that 
$\E[\norm{X_{T^c}-X_{T^c}'}_{1, \beta_{T^c}}] \le 2M_X$ as
$\norm{\beta_{T^c}}_1 \le 1$. 
Substitute it into equation~\eqref{eqn:two-step-zeta-l-1}. Since $\supp(\mu) \subseteq [0, M_\mu]$, we get
\begin{equation*}
\begin{split}
	\zeta_{l, 1}(\beta_{T^c}) 
		&\ge e^{-2 M_\mu M_X} \cdot  
			\int \bigg(\int \left|\E[e^{i \langle \omega_T, X_T\rangle} F_l(X_T)]\right|^2 \prod_{i\in T} \psi_t(\omega_i) d\omega_T\bigg)\mu(dt) \\
		&=e^{-2 M_\mu M_X} \cdot \E \left[\iint e^{i \langle \omega_T, X_T - X_T'\rangle} \prod_{i\in T} \psi_t(\omega_i) d\omega_T \mu(dt) \cdot F_l(X_T) F_l(X_T')\right].
\end{split}
\end{equation*}
Recall that $h(\norm{z_T}_1) = \int e^{i \langle \omega_T, z_T\rangle}Q(\omega_T) d\omega$. Hence, we obtain that
\begin{equation*}
\begin{split}
	\zeta_{l, 1}(\beta_{T^c})  &\ge e^{-2 M_\mu M_X} \cdot
			\E\left[h (\norm{X_T - X_T'}_1) F_l(X_T) F_l(X_T')\right] = e^{-2 M_\mu M_X} \cdot \effect_l(X_T).
\end{split}
\end{equation*}
%This proves the desired lower bound on  $\zeta_{l, 1}(\beta_{T^c})$.

\item We upper bound $\zeta_{l, 2}(\beta_{T^c})$. As $\E[F_l(X_S) - F_l(X_T) \mid X_T] = 0$, we obtain that 
\begin{equation*}
R_{l, T, 2}(\beta_{T^c}) = \E\Big[e^{i \langle \omega_T, X_T \rangle} \cdot 
	(e^{i \langle \omega_{T^c}, \beta_{T^c} \odot X_{T^c}\rangle} - 1)\cdot \left(F_l(X_S) - F_l(X_T)\right)\Big]
\end{equation*}
After applying Cauchy Schwartz inequality to $R_{l, T, 2}(\beta_{T^c})$, we obtain that
\begin{equation*}
\left|R_{l, T, 2}(\beta_{T^c})\right|^2 \le \E\left[|e^{i \langle \omega_{T^c}, \beta_{T^c} \odot X_{T^c}\rangle} - 1|^2\right] 
	\cdot  \E\left[(F_l(X_S) - F_l(X_T))^2\right].
\end{equation*}
Note (i) $|e^{it}-1| \le 2 \cdot \min\{|t|, 1\}$ for any $t \in \R$ and (ii) 
$ \E\left[(F_l(X_S) - F_l(X_T))^2\right] \le \E[Y^2]$ by ANOVA analysis. Consequently, 
this yields the bound 
\begin{equation}
\label{eqn:upper-bound-R-l-T-2}
|R_{l, T, 2}(\beta_{T^c})|^2 \le 4 \cdot \E\left[\min\{\left|\langle \omega_{T^c}, \beta_{T^c} \odot X_{T^c}\rangle\right|, 1\}^2\right] \cdot M_Y.
\end{equation} 
By substituting it into the definition of $\zeta_{l, 2}(\beta_{T^c})$, we obtain that 
\begin{equation}
\label{eqn:zeta-l-2-first-bound}
\begin{split}
\zeta_{l, 2}(\beta_{T^c}) &\le 4M_Y \cdot 
		\E\left[\int \min\{\left|\langle \omega_{T^c}, \beta_{T^c} \odot X_{T^c}\rangle\right|, 1\}^2 \wtilde{Q}(\omega_{T^c}) d\omega_{T^c}\right] \\
	& = 4M_Y \cdot \int_0^\infty 
		\E\left[\int  \min\{\left|\langle \omega_{T^c}, \beta_{T^c} \odot X_{T^c}\rangle\right|, 1\}^2 \prod_{l' \in T^c} \psi_t(\omega_{l'}) d\omega_{T^c}\right] \cdot t\mu(dt) .
\end{split}
\end{equation}
Now we bound the integral in the bracket. %To simplify the analysis we use the following observation. 
Recall that the $\psi_t$ are Cauchy density with parameter $t$ (since $q=1$). Let $W \in \R^p$ be a random 
vector whose coordinates are independent standard Cauchy random variables (with parameter $1$). By
introducing this standard Cauchy vector $W$, we can rewrite the integral into expectation: 
\begin{equation*}
	\int  \min\{\left|\langle \omega_{T^c}, \beta_{T^c} \odot X_{T^c}\rangle\right|, 1\}^2 \prod_{l' \in T^c} \psi_t(\omega_{l'}) d\omega_{T^c}
		= \E[\min\{\left|\langle t \cdot W_{T^c}, \beta_{T^c} \odot X_{T^c}\rangle\right|, 1\}^2 | X].
\end{equation*}
Here comes the crucial observation: any linear combination of 
independent Cauchy variables is Cauchy. In particular, the random variable 
$t \cdot \langle W_{T^c}, \beta_{T^c} \odot X_{T^c}\rangle$ 
(conditional on $X$) is Cauchy distributed with scale parameter $\alpha_t(X) = t \langle\beta_{T^c}, |X_{T^c}| \rangle \ge 0$.
Hence
\begin{equation*}
	\int  \min\{\left|\langle \omega_{T^c}, \beta_{T^c} \odot X_{T^c}\rangle\right|, 1\}^2 \prod_{l' \in T^c} \psi_t(\omega_{l'}) d\omega_{T^c}
		= \E[\min\{\alpha_t(X) \cdot Z, 1\}^2 | X]
\end{equation*}
where $Z$ is a standard Cauchy random variable. A simple calculation shows that 
\begin{equation*}
\begin{split}
	\E[\min\{\alpha |Z|, 1\}^2] &= \frac{2}{\pi}\int_0^\infty \min\{\alpha z, 1\}^2 \cdot \frac{1}{z^2 + 1} dz \le \frac{4}{\pi} \alpha~~\text{for all $\alpha \ge 0$}.
\end{split}
\end{equation*}
As a result, this yields the following upper bound
\begin{equation*}
	\int  \min\{\left|\langle \omega_{T^c}, \beta_{T^c} \odot X_{T^c}\rangle\right|, 1\}^2 \prod_{l' \in T^c} \psi_t(\omega_{l'}) d\omega_{T^c}
		\le \frac{4}{\pi} \alpha_t(X) = \frac{4}{\pi} t \langle\beta_{T^c}, |X_{T^c}| \rangle.
\end{equation*}
Substitute it back into equation~\eqref{eqn:zeta-l-2-first-bound}. This proves that 
(for $\wtilde{C} = \frac{16}{\pi} \cdot M_X M_Y \cdot |h''(0)|$):
\begin{equation*}
\begin{split}
\zeta_{l, 2}(\beta_{T^c}) &\le \frac{16}{\pi} \cdot M_Y \cdot \E[\langle\beta_{T^c}, |X_{T^c}| \rangle] 
	\cdot \int_0^\infty  t^2 \mu(dt) \le \wtilde{C} \cdot \norm{\beta_{T^c}}_1.
\end{split}
\end{equation*}
%where the constant $\wtilde{C} = \frac{16}{\pi} \cdot M_X M_Y \cdot |h''(0)|$. This proves the second claim. 
\end{enumerate}

\subsubsection{Proof of Lemma~\ref{lemma:upper-bound-tilde-J-hier}.}
\label{sec:proof-upper-bound-tilde-J-hier}

The proof of Lemma~\ref{lemma:upper-bound-tilde-J-hier} largely follows that of 
Lemma~\ref{lemma:upper-bound-tilde-J}. Throughout the proof, we can W.L.O.G. 
assume that $M = 1$. Fix  $T$ where $[m] \subseteq T \subseteq [N_k]$.
Introduce notation: $S_{k, T} = \{(k, l) \mid l \in T\}$, and $S_{k, T}^c = S_{k, N_k} \backslash S_{k, T}$.
%We adopt the notation that appear in the proof of Lemma~\ref{lemma:upper-bound-tilde-J} (see the beginning of 
%Section~\ref{sec:proof-lemma-upper-bound-tilde-J}). 

%The proof of Lemma~\ref{lemma:upper-bound-tilde-J-hier} is divided into four steps. Below we highlight the difference 
%between the proof from that of Lemma~\ref{lemma:upper-bound-tilde-J} for reader's convenience. 
%Notice that we can W.L.O.G. assume that $M = 1$ due to a rescaling argument. 

We start by proving $V(\beta) \le e^{-2M_XM_\mu} \cdot \wtilde{V}(\beta)$
where $\wtilde{V}(\beta)$ is defined by 
\begin{equation*}
\begin{split}
\wtilde{V}(\beta) &= 
\frac{1}{\lambda} \cdot 
	\iint \left|\E\Big[e^{i \langle \zeta_{k, m}, X_{k, m} \rangle + i \langle \omega_{\wbar{S}_k \backslash (k, m)}, 
		\beta_{\wbar{S}_k \backslash (k, m)} \odot X_{\wbar{S}_k \backslash (k, m)}\rangle}
		 	\cdot  F_{k, m}(X_{\wbar{S}_k})\Big]\right|^2  \\
		&\,\,\,\,\,\,\,\,\,\,\,\,\,\,\,\,\,\,\,\,\,\,\,\,\,\,\,\,\,\,\,\,\,\,\,\,\,\,\,\,\,\,\,\,\,\,\,\,\,\,\,\,\,\,\,\,
				\cdot q_0(\omega_{(k, m)}) \wtilde{Q}(\omega_{\wbar{S}_k \backslash (k, m)}) d\zeta_{(k,m)} d\omega_{\wbar{S}_k\backslash (k,m)}.
\end{split}
\end{equation*}
The proof is based on straightforward computation. By definition, 
\begin{equation*}
%\label{eqn:one-first-step-bound-V}
\begin{split}
V(\beta) %&= - \frac{1}{\lambda} \cdot 
%		\iint \left|\E\Big[e^{i\zeta_m X_{k, m} + i \langle \omega_{S\backslash (k, m)},\beta_{S\backslash (k, m)} \odot X_{S\backslash (k, m)}\rangle}
%		 	\cdot  F_{k, m}(X_{\wbar{S}_k})\Big]\right|^2 \\
%	&~~~~~~~~~~~~~~~~~~~~~~~~~~~~~~~~~~~~~~~~~~~~~
%		\cdot q_0(\omega_{(k, m)}) \wtilde{Q}(\omega_{S \backslash (k, m)}) d\zeta_{(k,m)} d\omega_{S \backslash (k,m)} \\
	&=- \frac{1}{\lambda} \cdot 
		\iiint \left|\E\Big[e^{i\zeta_m X_{k, m} + i \langle \omega_{S\backslash (k, m)},\beta_{S\backslash (k, m)} \odot X_{S\backslash (k, m)}\rangle}
		 	\cdot  F_{k, m}(X_{\wbar{S}_k})\Big]\right|^2 \\
	&\,\,\,\,\,\,\,\,\,\,\,\,\,\,\,\,\,\,\,\,\,\,\,\,\,\,\,\,\,\,\,\,\,\,\,\,\,\,\,\,\,\,\,\,\,\,\,\,\,\,\,\,\,\,\,\,
		\cdot q_0(\omega_{(k, m)}) \cdot \prod_{i \in S \backslash (k, m)} \psi_t(\omega_i) \cdot d\zeta_{(k,m)} d\omega_{S \backslash (k,m)} t\mu(dt) \\
	&= - \frac{1}{\lambda} \cdot 
		\iiint \left|\E\Big[e^{i\zeta_m X_{k, m} + i \langle \omega_{\wbar{S}_k\backslash (k, m)},
			\beta_{\wbar{S}_k \backslash (k, m)} 
			\odot X_{S\backslash (k, m)}\rangle}\cdot  F_{k, m}(X_{\wbar{S}_k})\Big]\right|^2
				 \cdot \left|\E \Big[e^{i \langle \omega_{S \backslash \wbar{S}_k}, 
				\beta_{S \backslash \wbar{S}_k} \odot X_{S \backslash \wbar{S}_k}\rangle}\Big]\right|^2 \\
	&\,\,\,\,\,\,\,\,\,\,\,\,\,\,\,\,\,\,\,\,\,\,\,\,\,\,\,\,\,\,\,\,\,\,\,\,\,\,\,\,\,\,\,\,\,\,\,\,\,\,\,\,\,\,\,\,
		\cdot q_0(\omega_{(k, m)}) \cdot \prod_{i \in S \backslash (k, m)} \psi_t(\omega_i) \cdot d\zeta_{(k,m)} d\omega_{S \backslash (k,m)} t\mu(dt) \\
	&=-  \frac{1}{\lambda} \cdot 
		\iiint \left|\E\Big[e^{i\zeta_m X_{k, m} + i \langle \omega_{\wbar{S}_k\backslash (k, m)},
			\beta_{\wbar{S}_k \backslash (k, m)} 
			\odot X_{S\backslash (k, m)}\rangle}\cdot  F_{k, m}(X_{\wbar{S}_k})\Big]\right|^2
				 \\ 
	&\,\,\,\,\,\,\,\,\,
		\times \E \left[e^{-t \normsmall{X_{S\backslash \wbar{S}_k} - 
			X_{S\backslash \wbar{S}_k}}_{1, \beta_{S \backslash \wbar{S}_k}}}\right]
				\cdot q_0(\omega_{(k, m)}) \cdot \prod_{i \in \wbar{S}_k \backslash (k, m)} 
					\psi_t(\omega_i) \cdot d\zeta_{(k,m)} d\omega_{\wbar{S}_k\backslash (k,m)} t\mu(dt)  \\
\end{split}
\end{equation*}
where the second line uses the independence between $X_{\wbar{S}_k}$
and $X_{S\backslash \wbar{S}_k}$, and the third line uses the fact that 
$\psi_t(\omega_i)$ is Cauchy whose Fourier transform is Laplace. For $t \in \supp(\mu)$,
\begin{equation*}
	\E \left[e^{-t \normsmall{X_{S\backslash \wbar{S}_k} - 
			X_{S\backslash \wbar{S}_k}}_{1, \beta_{S \backslash \wbar{S}_k}}}\right]
	\ge e^{-M_\mu \cdot \E\normsmall{X_{S\backslash \wbar{S}_k} - 
			X_{S\backslash \wbar{S}_k}}_{1, \beta_{S \backslash \wbar{S}_k}}}
	\ge e^{-2M_\mu M_X}.
\end{equation*}
where we have used Jenson's inequality. %Substitute it into equation~\eqref{eqn:one-first-step-bound-V}. 
This shows that $V(\beta) \le e^{-2M_XM_\mu} \cdot \wtilde{V}(\beta)$ as desired. 

Now, we upper bound on $V(\beta)$. Following the proof of~\eqref{eqn:upper-bound-one-tilde-J}, we derive similarly
\begin{equation}
\label{eqn:upper-bound-one-tilde-J-V}
\wtilde{V}(\beta) \le 
	- \frac{1}{\lambda} \cdot \bigg(\prod_{i \in T \backslash [m]} \beta_{k, i}\bigg)  \cdot \zeta_l(\beta_{S^c_{k, T}})  
\end{equation}
where $\zeta_l(\beta_{S^c_{k, T}})$ is defined by 
\begin{equation*}
\begin{split}
\zeta_l(\beta_{S^c_{k, T}}) &= \frac{1}{\lambda} \cdot 
	\iint \left|\E\Big[e^{i \langle \omega_{S_{k, T}}, X_{S_{k, T}} \rangle + i \langle \omega_{{S^c_{k, T}}}, \beta_{{S^c_{k, T}}} \odot X_{S^c_{k, T}}\rangle}
		 	\cdot  F_{k, m}(X_{\wbar{S}_k})\Big]\right|^2  \\
		&\,\,\,\,\,\,\,\,\,\,\,\,\,\,\,\,\,\,\,\,\,\,\,\,\,\,\,\,\,\,\,\,\,\,\,\,\,\,\,\,\,\,\,\,\,\,\,\,\,\,\,\,\,\,\,\,
				\cdot q_0(\omega_{(k, m)}) \wtilde{Q}(\omega_{\wbar{S}_k \backslash (k, m)}) d\zeta_{(k,m)} d\omega_{\wbar{S}_k\backslash (k,m)}.
\end{split}
\end{equation*}
Note $\prod_{i \in T \backslash [m]} \beta_{k, i} = \tau^{m-1} \cdot \prod_{i \in T \backslash [m]} \beta_{k, i}$ 
since $[m] \subseteq T$ and $\beta_{S_{k, m-1}} = \tau\mathbf{1}_{S_{k, m-1}}$. Hence, 
\begin{equation}
\label{eqn:upper-bound-one-tilde-J-V-'}
\wtilde{V}(\beta) \le 
	- \frac{1}{\lambda} \cdot \tau^{m-1}  \cdot \zeta_l(\beta_{S^c_{k, T}})  \cdot \prod_{i \in T \backslash [m]} \beta_{k, i}.
\end{equation}
Now, following the proof of  equation~\eqref{eqn:upper-bound-two-tilde-J}, we can derive analogously, 
\begin{equation}
\label{eqn:upper-bound-two-tilde-J-V}
\zeta_l(\beta_{S^c_{k, T}})  \ge 
	\Big(c \cdot \effect_{k, m}(X_{S_{k, j(T)}}) - C \cdot \sum_{w \in [N_k] \backslash T} \beta_{k, w}\Big)_+.
\end{equation} 
where $c, C > 0$ depend only on $M_X, M_Y, \mu$. This completes the proof of 
Lemma~\ref{lemma:upper-bound-tilde-J-hier}.

\subsection{Proof of Propositions}

\subsubsection{Proof of Proposition~\ref{prop:main-effect-positive}.}
\label{sec:proof-proposition-main-effect-positive}
Using the Fourier representation of the kernel, we obtain 
\begin{equation}
\label{eqn:effect-l-X-T-Fourier-representation}
\begin{split}
	\effect_l(X_T) &=  \iint \E \left[F_l(X_T) F_l(X_T') e^{i \langle \omega_T, X_T - X_T' \rangle}\right]  \prod_{i\in T} \psi_t(\omega_i) d\omega_T \mu (dt) \\
		&= \iint \left|\E\big[F_l(X_T) e^{i \langle \omega_T, X_T\rangle}\big]\right|^2 \prod_{i\in T} \psi_t(\omega_i) d\omega_T \mu (dt).
\end{split}
\end{equation}
Hence, $\effect_l(X_T) \ge 0$. Moreover $\effect_l(X_T) > 0$ whenever $F_l(X_T) \neq 0$.

Now, suppose that $f_l^*(X_l) \neq 0$. Then $F_l(X_T) \neq 0$ whenever $l \in T$. Hence,
$\effect_l > 0$. 

\subsubsection{Proof of Proposition~\ref{prop:hier-positive}.}
\label{sec:proof-proposition-hier-positive}
Following the proof of Proposition~\ref{prop:main-effect-positive}, we derive
\begin{equation*}
\begin{split}
	\effect_{k, m}(X_{S_{k, j}}) %&=  \frac{1}{(2\pi)^p}\int_{\R^p} \E \left[F_l(X_T) F_l(X_T') e^{i \langle \omega_T, X_T - X_T' \rangle}\right] Q(\omega) d\omega \\
		&=\iint \left|\E\big[F_{k, m}(X_{S_{k, j}}) e^{i \langle \omega_{S_{k, j}}, X_{S_{k, j}}\rangle}\big]\right|^2 
			 \prod_{i \in S_{k, j}} \psi_t(\omega_i) \mu(dt).
\end{split}
\end{equation*}
Hence, $\effect_{k, m}(X_{S_{k, j}}) \ge 0$. Moreover $\effect_{k, m}(X_{S_{k, j}}) > 0$ as long 
as $F_{k, m}(X_{S_{k, j}}) \neq 0$.

Now, suppose for some variable $X_l$, $f_{S_{k, j}}^*(X_{S_{k, j}}) \neq 0$ for all $1\le j\le l$. Then, it implies 
$F_{k, m}(X_{S_{k, j}}) \neq 0$ for any $1\le m \le j \le l$ . This implies $\effect_{k, m}(X_{S_{k, j}}) > 0$ and
hence $\effect_l > 0$.

\section{Proof of Concentration Results: Theorem~\ref{thm:concentration-of-gradients}}
\label{sec:proof-of-concentration-results-details}

\subsection{Introduction of Notation.}
We use $\{X^{(i)}, Y^{(i)}\}_{i=1}^n$ to denote the i.i.d original data. The notation 
$\what{\P}_n, \what{\E}_n$ denote the probability and expectation w.r.t the 
empirical distribution of the original data. As a shorthand, 
$(X^{(1:n)}, Y^{(1:n)})$ denotes the original data $\{X^{(i)}, Y^{(i)}\}_{i=1}^n$.

%of the data, e.g., 
%$\what{\E}_n \left[h(X, Y)\right] = \frac{1}{n} \sum_{i=1}^n h(X^{(i)}, Y^{(i)})$.

Draw independently another group of data $\{\wbar{X}^{(i)}, \wbar{Y}^{(i)}\}_{i=1}^n$ i.i.d from 
distribution $\P$. 
The reason to introduce $\{\wbar{X}^{(i)}, \wbar{Y}^{(i)}\}_{i=1}^n$ is to decouple the statistical dependencies 
(see equation~\eqref{eqn:empirical-gradient}) so as to facilitate the proof of the concentration results of the gradients.  

Let $\wbar{\P}_n, \wbar{\E}_n$ denote the probability and 
expectation w.r.t the empirical distribution of $\{\wbar{X}^{(i)}, \wbar{Y}^{(i)}\}_{i=1}^n$.
As a shorthand, 
$(\wbar{X}^{(1:n)}, \wbar{Y}^{(1:n)})$ denotes the generated data $\{\wbar{X}^{(i)}, \wbar{Y}^{(i)}\}_{i=1}^n$.
%Basically, the dependency issue arises because the gradient $\grad \obj_n(\beta)$ 
%is the average of $\what{r}_\beta(\beta^{1/q} \odot X; Y) \what{r}_\beta(\beta^{1/q} \odot X'; Y') h(\normsmall{X-X'}_{q, \beta}^q)|X_l-X_l'|$
%over the \emph{same} data $\{X^{(i)}, Y^{(i)}\}_{i=1}^n$ used in the construction of $\what{r}_\beta$. To decouple this 
%statistical dependency, the idea is to construct an estimator $\wbar{r}_\beta$ that's independent from $\{X^{(i)}, Y^{(i)}\}_{i=1}^n$.
%For this reason, we introduce an independent group of data $\{\wbar{X}^{(i)}, \wbar{Y}^{(i)}\}_{i=1}^n$ 
%drawn i.i.d from the distribution $\P$. The notation $\wbar{\P}_n, \wbar{\E}_n$ denote the probability and 
%expectation w.r.t the empirical distribution of the data $\{\wbar{X}^{(i)}, \wbar{Y}^{(i)}\}_{i=1}^n$.
Let $\wbar{f}_\beta(x)$, $\wbar{r}_\beta(x, y)$ denote the solution and residual of kernel ridge regression
under $\wbar{\P}_n$:
\begin{equation*}
\begin{split}
	\wbar{f}_\beta(x) = \argmin_{f\in \H}  \half ~\wbar{\E}_n \left[(Y - f(\beta^{1/q} \odot X))^2\right] + \frac{\lambda}{2} \norm{f}_{\H}^2,~~
	\wbar{r}_\beta(x, y) = y - \wbar{f}_\beta(x).
\end{split}
\end{equation*}
Introduce the covariance operator and covariance function $\wbar{\Sigma}_\beta$ and $\wbar{h}_\beta$ 
under $\wbar{\P}_n$:  for $f\in \H$
\begin{equation*}
\begin{split}
\wbar{\Sigma}_\beta f &= \wbar{\E}_n\left[k(\beta^{1/q}\odot X, \cdot)f(\beta^{1/q} \odot X)\right]~~\text{and}~~
\wbar{h}_\beta = \wbar{\E}_n[k(\beta^{1/q}\odot X, \cdot)Y].
\end{split}
\end{equation*}

%For notational simplicity, we use $(X^{(1:n)}, Y^{(1:n)})$ as a notation shorthand for the original data $\{X^{(i)}, Y^{(i)}\}_{i=1}^n$
%and $(\wbar{X}^{(1:n)}, \wbar{Y}^{(1:n)})$ as a notation shorthand for the new data $\{\wbar{X}^{(i)}, \wbar{Y}^{(i)}\}_{i=1}^n$.

\subsection{Roadmap of the Proof. (Heuristics and Main Ideas)}
\label{sec:tackle-statistical-dependency}
Recall the representation of the empirical and population gradients (equations~\eqref{eqn:population-gradient}-\eqref{eqn:empirical-gradient})
\begin{equation*}
\begin{split}
(\grad \obj(\beta))_l = -\frac{1}{\lambda} \cdot \E\left[r_\beta(\beta\odot X; Y)r_\beta(\beta\odot X'; Y') 
	h'(\langle \beta, |X- X'| \rangle)|X_l- X_l'|\right] \\
(\grad \obj_n(\beta))_l = -\frac{1}{\lambda} \cdot \what{\E}_n
	\left[\what{r}_\beta(\beta\odot X; Y) \what{r}_\beta(\beta\odot X'; Y') h'(\langle \beta, |X- X'| \rangle)|X_l- X_l'|\right]
\end{split}
\end{equation*}
As mentioned in the main text, complicated statistical dependencies appear on the RHS of the empirical gradient 
$\grad \obj_n(\beta)$ since the RHS is averaging over, under the empirical distribution of the original data 
$(X^{(1:n)}, Y^{(1:n)})$, quantities that involve $\what{r}_\beta$ which is dependent of $(X^{(1:n)}, Y^{(1:n)})$. 
This statistical dependence makes it hard to establish concentration. 

To alleviate this technical challenge, our idea is to replace $\what{r}_\beta$ by $\wbar{r}_\beta$, which is 
independent of the original data $(X^{(1:n)}, Y^{(1:n)})$. Formally, we construct for each $\beta \ge 0$ and $l \in [p]$, 
\begin{equation*}
\begin{split}
(\wtilde{\grad \obj(\beta)})_l &= -\frac{1}{\lambda} \cdot \E\left[\wbar{r}_\beta(\beta^{1/q}\odot X; Y) \wbar{r}_\beta(\beta^{1/q}\odot X'; Y') 
	h'(\normsmall{X-X'}_{q, \beta}^q)|X_l- X_l'|^q | \wbar{X}^{(1:n)}, \wbar{Y}^{(1:n)}\right] \\
	%\label{eqn:population-gradient-fake}\\
(\wtilde{\grad \obj_n(\beta)})_l &= -\frac{1}{\lambda} \cdot \what{\E}_n
	\left[\wbar{r}_\beta(\beta^{1/q}\odot X; Y) \wbar{r}_\beta(\beta^{1/q}\odot X'; Y') h'(\normsmall{X-X'}_{q, \beta}^q)|X_l- X_l'|^q
		| \wbar{X}^{(1:n)}, \wbar{Y}^{(1:n)}\right]
\end{split}.
\end{equation*}
Below we show how the introduction of the auxiliary quantities make it easy to establish concentration. Indeed, 
recall that our goal is to show that $\grad \obj_n(\beta) \approx \grad \obj(\beta)$ are uniformly close over 
$\beta \in \mathcal{B}_M$ with high probability. Now we can divide our proof into two steps. 

\begin{itemize}
\item In the first step, we show that the auxiliary quantities are uniformly close to the original ones with 
high probability. This means that we show uniformly 
\begin{equation}
\label{eqn:concentration-first-step}
 \wtilde{\grad \obj_n(\beta)} \approx \grad \obj_n(\beta)~~\text{and}~~
 \wtilde{\grad \obj(\beta)} \approx \grad \obj(\beta).
\end{equation}
The key to prove this is to show the uniform closeness: 
$\wbar{r}_\beta \approx r_\beta$ and $\wbar{r}_\beta \approx \what{r}_\beta$. 

\item In the second step, we show that the empirical version and the population version of the auxiliary 
quantities are close to each other. This means that we show uniformly 
 \begin{equation}
 \label{eqn:concentration-second-step}
 	\wtilde{\grad \obj_n(\beta)} \approx \wtilde{\grad\obj(\beta)}.
\end{equation}
This is easy to achieve. We can use standard concentration results from the empirical process theory 
to prove this since $\wbar{r}_\beta$ is \emph{independent} of the empirical measure $\what{\P}_n$. 
%, we can 
%safely use the established techniques from the empirical process theory to easily prove this standard concentration result.
\end{itemize}

We hope that the above explanations help clarify the main idea behind the proof of the concentration result. 
Below we will formalize the two steps in equations~\eqref{eqn:concentration-first-step} and~\eqref{eqn:concentration-second-step}.

\subsection{Proof of Theorem~\ref{thm:concentration-of-gradients}.}
The proof contains two steps. 

In the first step, we establish Proposition~\ref{proposition:auxiliary-close-to-original-ones}
(which formalizes equation~\eqref{eqn:concentration-first-step}). The proof of 
Proposition~\ref{proposition:auxiliary-close-to-original-ones} 
is deferred in Section~\ref{sec:proof-proposition-auxiliary-close-to-original-ones}.
\begin{proposition}
\label{proposition:auxiliary-close-to-original-ones}
Let $M, t > 0$. There exists constants $c, C > 0$ depending only on $M, \sigma_X, \sigma_Y, \mu$ such that the following hold. 
Then we have with probability at least $1-e^{-cn} - n^{-10} - e^{-t}$
\begin{equation*}
\begin{split}
	\sup_{\beta \in \mathcal{B}_M}\normbig{\wtilde{\grad \obj_n(\beta)} - \grad \obj_n(\beta)}_\infty &\le
		 \frac{C}{\lambda^2} \cdot \left(\sqrt[4]{\frac{\log n\log p}{n}} + \sqrt{\frac{t}{n}}\right) \\
	\sup_{\beta \in \mathcal{B}_M}\normbig{\wtilde{\grad \obj(\beta)} - \grad \obj(\beta)}_\infty &\le
		 \frac{C}{\lambda^2} \cdot \left(\sqrt[4]{\frac{\log n\log p}{n}} + \sqrt{\frac{t}{n}}\right),
\end{split}
\end{equation*}
whenever the condition $\lambda \ge C \sqrt[4]{\log n\log p/n}$ holds. 
\end{proposition}

In the second step, we establish Proposition~\ref{proposition:auxiliary-empirical-close-to-population}
(which formalizes equation~\eqref{eqn:concentration-second-step}). The proof of 
Proposition~\eqref{proposition:auxiliary-empirical-close-to-population} 
is deferred in Section~\ref{sec:proof-proposition-auxiliary-empirical-close-to-population}.
\begin{proposition}
\label{proposition:auxiliary-empirical-close-to-population}
Let $M, t > 0$. There exists constants $c, C > 0$ depending only on $M, \sigma_X, \sigma_Y, \mu$ such that the following hold. 
Then we have with probability at least $1-e^{-cn} - e^{-t} - n^{-3}$
\begin{equation*}
\begin{split}
	\sup_{\beta \in \mathcal{B}_M}\normbig{\wtilde{\grad \obj_n(\beta)} - \wtilde{\grad \obj(\beta)}}_\infty &\le
		 \frac{C\log^2(n)}{\min\{\lambda, 1\}^{7/2}}  \cdot \left(\sqrt{\frac{\log p}{n}} + \sqrt{\frac{t}{n}}\right).
\end{split}
\end{equation*}
whenever the condition $\lambda \ge C \sqrt[4]{\log n\log p/n}$ holds. 
\end{proposition}

Theorem~\ref{thm:concentration-of-gradients} now follows from Proposition~\ref{proposition:auxiliary-close-to-original-ones} and 
Proposition~\ref{proposition:auxiliary-empirical-close-to-population}.

\subsection{Proof of Proposition~\ref{proposition:auxiliary-close-to-original-ones}.}
\label{sec:proof-proposition-auxiliary-close-to-original-ones}

The key to the proof is to show that 
$\wbar{r}_\beta \approx r_\beta$ and $\wbar{r}_\beta \approx \what{r}_\beta$. This is given in 
Lemma~\ref{lemma:residual-close-high-probability} below. The proof of 
Lemma~\ref{lemma:residual-close-high-probability} is deferred to Section~\ref{sec:proof-lemma-residual-close-high-probability}. 

\begin{lemma}
\label{lemma:residual-close-high-probability}
There exists constant $c, C > 0$ that depends only on $\sigma_X, \sigma_Y, M, |h(0)|, |h^\prime(0)|$ such that 
the following bound holds with probability at least $1-n^{-10}- e^{-cn}- e^{-t}$:
\begin{equation*}
\begin{split}
\E_{\Q} \left[(r_\beta(\beta^{1/q}\odot X; Y) - \wbar{r}_\beta(\beta^{1/q} \odot X; Y))^2 \mid \wbar{X}^{(1:n)}, \wbar{Y}^{(1:n)}\right] 
	\le \frac{C}{\lambda} \cdot \left(\sqrt[4]{\frac{\log n\log p}{n}} + \sqrt{\frac{t}{n}}\right) \\
\E_{\Q} \left[(\what{r}_\beta(\beta^{1/q}\odot X; Y) - \wbar{r}_\beta(\beta^{1/q} \odot X; Y))^2 \mid \wbar{X}^{(1:n)}, \wbar{Y}^{(1:n)}\right] 
	\le \frac{C}{\lambda} \cdot \left(\sqrt[4]{\frac{\log n\log p}{n}} + \sqrt{\frac{t}{n}}\right)
\end{split}
\end{equation*}
whenever the condition $\lambda \ge C\sqrt[4]{\log n\log p/n}$ holds.
\end{lemma}

Given Lemma~\ref{lemma:residual-close-high-probability}, we are ready to prove 
Proposition~\ref{proposition:auxiliary-close-to-original-ones}. We shall only detail the proof 
for $\sup_{\beta \in \mathcal{B}_M}\normsmall{\wtilde{\grad \obj}_n(\beta) - \grad \obj_n(\beta)}_\infty$.
The proof for $\sup_{\beta \in \mathcal{B}_M}\normsmall{\wtilde{\grad \obj}(\beta) - \grad \obj(\beta)}_\infty$
is similar.

To start with, we pick any $l \in [p]$. Write $\wbar{\Delta}_\beta = r_\beta - \wbar{r}_\beta = \wbar{f}_\beta - f_\beta$. 
Note the decomposition:
$(\wtilde{\grad \obj_n(\beta)}_l - ({\grad \obj_n(\beta)})_l
	= \err_{1, l}(\beta) + \err_{2, l}(\beta) + \err_{3, l}(\beta)$ where
\begin{equation*}
\begin{split}
	\err_{1, l}(\beta) &= - \frac{1}{\lambda} \cdot \what{\E}_n\left[\wbar{\Delta}_\beta(\beta^{1/q} \odot X; Y) \wbar{r}_\beta(\beta^{1/q} \odot X'; Y') 
		h'(\normsmall{X-X'}_{q, \beta}^q)|X_l- X_l'|^q\right] \\
	\err_{2, l}(\beta) &= - \frac{1}{\lambda} \cdot \what{\E}_n\left[\wbar{r}_\beta(\beta^{1/q} \odot X; Y) \wbar{\Delta}_\beta(\beta^{1/q} \odot X'; Y') 
		h'(\normsmall{X-X'}_{q, \beta}^q)|X_l- X_l'|^q\right] \\
	\err_{3, l}(\beta) &= +\frac{1}{\lambda} \cdot \what{\E}_n\left[\wbar{\Delta}_\beta(\beta^{1/q} \odot X; Y) \wbar{\Delta}_\beta(\beta^{1/q} \odot X'; Y') 
		h'(\normsmall{X-X'}_{q, \beta}^q)|X_l- X_l'|^q\right]
\end{split}.
\end{equation*}
By triangle inequality, we have the following bound
\begin{equation}
\label{eqn:basic-bound-after-decomposition}
	\sup_{\beta \in \mathcal{B}_M}\normbig{\wtilde{\grad \obj_n(\beta)} - \grad \obj_n(\beta)}_\infty
		\le \sum_{1\le i \le 3} \sup_{\beta \in \mathcal{B}_M} \max_{l \in [p]}|\err_{1, l}(\beta)|.
\end{equation}
Below we will show with probability at least $1-2e^{-cn}$, the following bound holds 
\begin{equation}
\label{eqn:upper-bound-on-err-i}
\begin{split}
	\sup_{\beta \in \mathcal{B}_M} \max_{l \in [p]}| \max_{j = 1,2 ,3}\err_{j, l}(\beta)|  &\le  \frac{C}{\lambda} \cdot
		\normbigPn{\wbar{\Delta}_\beta(\beta^{1/q}\odot X; Y)} \\
	%\sup_{\beta \in \mathcal{B}_M} \max_{l \in [p]}|\err_{2, l}(\beta)|  &\le  \frac{C}{\lambda} \cdot
	%	\normbigPn{\wbar{\Delta}_\beta(\beta^{1/q}\odot X; Y)} \\
	%\sup_{\beta \in \mathcal{B}_M} \max_{l \in [p]}|\err_{3, l}(\beta)| &\le  \frac{C}{\lambda}  \cdot 
	%	\normbigPn{\wbar{\Delta}_\beta(\beta^{1/q}\odot X; Y)}^2. % \cdot \normPn{Y}
\end{split}
\end{equation}
where $C > 0$ depends only on $M, \sigma_X, \sigma_Y, \mu$. To avoid interruption of the 
flow, we defer the proof of equation~\eqref{eqn:upper-bound-on-err-i} to the end. 
Since Lemma~\ref{lemma:residual-close-high-probability} implies that with high probability
\begin{equation}
\label{eqn:wbar-bound}
	\sup_{\beta \in \mathcal{B}_M}
		\normbigPn{\wbar{\Delta}_\beta(\beta^{1/q}\odot X; Y)} \le \frac{C}{\lambda} \cdot \left(\sqrt[4]{\frac{\log n\log p}{n}} + \sqrt{\frac{t}{n}}\right).
\end{equation}
Proposition~\ref{proposition:auxiliary-close-to-original-ones} now follows from 
equations~\eqref{eqn:basic-bound-after-decomposition}--\eqref{eqn:wbar-bound} and the union bound. 

It remains to prove the deferred high probability bound~\eqref{eqn:upper-bound-on-err-i}.
The proofs of these bounds follow the same theme, and indeed from the facts below. 
\begin{itemize}
\item $|\what{\E}_n[{Z_1 Z_2 Z_3 Z_4}]| \le \normPn{Z_1} \normPn{Z_2} \normPninfty{Z_3} \normPninfty{Z_4}$ by H\"{o}lder's inequality. 
\item Almost surely, $|h'(\normsmall{X-X'}_{q, \beta}^q)| \le \sup_x|h^\prime(x)| \le |h^\prime(0)|$ and 
	$|X_l - X_l'|^q \le (2\sigma_X)^q$. 
\item By Proposition~\ref{prop:kernel-beta-fix}, $\normPn{\what{r}_\beta(\beta^{1/q} \odot X; Y)} 
	= \normPn{\what{r}_\beta(\beta^{1/q} \odot X'; Y')} \le \normPn{Y}$. 
	As $Y$ is $\sigma_Y$ subgaussian, $\normPn{Y} \le 2\normP{Y}$ with probability at least $1-e^{-cn}$.
\item By Lemma~\ref{lemma:deterministic-bound-high-prob-bound-f-r-beta} and Lemma~\ref{lemma:high-prob-bound-h-Sigma-beta}, we have 
	$\normbigPn{\what{r}_\beta(\beta^{1/q}\odot X; Y) - \wbar{r}_\beta(\beta^{1/q} \odot X; Y)} \le C$ with probability at least $1-e^{-cn}$ where 
	$C > 0$ is a constant depends on $\sigma_X, \sigma_Y, \mu$. 
\end{itemize}
%This shows that equation~\eqref{eqn:upper-bound-on-err-i} holds with probability at least $1-2e^{-cn}$.

\subsection{Proof of Lemma~\ref{lemma:residual-close-high-probability}.}
\label{sec:proof-lemma-residual-close-high-probability}
The proof of Lemma~\ref{lemma:residual-close-high-probability} contains two steps. 
\begin{itemize}
\item In the first step, Lemma~\ref{lemma:deterministic-bound-high-prob-bound-f-r-beta} shows that it suffices to prove 
that the difference between the covariance operators and covariance functions are small, i.e., 
$\Sigma_\beta \approx \wbar{\Sigma}_\beta$, $\what{\Sigma}_\beta \approx \wbar{\Sigma}_\beta$ (measured by the 
norm $\opnorm{\cdot}$) and $h_\beta \approx \wbar{h}_\beta$, $\what{h}_\beta \approx \wbar{h}_\beta$
(measured by the norm $\norm{\cdot}_{\H}$).
\item In the second step, Lemma~\ref{lemma:high-prob-bound-h-Sigma-beta} shows that uniformly over $\beta \in\mathcal{B}_M$, 
we have with high probability $\Sigma_\beta \approx \wbar{\Sigma}_\beta$, $\what{\Sigma}_\beta \approx \wbar{\Sigma}_\beta$ 
and $h_\beta \approx \wbar{h}_\beta$, $\what{h}_\beta \approx \wbar{h}_\beta$. The proof uses advanced tools from 
convex geometry and high dimensional  probability theory.  
\end{itemize}

The proof of Lemma~\ref{lemma:deterministic-bound-high-prob-bound-f-r-beta} 
and~\ref{lemma:high-prob-bound-h-Sigma-beta} are given in 
Section~\ref{sec:proof-deterministic-bound-high-prob-bound-f-r-beta} 
and~\ref{sec:proof-lemma-high-prob-bound-h-Sigma-beta} respectively. 

\begin{lemma}
\label{lemma:deterministic-bound-high-prob-bound-f-r-beta}
Assume $\opnormbig{\what{\Sigma}_\beta - \Sigma_\beta} \le \lambda$ at some $\beta \in \mathcal{B}_M$. 
Then we have for $\Q \in \{\P_n, \P\}$: 
\begin{equation}
\label{eqn:main-deterministic-bound}
\begin{split}
	&\E_{\Q} \left[(r_\beta(\beta^{1/q}\odot X; Y) - \wbar{r}_\beta(\beta^{1/q} \odot X; Y))^2 \mid \wbar{X}^{(1:n)}, \wbar{Y}^{(1:n)}\right]  \\
		&\le \frac{1}{\lambda} \cdot \left(\opnormbig{\Sigma_\beta - \wbar{\Sigma}_\beta} \cdot \norm{Y}_{\mathcal{L}_2(\Q)}+ |h(0)|^{1/2} \cdot 
			\normbig{h_\beta - \wbar{h}_\beta}_{\H}\right).
\end{split}
\end{equation}
and 
\begin{equation}
\label{eqn:main-deterministic-bound-'}
\begin{split}
	&\E_{\Q} \left[(\what{r}_\beta(\beta^{1/q}\odot X; Y) - \wbar{r}_\beta(\beta^{1/q} \odot X; Y))^2 \mid \wbar{X}^{(1:n)}, \wbar{Y}^{(1:n)}\right]   \\
		&\le \frac{1}{\lambda} \cdot \left(\opnormbig{\what{\Sigma}_\beta - \wbar{\Sigma}_\beta} \cdot \norm{Y}_{\mathcal{L}_2(\Q)}+ |h(0)|^{1/2} \cdot 
			\normbig{\what{h}_\beta - \wbar{h}_\beta}_{\H}\right).
\end{split}
\end{equation}
\end{lemma}

\begin{lemma}
\label{lemma:high-prob-bound-h-Sigma-beta}
Let $M, t > 0$. 
\begin{enumerate}[(a)]
\item The following bound holds with probability at least $1-e^{-t} -e^{-cn}$, 
\begin{equation*}
\sup_{\beta \in \mathcal{B}_M}\normbig{h_\beta - \what{h}_\beta}_\H \le 
	C \cdot \left(\sqrt[4]{\frac{\log n\log p}{n}} + \sqrt{\frac{t}{n}}\right).
\end{equation*}
Above, the constant $c > 0$ is absolute, and the constant $C > 0$ depends on the parameters 
$M, |h(0)|, |h^\prime(0)|, \sigma_X, \sigma_Y$. The same high probability holds for 
$\sup_{\beta \in \mathcal{B}_M}\normbig{h_\beta - \wbar{h}_\beta}_\H$.
\item The following bound holds with probability at least $1-e^{-t} - e^{-cn}$, 
\begin{equation*}
\sup_{\beta \in \mathcal{B}_M} \opnormbig{\Sigma_\beta - \what{\Sigma}_\beta} \le 
	C \cdot \left(\sqrt[4]{\frac{\log n\log p}{n}} + \sqrt{\frac{t}{n}}\right).
\end{equation*}
Above, the constant $c > 0$ is absolute, and the constant $C > 0$ depends on the parameters 
$M, |h(0)|, |h^\prime(0)|, \sigma_X, \sigma_Y$. The same high probability holds for 
$\sup_{\beta \in \mathcal{B}_M}\opnormbig{\Sigma_\beta - \wbar{\Sigma}_\beta} $.
\end{enumerate}
\end{lemma}

\subsection{Proof of Lemma~\ref{lemma:deterministic-bound-high-prob-bound-f-r-beta}.}
\label{sec:proof-deterministic-bound-high-prob-bound-f-r-beta}
Below we only prove equation~\eqref{eqn:main-deterministic-bound}. The proof of 
equation~\eqref{eqn:main-deterministic-bound-'} is similar.  
Note that $r_\beta - \wbar{r}_\beta = -(f_\beta - \bar{f}_\beta)$. Moreover, for any function 
$f \in \H$,  $\normPn{f(\beta^{1/q}\odot X)} = \normbig{\what{\Sigma}_\beta^{1/2} f}_{\H}$
and $\normP{f(\beta^{1/q}\odot X)} = \normbig{{\Sigma}_\beta^{1/2} f}_{\H}$. Hence, for $\beta \ge 0$,
\begin{equation*}
\begin{split}
\what{\E}_n\left[(r_\beta(\beta^{1/q}\odot X; Y) - \wbar{r}_\beta(\beta^{1/q} \odot X; Y))^2 | \wbar{X}^{(1:n)}, \wbar{Y}^{(1:n)}\right]
	&= \normbig{\what{\Sigma}_\beta^{1/2}(f_\beta - \bar{f}_\beta)}_\H \\
\E\left[(r_\beta(\beta^{1/q}\odot X; Y) - \wbar{r}_\beta(\beta^{1/q} \odot X; Y))^2 | \wbar{X}^{(1:n)}, \wbar{Y}^{(1:n)}\right]
	&= \normbig{{\Sigma}_\beta^{1/2}(f_\beta - \bar{f}_\beta)}_\H 
\end{split}
\end{equation*}
Assume $\opnormbig{\what{\Sigma}_\beta-  \Sigma_\beta} \le \lambda$. Now it suffices to prove the following deterministic bound
%the following deterministic bound holds for all 
%$\beta \in \mathcal{B}_M$:
\begin{equation}
\begin{split}
\label{eqn:main-deterministic-bound-two}
	\normbig{\what{\Sigma}_\beta^{1/2}(f_\beta - \bar{f}_\beta)}_\H &\le 
	\frac{1}{\lambda} \cdot \left(\opnormbig{\Sigma_\beta - \wbar{\Sigma}_\beta} \cdot \normP{Y} + |h(0)|^{1/2} \cdot \normbig{h_\beta - \wbar{h}_\beta}_{\H}\right)  \\
	\normbig{\Sigma_\beta^{1/2}(f_\beta - \bar{f}_\beta)}_\H &\le 
	\frac{1}{\lambda} \cdot \left(\opnormbig{\Sigma_\beta - \wbar{\Sigma}_\beta} \cdot \normP{Y} + |h(0)|^{1/2} \cdot \normbig{h_\beta - \wbar{h}_\beta}_{\H}\right) 
\end{split}.
\end{equation}
%The proof of equation~\eqref{eqn:main-deterministic-bound-two} is based on a series of careful use of triangle inequality.
The proof of the two inequalities in equation~\eqref{eqn:main-deterministic-bound-two} are essentially the same; 
below we only detail the proof for the first one (which is also the technically slightly harder one).% inequality.  

To see this, recall that $f_\beta = (\Sigma_\beta + \lambda I)^{-1} h_\beta$ and 
$\bar{f}_\beta = (\wbar{\Sigma}_\beta + \lambda I)^{-1} \wbar{h}_\beta$ by 
Proposition~\ref{prop:kernel-beta-fix}. Algebraic manipulation yields the decomposition
$\what{\Sigma}_\beta^{1/2}(f_\beta - \bar{f}_\beta) =  \err_1 + \err_2$ where 
%where %$\err_1$ and $\err_2$ are defined by 
\begin{equation*}
\begin{split}
	\err_1 = \what{\Sigma}_\beta^{1/2}  \left(({\Sigma}_\beta + \lambda I)^{-1} - (\wbar{\Sigma}_\beta + \lambda I)^{-1}\right) {h}_\beta, ~~\text{and}~~
	\err_2 = \what{\Sigma}_\beta^{1/2}  (\wbar{\Sigma}_\beta + \lambda I)^{-1} (h_\beta - \wbar{h}_\beta).
\end{split}
\end{equation*}
Hence, $\normsmall{\what{\Sigma}_\beta^{1/2}(f_\beta - \bar{f}_\beta)}_\H \le \norm{\err_1}_\H + \norm{\err_2}_\H$.
It remains to bound $\norm{\err_1}_\H$ and $\norm{\err_2}_\H$.

\noindent\noindent
(a) Bound on $\norm{\err_1}_{\H}$. The following representation of $\err_1$ is particularly useful: 
	\begin{equation*}
		\err_1 = \left(\what{\Sigma}_\beta^{1/2}(\Sigma_\beta + \lambda I)^{-1/2} \right) 
			\left(I - ({\Sigma}_\beta + \lambda I)^{1/2} (\wbar{\Sigma}_\beta + \lambda I)^{-1} ({\Sigma}_\beta + \lambda I)^{1/2}\right) 
				\left(({\Sigma}_\beta + \lambda I)^{-1/2} h_\beta\right).  
	\end{equation*}
%	As an immediate consequence of the above identity, we obtain
%	\begin{equation*}
%		\norm{\err_1}_{\H} \le \opnormbig{\what{\Sigma}_\beta^{1/2}({\Sigma}_\beta + \lambda I)^{-1/2}}
%			\cdot \opnormbig{I - ({\Sigma}_\beta + \lambda I)^{1/2} (\wbar{\Sigma}_\beta + \lambda I)^{-1} ({\Sigma}_\beta + \lambda I)^{1/2}}
%			\cdot \norm{({\Sigma}_\beta + \lambda I)^{-1/2} {h}_\beta}_{\H}
%	\end{equation*}
	Now we bound each of the three terms on the RHS. 
	\begin{itemize}
	\item $\Sigma_\beta$ is a positive operator. Hence, $\opnormbig{\what{\Sigma}_\beta^{1/2}({\Sigma}_\beta + \lambda I)^{-1/2}} \le 1$ 
		when $\opnormbig{\what{\Sigma}_\beta - \Sigma_\beta} \le \lambda$. 
	\item We use the following fundamental fact in functional analysis. For any linear operator $A: \H \to \H$, denoting $A^*$
		to be the adjoint operator of $A$, then $I - A^*A$ has the same spectrum as $I - AA^*$. Applying this fact to 
		$A = (\Sigma_\beta + \lambda I)^{1/2} (\wbar{\Sigma}_\beta + \lambda I)^{-1/2}$, we obtain 
		\begin{equation}
		\begin{split}
			&\opnormbig{I - ({\Sigma}_\beta + \lambda I)^{1/2} (\wbar{\Sigma}_\beta + \lambda I)^{-1} ({\Sigma}_\beta + \lambda I)^{1/2}} \\
			&= \opnormbig{I - (\wbar{\Sigma}_\beta + \lambda I)^{-1/2} ({\Sigma}_\beta + \lambda I) (\wbar{\Sigma}_\beta + \lambda I)^{-1/2}} \\
			&= \opnormbig{(\wbar{\Sigma}_\beta + \lambda I)^{-1/2} (\wbar{\Sigma}_\beta - {\Sigma}_\beta) (\wbar{\Sigma}_\beta + \lambda I)^{-1/2}}
				\le \frac{1}{\lambda} \opnormbig{ \wbar{\Sigma}_\beta - {\Sigma}_\beta}. 
		\end{split}
		\end{equation}
		%The last line uses $\opnormbig{(\wbar{\Sigma}_\beta + \lambda I)^{-1/2}} \le \lambda^{-1/2}$ since $\Sigma_\beta$ is a positive operator. 
	\item Finally, $\norm{({\Sigma}_\beta + \lambda I)^{-1/2} {h}_\beta}_{\H} \le \normP{Y}$. To see this, 
	let $g_\beta = ({\Sigma}_\beta + \lambda I)^{-1/2} {h}_\beta$. By Cauchy-Schwartz, 
	$\langle h_\beta, f\rangle = \E[f(\beta^{1/q}\odot X)Y] \le \langle f, \Sigma_\beta f\rangle_\H^{1/2}\normP{Y}$
	holds for all $f \in \H$. Hence, $\langle g_\beta, f\rangle \le \langle f, f\rangle^{1/2}_\H \normP{Y}$ for all $f \in \H$.
	Substituting $f = g_\beta$  yields the result. 
	\end{itemize}
	Summarizing the above bounds, we have derived that
	$\norm{\err_1}_{\H} \le \frac{1}{\lambda} \cdot \opnormbig{\Sigma_\beta - \what{\Sigma}_\beta} \cdot \normP{Y}$.

\noindent\noindent
(b) Bound on $\norm{\err_2}_{\H}$. Recall that 
	$\err_2 = \what{\Sigma}_\beta^{1/2}  (\wbar{\Sigma}_\beta + \lambda I)^{-1} (h_\beta - \wbar{h}_\beta)$.
	%$\norm{\err_2}_{\H} \le \opnormbig{\what{\Sigma}_\beta^{1/2}} \cdot 
	%		\opnormbig{(\wbar{\Sigma}_\beta + \lambda I)^{-1}} \cdot \normbig{h_\beta - \what{h}_\beta}_{\H}$
	%from the definition of $\err_2$.
	Note the following.  
	\begin{itemize}
	\item By definition, $\normbig{\what{\Sigma}_\beta^{1/2}f}_{\H} = \what{\E}_n \big[f(\beta^{1/q}\odot X)^2\big] 
				\le \norm{f}_\infty \le |h(0)|^{1/2}\norm{f}_{\H}$ for any $f \in \H$. As a result, 
				we obtain $\opnormbig{\what{\Sigma}_\beta^{1/2}} \le |h(0)|^{1/2}$. 
	\item $\opnormbig{(\wbar{\Sigma}_\beta + \lambda I)^{-1}} \le \frac{1}{\lambda}$ since 
		$\wbar{\Sigma}_\beta$ is a positive operator. 
	\end{itemize}
	Summarizing the above bounds, we have derived that
	$\norm{\err_2}_{\H} \le \frac{1}{\lambda} \cdot |h(0)|^{1/2} \cdot \normbig{h_\beta - \what{h}_\beta}_{\H}$.

\subsection{Proof of Lemma~\ref{lemma:high-prob-bound-h-Sigma-beta}.}
\label{sec:proof-lemma-high-prob-bound-h-Sigma-beta}

The proof for Part (a) and (b) are similar. To save space, we only detail the proof of 
Part (a) and sketch the proof of Part (b). %, where we emphasize the essential difference from that of Part (a). 

\subsubsection{Proof of Part (a). }
For clarity of exposition, we divide the proof into several steps. 

\paragraph{Step 1: Symmetrization and Reduction.}
Let $\{\eps^{(i)}\}_{i=1}^n$ be i.i.d Rademacher random variables. Define 
$\what{h}_\beta(\eps) = \what{\E}_n[\eps k(\beta^{1/q} \odot X, \cdot) Y] = \frac{1}{n} \sum_{i=1}^n \eps^{(i)} k(\beta^{1/q} \odot X^{(i)}, \cdot )Y^{(i)}$.
The standard symmetrization argument implies that for any convex and increasing mapping $\Phi: \R_+ \to \R_+$: 
\begin{equation*}
	\E\bigg[\Phi \Big(\sup_{\beta \in \mathcal{B}_M}\normbig{h_\beta - \what{h}_\beta}_\H\Big)\bigg]
		\le 
	\E\bigg[\Phi \Big(2 \cdot \sup_{\beta \in \mathcal{B}_M}\normbig{\what{h}_\beta(\eps)}_\H\Big)\bigg].
\end{equation*}
Armed with this, a classical reduction argument due to Panchenko (Lemma~\ref{lemma:Panchenko}) shows that it 
suffices to prove an exponential tail bound on the random variable 
$\sup_{\beta \in \mathcal{B}_M}\normbig{\what{h}_\beta(\eps)}_\H$.

\paragraph{Step 2: Evaluation and Simplification.}
We evaluate $\normbig{\what{h}_\beta(\eps)}_\H^2$ and leverage the reproducing 
property of RKHS to simplify the expression. Indeed, % Indeed, %it is simple to establish the following identity: 
\begin{equation*}
	\normsmall{\what{h}_\beta(\eps)}_\H^2 = W_\beta~~\text{where}~~
	W_\beta \defeq \frac{1}{n^2} \sum_{i, j} \eps^{(i)} \eps^{(j)} k(\beta^{1/q} \odot X^{(i)}, \beta^{1/q} \odot X^{(j)}) Y^{(i)}Y^{(j)}.  
\end{equation*}
It suffices to prove an exponential tail bound for $W = \sup_{\beta \in \mathcal{B}_M} W_\beta = 
\sup_{\beta \in \mathcal{B}_M} \normsmall{\what{h}_\beta(\eps)}_\H^2$.

\paragraph{Step 3: Centering---from $W_\beta$ to $\wbar{W}_\beta$.}
Let $\wbar{W}_\beta = W_\beta - \E[W_\beta]$. Note then 
\begin{equation}
\label{eqn:centering-step-concentration}
	\sup_{\beta \in \mathcal{B}_M}  \left|W_\beta - \wbar{W}_\beta\right| 
		=  \sup_{\beta \in \mathcal{B}_M}  \left| \E[W_\beta] \right| = \frac{1}{n} \cdot |h(0)| \cdot \E[Y^2]. 
\end{equation}
Write $\wbar{W} = \sup_{\beta \in \mathcal{B}_M} \wbar{W}_\beta$. Below we prove high probability 
bounds on $\wbar{W}$. 
%Below we shift our focus to proving high probability bound to the supremum of the centered process  
%$\wbar{W} = \sup_{\beta \in \mathcal{B}_M} \wbar{W}_\beta$.

\paragraph{Step 4: $\{\wbar{W}_\beta\}_{\beta \in \mathcal{B}_M}$ is a Sub-exponential Process.}
We prove that $\{\wbar{W}_\beta\}_{\beta\in \mathcal{B}_M}$ is a sub-exponential process 
(see Definition~\ref{definition:sub-exponential-process}). More precisely, introduce the 
semi-norm $\norm{\cdot}_X$:
\begin{equation*}
	\norm{\cdot}_X = \max_{1\le i, j\le n} |\langle \cdot,T^{(ij)} \rangle|~~\text{where}~~T^{(ij)} = |X^{(i)}-  X^{(j)}|^q
\end{equation*}
We shall show that $ 
\wbar{W}_\beta - \wbar{W}_{\beta'}$ is $\sigma_{\beta, \beta'}$ sub-exponential 
where $\sigma_{\beta, \beta'}  = (2 |h'(0)| \sigma_Y^2  \norm{\beta - \beta'}_X)/n$.

To prove this, the core technique is the Hanson-Wright's inequality. Introduce notation.  
\begin{itemize}
\item Let $\Delta_{\beta, \beta'} = \wbar{W}_\beta - \wbar{W}_{\beta'}$.
\item Let $Z \in \R^n$ be such that $Z_i = \eps^{(i)} Y^{(i)}$. 
\item Let $A_\beta \in \R^{n \times n}$ be the matrix where its $(i, j)$-th entry is defined by 
	\begin{equation*}
		(A_\beta)_{i, j} = k(\beta^{1/q} \odot X^{(i)}, \beta^{1/q} \odot X^{(j)}) - 
	\E[k(\beta^{1/q} \odot X^{(i)}, \beta^{1/q} \odot X^{(j)})].
	\end{equation*}
	By definition, $(A_\beta)_{i, j} = h(\normsmall{X^{(i)} - X^{(j)}}_{q, \beta}^q) - \E[ h(\normsmall{X^{(i)} - X^{(j)}}_{q, \beta}^q)]$.
\item Let $\Delta^A_{\beta, \beta'}$ be the matrix with $\Delta_{\beta, \beta'}^A = A_\beta - A_{\beta'}$.
\end{itemize}
Note then $\wbar{W}_\beta = \frac{1}{n^2} Z^T A_\beta Z$ by definition. Note the following observations. 
\begin{itemize}
\item First, $Z_i$ is $\sigma_Y$-subgaussian since 
			$\E[e^{tZ_i}] = \half \left(\E[e^{tY_i}] + \E[e^{-tY_i}]\right) \le e^{\half \sigma_Y^2 t^2}$ for $t \in \R$. 
\item Next, since $Z$ has i.i.d $\sigma_Y$-subgaussian coordinates, Hanson-Wright's inequality 
	(see~\cite{RudelsonVe13}) %to the quadratic form $\Delta_{\beta, \beta'} = \frac{1}{n^2} Z^T \Delta^A_{\beta, \beta'} Z$ 
	implies the following inequality: for some absolute constant $c > 0$
	%\begin{equation*}
	%	\P\left(\wbar{W}_\beta \ge t\mid X\right) \le 2\exp\left(-c \cdot \min\left\{-n^2t/(\sigma_Y^2\opnorm{A_\beta}), 
	%		n^4 t^2/(\sigma_Y^4\matrixnorm{A_\beta}_F^2)\right\}\right).
	%\end{equation*}
	%A similar reasoning gives that, for some numerical constant $c > 0$
	\begin{equation*}
		\P\left(\Delta_{\beta, \beta'} \ge t\mid X\right) \le 2\exp\left(-c \cdot \min\left\{-n^2t/(\sigma_Y^2\opnorm{\Delta_{\beta, \beta'}^A}), 
			n^4 t^2/(\sigma_Y^4\matrixnorm{\Delta_{\beta, \beta'}^A}_F^2)\right\}\right).
	\end{equation*}
\item Third, we bound $\Delta^A_{\beta, \beta'}$. Since $h$ is strictly 
	completely monotone, each entry of $\Delta^A_{\beta, \beta'}$ is bounded by 
	$|h^\prime(0)| \cdot \norm{\beta- \beta'}_X$. Hence, %this shows that 
	%$\opnorm{A_\beta}, \matrixnorm{A_\beta}_F \le n |h(0)|$ and 
	$\matrixnorms{\Delta_{\beta, \beta'}^A}_F, \opnorms{\Delta_{\beta, \beta'}^A}
	\le n|h^\prime(0)| \cdot \norm{\beta- \beta'}_X$. 
\end{itemize}
Summarizing the above results, we have shown %$\wbar{W}_\beta$ is $\sigma_\beta$ sub-exponential 
	%where $\sigma_\beta \le (2  |h(0)|  \sigma_Y^2)/n$ and 
	$\Delta_{\beta, \beta'} = 
	\wbar{W}_\beta - \wbar{W}_{\beta'}$ is $\sigma_{\beta, \beta'}$ sub-exponential.

\newcommand{\diam}{{\rm diam}}
\newcommand{\Quantile}{Q}

\paragraph{Step 5: Chaining.}
Since $\wbar{W}_\beta$ is a (centered) sub-exponential process, we can use standard
chaining argument (Theorem~\ref{theorem:chaining-sub-exponential}) to derive a high 
probability upper bound onto the supremum $\wbar{W} = \sup_{\beta \in \mathcal{B}_M} \wbar{W}_\beta$.  
Introduce notation below. 
\begin{enumerate}[-]
\item We use $\diam(\mathcal{B}_M, \norm{\cdot}_X)$ to denote the diameter of the set $\mathcal{B}_M$ under the 	
	norm $\norm{\cdot}_X$. 
\item We use $S(\mathcal{B}_M, \norm{\cdot}_X, \eps)$ to denote the set of $\eps$-covering (using $\norm{\cdot}_X$ ball) 
	of $\mathcal{B}_M$. 
\item We use $N(\mathcal{B}_M, \norm{\cdot}_X, \eps)$ to denote the cardinality
	$N(\mathcal{B}_M, \norm{\cdot}_X, \eps) = |S(\mathcal{B}_M, \norm{\cdot}_X, \eps)|$.
\end{enumerate}
%Now we establish a high probability upper bound onto the supremum of the sub-exponential process $\{\wbar{W}_\beta\}_{\beta\in \mathcal{B}_M}$. 
Let $\beta_0 = 0$. Hence $\wbar{W}_{\beta_0} = 0$. For any $\delta > 0$, 
Theorem~\ref{theorem:chaining-sub-exponential} shows that 
with probability $1- e^{-t}$: 
\begin{equation}
\label{eqn:chaining-result}
\begin{split}
	\wbar{W} &\le 
		\sup_{\stackrel{\beta_1, \beta_2\in \mathcal{B}_M}{\norm{\beta_1-\beta_2}_X \le \delta}} |\wbar{W}_{\beta_1} - \wbar{W}_{\beta_2}| \\
			&~~~~~~~~~~~
				+ C |h^\prime(0)| \cdot \left(\frac{\sigma_Y^2}{n}  \cdot \int_\delta^{\diam(\mathcal{B}_M; \norm{\cdot}_X)}
					\log N(\mathcal{B}_M, \norm{\cdot}_X, \eps) d\eps + \frac{\sigma_Y^2}{n}\cdot\diam(\mathcal{B}_M; \norm{\cdot}_X) t\right)
\end{split}
\end{equation}
where $C > 0$ is a numerical constant. Now we simplify the RHS. First, 
 $\diam(\mathcal{B}_M; \norm{\cdot}_X) \le M(2\sigma_X)^q$ since $\norm{X}_\infty \le \sigma_X$ by assumption.
Next, we give a high probability bound on the first term on the RHS 
of equation~\eqref{eqn:chaining-result}. Note the following deterministic bound
\begin{equation*}
	|\wbar{W}_{\beta_1} - \wbar{W}_{\beta_2}| = \frac{1}{n} \norm{Z}_2^2 \cdot \opnorm{\Delta_{\beta, \beta'}}
		\le \norm{Z}_2^2 \cdot |h^\prime(0)| \cdot \norm{\beta- \beta'}_X
\end{equation*}
where we use the fact that each entry of $\Delta_{\beta, \beta'}$ is bounded by $|h^\prime(0)| \cdot \norm{\beta- \beta'}_X$. 
As $Z$ has i.i.d $\sigma_Y$ sub-gaussian entries, $\norm{Z}_2^2 \le 2\sigma_Y^2$
with probability at least $1-e^{-cn}$ for some constant $c > 0$. Consequently, it means that with probability at least $1-e^{-cn}$, 
\begin{equation*}
	\sup_{\beta_1, \beta_2\in \mathcal{B}_M, \norm{\beta_1-\beta_2}_X \le \delta} |\wbar{W}_{\beta_1} - \wbar{W}_{\beta_2}|\le 2|h^\prime(0)|
		 \delta \sigma_Y^2.
\end{equation*}
Using equation~\eqref{eqn:chaining-result} and union bound, we know with probability at least $1- e^{-t} -e^{-cn}$:
\begin{equation}
\label{eqn:chaining-result-final}
	\wbar{W} \le C |h^\prime(0)|\cdot \left(\delta\sigma_Y^2 + \frac{\sigma_Y^2}{n}  \cdot \int_\delta^{M(2\sigma_X)^q}
				\log N(\mathcal{B}_M, \norm{\cdot}_X, \eps) d\eps + \frac{\sigma_Y^2}{n} \cdot M \sigma_X^q t\right).
\end{equation}
%This is the main result of the chaining step. 

\paragraph{Step 6: Metric Entropy Bound.}
This step bounds the metric entropy $\log N(\mathcal{B}_M, \norm{\cdot}_X, \eps)$. We invoke a 
classical argument due to Maurey (Proposition~\ref{proposition:maurey-argument}).

Indeed, note that $\max_{i, j} \norm{T^{(ij)}}_\infty \le (2\sigma_X)^q$ by assumption. 
As a result, Maurey's argument (Proposition~\ref{proposition:maurey-argument}) 
implies the following upper bound on the metric entropy: 
\begin{equation*}
	\log N(\mathcal{B}_M, \norm{\cdot}_X,\sigma_X^q \cdot \eps) \le CM^2  \frac{\log n\log p}{\eps^2}. 
\end{equation*}
where $C > 0$ is an absolute constant. Consequently, we obtain the bound
\begin{equation*}
\int_\delta^{M(2\sigma_X)^q}
				\log N(\mathcal{B}_M, \norm{\cdot}_X, \eps) d\eps  \le C (M \sigma_X^{q})^2 \cdot \frac{\log n\log p}{\delta}. 
\end{equation*}
Back to equation~\eqref{eqn:chaining-result-final}. We obtain for all $\delta > 0$, %the following bound holds 
with probability at least $1- e^{-t} -e^{-cn}$
\begin{equation*}
	\wbar{W} \le C |h^\prime(0)|\cdot  \left(\delta\sigma_Y^2 + \frac{\sigma_Y^2}{n}  \cdot (M\sigma_X^{q})^2 \cdot \frac{\log n\log p}{\delta} + M\sigma_X^q t\right).
\end{equation*}
Take $\delta = (M \sigma_X^q) \cdot \sqrt{\log n\log p/n}$. This shows with probability at least $1- e^{-t} -e^{-cn}$,
\begin{equation}
\label{eqn:metric-entropy-final-result}
	\wbar{W} \le C |h^\prime(0)| M \sigma_X^q \sigma_Y^2 \cdot \left(\sqrt{\frac{\log n\log p}{n}} + \frac{t}{n}\right).
\end{equation}
%Above $C > 0$ is a numerical constant. 

%One can then show that 
%\begin{equation*}
%	\begin{split}
%		& \int_0^{\diam(\mathcal{B}_M; \norm{\cdot}_X)}
%				\log N(\mathcal{B}_M, \norm{\cdot}_X, \eps) d\eps \\
%		& \le CM \sigma_X^q \cdot \int_0^M \min\left\{\frac{\log n}{\eps^2}\log \left(1+ p\eps^2\right), p\log\left(1+ \frac{\log n}{\eps^2 p}\right)\right\}d\eps
%	\end{split}
%\end{equation*}

\paragraph{Step 7: Finalizing Argument.}
Summarizing, we get with probability at least $1- e^{-t} -e^{-cn}$
\begin{equation*}
	\sup_{\beta \in \mathcal{B}_M} \normbig{\hat{h}_\beta(\eps)}_{\H} = W^{1/2} \le C \cdot \left(\sqrt[4]{\frac{\log n\log p}{n}}+ \sqrt{\frac{t}{n}}\right).
\end{equation*}
Above $C > 0$ is a constant that depends on $M, \sigma_X, \sigma_Y, |h^\prime(0)|$. 
As discussed in Step 1, this can be translated to a high probability bound on 
$\sup_{\beta \in \mathcal{B}_M}\normbig{\what{h}_\beta - h_\beta}_{\H}$.

\subsubsection{Proof of Part (b).}
The proof is essentially the same as that of Part (a). Below we highlight difference. Introduce some notation. 
Let $\H \otimes \H$ be the tensor product of the space $\H$ and $\H$. Any element $h_1 \otimes h_2 \in \H \otimes \H$ can 
be viewed as a linear operator that maps $\H$ to $\H$ as follows: $(h_1 \otimes h_2) h = \langle h_2, h\rangle_{\H} \cdot h_1$.

\paragraph{Step 1: Symmetrization and Reduction.}
Let $\{\eps_i\}_{i=1}^n$ be i.i.d Rademacher random variables. Define
%$\what{\Sigma}_\beta(\eps) \in \H \otimes \H$ by
$\what{\Sigma}_\beta(\eps) = \what{\E}_n[\eps k(\beta^{1/q}\odot X, \cdot) \otimes k(\beta^{1/q}\odot X, \cdot)] = 
	\frac{1}{n} \sum_{i=1}^n \eps_i k(\beta^{1/q} \odot X_i, \cdot ) \otimes k(\beta^{1/q} \odot X_i, \cdot )$.
Symmetrization implies that for any convex and increasing mapping $\Phi: \R_+ \to \R_+$: 
\begin{equation*}
	\E\bigg[\Phi \Big(\sup_{\beta \in \mathcal{B}_M}\opnormbig{\Sigma_\beta - \what{\Sigma}_\beta}\Big)\bigg]
		\le 
	\E\bigg[\Phi \Big(2 \cdot \sup_{\beta \in \mathcal{B}_M}\opnormbig{\what{\Sigma}_\beta(\eps)}\Big)\bigg]
\end{equation*}
Lemma~\ref{lemma:Panchenko} shows that it suffices to prove 
an exponential tail bound onto  $\sup_{\beta \in \mathcal{B}_M}\opnormbig{\what{\Sigma}_\beta(\eps)}$.

\paragraph{Step 2: Evaluation and Simplification.}
Note that $\opnorms{\what{\Sigma}_\beta(\eps)} \le \matrixnorms{\what{\Sigma}_\beta(\eps)}_{\HS}$. Moreover, 
we can leverage the reproducing property of RKHS to obtain the identity
%The quantity $\opnormbig{\what{\Sigma}_\beta(\eps)}$ is hard to evaluate. The idea is to first upper bound it by
%$\matrixnormbig{\what{\Sigma}_\beta(\eps)}_{\HS}$ and then notice that it is simple to evaluate 
%the quantity $\matrixnormbig{\what{\Sigma}_\beta(\eps)}_{\HS}^2$ due to the Hilbert space structure. Indeed, we have 
\begin{equation*}
\begin{split}
	%\opnormbig{\what{\Sigma}_\beta(\eps)} \le \matrixnormbig{\what{\Sigma}_\beta(\eps)}_{\HS} ~~\text{and}~~
	\matrixnormbig{\what{\Sigma}_\beta(\eps)}_{\HS}^2 = U_\beta
	~~\text{where}~~U_\beta = \frac{1}{n^2} \sum_{i, j} \eps_i \eps_j k(\beta^{1/q} \odot X_i, \beta^{1/q} \odot X_j)^2.  
\end{split}
\end{equation*}
%Let $U_\beta = \frac{1}{n^2} \sum_{i, j} \eps_i \eps_j k(\beta^{1/q} \odot X_i, \beta^{1/q} \odot X_j)^2$. 
It suffices to prove an exponential tail bound for the supremum $U = \sup_{\beta \in \mathcal{B}_M} U_\beta$. 

\paragraph{Step 3: Centering---from $U_\beta$ to $\wbar{U}_\beta$.}
Introduce $\wbar{U}_\beta = U_\beta - \E[U_\beta]$. Then 
\begin{equation*}
	\sup_{\beta \in \mathcal{B}_M}  \left|U_\beta - \wbar{U}_\beta\right| 
		=  \sup_{\beta \in \mathcal{B}_M}  \left| \E[U_\beta] \right| = \frac{1}{n} \cdot |h(0)|^2. 
\end{equation*}
Write $\wbar{U} = \sup_{\beta \in \mathcal{B}_M} \wbar{U}_\beta$. Below we prove high probability 
bounds on $\wbar{U}$. 

\paragraph{Step 4: $\{\wbar{U}_\beta\}_{\beta \in \mathcal{B}_M}$ is a Sub-exponential Process.}
One can prove that $\{\wbar{U}_\beta\}_{\beta\in \mathcal{B}_M}$ is a sub-exponential process 
(see Definition~\ref{definition:sub-exponential-process}) w.r.t 
the semi-norm $\norm{\cdot}_X$ on the space of $\mathcal{B}_M$. The proof follows the same 
argument as appears in Step 4 of Part (a). %We omit the details for space considerations. 

\paragraph{Step 5: Chaining.}
One can use chaining (Theorem~\ref{theorem:chaining-sub-exponential}) to derive a high 
probability upper bound onto the quantity $\wbar{U}$: for any $\delta, t > 0$, 
we have with probability at least $1- e^{-t} -e^{-cn}$:
\begin{equation}
\label{eqn:chaining-result-final-U}
	\wbar{W} \le C |h(0)||h^\prime(0)|\cdot \left(\delta\sigma_Y^2 + \frac{\sigma_Y^2}{n}  \cdot \int_\delta^{M(2\sigma_X)^q}
				\log N(\mathcal{B}_M, \norm{\cdot}_X, \eps) d\eps + \frac{\sigma_Y^2}{n} \cdot M \sigma_X^q t\right).
\end{equation}
The proof follows exactly the same argument as appears in Step 5 of Part (a). %we omit the details for space considerations.

\paragraph{Step 6: Metric Entropy Bound.}
Perform exactly the same as appears in Step 6 of Part (a). We can analogously show that with probability at least $1- e^{-t} -e^{-cn}$,
\begin{equation}
\label{eqn:metric-entropy-final-result}
	\wbar{W} \le C |h^\prime(0)| |h(0)| M \sigma_X^q \sigma_Y^2 \cdot \left(\sqrt{\frac{\log n\log p}{n}} + \frac{t}{n}\right).
\end{equation}
%Above $C > 0$ is a numerical constant. 

\paragraph{Step 7: Finalizing Argument.}
Summarizing, we get with probability at least $1- e^{-t} -e^{-cn}$
\begin{equation*}
	\sup_{\beta \in \mathcal{B}_M} \normbig{\hat{h}_\beta(\eps)}_{\H} = U^{1/2} \le C  \cdot \left(\sqrt[4]{\frac{\log n\log p}{n}}+ \sqrt{\frac{t}{n}}\right).
\end{equation*}
where $C > 0$ is a constant that depends only on $M, \sigma_X, \sigma_Y, |h^\prime(0)|, |h(0)|$.
As discussed in Step 1, this can be translated to a high probability bound on 
$\sup_{\beta \in \mathcal{B}_M}\opnormbig{\Sigma_\beta - \what{\Sigma}_\beta}$.

\subsection{Proof of Proposition~\ref{proposition:auxiliary-empirical-close-to-population}.}
\label{sec:proof-proposition-auxiliary-empirical-close-to-population}

The core technique of the proof is empirical process theory. To facilitate readers' 
understanding, we divide the proof into several pieces, where each piece 
demonstrates one independent technical idea in the proof. Introduce the notation
\begin{equation*}
	G_{\beta, l}(x, y, x', y') = \wbar{r}_\beta(\beta^{1/q}\odot x; y) \wbar{r}_\beta(\beta^{1/q}\odot x'; y') 
	h'(\norm{x-x'}_{q, \beta}^q)|x_l- x_l'|^q.
\end{equation*}
This notation help simplify the expression of the gradients:  
\begin{equation}
\label{eqn:auxiliary-gradient-new-notation}
\begin{split}
(\wtilde{\grad \obj(\beta)})_l &= -\frac{1}{\lambda} \cdot \E\left[G_{\beta, l}(X, Y, X', Y') | \wbar{X}^{(1:n)}, \wbar{Y}^{(1:n)}\right] \\
(\wtilde{\grad \obj_n(\beta)})_l &= -\frac{1}{\lambda} \cdot \what{\E}_n
	\left[G_{\beta, l}(X, Y, X', Y') | \wbar{X}^{(1:n)}, \wbar{Y}^{(1:n)}\right]
\end{split}
\end{equation}
Introduce the quantity %the error
$\wtilde{\Delta}_n = 
	\sup_{l \in [p]}\sup_{\beta \in \mathcal{B}_M} |(\wtilde{\grad \obj_n(\beta)})_l - (\wtilde{\grad \obj(\beta)})_l|$.
Our goal is to provide high probability upper bound on the target quantity $\wtilde{\Delta}_n$.

\paragraph{Step 1: Reduction to Bounded $Y$.}
Write $\wtilde{\sigma_Y} = 3\sigma_Y \sqrt{\log n}$ and $\wtilde{Y} = Y \indic{|Y| \le \wtilde{\sigma_Y}}$. 
Consider% the following quantities: 
\begin{equation}
\label{eqn:auxiliary-gradient-truncation}
\begin{split}
(\wtilde{\grad \obj'(\beta)})_l &= -\frac{1}{\lambda} \cdot \E\left[G_{\beta, l}(X, \wtilde{Y}, X', \wtilde{Y'}) | \wbar{X}^{(1:n)}, \wbar{Y}^{(1:n)}\right] \\
(\wtilde{\grad \obj'_n(\beta)})_l &= -\frac{1}{\lambda} \cdot \what{\E}_n
	\left[G_{\beta, l}(X, \wtilde{Y}, X', \wtilde{Y}') | \wbar{X}^{(1:n)}, \wbar{Y}^{(1:n)}\right]
\end{split}
\end{equation}
Note the difference between the RHS of equation~\eqref{eqn:auxiliary-gradient-new-notation} and that of 
equation~\eqref{eqn:auxiliary-gradient-truncation}: we replace $Y$ by its truncated version 
$\wtilde{Y}$ which is almost surely bounded. Using the fact that $Y$ is $\sigma_Y$ sub-gaussian,
Lemma~\ref{lemma:truncation-small-effect} bounds the effect of such truncation.% does not alter the 
%values of the gradients $\wtilde{\grad \obj(\beta)}$ and $\wtilde{\grad \obj_n(\beta)}$ that much. 

\begin{lemma}
\label{lemma:truncation-small-effect}
We have the following results. 
\begin{itemize}
\item With probability at least $1-n^{-3}$, $(\wtilde{\grad \obj'_n(\beta)})_l = (\wtilde{\grad \obj_n(\beta)})_l$ for all 
	$\beta \in \mathcal{B}_M$ and $l \in [p]$.
\item With probability at least $1-n^{-3}$, the following bound $|(\wtilde{\grad \obj'(\beta)})_l - (\wtilde{\grad \obj(\beta)})_l| \le \frac{C}{n}$ 
	holds for all $\beta \in \mathcal{B}_M$ and $l \in [p]$. Here the constant $C$ depends only on $\sigma_Y$ and $|h(0)|$.
\end{itemize}
\end{lemma}
\noindent\noindent
As an immediate consequence, if we define 
$\wtilde{\Delta}_n' = \sup_{l \in [p]}\sup_{\beta \in \mathcal{B}_M} 
 \left|(\wtilde{\grad \obj'_n(\beta)})_l - (\wtilde{\grad \obj'(\beta)})_l \right|$, then
Lemma~\ref{lemma:truncation-small-effect} implies that 
$|\wtilde{\Delta}_n' - \wtilde{\Delta}_n | \le  \frac{C}{n}$ holds with probability at least $1-n^{-3}$. 
Below we shift our focus to high probability bounds on $\wtilde{\Delta}_n'$.

\paragraph{Step 2: A ``Good'' Event $\event_n$. }
Introduce the ``good'' event: %on the data $\{\wbar{X}^{(i)}, \wbar{Y}^{(i)}\}_{i=1}^n$ 
\begin{equation*}
	\event_n = \left\{\normsmall{\wbar{Y}}_{\mathcal{L}_2(\wbar{P}_n)} \le 2\sigma_Y \right\}.
\end{equation*}
%In the later steps, we will obtain tight high probability bounds on our target quantity $\wtilde{\Delta}_n'$ on the event $\event_n$. 
%In this step, we will collect some basic properties of the event $\event_n$. 
Note that $\event_n$ happens with high probability since $Y$ is $\sigma_Y$ subgaussian~\cite{Vershynin18}. 
\begin{lemma}
\label{lemma:event-n-high-prob}
$\event_n$ happens with probability at least $1-e^{-cn}$ where $c$ is an absolute constant.  
\end{lemma}

We show that $G_{\beta, l}(X, \wtilde{Y}, X', \wtilde{Y}')$ is bounded and Lipschitz on the event $\event_n$.  
Write $z_{\lambda, n} = \frac{1}{\sqrt{\lambda}} + \sqrt{\log n}$. The proof of Lemma~\ref{lemma:event-n-good-things-happen} is 
deferred to Section~\ref{sec:proof-lemma-event-n-good-things-happen}.

\begin{lemma}
\label{lemma:event-n-good-things-happen}
On the event $\event_n$, we have the following results. 
\begin{itemize}
\item There exists $C > 0$ depending only on $|h^\prime(0)|$, $\sigma_X$, $\sigma_Y$ so that 
	for all $\beta \in \mathcal{B}_M$ and $l \in [p]$:
\begin{equation*}
	\left|G_{\beta, l}(X, \wtilde{Y}, X', \wtilde{Y}')\right| \le C z_{\lambda, n}^2.
\end{equation*}
\item There exists $C > 0$ depending only on $|h^\prime(0)|$, $\sigma_X$, $\sigma_Y$ so that 
	for all $\beta, \beta' \in \mathcal{B}_M$, $l \in [p]$: 
\begin{equation*}
\begin{split}
	\left|G_{\beta, l}(X, \wtilde{Y}, X', \wtilde{Y}') - G_{\beta', l}(X, \wtilde{Y}, X', \wtilde{Y}') \right| 
	\le \frac{C}{\lambda^2} \cdot z_{\lambda, n} \cdot \max_{T \in \mathcal{T}(X, X')} |\langle \beta - \beta', T\rangle|
	%
	%&\le C \cdot \max \left\{\max_{1\le i, j \le n} |\langle \beta- \beta', |X^{(i)} - X^{(j)}|^q\rangle|, \max_{1\le i\le n} |\langle \beta -\beta', |X-X^{(i)}|^q\rangle|, 
	%\max_{1\le i\le n} |\langle \beta -\beta', |X'-X^{(i)}|^q\rangle| \right\}
\end{split}
\end{equation*}
where the set $\mathcal{T}(X, X') = \left\{|X-X^{(i)}|^q\right\}_{i=1}^n \cup \left\{|X'-X^{(i)}|^q\right\}_{i=1}^n \cup \left\{|X^{(i)}-X^{(j)}|^q\right\}_{i, j=1}^n$
\end{itemize}
\end{lemma}

\paragraph{Step 3: Concentration of $\wtilde{\Delta}_n'$ Conditional on the Event $\Lambda_n$.}
We show that $\wtilde{\Delta}_n'$ is concentrated around its mean conditional on the event $\Lambda_n$. 
Introduce the notation
\begin{equation*}
	\wbar{G}_{\beta, l}(X, \wtilde{Y}, X', \wtilde{Y}') = G_{\beta, l}(X, \wtilde{Y}, X', \wtilde{Y}') 
		- \E[G_{\beta, l}(X, \wtilde{Y}, X', \wtilde{Y}') \mid \wbar{X}^{(1:n)}, \wbar{Y}^{(1:n)}]
\end{equation*}
By definition, it is easy to see that 
\begin{equation}
\label{eqn:Delta-n-represent}
\begin{split}
	&\wtilde{\Delta}_n' = \frac{1}{\lambda} \cdot \max_{l \in [p]} U_{n, l} 
	~~\text{where}~~
	U_{n, l} = \sup_{\beta\in \mathcal{B}_M} 
		\left|\what{\E}_n[\wbar{G}_{\beta, l}(X, \wtilde{Y}, X', \wtilde{Y}')]\right|
\end{split}.
\end{equation}
Lemma~\ref{lemma:concentration-of-tilde-delta} shows that $U_{n, l}$ is concentrated. The proof is 
given in Section~\ref{sec:lemma-concentration-of-tilde-delta}.
\begin{lemma}
\label{lemma:concentration-of-tilde-delta}
There exists a constant $C > 0$ depending only on $|h(0)|, |h^\prime(0)|, \sigma_X, \sigma_Y$ such that conditional 
on the event $\Lambda_n$, the following happens with probability at least $1-e^{-t}$: 
\begin{equation*}
	\max_{l \in [p]} \left|U_{n, l} - \E[U_{n, l} \mid \wbar{X}^{(1:n)}, \wbar{Y}^{(1:n)}]\right| \le 
		C z_{\lambda, n}^2 \cdot \left(\sqrt{\frac{\log p}{n}} + \sqrt{\frac{t}{n}}\right).
\end{equation*}
\end{lemma}
As a result of
Lemma~\ref{lemma:concentration-of-tilde-delta}, we have conditional on $\event_n$, with probability at least  $1-e^{-t}$:
\begin{equation}
\label{eqn:concentration-final-result}
	\left|\wtilde{\Delta}_n' - \frac{1}{\lambda} \max_{l \in [p]} \E[U_{n, l} \mid \wbar{X}^{(1:n)}, \wbar{Y}^{(1:n)}]\right|
		\le C \cdot \frac{1}{\lambda} \cdot z_{\lambda, n}^2 \cdot \left(\sqrt{\frac{\log p}{n}} + \sqrt{\frac{t}{n}}\right).
\end{equation}
Below we seek bounds on $\E[U_{n, l} \mid \wbar{X}^{(1:n)}, \wbar{Y}^{(1:n)}]$ for each $l \in [p]$.

\paragraph{Step 4: Symmetrization.}
We wish to invoke standard symmetrization argument to $\E[U_{n, l} | \wbar{X}^{(1:n)}, \wbar{Y}^{(1:n)}]$. 
However, there is a technical issue we need to take care of---
$U_{n, l} = \sup_{\beta\in \mathcal{B}_M} |\what{\E}_n[\wbar{G}_{\beta, l}(X, \wtilde{Y}, X', \wtilde{Y}')]|$
where $\what{\E}_n$ averages over \emph{dependent} random variables. Consequently, to 
overcome this technical issue, we need to
apply a decoupling argument~\cite{Hoeffding94}. %Here is the notation. 
\begin{itemize}
\item Let $\sigma_{i,i'}$ be independent Rademacher random variables. 
\item Let $\mathcal{I} = \left\{(i, i') \mid i \neq i', 1\le i, i' \le n\right\}$ be the set of distinct indices. 
	Note that we can decompose $\mathcal{I} = \cup_{j=1}^I \mathcal{I}_j$ where 
	$I \le n$, $|\mathcal{I}_j| \ge \floor{\frac{n}{2}}$,  
	so that any two different tuples $(i_1, i_2), (i_3, i_4) \in \mathcal{I}_j$ for some $j \in [I]$ must satisfy 
	$i_k \neq i_l$ for $1\le k < l \le 4$. Let $\what{\E}_{n, j}$ denote the empirical average over the distinct tuples 
	$(i_1, i_2) \in \mathcal{I}_j$ for any $j \in [I]$. %For instance, we have 
	%\begin{equation*}
	%\what{\E}_{n, j} [\wbar{G}_{\beta, l}(X, \wtilde{Y}, X', \wtilde{Y}')] = \frac{1}{|\mathcal{I}_j|}\sum_{(i, i') \in \mathcal{I}_j}
	%	\wbar{G}_{\beta, l}(X^{(i)}, \wtilde{Y}^{(i)}, X^{(i')}, \wtilde{Y}^{(i')}).
	%\end{equation*}
\item Let $\what{\E}_{n, 0}$ denote the empirical average over the indices $\{(i, i)\}_{i=1}^n$. 
	%For instance, we have 
	%\begin{equation*}
	%	\what{\E}_{n, 0} [\wbar{G}_{\beta, l}(X, \wtilde{Y}, X', \wtilde{Y}')] = \frac{1}{n}\sum_{i=1}^n
	%	\wbar{G}_{\beta, l}(X^{(i)}, \wtilde{Y}^{(i)}, X^{(i)}, \wtilde{Y}^{(i)}).
	%\end{equation*}
\end{itemize}
As a result, we obtain, using triangle inequality, that 
\small
\begin{equation}
\label{eqn:symmetrization-one}
\begin{split}
	& \E[U_{n, l} \mid \wbar{X}^{(1:n)}, \wbar{Y}^{(1:n)}]  \\
	 &=  \E \left[\sup_{\beta\in \mathcal{B}_M} 
		\left|\what{\E}_n[\wbar{G}_{\beta, l}(X, \wtilde{Y}, X', \wtilde{Y}')]\right| \mid \wbar{X}^{(1:n)}, \wbar{Y}^{(1:n)}\right]\\
	 &\le \frac{1}{I} \sum_{j=1}^I 
			\E \left[\sup_{\beta\in \mathcal{B}_M} 
				\left|\what{\E}_{n, j} [\wbar{G}_{\beta, l}(X, \wtilde{Y}, X', \wtilde{Y}')]\right|\mid \wbar{X}^{(1:n)}, \wbar{Y}^{(1:n)}\right]
			+ \frac{1}{n} \E \left[\sup_{\beta\in \mathcal{B}_M} 
				\left|\what{\E}_{n, 0} [\wbar{G}_{\beta, l}(X, \wtilde{Y}, X', \wtilde{Y}')]\right|\mid \wbar{X}^{(1:n)}, \wbar{Y}^{(1:n)}\right]\\
%	&\le \max_j  \E \left[\sup_{\beta\in \mathcal{B}_M} 
%				\left|\what{\E}_{n, j} [\wbar{G}_{\beta, l}(X, \wtilde{Y}, X', \wtilde{Y}')]\right| \mid \wbar{X}^{(1:n)}, \wbar{Y}^{(1:n)}\right]
%		+  \frac{1}{n} \E \left[\sup_{\beta\in \mathcal{B}_M} 
%				\left|\what{\E}_{n, 0} [\wbar{G}_{\beta, l}(X, \wtilde{Y}, X', \wtilde{Y}')]\right|\mid \wbar{X}^{(1:n)}, \wbar{Y}^{(1:n)}\right]
\end{split}
\end{equation}
\normalsize
Now for each $j \in [I]$, the random variable $\what{\E}_{n, j} [\wbar{G}_{\beta, l}(X, \wtilde{Y}, X', \wtilde{Y}')]$ is the empirical average of 
independent random variables. The standard symmetrization argument gives for $j \in [I]$
\begin{equation}
\label{eqn:symmetrization-two}
\begin{split}
&\E \left[\sup_{\beta\in \mathcal{B}_M} 
				\left|\what{\E}_{n, j} [\wbar{G}_{\beta, l}(X, \wtilde{Y}, X', \wtilde{Y}')]\right| \mid \wbar{X}^{(1:n)}, \wbar{Y}^{(1:n)}\right] 
\le 2 \cdot \E \left[\sup_{\beta\in \mathcal{B}_M} 
				\left|H_{\beta, l, j}\right|\mid \wbar{X}^{(1:n)}, \wbar{Y}^{(1:n)}\right].
\end{split}
\end{equation}
where $
	H_{\beta,l, j} = \frac{1}{|\mathcal{I}_j|} \cdot \sum_{(i, i') \in \mathcal{I}_j} \sigma_{i, i'} \cdot 
					G_{\beta, l}(X^{(i)}, \wtilde{Y}^{(i)}, X^{(i')}, \wtilde{Y}^{(i')} )$.
By Lemma~\ref{lemma:event-n-good-things-happen}, 
\begin{equation}
\label{eqn:symmetrization-three}
	\frac{1}{n} \E \left[\sup_{\beta\in \mathcal{B}_M} 
				\left|\what{\E}_{n, 0} [\wbar{G}_{\beta, l}(X, \wtilde{Y}, X', \wtilde{Y}')]\right|\mid \wbar{X}^{(1:n)}, \wbar{Y}^{(1:n)}\right] 
		\le \frac{1}{n} \cdot C z_{n, \lambda}^2
\end{equation}
holds on the event $\event_n$.
As a consequence, we obtain that 
on the event $\event_n$: 
\begin{equation}
\label{eqn:symmetrization-four}
 \E[U_{n, l} \mid \wbar{X}^{(1:n)}, \wbar{Y}^{(1:n)}] \le  \frac{1}{n} \cdot C z_{n, \lambda}^2 + 
 	2 \cdot \max_{j \in [p]}\E \left[\sup_{\beta\in \mathcal{B}_M} 
				\left|H_{\beta, l, j}\right|\mid \wbar{X}^{(1:n)}, \wbar{Y}^{(1:n)}\right].
\end{equation}
Now we shift our focus to bound the conditional expectation on the RHS of equation~\eqref{eqn:symmetrization-four}.

\paragraph{Step 5: Chaining.}
Fix $l \in [p]$ and $j \in [I]$. This step uses the standard chaining argument to bound
$U_{n, l, j}$ on the event $\Lambda_n$. 
%\begin{equation*}
%U_{n, l, j} \defeq \E \left[\sup_{\beta\in \mathcal{B}_M}  |H_{\beta, l, j}| \mid \wbar{X}^{(1:n)}, \wbar{Y}^{(1:n)}\right].
%\end{equation*}
By the law of iterated expectations, we have 
\begin{equation}
\label{eqn:lie-U}
U_{n, l, j} =\E \left[\E \Big[\sup_{\beta\in \mathcal{B}_M}  |H_{\beta, l, j}| \mid X^{(1:n)}, Y^{(1:n)}, \wbar{X}^{(1:n)}, \wbar{Y}^{(1:n)}\Big]
	\mid \wbar{X}^{(1:n)}, \wbar{Y}^{(1:n)} \right].
\end{equation}
The key to bound $U_{n, l, j}$ is: the process $\beta \mapsto H_{\beta, l, j}$ is a sub-gaussian process conditional on 
$\{X^{(1:n)}, Y^{(1:n)}, \wbar{X}^{(1:n)}, \wbar{Y}^{(1:n)}\}$ and on the event $\Lambda_n$. Introduce the semi-norm:
\begin{equation*}
	\norm{\cdot}_X = \max\left\{ \max_{1\le i, j \le n} |\langle \cdot, T^{(ij)}\rangle|,  \max_{1\le i, j \le n} |\langle \cdot, \wbar{T}^{(ij)}\rangle| \right\}
\end{equation*}
where $T^{(ij)} = |X^{(i)} - X^{(j)}|^q$ and $\wbar{T}^{(ij)} = |X^{(i)} - \wbar{X}^{(j)}|^q$. We have the following result. 

\begin{lemma}
\label{lemma:sub-gaussian-of-H}
Condition on $\{X^{(i)}, Y^{(i)}, \wbar{X}^{(i)}, \wbar{Y}^{(i)}\}_{i=1}^n$ and on the event $\Lambda_n$. We have  
\begin{itemize}
\item The random variable $H_{\beta, l, j}$ is $\sigma_{\beta, l, j}$ sub-gaussian where $\sigma_{\beta, l, j} 
	= \frac{C}{\sqrt{n}}z_{\lambda, n}^2$.
\item The difference $H_{\beta, l, j} - H_{\beta', l, j}$ is $\sigma_{\beta, \beta', l, j}$ sub-gaussian where 
	$\sigma_{\beta, \beta', l, j} = \frac{C}{\lambda^2 \sqrt{n}} z_{\lambda, n} \norm{\beta- \beta'}_X$. 
\end{itemize}
In above, the constant $C$ depends only on $|h^\prime(0)|$, $\sigma_X$, $\sigma_Y$. 
\end{lemma}

Lemma~\ref{lemma:sub-gaussian-of-H} enables us to use chaining to upper bound $U_{n, l, j}$. 
Introduce notation. 
\begin{enumerate}[-]
\item Let $\diam(\mathcal{B}_M, \norm{\cdot}_X)$ denote the diameter of the set $\mathcal{B}_M$ under the 	
	norm $\norm{\cdot}_X$. 
\item Let $S(\mathcal{B}_M, \norm{\cdot}_X, \eps)$ denote the set of $\eps$-covering (using $\norm{\cdot}_X$ ball) 
	of $\mathcal{B}_M$. 
\item Let $N(\mathcal{B}_M, \norm{\cdot}_X, \eps)$ denote the cardinality
	$N(\mathcal{B}_M, \norm{\cdot}_X, \eps) = |S(\mathcal{B}_M, \norm{\cdot}_X, \eps)|$.
\end{enumerate}
Applying chaining argument~\cite[Chapter 5]{Wainwright19} yields the following: there exists 
$C > 0$ depending only on $|h^\prime(0)|, \sigma_X, \sigma_Y$ such that on $\Lambda_n$, 
the following holds for any $\delta > 0$, 
\begin{equation}
\label{eqn:chaining-result-final-expectation}
\begin{split}
& \E \Big[\sup_{\beta\in \mathcal{B}_M}  |H_{\beta, l, j}| \mid X^{(1:n)}, Y^{(1:n)}, \wbar{X}^{(1:n)}, \wbar{Y}^{(1:n)}\Big] \\
&~~~~\le \sup_{\beta \in \mathcal{B}_M} 
	\E \left[ |H_{\beta, l, j} | 
	\mid X^{(1:n)}, Y^{(1:n)}, \wbar{X}^{(1:n)}, \wbar{Y}^{(1:n)}\right] \\
&~~~~~~~~+ \E \left[\sup_{\beta, \beta' \in \mathcal{B}_M, \norm{\beta-\beta'}_X \le \delta} |H_{\beta, l, j} - H_{\beta', l, j}| 
	\mid X^{(1:n)}, Y^{(1:n)}, \wbar{X}^{(1:n)}, \wbar{Y}^{(1:n)}\right] \\
&~~~~~~~~+ C \cdot \frac{z_{\lambda, n}}{\lambda^2 \sqrt{n}}  \cdot \int_\delta^{\diam(\mathcal{B}_M; \norm{\cdot}_X)}
					\sqrt{\log N(\mathcal{B}_M, \norm{\cdot}_X, \eps)}d\eps
\end{split}
\end{equation}
%This is the main result of the chaining step. 

\paragraph{Step 6: Metric Entropy Bound.}
We upper bound the RHS of equation~\eqref{eqn:chaining-result-final-expectation}. 
Note that the RHS of equation~\eqref{eqn:chaining-result-final-expectation} involves three terms. It turns out that 
boundiing the last term is most challenging since we need to carefully upper bound the metric entropy 
$\log N(\mathcal{B}_M, \norm{\cdot}_X, \eps)$, where we invoke a classical geometric 
argument due to Maurey~\cite{Pisier81}. In the below discussion, to simplify the reasoning, we assume we are 
always on the event $\Lambda_n$. 

\begin{itemize}
\item To bound the first term on the RHS of equation~\eqref{eqn:chaining-result-final-expectation}, 
	we use Lemma~\ref{lemma:sub-gaussian-of-H} which shows that 
	$H_{\beta, l, j}$ is $\frac{C}{\sqrt{n}}z_{\lambda, n}^2$ sub-gaussian for any $\beta \in \mathcal{B}_M$ 
	%where $C > 0$ depends only on $\sigma_X, \sigma_Y, |h^\prime(0)|$. 
	As a result, we obtain that 
	\begin{equation*}
		\sup_{\beta \in \mathcal{B}_M} 
			\E \left[ |H_{\beta, l, j} | \mid X^{(1:n)}, Y^{(1:n)}, \wbar{X}^{(1:n)}, \wbar{Y}^{(1:n)}\right]
				\le \frac{C}{\sqrt{n}} z_{\lambda, n}^2.
	\end{equation*}
	Above, the constant $C > 0$ depends only on $\sigma_X, \sigma_Y, |h^\prime(0)|$.
\item To bound the second term on the RHS of equation~\eqref{eqn:chaining-result-final-expectation},
	we use Lemma~\ref{lemma:event-n-good-things-happen} to obtain that 
	$
		|H_{\beta, l, j} - H_{\beta', l, j}| \le \frac{C}{\lambda^2} \cdot z_{\lambda, n} \cdot \norm{\beta- \beta'}_X
	$
	%where the constant $C > 0$ depends only on $\sigma_X, \sigma_Y, |h^\prime(0)|$. 
	As a result, we obtain
	\begin{equation*}
		\E \left[\sup_{\beta, \beta' \in \mathcal{B}_M, \norm{\beta-\beta'}_X \le \delta} |H_{\beta, l, j} - H_{\beta', l, j}| 
			\mid X^{(1:n)}, Y^{(1:n)}, \wbar{X}^{(1:n)}, \wbar{Y}^{(1:n)}\right] \le \frac{C}{\lambda^2} \cdot z_{\lambda, n} \cdot \delta.
	\end{equation*}
	Above, the constant $C > 0$ depends only on $\sigma_X, \sigma_Y, |h^\prime(0)|$.
\item To bound the last term on the RHS of equation~\eqref{eqn:chaining-result-final-expectation}, 
	note
	$\max_{i, j} \normsmall{T^{(ij)}}_\infty \le (2\sigma_X)^q$, 
	$\max_{i, j} \normsmall{\wbar{T}^{(ij)}}_\infty \le (2\sigma_X)^q$ by assumption.
	Hence, $\diam(\mathcal{B}_M; \norm{\cdot}_X) \le M(2\sigma_X)^q$.
	Maurey's argument (Proposition~\ref{proposition:maurey-argument}) yields for some numerical constant $C > 0$
	\begin{equation*}
		\sqrt{\log N(\mathcal{B}_M, \norm{\cdot}_X,\sigma_X^q \cdot \eps)} \le C\cdot \frac{1}{\eps} \cdot M\sqrt{\log n\log p}.
	\end{equation*}
	As a result, we obtain that 
	\begin{equation*}
	\int_\delta^{\diam(\mathcal{B}_M; \norm{\cdot}_X)}
					\sqrt{\log N(\mathcal{B}_M, \norm{\cdot}_X, \eps)}d\eps
		  \le C (M \sigma_X^{q})^2 \cdot \sqrt{\log p\log n} \cdot \log \frac{M(2\sigma_X)^q}{\delta}. 
	\end{equation*}
\end{itemize}
Substitute above bounds into equation~\eqref{eqn:chaining-result-final-expectation}. We obtain for
some constant $C > 0$ depending only on $|h^\prime(0)|, \sigma_X, \sigma_Y$, conditioning on the event 
$\Lambda_n$, the following bound holds for any $\delta > 0$, 
\begin{equation*}
\begin{split}
&\E \Big[\sup_{\beta\in \mathcal{B}_M}  |H_{\beta, l, j}| \mid X^{(1:n)}, Y^{(1:n)}, \wbar{X}^{(1:n)}, \wbar{Y}^{(1:n)}\Big]  \\
&\le  C \cdot \left(\frac{1}{\sqrt{n}} z_{\lambda, n}^2 + \frac{z_{\lambda, n}}{\lambda^2} \cdot \delta + \frac{z_{\lambda, n}}{\lambda^2 \sqrt{n}}
	\cdot\sqrt{\log p\log n} \cdot \log \frac{1}{\delta}\right).
\end{split}
\end{equation*}
Pick $\delta = 1/\sqrt{n}$. This gives us that, conditional on the event $\Lambda_n$, the following bound holds: 
\begin{equation*}
	\E \Big[\sup_{\beta\in \mathcal{B}_M}  |H_{\beta, l, j}| \mid X^{(1:n)}, Y^{(1:n)}, \wbar{X}^{(1:n)}, \wbar{Y}^{(1:n)}\Big] \le 
		\frac{C\log^2(n)}{\min\{\lambda, 1\}^{5/2} \sqrt{n}}  \cdot \sqrt{\log p}.
\end{equation*}
Equation~\eqref{eqn:lie-U} immediate yields the following bound that holds on the event $\Lambda_n$: 
\begin{equation}
\label{eqn:metric-entropy-final-result-expectation}
	U_{l, n, j} \le\frac{C\log^2(n)}{\min\{\lambda, 1\}^{5/2} \sqrt{n}}  \cdot \sqrt{\log p}.
\end{equation}
Again, the constant $C > 0$ depends only on $|h^\prime(0)|, \sigma_X, \sigma_Y$.

\paragraph{Step 7: Finalizing Argument.}
Substituting equation~\eqref{eqn:metric-entropy-final-result-expectation}
into equation~\eqref{eqn:symmetrization-four}, we obtain that the following holds on the event $\Lambda_n$:
\begin{equation}
\E \left[U_{l, n, j} \mid \wbar{X}^{(1:n)}, \wbar{Y}^{(1:n)} \right] \le \frac{C\log^2(n)}{\min\{\lambda, 1\}^{5/2} \sqrt{n}}  \cdot \sqrt{\log p}.
\end{equation}
Recall equation~\eqref{eqn:concentration-final-result} in Step 3. As its consequence, we obtain the following high probability bound: 
conditional on the event $\Lambda_n$, with probability at least $1-e^{-t}$
\begin{equation}
	\wtilde{\Delta}_n' \le C \cdot \frac{C\log^2(n)}{\min\{\lambda, 1\}^{7/2} \sqrt{n}}  \cdot (\sqrt{\log p} + \sqrt{t}).
\end{equation}
Recall that high probability, $\Lambda_n$ happens and $\Delta_n$ is close to $\wtilde{\Delta}_n'$ up to $C/n$. 
Hence, we can translate the above bound to $\Delta_n$. This completes the proof of 
Proposition~\ref{proposition:auxiliary-empirical-close-to-population}.

%A simple union bound gives the desired high probability bound. % onto the quantity.

\subsection{Proof of Lemma~\ref{lemma:truncation-small-effect}.}
\begin{itemize}
\item 
Since $Y$ is $\sigma_Y$ subgaussian,  $\P(Y \ge \wtilde{\sigma}_Y) \le \exp(- \wtilde{\sigma}_Y^2/2\sigma_Y^2) \le n^{-4}$. 
A union bound gives
\begin{equation*}
\P\left(\exists 1\le i\le n~\text{such that}~Y^{(i)} \neq \wtilde{Y}^{(i)}\right) \le 
	\sum_{i=1}^n \P\left(Y^{(i)} \ge \wtilde{\sigma}_Y\right) \le n^{-3}. 
\end{equation*}
Hence with probability at least $1-n^{-3}$ truncation has no effect: 
$Y^{(i)} = \wtilde{Y}^{(i)}$ for all $i \in [n]$. Note that on this event, 
$(\wtilde{\grad \obj'_n(\beta)})_l= (\wtilde{\grad \obj_n(\beta)})_l$ holds for all $\beta \in \mathcal{B}_M$ and $l \in [p]$.
\item Elementary algebraic manipulation yields the following inequality: 
\begin{equation*}
\begin{split}
	& |G_{\beta, l}(x, \wtilde{y}, x', \wtilde{y'}) - G_{\beta, l}(x, y, x', y')| \\
	&\le |y - \wtilde{y}| \cdot |\wbar{r}_\beta(\beta^{1/q}\odot x'; y') | \cdot
			h'(\normsmall{x-x'}_{q, \beta}^q) \cdot |x_l- x_l'|^q \\
	& ~~~+  |\wbar{r}_\beta(\beta^{1/q}\odot x; y) | \cdot |y' - \wtilde{y'}|
			\cdot h'(\normsmall{x-x'}_{q, \beta}^q) \cdot |x_l- x_l'|^q.
	%\left|(\wtilde{\grad \obj(\beta)})_l^{(1)} - (\wtilde{\grad \obj(\beta)})_l \right| \le
	%	2\E[|Y - \wtilde{Y}| \wbar{r}_\beta(\beta^{1/q}\odot X'; Y') 
	%		h'(\normsmall{X-X'}_{q, \beta}^q)|X_l- X_l'|^q]
\end{split}
\end{equation*}
Note (i) $\sup_x |h^\prime(x)| \le |h^\prime(0)|$ and (ii) $\norm{X}_\infty \le \sigma_X$. Hence, 
for all $l \in [p]$ and $\beta \in \mathcal{B}_M$
\begin{equation*}
	|(\wtilde{\grad \obj'_n(\beta)})_l - (\wtilde{\grad \obj(\beta)})_l|
		\le C\cdot \E[|Y- \tilde{Y}|] \cdot \E[|\wbar{r}_\beta(\beta^{1/q}\odot X; Y) | \mid \wbar{X}^{(1:n)}, \wbar{Y}^{(1:n)}].
\end{equation*}
where $C =  2^{q+1} |h^\prime(0)| \sigma_X^q$. Note $\E[|Y - \tilde{Y}|] 
= \E[|Y| \indic{|Y| \ge \wbar{\sigma}_Y}] \le \frac{1}{n}$.
Recall the assumption on $\lambda$. By Lemma~\ref{lemma:deterministic-bound-high-prob-bound-f-r-beta} and 
Lemma~\ref{lemma:high-prob-bound-h-Sigma-beta}, one gets with probability
$1-n^{-3}$: 
\begin{equation*}
\begin{split}
	\E[|\wbar{r}_\beta(\beta^{1/q}\odot X; Y) | \mid \wbar{X}^{(1:n)}, \wbar{Y}^{(1:n)}]
	% \le(\E[|\wbar{r}_\beta(\beta^{1/q}\odot X; Y) |^2 \mid \wbar{X}^{(1:n)}, \wbar{Y}^{(1:n)}])^{1/2}
	 \le 2\sigma_Y + |h(0)|^{1/2}.
\end{split}
\end{equation*}
\end{itemize}

\subsection{Proof of Lemma~\ref{lemma:event-n-good-things-happen}.}
\label{sec:proof-lemma-event-n-good-things-happen}
By definition, 
\begin{equation*}
G_{\beta, l}(X, \wtilde{Y}, X', \wtilde{Y}') = \wbar{r}_\beta(\beta^{1/q}\odot X; \wtilde{Y})\wbar{r}_\beta(\beta^{1/q}\odot X'; \wtilde{Y}')
	h^\prime(\normsmall{X-X'}_{q, \beta}^q) |X_l - X_l'|^q.
\end{equation*}
The key to the proof of Lemma~\ref{lemma:event-n-good-things-happen} is to show that $\beta \mapsto \wbar{r}_\beta(\beta^{1/q}\odot X; \wtilde{Y})$
is uniformly bounded and Lipschitz. This is formally stated in Lemma~\ref{lemma:event-n-good-things-happen-details}, 
whose proof is given in Section~\ref{sec:proof-lemma-event-n-good-things-happen-details}.

\begin{lemma}
\label{lemma:event-n-good-things-happen-details}
On the event $\Lambda_n$, the following things happen. 
\begin{itemize}
\item The family of functions $\{\wbar{r}_\beta\}_{\beta \in \mathcal{B}_M}$ is uniformly bounded: with probability one,
	\begin{equation}
	\label{eqn:uniform-bound-on-r}
		\sup_{\beta \in \mathcal{B}_M} |\wbar{r}_\beta(\beta^{1/q} \odot X; \wtilde{Y})| \le Cz_{\lambda, n}. 
	\end{equation}
	Here $C > 0$ is a constant that depends only on $|h(0)|$, $\sigma_Y$. 
\item The family of functions $\{\wbar{r}_\beta\}_{\beta \in \mathcal{B}_M}$ is Lipschitz: 
	there exists a constant $C > 0$ that depends only on $|h^\prime(0)|$ such that the following holds
	for any $\beta, \beta' \in \mathcal{B}_M$
	\begin{equation}
	\label{eqn:Lipschitzness-on-r}
		\left|\wbar{r}_\beta(\beta^{1/q}\odot X, Y) - \wbar{r}_{\beta'}({\beta'}^{1/q}\odot X, Y)\right| \le 
			\frac{C}{\lambda^2} \cdot \max_{T \in \wbar{\mathcal{T}}(X)} |\langle \beta -\beta', T\rangle|. 
	\end{equation}
	Here $\wbar{\mathcal{T}}(X) = \{|X-X^{(i)}|^q\}_{i=1}^n \cup \{|X^{(i)} - X^{(j)}|^q\}_{i, j=1}^n$
\end{itemize}
\end{lemma}

Note that $\sup_x |h^\prime(x)| \le |h^\prime(0)|$ (since $h$ is completely monotone) and $|X-X'|^q \le (2\sigma_X)^q$ by assumption. 
Now Lemma~\ref{lemma:event-n-good-things-happen} follows easily from Lemma~\ref{lemma:event-n-good-things-happen-details} 
and the triangle inequality.

\subsection{Proof of Lemma~\ref{lemma:concentration-of-tilde-delta}.}
\label{sec:lemma-concentration-of-tilde-delta}
Fix $l \in [p]$. The key is to show that, conditional on the event $\Lambda_n$, $U_{n, l}$ 
with probability at least $1-e^{-t}$ satisfies the bound 
\begin{equation*}
	\left|U_{n, l} - \E[U_{n, l} \mid \wbar{X}^{(1:n)}, \wbar{Y}^{(1:n)}]\right| \le C z_{\lambda, n}^2 \cdot \sqrt{\frac{t}{n}}. 
\end{equation*}
Above, the constant $C$ depends only on $|h(0)|, |h^\prime(0)|, \sigma_X, \sigma_Y$, and does not depend on $l \in [p]$. 
With this bound at hand, Lemma~\ref{lemma:concentration-of-tilde-delta} follows immediately from the union bound. 

To prove the aforementioned concentration, recall
Lemma~\ref{lemma:event-n-good-things-happen}, which shows that, on the event $\Lambda_n$, 
$G_{\beta, l}(X, \wtilde{Y}, X', \wtilde{Y}') \le Cz_{\lambda, n}^2$
where $C$ does not depend on $\beta \in \mathcal{B}_M$ and $l \in [p]$. This shows that, the 
random variable $U_{n, l}$, as a function of the i.i.d pair $Z_i = (X_i, \wtilde{Y}_i)$, is of bounded 
difference conditional on the event $\Lambda_n$.
The desired concentration now follows from the McDiarmid bounded difference concentration.

\subsection{Proof of Lemma~\ref{lemma:sub-gaussian-of-H}.}
Recall our definition of the random variable $H_{\beta, l, j}$: 
\begin{equation*}
H_{\beta,l, j} = \frac{1}{|\mathcal{I}_j|} \cdot \sum_{(i, i') \in \mathcal{I}_j} \sigma_{i, i'} \cdot 
					G_{\beta, l}(X^{(i)}, \wtilde{Y}^{(i)}, X^{(i')}, \wtilde{Y}^{(i')} ),
\end{equation*}
where $\sigma_{i, i'}$ is independent Radamacher random variable which is $1$-subgaussian. The result follows immediately 
from the fact that (i) $|\mathcal{I}_j| \ge \floor{\half n}$ and (ii) by Lemma~\ref{lemma:event-n-good-things-happen}, 
the random variable $G_{\beta, l}$ is bounded by $C z_{\lambda, n}^2$ on the event $\Lambda_n$ and (iii) 
by Lemma~\ref{lemma:event-n-good-things-happen}, on the event $\Lambda_n$, 
the mapping $\beta \mapsto G_{\beta, l}(X^{(i)}, \wtilde{Y}^{(i)}, X^{(i')}, \wtilde{Y}^{(i')} )$ is Lipschitz for any $1\le i, i'\le n$.
%:
%\begin{equation*}
%|G_{\beta, l}(X^{(i)}, \wtilde{Y}^{(i)}, X^{(i')}, \wtilde{Y}^{(i')} )-G_{\beta', l}(X^{(i)}, \wtilde{Y}^{(i)}, X^{(i')}, \wtilde{Y}^{(i')} )|
%	\le \frac{C}{\lambda^2} \cdot z_{\lambda, n} \cdot \norm{\beta-\beta'}_X.
%\end{equation*}
%Above the constant $C$ depends only on $|h(0)|,|h^\prime(0)|, \sigma_X, \sigma_Y$ and is independent of $\beta$.

\subsection{Proof of Lemma~\ref{lemma:event-n-good-things-happen-details}.}
\label{sec:proof-lemma-event-n-good-things-happen-details}
%We divide the proof into two parts. 
Recall $\wbar{r}_\beta(\beta^{1/q}\odot X; \wtilde{Y}) = \wtilde{Y} - \wbar{f}_\beta(\beta^{1/q} \odot X)$.
%Recall the following definition
%\begin{equation*}
%	\wbar{r}_\beta(\beta^{1/q}\odot X; \wtilde{Y}) = \wtilde{Y} - \wbar{f}_\beta(\beta^{1/q} \odot X). 
%\end{equation*}

\begin{itemize}
\item Proof of Part (i) of Lemma~\ref{lemma:event-n-good-things-happen-details}.
	Note $\wtilde{Y} \le 3\sigma_Y\sqrt{\log n}$ by construction. It suffices to prove that 
	$\sup_x |\wbar{f}_\beta(x)| \le C/{\sqrt{\lambda}}$
	holds on $\Lambda_n$ where $C$ depends only on $|h(0)|$, $\sigma_Y$. 
	This is implied by (i) $\sup_x |\wbar{f}_\beta(x)| \le |h(0)|^{1/2}\norm{\wbar{f}_\beta}_{\H}$
	and (ii) $\lambda \norm{\wbar{f}_\beta}_\H^2 \le \norm{Y}_{\mathcal{L}_2(\wbar{\P}_n)}^2$ since 
	$\wbar{f}_\beta$ is the minimum of the kernel ridge regression w.r.t $\wbar{\P}_n$. Now that 
	$\norm{Y}_{\mathcal{L}_2(\wbar{\P}_n)} \le 2\sigma_Y$ on the event $\Lambda_n$. Hence, 
	$\sup_x |\wbar{f}_\beta(x)| \le \frac{C}{\sqrt{\lambda}}$ on $\Lambda_n$ where $C = 2 |h(0)|^{1/2} \sigma_Y$. 
\item Proof of Part (ii) of Lemma~\ref{lemma:event-n-good-things-happen-details}. It suffices 
	to prove for some $C$ depending only on $|h^\prime(0)|$
	\begin{equation}
	\label{eqn:f-beta-Lipschitz}
		\left|\wbar{f}_\beta(\beta^{1/q} \odot X) - \wbar{f}_\beta'(\beta'^{1/q} \odot X)\right| \le 
			\frac{C}{\lambda^2} \cdot \max_{T \in \wbar{\mathcal{T}}(X)} |\langle \beta -\beta', T\rangle|. 
	\end{equation}
	Let $K_\beta \in \R^{n \times n}$ be the matrix where 
	$(K_\beta)_{ij} = k(\beta^{1/q} \odot \wbar{X}^{(i)}, \beta^{1/q} \odot \wbar{X}^{(j)})$, 
	$\wtilde{k}_\beta(X) \in \R^n$ be the vector where 
	$(\wtilde{k}_\beta)_i = k(\beta^{1/q} \odot \wbar{X}^{(i)}, \beta^{1/q} \odot X)$, and 
	$\wbar{y} \in \R^n$ be the vector $(\wbar{y})_i = \wbar{Y}^{(i)}$. Note then $\wbar{f}_\beta(\beta^{1/q} \odot X) = 
	\frac{1}{n}  (\wtilde{k}_\beta(X))^T \big(\frac{1}{n} K_\beta + \lambda I\big)^{-1} \wbar{y}$, 
	and the desired equation~\eqref{eqn:f-beta-Lipschitz} follows from the following facts
	(i) $\beta \mapsto \frac{1}{\sqrt{n}}\wtilde{k}_\beta(X)$ is Lipschitz: 
	$\frac{1}{\sqrt{n}}  \ltwo{\wtilde{k}_\beta(X) - \wtilde{k}_\beta'(X)} \le |h^\prime(0)| 
			\cdot \max_{T \in \wbar{\mathcal{T}}(X)} |\langle \beta -\beta', T\rangle|$
	(ii) $\beta \mapsto \frac{1}{n} K_\beta$ is Lipschitz: 
	$\frac{1}{n} \opnorm{K_\beta - K_{\beta'}} \le  
		|h^\prime(0)| \cdot \max_{T \in \wbar{\mathcal{T}}(X)} |\langle \beta -\beta', T\rangle|$
	and (iii) $\frac{1}{\sqrt{n}} \ltwo{\wbar{y}} \le 2\sigma_Y$ is bounded on the event $\Lambda_n$.
\end{itemize}

\section{Finite Sample Guarantees}
\label{sec:proof-main-corollaries}

\subsection{Proof of Corollary~\ref{corollary:no-false-positive-finite-sample}.}
\label{sec:proof-corollary-no-false-positive-finite-sample}
The proof is pretty much the same as that of Theorem~\ref{thm:no-false-positive}. The key to the proof is  
Lemma~\ref{lemma:grad-noise-variable-beta-noise-0-finite-sample}, which follows from
Lemma~\ref{lemma:grad-noise-variable-beta-noise-0} and Theorem~\ref{thm:concentration-of-gradients}.
\begin{lemma}
\label{lemma:grad-noise-variable-beta-noise-0-finite-sample}
Assume the assumptions in Corollary~\ref{corollary:no-false-positive-finite-sample}. %Fix $l \in S$. 
Then the following holds with probability at least $1-e^{-cn}-e^{-t}$: for all $\beta$ such that 
$\beta_{S^c} = 0$, $\partial_{\beta_l} \obj_{n, \gamma}(\beta) \ge 0$ for $l \not\in S$.
\end{lemma} \noindent\noindent
With Lemma~\ref{lemma:grad-noise-variable-beta-noise-0-finite-sample} at hand, we can prove 
$\beta^{(k)}_{S^c} = 0$ for all $k \in \N$. The proof is via induction.  
\begin{itemize}%[(i)]
\item The base case $\beta^{(0)}_{S^c} = 0$ (this is the only part where we use the assumption $\beta^{(0)} = 0$). 
\item Suppose $\beta^{(k)}_{S^c} = 0$. Fix $l \in S^c$. Then 
	$\partial_{\beta_l}  \obj_{n, \gamma}(\beta^{(k)})\ge 0$ by
	Lemma~\ref{lemma:grad-noise-variable-beta-noise-0}. This yields
	\begin{equation*}
		\beta^{(k+\half)}_l \equiv \beta^{(k)}_l - \alpha \cdot \partial_{\beta_l}  \obj_{n, \gamma}(\beta^{(k)}) \le 0.
	\end{equation*}
	Hence $\beta_l^{(k+1)} = (\proj_{\B_M}(\beta^{(k+\half)}))_l = 0$ thanks to 
	Lemma~\ref{lemma:projection-onto-ell-one-ball}. This proves that $\beta^{(k+1)}_{S^c} = 0$.
\end{itemize}

%\subsection{Proof of Corollary~\ref{corollary:add-main-effect}}
%The proof is pretty much the same as that of Theorem~\ref{thm:add-main-effect}. The core is
%the following result, which is a simple consequence of equation~\eqref{eqn:desired-goal-on-obj-main-add} and
%Theorem~\ref{thm:concentration-of-gradients}. 
%
%\begin{lemma}
%\label{lemma:core-corollary-add-main-effect}
%There exists a sufficiently large constant $\wbar{C}$ 
%that depends only on $M, M_X, \sigma_Y, \mu$ such that under the given condition of Corollary~\ref{corollary:add-main-effect}, 
%we have with probability at least $1-e^{-cn}-e^{-t}$, the following holds at any $\beta$ with $\beta_l = 0$ and $\beta_{S^c} = 0$
%\begin{equation*}
%	\partial_{\beta_l} \erobj(\beta) = \partial_{\beta_l} \erobj(\beta) + \gamma < 0. 
%\end{equation*}
%\end{lemma}\noindent\noindent
%Now we are ready to prove Corollary~\ref{corollary:add-main-effect}.
%By Proposition~\ref{prop:grad-obj-lipschitz-beta} and Lemma~\ref{lemma:gradient-ascent-increases-objective}, 
%the algorithm must reach a stationary point. It suffices to show that with the desired high probability any stationary point 
%$\beta^*$ reached by the algorithm must have $\beta_l^* > 0$. 
%
%To see this, Theorem~\ref{thm:no-false-positive} already shows that with high probability 
%any such stationary point $\beta^*$ must exclude noise variables, i.e., $\beta^*_{S^c} = 0$ holds. 
%Lemma~\ref{lemma:core-corollary-add-main-effect} shows that with high probability any $\beta$ with 
%$\beta_l = 0$ and $\beta_{S^c} = 0$ can't be stationary. This completes the proof. 

\subsection{Proof of Corollary~\ref{corollary:add-main-mix-effect}.}
\label{sec:proof-of-corollary-add-main-mix-effect}
The proof is similar to that of Theorem~\ref{thm:add-main-mix-effect}. The key
Lemma~\ref{lemma:core-corollary-add-main-mix-effect}  follows from %the triangle inequality, 
%Lemma~\ref{lemma:core-corollary-add-main-mix-effect} is basically a consequence of 
Lemma~\ref{lemma:key-to-thm-add-main-mix-effect}, Theorem~\ref{thm:concentration-of-gradients}, and 
the fact that $\gamma$ upper bounds with high probability the deviation between the empirical and true gradient
(as $\gamma$ satisfies~\eqref{eqn:lower-bound-on-ell-1}). 

\begin{lemma}
\label{lemma:core-corollary-add-main-mix-effect}
Assume the assumptions in Corollary~\ref{corollary:add-main-mix-effect}. Then the following holds 
with probability at least $1-e^{-cn}-e^{-t}$: for all $\beta$ such that $\beta_l = 0$ and $\beta_{S^c} = 0$: 
\begin{equation}
\label{eqn:grad-main-add-mix-effect-high-prob}
	\partial_{\beta_l} \obj_{n, \gamma}(\beta) \le \frac{1}{\lambda} \cdot  \left(-c \cdot \effect_l + C\lambda^{1/2}(1+\lambda^{1/2})
		+ 2\lambda \gamma\right).
\end{equation}
Above, the constant $c, C > 0$ depend only on $M, \mu_X, \mu_Y, \mu$.
\end{lemma}\noindent\noindent
Now we are ready to prove Corollary~\ref{corollary:add-main-mix-effect}.
By Proposition~\ref{prop:grad-obj-lipschitz-beta} and Lemma~\ref{lemma:gradient-ascent-increases-objective}, 
the algorithm must reach a stationary point. It suffices to show that with the desired high probability any stationary point 
$\beta^*$ reached by the algorithm must have $\beta_l^* > 0$. 

To see this, Corollary~\ref{corollary:no-false-positive-finite-sample} already shows that with high probability 
any such stationary point $\beta^*$ must exclude noise variables, i.e., $\beta^*_{S^c} = 0$ holds. 
Lemma~\ref{lemma:core-corollary-add-main-mix-effect} shows that with high probability any $\beta$ with 
$\beta_l = 0$ and $\beta_{S^c} = 0$ can't be stationary as long as the constant $\wbar{C} > 0$ in 
equation~\eqref{eqn:effect-signal-add-main-mix-finite-sample} is sufficiently large. This completes the proof.

\subsection{Proof of Corollary~\ref{corollary:hier-effect}.}
\label{sec:proof-of-corollary-hier-effect}
The proof is similar to that of Theorem~\ref{thm:hier-effect}. 
The key Lemma~\ref{lemma:core-corollary-hier-effect} below follows from Lemma~\ref{lemma:key-to-thm-hier-effect}, 
Theorem~\ref{thm:concentration-of-gradients}, and the fact that $\gamma$ upper bounds with high probability 
the deviation between the empirical and the true gradient.

\begin{lemma}
\label{lemma:core-corollary-hier-effect}
Assume the assumptions in Corollary~\ref{corollary:hier-effect}. With probability 
at least $1-e^{-cn}-e^{-t}$, we have for all $k \le K, m \le N_k$, and 
$\beta$ with $\beta_{S_{k, m-1}} = \tau \mathbf{1}_{S_{k, m-1}}$, 
$\beta_{k, m} = 0$, $\beta_{S^c} = 0$: 
\begin{equation}
\label{eqn:grad-hier-finite-sample}
	\partial_{\beta_{k, m}} \obj_{n, \gamma}(\beta) \le \frac{1}{\lambda} \cdot  \left(-c \cdot \min\{\tau^m, 1\} 
		\cdot \effect_{k, l} + C\lambda^{1/2}(1+\lambda^{1/2})+2\lambda \gamma\right).
\end{equation}
Above, the constant $c, C > 0$ depend only on $M, \mu_X, \mu_Y, \mu$.
\end{lemma}\noindent\noindent

Now we prove Corollary~\ref{corollary:hier-effect}. Corollary~\ref{corollary:no-false-positive-finite-sample} already 
shows that with high probability any such stationary point $\beta^*$ must exclude noise variables, i.e., 
$\beta^*_{S^c} = 0$ holds. Lemma~\ref{lemma:core-corollary-hier-effect} also shows that with high probability 
any $\beta$ such that $\beta_{S_{k, m-1}} = \tau \mathbf{1}_{S_{k, m-1}}$, $\beta_{k, m} = 0$, $\beta_{S^c} = 0$
must satisfy $\partial_{\beta_{k, m}} \obj_{n, \gamma}(\beta) < 0$
provided that the constant $\wbar{C} > 0$ in equation~\eqref{eqn:effect-signal-hier-finite-sample} is sufficiently large. 
Condition on these high probability events, one can proceed with the same logic in the proof of Theorem~\ref{thm:hier-effect}
to show that $S_{k, m} \subseteq \what{S}_m$ for all $0\le m \le l$. This proves Corollary~\ref{corollary:hier-effect}.

\newcommand{\circup}[1]{\accentset{\circ}{#1}}
\section{Recovery Guarantees Without Independence Assumption}
\label{sec:proof-under-dependent-covariates}

This section studies kernel feature selection under the general setting where we allow the 
covariates within the signal set to be dependent. In particular, we show that kernel feature 
section would recover signal variables that have nontrivial explanatory powers 
conditional on all the rest signal variables. 

To set the stage, we introduce assumptions and notation. A subset $T$ is called \emph{sufficient} if 
$\E[Y | X] = \E[Y | X_T]$. Otherwise the subset $T$ is called \emph{insufficient}. Sufficient subsets are also 
called \emph{Markov Blanket} in the literature on graphical models~\cite{Pearl14}, and are also closely related 
to \emph{sufficient dimension reduction} in the literature of statistics~\cite{Li91, Cook07}.

We make the following assumption on any insufficient set $T \subsetneq S$. 
The assumption says that if a set $T \subsetneq S$ is insufficient, then there exists some index 
$j \in S$ such that appending the variable $X_j$ to $X_T$ strictly increases the explanative power of $Y$.

\setcounter{assumption}{2}
\renewcommand\theassumption{\arabic{assumption}}
\begin{assumption}
\label{assumption:insufficient-set-T}
For any \emph{insufficient} subset $T \subsetneq S$, there exists an index $j \in S$ such that 
\begin{equation*}
	\E[Y| X_T] \neq \E[Y | X_{T \cup \{j\}}].
\end{equation*}
\end{assumption}

Now we show that on population the subset $\what{S}$ returned by the algorithm is sufficient, i.e., $\E[Y|X_{\what{S}}] = \E[Y|X_S]$,
provided that the signal size is beyond a certain threshold. To formally describe the signal size, we need the following definition 
of $\effect_T$. 

\begin{definition}
\label{eqn:effect_T}
Let $T \subsetneq S$ be any insufficient set $T$. Define $\effect_T = \max_{j \in S\backslash T} \effect_{j; T}$ where
\begin{equation*}
	\effect_{j; T} = \min_{g: \E[g^2(X_T)] \le \E[Y^2]} \E\left[(Y-g(X_T)) (Y'-g(X_T')) f(\normsmall{X_{T \cup \{j\}} - X'_{T \cup \{j\}} })\right]
\end{equation*}
\end{definition}
\begin{remark}
For any insufficient set $T$, $\effect_{j;T}$ quantifies the added 
explanatory power when one append $X_j$ to the 
variables $X_T$. As a result, $\effect_T$ quantifies the maximal 
increase of explanatory power over all possible $X_j$ that 
could be appended to $X_T$.
\end{remark}

\begin{proposition}
\label{proposition:why-effect-T-nonzero}
Let $T \subsetneq S, j\in S$ be such that $\E[Y| X_T] \neq \E[Y | X_{T \cup \{j\}}]$. Then $\effect_{j; T} > 0$.
\end{proposition}

Proposition~\ref{proposition:why-effect-T-nonzero} shows that $\effect_T > 0$ whenever
$T \subsetneq S$ is insufficient. Theorem~\ref{thm:dependent-covariates}---the main result 
of the section---shows that, on population, the output $\what{S}$ Algorithm~\ref{alg:kernel-feature-selection-hier} is 
sufficient as long as $\effect_T$---for any insufficient set $T$---is large enough. 
The proof is given in Appendix~\ref{sec:proof-of-thm-dependent-covariates}.

\begin{theorem}
\label{thm:dependent-covariates}
Assume Assumptions~\ref{assumption:mu-compact}, ~\ref{assumption:X-Y-bound} and~\ref{assumption:insufficient-set-T}. There exists a 
$\wbar{C} > 0$ depending only on 
$\tau, |S|, M, M_X, M_Y, M_\mu$ such that the following holds. Suppose the following condition holds: 
\begin{equation}
\label{eqn:condition-dependent-covariates}
	\min_{T\subsetneq S: T~\text{is insufficient}} \effect_T \ge \wbar{C} \cdot (\lambda^{1/2}(1+\lambda^{1/2}) + \lambda \gamma).
\end{equation}
Consider Algorithm~\ref{alg:kernel-feature-selection-hier} with the initializers $\{\beta^{(0; T)}\}_{T \in 2^{[p]}}$ where 
$\beta^{(0; T)}_T = \tau \mathbf{1}_T$ and $\beta^{(0; T)}_{T^c} = 0$ and with the stepsize 
$\alpha \le \frac{\lambda^2}{\wbar{C} p}$. 
Then the output  $\what{S}$ of Algorithm~\ref{alg:kernel-feature-selection-hier} is sufficient. 
\end{theorem}

\begin{remark}
%Here we compare our result with the existent ones in the literature. 
%Our recovery guarantee is comparable to the one in the literature~\cite{AzadkiaCh19}. However, our 
%assumptions are weaker in the sense that we do not require the smoothness assumption of the conditional probability 
%of $Y$ given $X$, as is required in the literature~\cite{AzadkiaCh19}.
We compare Theorem~\ref{thm:dependent-covariates} with Theorem~\ref{thm:add-main-mix-effect} 
and 
Theorem~\ref{thm:hier-effect}.% in the main text. 
\begin{itemize}
\item Theorem~\ref{thm:dependent-covariates} does not require the independence assumption
	needed in Theorem~\ref{thm:add-main-mix-effect} and~\ref{thm:hier-effect}.
\item Theorem~\ref{thm:add-main-mix-effect} shows that one round of 
	Algorithm~\ref{alg:kernel-feature-selection-hier} (i.e., Algorithm~\ref{alg:kernel-feature-selection})
	recovers the main effects; Theorem~\ref{thm:hier-effect} shows that 
	Algorithm~\ref{alg:kernel-feature-selection-hier}, with $\max_{k\le K}N_k$ rounds 
	recovers the hierarchical interactions. These information on number of rounds are absent 
	in Theorem~\ref{thm:dependent-covariates}. 
\end{itemize}
\end{remark}

Corollary~\ref{corollary:dependent-covariates} is a finite sample version of Theorem~\ref{thm:dependent-covariates}. 
The proof is similar to that of Corollary~\ref{corollary:add-main-mix-effect} and Corollary~\ref{corollary:hier-effect}.
\begin{corollary}
\label{corollary:dependent-covariates}
Make Assumption~\ref{assumption:mu-compact}, ~\ref{assumption:sub-gaussian-tail-Y} and \ref{assumption:insufficient-set-T}. 
%Let $c, C > 0$ be the constants as appeared in Theorem~\ref{thm:concentration-of-gradients}. 
Let $t > 0$. Assume $\gamma$ satisfies equation~\eqref{eqn:lower-bound-on-ell-1}. 

Then for some $c, \wbar{C} > 0$ depending on $\tau, |S|, M, M_X, M_Y, M_\mu$, if the condition 
%Suppose %the effective signal size $\effect_l$ satisfies
\begin{equation}
\label{eqn:effect-signal-hier-finite-sample}
	\min_{T\subsetneq S: T~\text{is insufficient}} \effect_T \ge 
		\wbar{C} \cdot (\lambda^{1/2}(1+\lambda^{1/2}) + \lambda \gamma) %~~\text{for some $l \in S$}, 
\end{equation}
holds, then Algorithm~\ref{alg:kernel-feature-selection-hier} with initializers $\{\beta^{(0; T)}\}_{T \in 2^{[p]}}$ where 
$\beta^{(0; T)}_T = \tau \mathbf{1}_T$ and $\beta^{(0; T)}_{T^c} = 0$
and with the stepsize $\alpha \le \frac{\lambda^2}{\wbar{C} p}$ outputs a set $\what{S}$ that is (i)
sufficient and (ii) $\what{S} \subseteq S$ with probability at least $1-e^{-cn}-e^{-t}$.
\end{corollary}

%\begin{remark}
%Here we compare our result with the existent ones in the literature. 
%Our recovery guarantee is comparable to the one in the literature~\cite{AzadkiaCh19} in the sense that 
%the output of the signal set is sufficient. However, our 
%assumptions are weaker in the sense that we do not require the smoothness assumption of the conditional probability 
%of $Y$ given $X$, as is required in the literature~\cite{AzadkiaCh19}.
%\end{remark}

\subsection{Proof of Theorem~\ref{thm:dependent-covariates}.}
\label{sec:proof-of-thm-dependent-covariates}

The proof idea of Theorem~\ref{thm:dependent-covariates} is similar to that of Theorem~\ref{thm:hier-effect}. 
Lemma~\ref{lemma:general-gradient-lower-bound} is the kety to the proof. 
Let $\what{S}_m$ denote the variables selected by the $m$-th iteration of the algorithm. 
It suffices to prove the following results. 
\begin{itemize}
\item No false positive: $\what{S}_m \subseteq S$ holds.
\item Recovery: the algorithm does not terminate at $m$-th round unless $\E[Y|\what{S}_m] = \E[Y|X_S]$.
\end{itemize}
It's easy to prove the first result; indeed it follows from a simple adaptation of the proof of 
Theorem~\ref{thm:no-false-positive}. Below we establish the second point. That completes 
the proof of Theorem~\ref{thm:dependent-covariates}.  

To see the second point, recall that in the m-th round, the algorithm aims to solve 
\begin{equation*}
(O_m):~~~~
\begin{aligned}
	\minimize_{\beta} ~~&\mathcal{J}_\gamma(\beta)  \\
	\subjectto ~~& \beta \ge 0 ~\text{and}~\beta_{\what{S}_{m}} = \tau  \mathbf{1}_{\what{S}_{m}}.
\end{aligned}
\end{equation*}
Let $\beta^*$ denote the solution returned by the gradient descent for $(O_m)$. We shall prove that 
$\what{S}_m \subsetneq \supp(\beta^*)$ unless $\what{S}_m$ is a sufficient set. Assume below W.L.O.G. 
that $\what{S}_m$ is insufficient. Since the solution $\beta^*$ is stationary w.r.t 
$(O_m)$ (followed by an identical argument in the proof of Theorem~\ref{thm:hier-effect}), 
it suffices to show that any stationary point $\beta$ must satisfy $\what{S}_m \subsetneq \supp(\beta^*)$. 
Towards this goal, Lemma~\ref{lemma:general-gradient-lower-bound} is the key, which basically 
says that, for any subset $T$ that's insufficient, and for any variable $l \not\in T$, the gradient w.r.t 
the variable $j$ is non-positive as long as $\effect_{j; T}$ is sufficiently large. 
The proof of Lemma~\ref{lemma:general-gradient-lower-bound} is in 
Section~\ref{sec:proof-theorem-general-gradient-lower-bound}. 

\begin{lemma}
\label{lemma:general-gradient-lower-bound}
Assume Assumptions~\ref{assumption:mu-compact}---\ref{assumption:X-Y-bound}. Let $T \subseteq S$ and 
$l \in S$ be any pair of subset and variable such that $\E[Y| X_T] \neq \E[Y | X_{T \cup \{l\}}]$. Then there exist constants 
$c, C > 0$ depending only on $\tau, |T|, M_X, M_Y, M_\mu$ such that at the point $\beta^0$ where $\beta_T^0 = \tau \mathbf{1}_T$ and $\beta_{T^c}^0 = 0$, 
\begin{equation*}
	\partial_{\beta_l}  \obj_\gamma(\beta)\mid_{\beta= \beta^0} 
		\le \frac{1}{\lambda} \cdot (-c \cdot \effect_{l; T} + C\lambda^{1/2}(1+\lambda^{1/2}) + \lambda \gamma).
\end{equation*}
\end{lemma}

By Lemma~\ref{lemma:general-gradient-lower-bound} and Assumption~\ref{assumption:insufficient-set-T}, 
and that $\effect_{\what{S}_m} = \max_{l \not\in \what{S}_m}\effect_{l; \what{S}_m}$, we can find 
some $l \not\in \what{S}_m$
\begin{equation}
\label{eqn:desired-goal-on-obj-dependence}
	\partial_{\beta_{l}} \regobj(\beta)\mid_{\beta = \beta^0} \le 
		\frac{1}{\lambda} \cdot (-c \cdot \effect_{\what{S}_m} + C\lambda^{1/2}(1+\lambda^{1/2}) + \lambda \gamma).
\end{equation}
Above, $c, C > 0$ depend only on $\tau, |S|, M, M_X, M_Y, M_\mu$. 
Note $\effect_{\what{S}_m} \ge \min_{T\subsetneq S: T~\text{is insufficient}} \effect_T$.
Hence, if the constant 
$\wbar{C} > 0$ in the condition~\eqref{eqn:condition-dependent-covariates} is sufficiently large, then 
equation~\eqref{eqn:desired-goal-on-obj-dependence} implies that 
$\beta^0$ can not be stationary. This concludes the proof of  Theorem~\ref{thm:dependent-covariates}.

\subsection{Proof of Lemma~\ref{lemma:general-gradient-lower-bound}.}
\label{sec:proof-theorem-general-gradient-lower-bound}
The key is to bound the surrogate gradient $\wtilde{\partial_{\beta_l} \obj(\beta^0)}$. Recall 
\begin{equation}
\label{eqn:rewrite-definition-of-surrogate-gradient}
\wtilde{\partial_{\beta_l} \obj(\beta)} = 
	-\frac{1}{\lambda} \cdot \iiint \left|\Cov \left({R}_{\beta, \omega}(\beta \odot X; Y), e^{i \zeta_l X_l }\right)\right|^2
				\cdot \psi_0(\zeta_l)  t \cdot \prod_i \psi_t(\omega_i) \cdot d\zeta_l d\omega\mu(dt)
\end{equation}
where $\psi_0(\zeta) = \frac{1}{\zeta^2}$. 
By definition $t\psi_0(\zeta_l) =  \frac{t}{\pi \zeta_l^2} \ge \frac{t}{\pi(t^2+\zeta_l^2)} = \psi_t(\zeta_l)$. 
Furthermore, we have 
\begin{equation*}
%\label{eqn:covariance-square-lower-bound}
\begin{split}
	&\left|\Cov \big({R}_{\beta, \omega}(\beta \odot X; Y), e^{i \zeta_l X_l }\big)\right|^2  
	\ge \half \left|\E[{R}_{\beta, \omega}(\beta \odot X; Y)e^{i\zeta_l X_l}]\right|^2 - 4|\E[{R}_{\beta, \omega}(\beta \odot X; Y)]|^2 
\end{split}
\end{equation*}
which follows from the basic inequality $|\Cov(W_1, W_2)|^2 \ge \half |\E[W_1W_2]|^2 - 4 |\E[W_1]|^2 |\E[W_2]|^2$
that holds for any (complex-valued) random variables $W_1, W_2$. As a result, we obtain that 
\begin{equation}
\label{eqn:upper-bound-surrogate-gradient-one}
\begin{split}
	\wtilde{\partial_{\beta_l} \obj(\beta)}& \le  \circup{\partial_{\beta_l}} \obj(\beta) 
		+ \frac{4}{\lambda} \cdot \int  |\E[{R}_{\beta, \omega}(\beta \odot X; Y)]|^2 \cdot Q(\omega)  d\omega.  \\
	\text{where}~~
	\circup{\partial_{\beta_l}} \obj(\beta) &= -\frac{1}{2\lambda} \cdot
		\iiint \left |\E\left[{R}_{\beta, \omega}(\beta \odot X; Y)e^{i \zeta_l X_l }\right]\right|^2
				\cdot \psi_t(\zeta_l) \cdot \prod_i \psi_t(\omega_i)d\zeta_l d\omega \mu(dt) 
\end{split}
\end{equation}
Using Lemma~\ref{lemma:KKT-condition}, Proposition~\ref{proposition:norm-of-H} and 
$\lambda \norm{f_\beta}_{\H}^2\le \energy(\beta, f_\beta) \le \energy(\beta, 0) \nolinebreak= \E[Y^2]$, we get
\begin{equation*}
\frac{1}{\lambda} \cdot\int  |\E[{R}_{\beta, \omega}(\beta \odot X; Y)]|^2 \cdot  Q(\omega) d\omega
=  \frac{\lambda}{(2\pi)^p} \cdot \int  \frac{|\F(f_\beta)(\omega)|^2}{Q(\omega)}d\omega = \lambda \norm{f_\beta}_{\H}^2
\le  \E[Y^2].
\end{equation*}
All of the above bounds, in conjunction with Lemma~\ref{lemma:approximate-gradient-bound}, establish the 
following fact: 
\begin{equation*}
	 \partial_{\beta_l} \obj(\beta) \le \circup{\partial_{\beta_l}} \obj(\beta)
	 	+ C \cdot (1+\frac{1}{\sqrt{\lambda}})
\end{equation*}
holds for any $\beta \in \R_+^p$ and $l \in [p]$, where $C > 0$ depends only on $M_X, M_Y, M_\mu$. 
Using the fact that $\partial_{\beta_l}  \obj_\gamma(\beta) = \partial_{\beta_j}  \obj(\beta) + \gamma$, it 
remains to prove that at $\beta = \beta^0$, $\circup{\partial_{\beta_l}} \obj(\beta)\mid_{\beta= \beta^0}  \le -c \cdot \effect_{l; T}$
where $c > 0$ depending only on $\tau, |T|$.

Below we prove this. Note $g_{\tau, T}(X_T) = f_{\beta^0}(\beta^0 \odot X)$ depends only on $\tau, T, X_T$. Hence,
\begin{equation*}
	\E \left[{R}_{\beta, \omega}(\beta \odot X; Y) 
		e^{i \langle \beta \odot \omega, X\rangle}e^{i \zeta_l X_l }\right] \mid_{\beta = \beta^0}
		= \E \left[(Y-g_{\tau, T}(X_T)) e^{i \cdot \tau \langle \omega_T, X_T\rangle}e^{i \zeta_l X_l }\right]
\end{equation*}
does not depend on $\omega_{T^c}$. Since $\int \psi_t(\omega_i) d\omega_i = 1$ for $i \not\in T$, we obtain 
\begin{equation*}
\begin{split}
\accentset{\circ}{\partial_{\beta_l}}  \obj(\beta) \mid_{\beta = \beta^0} 
	%&= - \frac{1}{\lambda} \cdot  \iiint \left| \E \left[(Y-g_{\tau, T}(X_T)) e^{i \cdot \tau \langle \omega_T, X_T\rangle}e^{i \zeta_l X_l }\right]\right|^2
	%	\cdot \psi_t(\zeta_l) \cdot \prod_{i \in [p]} \psi_t(\omega_i) d\zeta_l  d\omega \mu(dt) \\
	&= - \frac{1}{\lambda} \cdot  \iiint \left| \E \left[(Y-g_{\tau, T}(X_T)) e^{i \cdot \tau \langle \omega_T, X_T\rangle}e^{i \zeta_l X_l }\right]\right|^2
		\cdot \psi_t(\zeta_l) \cdot \prod_{i \in T} \psi_t(\omega_i) d\zeta_l  d\omega_T \mu(dt) \\
	&=- \frac{1}{\lambda} \cdot  \iiint \left| \E \left[(Y-g_{\tau, T}(X_T)) e^{i \cdot \langle \omega_T, X_T\rangle}e^{i \zeta_l X_l }\right]\right|^2
		\cdot \psi_t(\zeta_l) \cdot \prod_{i \in T} \psi_{\tau t}(\omega_i) d\zeta_l  d\omega_T \mu(dt)
\end{split}
\end{equation*}
where the second line follows from a change of variable ($\omega_T \mapsto \tau \omega_T$). 
Since we have 
$\psi_{\tau t}(\omega) = \frac{\tau t}{\tau^2 t^2 + \omega^2} \ge \frac{\tau}{(1+\tau^2)} \cdot \frac{t}{t^2+\omega^2} 
		= \frac{\tau}{(1+\tau^2)} \cdot \psi_t(\omega)$, this immediately implies 
(write $\wbar{\tau} = \tau/(1+\tau^2)$)
\begin{equation*}
\begin{split}
\accentset{\circ}{\partial_{\beta_l}}  \obj(\beta) \mid_{\beta = \beta^0}  
&\le -\frac{\wbar{\tau}^{|T|}}{\lambda} 
	\cdot \iiint  \left| \E \left[(Y-g_{\tau, T}(X_T)) e^{i \cdot \langle \omega_T, X_T\rangle}e^{i \zeta_l X_l }\right]\right|^2
		\cdot \psi_t(\zeta_l) \prod_{i \in T} \psi_t(\omega_i) \cdot d\zeta_l  d\omega_T \mu(dt) \\
%&=-\frac{\wbar{c}^{|T|}}{\lambda} \cdot \E \left[(Y-g_{\tau, T}(X_T))(Y'-g_{\tau, T}(X_T')) \cdot \iiint  e^{i \cdot \langle \omega_T, X_T-X_T'\rangle}e^{i \zeta_l (X_l -X_l') }
%		\cdot \psi_t(\zeta_l) \cdot \prod_{i \in T} \psi_{t}(\omega_i) d\zeta_l  d\omega_T \mu(dt)\right] \\
&= -\frac{\wbar{\tau}^{|T|}}{\lambda}
	\cdot \E \left[(Y-g_{\tau, T}(X_T))(Y'-g_{\tau, T}(X_T')) h\Big(\normsmall{X_{T\cup\{l\}}- X_{T\cup\{l\}}'}_1\Big)\right],
\end{split}
\end{equation*}
where the last line uses the Fourier representation of the kernel $(x, x') \mapsto h(\norm{x-x'}_1)$.
By Proposition~\ref{prop:second-moment-bound}, $\E[g_{\tau, T}(X_T)^2] \le \E[Y^2]$. Hence
$\accentset{\circ}{\partial_{\beta_l}}  \obj(\beta) \mid_{\beta = \beta^0}  
		\le -\frac{1}{\lambda} \cdot \wbar{c}^{|T|} \cdot \effect_{l; T}$ as desired.
%This proves equation~\eqref{eqn:goal-of-prop-general-gradient-lower-bound} as desired. 

\subsection{Proof of Proposition~\ref{proposition:why-effect-T-nonzero}.}
\label{sec:proof-proposition-why-effect-T-nonzero}

%The proof of Proposition~\ref{proposition:why-effect-T-nonzero} consists of two steps. 
Let $\wbar{\G} = \{g: \E[g^2(X_T)]\}$ denote the Hilbert space 
with the inner product $\langle g_1, g_2\rangle = \E[g_1(X_{T})g_2(X_{T})]$.
and $\mathcal{G} = \left\{g: \E[g^2(X_T)] \le \E[Y^2]\right\}$.  Introduce
\begin{equation*}
	\effect_{j; T}(g) = \E\left[(Y-g(X_{T}))(Y'-g(X'_{T}))
		h\big(\normsmall{X_{T \cup \{j\}} - X_{T \cup \{j\}}'}_{1}\big) \right].
\end{equation*}
Note then $\effect_{j; T}  =  \inf_{g \in \G} \effect_{j; T}(g)$. %is the minimum of $\effect_l(T, \beta, g)$
%over all $T$, $\beta_{T} \in \mathcal{B}_T(l)$ and $g \in \G_{T\backslash l}$. 
Below we prove $\effect_{j; T} > 0$ when $\E[Y | X_{T \cup \{j\}}] \neq \E[Y | X_T]$.

To see this, we first prove that $\effect_{j; T}(g) > 0$ for any fixed $g\in \G$. Indeed, using the Fourier representation
of the kernel $(x, x') \mapsto h(\norm{x-x'}_1)$, we obtain the expression 
\begin{equation*}
\begin{split}
	\effect_{j; T}(g) = \iint \left|\E[(Y - g(X_{T})) e^{i \langle \omega_{T\cup\{j\}}, X_{T\cup\{j\}}\rangle} ]\right|^2 
		\cdot \prod_{i \in T\cup\{j\}}\psi_t(\omega_i) d\omega_{T\cup\{j\}} \mu(dt) \\
	%	&= \iint \left|\E[(\E[Y\mid X_T] - g(X_{T \backslash l})) e^{i\langle \omega_T, x_T\rangle}]\right|^2 \cdot \prod_{i \in T}\psi_t(\omega_i) d\omega_T \mu(dt).
\end{split}
\end{equation*}
Now suppose on the contrary that $\effect_{j; T}(g) = 0$ for some $g \in \G$. Then it implies for all $\omega$, 
\begin{equation*}
	\E[(Y - g(X_{T})) e^{i \langle \omega_{T\cup\{j\}}, X_{T\cup\{j\}}\rangle} ]= 0,
\end{equation*}
This shows that $\E[Y| X_{T \cup \{j\}}] = \E[Y|X_T]$. Contradiction! Hence, $\effect_{j; T}(g) > 0$.

Next, we prove that $\effect_{j; T}  =  \inf_{g \in \G} \effect_{j; T}(g) > 0$. This follows the facts that 
(i) the mapping $g \mapsto \effect_{j; T}(g)$ is lower-semicontinuous w.r.t the weak-* topology of 
$\wbar{\G}$, and (ii) the set $\G$ is compact w.r.t the weak-* topology of the Hilbert space $\wbar{\G}$
(Banach-Alaoglu Theorem~\cite{Conway19}).

\section{Deferred Proof of Results in Section~\ref{sec:preliminary-main-text}}
\label{sec:proof-of-section-of-preliminary-results}

\subsection{Proof of Proposition~\ref{proposition:k-pos-definite-representation}.}
\label{sec:proof-k-pos-definite-representation}
As explained in the main text, when $q= 2$, Proposition~\ref{proposition:k-pos-definite-representation} 
is known as Schoenberg's theorem in the literature~\cite[Theorem 7.13]{Wendland04}. Despite the 
wealth of information available, we are not able to find a good reference of 
Proposition~\ref{proposition:k-pos-definite-representation} for $q=1$ in the literature. For convenience of 
the readers, we provide a self-contained proof. The proof below for $q=1$ mimics that for $q=2$ in the literature 
\cite[Theorem 7.13]{Wendland04}.

\subsubsection{Sufficiency.}
Assume that the function $h$ satisfies equation~\eqref{eqn:l_1_kernel-prop}. Then
\begin{equation*}
	k(x, x') = \int_0^\infty e^{-t \normsmall{x-x'}_q^q} \mu(dt). 
\end{equation*}
It is well-known that $(x, x') \mapsto e^{-t \normsmall{x-x'}_q^q}$ is the positive definite Laplace kernel when $q=1$ and Gaussian 
kernel when $q = 2$. As a weighted average of the 
Laplace and Gaussian kernel over different scale $t > 0$, the function $(x, x') \mapsto k(x, x')$ must also be positive definite. 

\subsubsection{Necessity.}
Assume that the function $k(x, x') = h(\norm{x-x'}_q^q)$ is positive definite. We show that $h$ admits 
the integral representation in equation~\eqref{eqn:l_1_kernel-prop}. %The proof is based on the standard approximation-theoretic arguments. 
Recall the following definition of finite-difference operator $\Delta_s^{k}$ in approximation theory~\cite{Wendland04, SchillingSoVo12}. 
\begin{definition}
For any function $\phi: \R_+ \to \R$, $s \ge 0$, we define its $k$th-order difference by 
\begin{equation*}
	[\Delta_s^{k} \phi] (x) \defeq \sum_{j=0}^k (-1)^{k-j}  \choose{k}{j} \phi(x + js).
\end{equation*}
\end{definition}
We argue that it suffices to show that $(-1)^k\Delta_s^{k}[h] \ge 0$ for all $s \ge 0$ and $k \in \N$. Indeed, 
the desired result would follow from the following theorem on completely monotone functions.
%independently treated by Bernstein in 1914, by Hausdorff in 1921, and by Widder in 1931. 
\begin{theorem}[Hausdorff--Bernstein--Widder]
A function $\phi \in \mathcal{C}^\infty[0, \infty)$ satisfies $[\Delta_s^{k} \phi] (x) \ge 0$ for all $s \ge 0, x\ge 0, k\in \N$ if 
and only if it is the Laplace transform of a nonnegative finite Borel measure $\mu$, i.e., 
\begin{equation*}
	\phi(x) = \int_0^\infty e^{-t x} \mu(dt).
\end{equation*}
\end{theorem}

Now we prove that $(-1)^k\Delta_s^{k}[h] \ge 0$ for all $s \ge 0$ and $k \in \N$. The proof is based on induction. 
To see why $k = 0$ holds, set $x_j = r^{1/q} \cdot e_j /2$ where $\{e_j\}_{j \in [d]}$ is the standard basis 
in $\R^d$. Since $(x, x') \mapsto h(\norm{x-x'}_q^q)$ is positive definite, we obtain that 
$0 \le \sum_{j, k=1}^N h(\norm{x_j - x_k}_q^q) = N\phi(0) + N(N-1) \phi(r)$.
Letting $N$ tend to infinity shows that $\phi(r) \ge 0$ for all $r  \ge 0$.

For the induction step, as $\Delta_s^{k} = \Delta_s^{k-1} \circ \Delta_s^1$, it suffices to show that 
$(x, x') \mapsto -\Delta_s^1[h](\norm{x-x'}_q^q)$ is positive definite. To do this, suppose 
$x_1, x_2, \ldots, x_N \in \R^d$ and $\alpha \in \R^N$ are given. Treat $x_j$ as elements of 
$\R^{d+1}$ (by concatenating $0$ to the $d+1$th coordinate of $x$) and define 
\begin{equation*}
	y_j = \begin{cases}
			x_j ~& 1\le j\le N \\
			x_{j-N}+s^{1/q}\cdot e_{d+1}~&N+1 \le j \le 2N
		\end{cases},
~~~~~~~
	\beta_j = \begin{cases}
			\alpha_j ~& 1\le j\le N \\
			-\alpha_{j-N}~& N+1 \le j \le 2N
		\end{cases}.
\end{equation*}
Since $h$ is positive definite on $\R^{d+1}$, we get %we can then obtain  
$
	0 \le \sum_{j, k = 1}^{2N} \beta_j \beta_k h(\norm{y_j - y_k}_q^q) 
		%= 2 \sum_{j,k=1}^N \alpha_j \alpha_k (h(\norm{x_j - x_k}_1) - h(\norm{x_j - x_k}_1 + s)) 
		= - 2\sum_{j, k =1}^N \alpha_j \alpha_k \cdot \Delta_{s}^1[h](\norm{x_j - x_k}_q^q). 
$
Hence $(x, x') \mapsto -\Delta_s^1[h](\norm{x-x'}_q^q)$ is positive definite as desired. 

\subsection{Proof of Proposition~\ref{proposition:norm-of-H}.}
\label{sec:proof-norm-of-H}
The proof uses standard argument in RKHS theory~\cite{BerlinetTh11, Paulsen16}. Consider the following Hilbert space
\begin{equation*}
	\H' = \left\{f\in \mathcal{C}(\R^p) \cap \mathcal{L}_2(\R^p): \int \frac{|\F(f)(\omega)|^2}{Q(\omega)} d\omega < \infty\right\},
\end{equation*}
and define its inner product as 
\begin{equation*}
	\langle f, g\rangle_{\H'} = \frac{1}{(2\pi)^p} \int \frac{\F(f)(\omega) \wbar{\F(g)(\omega)}}{Q(\omega)} d\omega
\end{equation*}
Now it suffices to show that $\H' = \H$. By the uniqueness part of Moore-Aronszajn's theorem, 
it suffices to check that $f(x) = \langle f, k(\cdot, x)\rangle_{\H'}$ holds for any function $f \in \H'$ and any $x \in \R^p$.

To see this, note the Fourier transform of Laplace (Gaussian) is Cauchy (Gaussian resp.): 
\begin{equation*}
\begin{split}
	\exp(-|z|) &= \int e^{i \omega z} \frac{1}{\pi(1+\omega^2)}d \omega \\
	\exp(-|z|^2) &= \int e^{i \omega z} \frac{1}{2\sqrt{\pi}}e^{-\omega^2/4}d \omega
\end{split}
\end{equation*}
Consequently, we obtain the integral formula that holds for both $q=1$ and $q=2$: 
\begin{equation*}
k(x', x) = \int_0^\infty e^{-t\norm{x'-x}_q^q} \mu(dt) = \int_0^\infty \int_{\R^p} e^{i \langle \omega, x'- x\rangle}q_t(\omega)\mu(dt)
	= \int_{\R^p} e^{i \langle \omega, x'- x\rangle} Q(\omega) d\omega.
\end{equation*}
Note that (i) $k(\cdot, x) \in \mathcal{L}_1(\R^p)$ since $h \in \mathcal{L}_1(\R^p)$ by assumption, and (ii)
$Q(\cdot) \in \mathcal{L}^1(\R^p)$ since $\int Q(\omega) d\omega = \int_0^\infty \mu(dt) = h(0) < \infty$.
Hence, Fourier's inversion theorem implies that 
\begin{equation*}
	\F(k(\cdot, x))(\omega) = (2\pi)^{p/2} Q(\omega)e^{-i\langle \omega, x\rangle},
\end{equation*}
where the equality holds a.e. under the Lebesgue measure. 
Now suppose  $\norm{f}_{\H'} < \infty$. 
This would imply that $\omega \mapsto \F(f)(\omega)$ is integrable since by Cauchy-Schwartz inequality, 
\begin{equation*}
	\left(\int |\F(f)(\omega)| d\omega\right)^2 \le \int \frac{|\F(f)(\omega)|^2}{Q(\omega)}d\omega
			\cdot \int Q(\omega)d\omega = (2\pi)^{p} |h(0)| \norm{f}_{\H'}^2 <\infty.
\end{equation*}
As $f \in \mathcal{C}(\R^p)$, Fourier's inversion theorem implies that the following holds for all $x$: 
\begin{equation*}
	f(x) = \frac{1}{(2\pi)^{p/2}}\int_{\R^p} F(f)(\omega) e^{i \langle \omega, x\rangle}  d\omega =\frac{1}{(2\pi)^p} \int_{\R^p} 
		\frac{F(f)(\omega) \wbar{\F(k(\cdot, x))(\omega)}}{Q(\omega)} d\omega. 
\end{equation*}
This shows that $f(x) = \langle f, k(\cdot, x)\rangle_{\H'}$ holds for all $x$ and $f \in \H'$. The proof is complete.

\section{Basics}
\label{sec:basics}

\subsection{RKHS}
Moore-Aronszajn Theorem is foundational to the RKHS theory~\cite{Aronszajn50}. 

\begin{theorem}[Moore-Aronszajn]
\label{theorem:Moore-Aronszajn}
Let $k(x, x')$ be a positive definite kernel on $X \times X$. Then there exists one unique 
reproducing kernel Hilbert space $\H$ associated with the kernel $k(x, x')$. It can be identified as follows.
First, let $\H_0$ be the space spanned by the functions 
$\{k(x, \cdot)\}_{x\in X}$. Then, $\H$ is the completion of the pre-Hilbert space 
$\H_0$ with respect to the inner product
\begin{equation*}
\langle f, g \rangle_{\H_0} = \sum_{i=1}^n \sum_{j=1}^m \alpha_{f, i} \alpha_{g, j} k(x_i, y_j).
\end{equation*}
where $f(\cdot) = \sum_{i=1}^n \alpha_{f, i} k(x_i, \cdot)$ and $g(\cdot) = \sum_{j=1}^m \alpha_{g, j} k(y_j, \cdot)$.
%In particular, any element $f \in \H$ can be obtained as a pointwise limit of a sequence $\{f_n\}_{n \in \N}$ that belongs to $\H_0$, i.e., $f(x) = \lim_{n \to \infty} f_n(x)$ for all $x\in X$.
\end{theorem}

\subsection{Functional Analysis}
Let $\H_1, \H_2$ be Hilbert spaces. 
\begin{definition}
The Hilbert-Schmidt norm for an operator $A:\H_1 \to \H_2$ is defined by:
\begin{equation*}
	\normbig{A}_{\HS}^2 =\sum_{i=1}^\infty \sum_{j=1}^\infty \langle \phi_i, A\psi_j\rangle_{\H_2}^2.
\end{equation*} 
Here $\{\psi_i\}_{i \in \N}$, $\{\phi_i\}_{i\in \N}$ are complete orthonormal system of 
$\H_1$ and of $\H_2$ respectively. The Hilbert-Schmidt inner-product 
for any two operators $A_1, A_2:\H_1 \to \H_2$ is defined by 
\begin{equation*}
	\langle A_1, A_2 \rangle_{\HS} =\sum_{i=1}^\infty \sum_{j=1}^\infty  \langle \phi_i, A_1\psi_j\rangle_{\H_2} \langle \phi_i, A_2\psi_j\rangle_{\H_2}.
\end{equation*} 
\end{definition}

\begin{definition}
The operator norm for an operator $A:\H_1 \to \H_2$ is defined by:
\begin{equation*}
	\opnorm{A} = \sup_{\psi: \normsmall{\psi}_{\H_1} = 1} \normbig{ A\psi }_{\H_2}
		= \sup_{\phi: \normsmall{\phi}_{\H_2} \le 1, \normsmall{\psi}_{\H_1} \le 1} \langle \phi, A\psi\rangle_{\H_2}.
\end{equation*} 
\end{definition}

\subsection{Concentration}
\label{sec:tool-concentrations}

\subsubsection{Definitions}

\begin{definition}
\label{definition: sub-gaussian-exponential}
A random variable $Z$ is sub-gaussian with parameter $\sigma$ if 
$\E[e^{tZ}] \le e^{\frac{\sigma^2 t^2}{2}}$ for all $t \in \R$.
A random variable $Z$ is sub-exponential with parameter $\sigma$ if and only if 
\begin{equation*}
	\E[e^{tZ}] \le e^{\frac{\sigma^2 t^2}{2}}~~\text{for all $|t| \le \frac{1}{\sigma}$}.
\end{equation*}
\end{definition}

\begin{definition}
\label{definition:sub-exponential-process}
Let $d$ be a semi-norm on a set $\mathcal{S}$. We call a random process 
$\{X_s\}_{s\in \mathcal{S}}$ on the space $(\mathcal{S}, d)$ sub-exponential if $\E[X_s] = 0$ and 
for all $s_1, s_2 \in S$, $X_{s_1} - X_{s_2}$ is sub-exponential with parameter $d(s_1, s_2)$.
\end{definition}

\subsubsection{Panchenko's lemma.}
Lemma~\ref{lemma:Panchenko} is due to Panchenko~\cite[Lemma 1]{Panchenko03}. 
\begin{lemma}
\label{lemma:Panchenko}
Let $Z_1, Z_2$ be two random variables such that 
\begin{equation*}
\E[\Phi(Z_1)] \le \E[\Phi(Z_2)]
\end{equation*}
for all convex and increasing function $\Phi: \R_+ \to \R_+$. If for some $a, c_1, \alpha \ge 1, c_2 > 0$, 
\begin{equation*}
\P(Z_2 \ge a + t) \le c_1 e^{-c_2 t^{\alpha}}~~\text{for all $t \ge 0$}.
\end{equation*}
then we have 
\begin{equation*}
\P(Z_1 \ge a + t) \le c_1 e^{1-c_2 t^{\alpha}}~~\text{for all $t \ge 0$}.
\end{equation*}
\end{lemma}

\subsubsection{Maurey's Sparsification Lemma.}
Let $z_1, z_2, \ldots, z_n$ be in $\R^p$ such that $\norm{z_i}_\infty \le M_Z$. Introduce
the seminorm $\norm{\cdot}_{@} \defeq \max_{i=1}^n |\langle \cdot, z_i\rangle|$. 
Let $\mathcal{B}_M = \left\{\beta: \norm{\beta}_1 \le M\right\}$. Proposition~\ref{proposition:maurey-argument} 
bounds $N(\mathcal{B}_M, \norm{\cdot}_{@}, \eps)$, the minimal covering number of 
$\mathcal{B}_M$ using $\eps$-$\norm{\cdot}_{@}$ ball. 
%The proof is based on Maurey's sparsification technique~\cite{Pisier81, Jones92, Barron93}. 
\begin{proposition}[Maurey's Sparsification Lemma~\cite{Pisier81, Jones92, Barron93}]
\label{proposition:maurey-argument}
There exists an absolute constant $C > 0$ such that 
\begin{equation*}
\log N(\mathcal{B}_M, \norm{\cdot}_{@}, M_Z \cdot \eps) \le CM^2 \cdot \frac{\log n \log p}{\eps^2}. %\log \left(1+p\eps^2\right).% , p\log\left(1+ \frac{\log n}{\eps^2 p}\right)\right\}.
\end{equation*}
\end{proposition}

\subsubsection{Chaining theorem for sub-exponential processes}

\begin{theorem}
\label{theorem:chaining-sub-exponential}
Let $d$ be a semi-norm on a set $\mathcal{S}$. 
Assume that the random process $\{X_s\}_{s\in \mathcal{S}}$ on the space $(\mathcal{S}, d)$ is sub-exponential 
(see Definition~\ref{definition:sub-exponential-process}). Then for any $\delta > 0$ and $t > 0$, and 
any choice of $s_0 \in \mathcal{S}$, we have with probability at least $1-e^{-t}$
\begin{equation*}
\sup_{s \in \mathcal{S}} X_s \le X_{s_0} + 
		\sup_{\stackrel{s_1, s_2 \in S}{d(s_1, s_2) \le \delta}} |X_{s_1} - X_{s_2}| 
			+  C \cdot \left( \int_{\delta}^{\diam(\mathcal{S})} \log(N(\mathcal{S}, d, \eps)) d\eps + \diam(\mathcal{S}) t\right),
\end{equation*}
where in above, $\diam(\mathcal{S}) = \sup \{d(s_1, s_2) \mid s_1, s_2 \in S\}$ denotes the diameter of the set 
$S$ w.r.t the seminorm $d$, $N(\mathcal{S}, d, \eps)$ denotes the cardinality of smallest $\eps$-covering set.
%and $Q(G; \delta)$ denotes the upper $\delta$-th quantile of a random variable $G$: $Q(G; \delta) = \inf_{s \in \R} \P(G \ge s) \le \delta$.
\end{theorem}

\begin{proof}
The proof is based on standard chaining argument~\cite{Ramon14}. 

%Assume W.L.O.G that $\mathcal{S}$ is a discrete space (i.e., $|\mathcal{S}| < \infty$). 
Let $k_0$ be the largest integer such that $2^{-k_0} \ge \diam(\mathcal{S})$. 
Let $k_1$ be the smallest integer such that $\delta \ge 2^{-k_1}$. For each $k \ge 0$, let $N_k$ 
be a $2^{-k}$ net, and choose $\pi_k(s) \in N_k$ such that $d(s, \pi_k(s)) \le 2^{-k}$. 
%As $|\mathcal{S}| < \infty$, we have $\pi_n(s) = s$ for $n$ sufficiently large. 
Fix $s_0$ and let $N_{k_0} = \{s_0\}$. Then by the telescoping of the sum, we have 
\begin{equation*}
	X_s - X_{s_0} = X_s - X_{\pi_{k_1}(s)} + \sum_{k_0 < k \le k_1} X_{\pi_k(s)} - X_{\pi_{k-1}(s)}.
\end{equation*}
Consequently, we obtain that 
\begin{equation*}
	\sup_{s \in \mathcal{S}} X_s \le X_{s_0} + 	
		\sup_{s \in \mathcal{S}} (X_s - X_{\pi_{k_1}(s)}) + \sum_{k > k_0} \sup_{s \in \mathcal{S}} \left\{X_{\pi_k(s)} - X_{\pi_{k-1}(s)}\right\}. 
\end{equation*}
By the sub-exponential property of the process $\{X_s\}_{s \in \mathcal{S}}$ and union bound, we get for $t > 0$:
\begin{equation*}
	\P\left(\sup_{s \in \mathcal{S}} \left\{X_{\pi_k(s)} - X_{\pi_{k-1}(s)}\right\} \ge 2^{-k}(\log (|N_k|) + t)\right) \le e^{-t}.
\end{equation*}
Thus, with high probability, every link $X_{\pi_k(s)} - X_{\pi_{k-1}(s)}$ at the scale $k$ is small. Now we 
show that all links are small simultaneously. To do so, we fix a sequence of $t_k$. Then 
\begin{equation*}
\begin{split}
\P\left(\Omega\right) &\defeq 
	\P \left(\exists k > k_0 ~~\text{s.t.}. ~\sup_{s \in \mathcal{S}} \{X_{\pi_k(s)} - X_{\pi_{k-1}(s)}\}
		\ge 2^{-k}(\log (|N_k|+3) + t) \right) \\
		&\le \sum_{k > k_0} \P \left(\exists k > k_0 ~~\text{s.t.}. ~\sup_{s \in \mathcal{S}} \{X_{\pi_k(s)} - X_{\pi_{k-1}(s)}\} 
		\ge 2^{-k}(\log (|N_k|+3) + t) \right)  \\
		&\le \sum_{k > k_0} \exp(-t_k).
\end{split}
\end{equation*}
A simple choice of $t_k = t + \sqrt{k-k_0}$ gives that $\P(\Omega) \le e^{-t} \cdot \sum_{k > 0}  \exp(- \sqrt{k}) \le Ce^{-t}$ 
for some absolute constant $C > 0$. Now note that on the event $\Omega^c$, we have 
\begin{equation*}
\begin{split}
	\sup_{s \in \mathcal{S}} X_s &\le X_{s_0} + \sup_{s \in \mathcal{S}} (X_s - X_{\pi_{k_1}(s)}) + \sum_{k_0 < k\le k_1} 2^{-k}(\log (|N_k|+3) + t)  \\
	&\le X_{s_0} + \sup_{s \in \mathcal{S}} (X_s - X_{\pi_{k_1}(s)}) + 
		\wbar{C} \cdot  \int_{\delta}^{\diam(\mathcal{S})} \log(N(\mathcal{S}, d, \eps)) d\eps + \diam(\mathcal{S}) t
\end{split}
\end{equation*}
where in the last inequality, we have used the fact that $\sum_{k > k_0} 2^{-k} \le 2^{-(k_0-1)} \le 4\diam(\mathcal{S})$. 
This completes the proof. 
%Hence, with probability at least $1-e^{-t}$, we have for some absolute constant $C > 0$,
%\begin{equation*}
%	\sup_{s \in \mathcal{S}} X_s \le X_{s_0} +
%		\sup_{\stackrel{s_1, s_2 \in S}{d(s_1, s_2) \le \delta}} |X_{s_1} - X_{s_2}|
%			+  C \cdot  \int_{\delta}^{\diam(\mathcal{S})} \log(N(\mathcal{S}, d, \eps)) d\eps + \diam(\mathcal{S}) t
%\end{equation*}
\end{proof}

\subsection{Optimization}
Lemma~\ref{lemma:gradient-ascent-increases-objective} is standard in nonlinear optimization~\cite[Prop 2.3.2]{Bertsekas97}.
\begin{lemma}
\label{lemma:gradient-ascent-increases-objective}
Consider the minimization problem 
\begin{equation*}
\begin{split}
\minimize_\beta ~J(\beta)~~ 
\subjectto ~\beta \in \mathcal{C}.
\end{split}
\end{equation*}
Assume that (i) the gradient $\beta \mapsto \grad J(\beta)$ is L-Lipschitz, i.e., 
$\ltwo{\grad J(\beta) - \grad J(\beta^\prime)} \le L \ltwo{\beta-\beta^\prime}$
for any $\beta, \beta^\prime \in \mathcal{C}$ and (ii) the constraint set 
$\mathcal{C}$ is convex. Consider projected gradient descent 
\begin{equation*}
\beta^{(k+1)} = \proj_{\mathcal{C}} \left(\beta^{(k)} - \alpha \grad J(\beta^{(k)})\right),
\end{equation*}
where the stepsize $\alpha \le 1/L$. Then (i)  $k \mapsto J(\beta^{(k)})$ is 
monotonically decreasing and (ii) any accumulation point $\beta^{\infty}$ of 
$\{\beta^{(k)}\}_{k \in \N}$ is stationary, i.e., $\langle \grad J(\beta^{\infty}), \beta' - \beta^\infty\rangle \ge 0$
for $\beta' \in \mathcal{C}$.
\end{lemma}\noindent\noindent
Lemma~\ref{lemma:projection-onto-ell-one-ball} characterizes the projection onto an $\ell_1$ ball. 
The proof can be found in~\cite{LiuRu20}.
\begin{lemma}
\label{lemma:projection-onto-ell-one-ball}
Let $\beta \in \R^p$ and $\mathcal{B} = \{\beta: \beta \ge 0, \norm{\beta}_1 \le b\}$. 
The $\ell_2$ projection $\tilde{\beta} = \proj_{\mathcal{B}}(\beta)$ satisfies 
$\tilde{\beta} = (\beta - \gamma)_+$ where $\gamma \ge 0$ is defined by 
$\gamma = \inf\{\gamma \ge 0: \sum_{i \in [p]}(\beta_i - \gamma)_+ \le b\}$.
\end{lemma}

\bibliographystyle{amsalpha}
\bibliography{bib}

\newcommand{\etalchar}[1]{$^{#1}$}
\providecommand{\bysame}{\leavevmode\hbox to3em{\hrulefill}\thinspace}
\providecommand{\MR}{\relax\ifhmode\unskip\space\fi MR }
\providecommand{\MRhref}[2]{%
  \href{http://www.ams.org/mathscinet-getitem?mr=#1}{#2}
}
\providecommand{\href}[2]{#2}
\begin{thebibliography}{WMC{\etalchar{+}}00}

\bibitem[AHM{\etalchar{+}}17]{Arras17}
Leila Arras, Franziska Horn, Gr{\'e}goire Montavon, Klaus-Robert M{\"u}ller,
  and Wojciech Samek, \emph{" what is relevant in a text document?": An
  interpretable machine learning approach}, PloS one \textbf{12} (2017), no.~8,
  e0181142.

\bibitem[All13]{Allen13}
Genevera~I Allen, \emph{Automatic feature selection via weighted kernels and
  regularization}, Journal of Computational and Graphical Statistics
  \textbf{22} (2013), no.~2, 284--299.

\bibitem[Aro50]{Aronszajn50}
Nachman Aronszajn, \emph{Theory of reproducing kernels}, Transactions of the
  American Mathematical Society \textbf{68} (1950), no.~3, 337--404.

\bibitem[Bak73]{Baker73}
Charles~R Baker, \emph{Joint measures and cross-covariance operators},
  Transactions of the American Mathematical Society \textbf{186} (1973),
  273--289.

\bibitem[Bar93]{Barron93}
Andrew~R Barron, \emph{Universal approximation bounds for superpositions of a
  sigmoidal function}, IEEE Transactions on Information Theory \textbf{39}
  (1993), no.~3, 930--945.

\bibitem[Ber97]{Bertsekas97}
Dimitri~P Bertsekas, \emph{Nonlinear programming}, Journal of the Operational
  Research Society \textbf{48} (1997), no.~3, 334--334.

\bibitem[BTA11]{BerlinetTh11}
Alain Berlinet and Christine Thomas-Agnan, \emph{Reproducing kernel hilbert
  spaces in probability and statistics}, Springer Science \& Business Media,
  2011.

\bibitem[CD12]{CommingesDa12}
La{\"e}titia Comminges and Arnak~S Dalalyan, \emph{Tight conditions for
  consistency of variable selection in the context of high dimensionality},
  Annals of Statistics \textbf{40} (2012), no.~5, 2667--2696.

\bibitem[CFJL16]{CandesFaJaLv16}
Emmanuel Candes, Yingying Fan, Lucas Janson, and Jinchi Lv, \emph{Panning for
  gold: Model-x knockoffs for high-dimensional controlled variable selection},
  arXiv preprint arXiv:1610.02351 (2016).

\bibitem[CLWY18]{CaiLuWaYa18}
Jie Cai, Jiawei Luo, Shulin Wang, and Sheng Yang, \emph{Feature selection in
  machine learning: A new perspective}, Neurocomputing \textbf{300} (2018),
  70--79.

\bibitem[Con19]{Conway19}
John~B Conway, \emph{A course in functional analysis}, vol.~96, Springer, 2019.

\bibitem[Coo07]{Cook07}
R~Dennis Cook, \emph{Fisher lecture: Dimension reduction in regression},
  Statistical Science \textbf{22} (2007), no.~1, 1--26.

\bibitem[CS02]{CuckerSm02}
Felipe Cucker and Steve Smale, \emph{On the mathematical foundations of
  learning}, American Mathematical Society \textbf{39} (2002), no.~1, 1--49.

\bibitem[CSS{\etalchar{+}}07]{CaoShSuYaCh07}
Bin Cao, Dou Shen, Jian-Tao Sun, Qiang Yang, and Zheng Chen, \emph{Feature
  selection in a kernel space}, Proceedings of the 24th International
  Conference on Machine learning, 2007, pp.~121--128.

\bibitem[CSWJ17]{ChenStWaJo17}
Jianbo Chen, Mitchell Stern, Martin~J Wainwright, and Michael~I Jordan,
  \emph{Kernel feature selection via conditional covariance minimization},
  Advances in Neural Information Processing (NIPS), vol.~31, 2017.

\bibitem[DR20]{DasRa20}
Samarendra Das and Shesh~N Rai, \emph{Statistical approach for biologically
  relevant gene selection from high-throughput gene expression data}, Entropy
  \textbf{22} (2020), no.~11, 1205.

\bibitem[FBJ04]{FukumizuBaJo04}
Kenji Fukumizu, Francis~R Bach, and Michael~I Jordan, \emph{Dimensionality
  reduction for supervised learning with reproducing kernel {H}ilbert spaces},
  Journal of Machine Learning Research \textbf{5} (2004), no.~Jan, 73--99.

\bibitem[FBJ09]{FukumizuBaJo09}
\bysame, \emph{Kernel dimension reduction in regression}, Annals of Statistics
  \textbf{37} (2009), no.~4, 1871--1905.

\bibitem[FHT01]{FriedmanHaTi01}
Jerome Friedman, Trevor Hastie, and Robert Tibshirani, \emph{The elements of
  statistical learning}, Springer, New York, 2001.

\bibitem[FM23]{FisherMa23}
Ronald~Aylmer Fisher and Winifred~A Mackenzie, \emph{Studies in crop variation
  {II}. the manurial response of different potato varieties}, The Journal of
  Agricultural Science \textbf{13} (1923), 311--320.

\bibitem[GBSS05]{GrettonBoSmSc05}
Arthur Gretton, Olivier Bousquet, Alex Smola, and Bernhard Sch{\"o}lkopf,
  \emph{Measuring statistical dependence with {H}ilbert-{S}chmidt norms},
  International Conference on Algorithmic Learning Theory, Springer, 2005,
  pp.~63--77.

\bibitem[GC02]{GrandvaletCa02}
Yves Grandvalet and St{\'e}phane Canu, \emph{Adaptive scaling for feature
  selection in {SVM}s}, Advances in Neural Information Processing, vol.~15,
  2002.

\bibitem[GMR{\etalchar{+}}18]{GuidottiMoRuTuGiPe18}
Riccardo Guidotti, Anna Monreale, Salvatore Ruggieri, Franco Turini, Fosca
  Giannotti, and Dino Pedreschi, \emph{A survey of methods for explaining black
  box models}, ACM computing surveys (CSUR) \textbf{51} (2018), no.~5, 1--42.

\bibitem[Gra08]{Grafakos08}
Loukas Grafakos, \emph{Classical fourier analysis}, Springer, 2008.

\bibitem[Hoe94]{Hoeffding94}
Wassily Hoeffding, \emph{Probability inequalities for sums of bounded random
  variables}, The Collected Works of Wassily Hoeffding, Springer, 1994,
  pp.~409--426.

\bibitem[H{\"o}r07]{Hormander07}
Lars H{\"o}rmander, \emph{The analysis of linear partial differential operators
  iii: Pseudo-differential operators}, Springer Science \& Business Media,
  2007.

\bibitem[Jon92]{Jones92}
Lee~K Jones, \emph{A simple lemma on greedy approximation in {H}ilbert space
  and convergence rates for projection pursuit regression and neural network
  training}, Annals of Statistics \textbf{20} (1992), no.~1, 608--613.

\bibitem[LCW{\etalchar{+}}17]{LiChWaMoTrRoTaLi17}
Jundong Li, Kewei Cheng, Suhang Wang, Fred Morstatter, Robert~P Trevino,
  Jiliang Tang, and Huan Liu, \emph{Feature selection: A data perspective}, ACM
  Computing Surveys \textbf{50} (2017), no.~6, 1--45.

\bibitem[Li91]{Li91}
Ker-Chau Li, \emph{Sliced inverse regression for dimension reduction}, Journal
  of the American Statistical Association \textbf{86} (1991), no.~414,
  316--327.

\bibitem[LR20]{LiuRu20}
Keli Liu and Feng Ruan, \emph{A self-penalizing objective function for scalable
  interaction detection}, arXiv preprint arXiv:2011.12215 (2020).

\bibitem[MFD10]{MasaeliFuDy10}
Mahdokht Masaeli, Glenn Fung, and Jennifer~G Dy, \emph{From
  transformation-based dimensionality reduction to feature selection},
  International Conference on Machine Learning, 2010.

\bibitem[Mil19]{Miller19}
Tim Miller, \emph{Explanation in artificial intelligence: Insights from the
  social sciences}, Artificial intelligence \textbf{267} (2019), 1--38.

\bibitem[MSK{\etalchar{+}}19]{MurdochSiKuAbYu19}
W~James Murdoch, Chandan Singh, Karl Kumbier, Reza Abbasi-Asl, and Bin Yu,
  \emph{Interpretable machine learning: definitions, methods, and
  applications}, arXiv preprint arXiv:1901.04592 (2019).

\bibitem[MXZ06]{MicchelliYuZh06}
Charles~A Micchelli, Yuesheng Xu, and Haizhang Zhang, \emph{Universal kernels},
  Journal of Machine Learning Research \textbf{7} (2006), no.~12.

\bibitem[Pan03]{Panchenko03}
Dmitry Panchenko, \emph{Symmetrization approach to concentration inequalities
  for empirical processes}, Annals of Probability (2003), 2068--2081.

\bibitem[Pea14]{Pearl14}
Judea Pearl, \emph{Probabilistic reasoning in intelligent systems: Networks of
  plausible inference}, Elsevier, 2014.

\bibitem[Pis81]{Pisier81}
Gilles Pisier, \emph{Remarques sur un r{\'e}sultat non publi{\'e} de {B}.
  {M}aurey}, S{\'e}minaire Analyse Fonctionnelle (1981), 1--12.

\bibitem[PR16]{Paulsen16}
Vern~I Paulsen and Mrinal Raghupathi, \emph{An introduction to the theory of
  reproducing kernel hilbert spaces}, Cambridge university press, 2016.

\bibitem[Rud19]{Rudin19}
Cynthia Rudin, \emph{Stop explaining black box machine learning models for high
  stakes decisions and use interpretable models instead}, Nature Machine
  Intelligence \textbf{1} (2019), no.~5, 206--215.

\bibitem[RV13]{RudelsonVe13}
Mark Rudelson and Roman Vershynin, \emph{Hanson-wright inequality and
  sub-gaussian concentration}, Electronic Communications in Probability
  \textbf{18} (2013).

\bibitem[SSG{\etalchar{+}}07]{SongSmGrBoBe07}
Le~Song, Alex Smola, Arthur Gretton, Karsten~M Borgwardt, and Justin Bedo,
  \emph{Supervised feature selection via dependence estimation}, Proceedings of
  the 24th International Conference on Machine Learning, 2007, pp.~823--830.

\bibitem[SSG{\etalchar{+}}12]{SongSmGrBeBo12}
Le~Song, Alex Smola, Arthur Gretton, Justin Bedo, and Karsten Borgwardt,
  \emph{Feature selection via dependence maximization}, Journal of Machine
  Learning Research \textbf{13} (2012), no.~5.

\bibitem[SSV12]{SchillingSoVo12}
Ren{\'e}~L Schilling, Renming Song, and Zoran Vondracek, \emph{Bernstein
  functions: Theory and applications}, Walter de Gruyter, 2012.

\bibitem[Ste87]{Stein87}
Michael Stein, \emph{Large sample properties of simulations using latin
  hypercube sampling}, Technometrics \textbf{29} (1987), no.~2, 143--151.

\bibitem[Ver18]{Vershynin18}
Roman Vershynin, \emph{High-dimensional probability: An introduction with
  applications in data science}, Cambridge University Press, 2018.

\bibitem[vH14]{Ramon14}
Ramon van Handel, \emph{Probability in high dimension}, Tech. report, PRINCETON
  UNIV NJ, 2014.

\bibitem[Wai19]{Wainwright19}
Martin~J Wainwright, \emph{High-dimensional statistics: A non-asymptotic
  viewpoint}, Cambridge University Press, 2019.

\bibitem[Wen04]{Wendland04}
Holger Wendland, \emph{Scattered data approximation}, Cambridge university
  press, 2004.

\bibitem[WMC{\etalchar{+}}00]{WestonMuChPoPoVa00}
Jason Weston, Sayan Mukherjee, Olivier Chapelle, Massimiliano Pontil, Tomaso
  Poggio, and Vladimir Vapnik, \emph{Feature selection for {SVM}s}, Advances in
  Neural Information Processing Systems, 2000.

\bibitem[YJS{\etalchar{+}}14]{YamadaJiSiXi14}
Makoto Yamada, Wittawat Jitkrittum, Leonid Sigal, Eric~P Xing, and Masashi
  Sugiyama, \emph{High-dimensional feature selection by feature-wise kernelized
  lasso}, Neural Computation \textbf{26} (2014), no.~1, 185--207.

\end{thebibliography}

%\rfcomment{Let's see how to make it better tomorrow.}

%\newpage

%\appendix

%\input{appendix}

\end{document}